\documentclass{amsart}
\usepackage[margin=1.5in]{geometry}

\usepackage{amsmath,amsthm, latexsym, amssymb,mathrsfs,  tikz-cd}
\usepackage{stackengine,scalerel, quiver, MnSymbol}
\usepackage[colorlinks]{hyperref}
\usepackage[capitalize]{cleveref}
\definecolor{dark-red}{rgb}{0.6,0,0}
\definecolor{dark-green}{rgb}{0,0.4,0}
\definecolor{medium-blue}{rgb}{0,0,0.5}
\hypersetup{
    colorlinks, linkcolor={dark-red},
    citecolor={dark-green}, urlcolor={medium-blue}
}

\newcommand{\Rep}{\mathrm{Rep}}

\newcommand{\BB}{\mr{BB}}
\newcommand{\Aut}{\mr{Aut}}
\newcommand{\Fil}{\mr{Fil}}
\newcommand{\an}{\mr{an}}

\newcommand{\gr}{\mr{gr}}
\newcommand{\Gr}{\mathrm{Gr}}
\newcommand{\Id}{\mr{Id}}
\newcommand{\Fl}{\mr{Fl}}
\newcommand{\Ad}{\mathrm{Ad}}
\newcommand{\cyc}{\mathrm{cyc}}
\newcommand{\Spd}{\mathrm{Spd}}
\newcommand{\GL}{\mathrm{GL}}
\newcommand{\HT}{\mr{HT}}

\newcommand{\mc}[1]{\mathcal{#1}}
\newcommand{\mbb}[1]{\mathbb{#1}}
\newcommand{\mr}[1]{\mathrm{#1}}
\newcommand{\mf}[1]{\mathfrak{#1}}

\newcommand{\et}{\mathrm{\acute{e}t}}

\newcommand{\dR}{\mathrm{dR}}
\newcommand{\BC}{\mathrm{BC}}

\newcommand{\Perf}{\mathrm{Perf}}
\newcommand{\adm}{\mathrm{adm}}

\newcommand{\AdicSpaces}{\mr{AdicSpc}}

\newcommand{\Sm}{\mr{Sm}}

\newcommand{\LZ}{\mathrm{LZ}}

\newcommand{\VB}{\mathrm{Vect}}
\newcommand{\Sch}{\mathrm{Sch}}

\newcommand{\univ}{\mathrm{univ}}

\newcommand{\Spa}{\mathrm{Spa}}
 
\newcommand{\Isom}{\mathrm{Isom}}

\newcommand{\FF}{\mathrm{FF}}
\DeclareMathOperator{\Lie}{Lie}
\DeclareMathOperator{\Spec}{Spec}

\DeclareMathOperator{\Gal}{Gal}

\DeclareMathOperator{\Hom}{Hom}

\newcommand{\ul}[1]{\underline{#1}}

\newcommand{\Hdg}{\mr{Hdg}}
\newcommand{\BG}{\mathrm{B}G}
\newcommand{\Sh}{\mathrm{Sh}}
\newcommand{\qpbar}{\overline{\mathbb{Q}}_p}
\newcommand{\gx}{(\mathsf{G},\mathsf{X})}

\newcommand{\std}{\mathrm{std}}
\newcommand{\End}{\mr{End}}
\numberwithin{equation}{subsection}
\numberwithin{equation}{subsubsection}

\theoremstyle{plain}

\newtheorem{maintheorem}{Theorem}

\newtheorem*{theorem*}{Theorem}

\newtheorem{theorem}[subsubsection]{Theorem}
\newtheorem{corollary}[subsubsection]{Corollary}

\newtheorem{proposition}[subsubsection]{Proposition}
\newtheorem{lemma}[subsubsection]{Lemma}

\theoremstyle{definition}

\newtheorem{example}[subsubsection]{Example}

\newtheorem{definition}[subsubsection]{Definition}
\newtheorem{remark}[subsubsection]{Remark}

\crefformat{section}{\S#2#1#3}
\crefformat{subsection}{\S#2#1#3}
\crefformat{subsubsection}{\S#2#1#3}

\newcommand{\triv}{\mathrm{triv}}

\newcommand{\qpbreve}{\breve{\mathbb{Q}}_p}

\newcommand{\alg}{\mr{alg}}
\newcommand{\AffPerf}{\mr{AffPerf}}

\newcommand{\red}{\mathrm{red}}

\newcommand{\M}{\mathscr{M}}
\newcommand{\Tw}{\mathrm{Tw}}\newcommand{\RH}{\mathcal{RH}}
\newcommand{\Igs}{\mathrm{Igs}}
\newcommand{\BL}{\mathrm{BL}}
\newcommand{\Bun}{\mathrm{Bun}}

\newcommand{\lfp}{{\mathrm{lf}^{+}}}

\newcommand{\cn}{\mathrm{cn}}

\newcommand{\lf}{\mathrm{lf}}
\newcommand{\lfid}{{\diamond_{\lf}}}

\newcommand{\Fqbar}{\overline{\mathbb{F}}_q}
\newcommand{\goodAS}{\mathrm{SSAdicSpc}}
\newcommand{\bs}{\backslash}

\title{Inscription, twistors, and $p$-adic periods}

\author{Sean Howe}

\begin{document}

\begin{abstract} We introduce the theory of inscribed $v$-sheaves, a differentiable extension of the theory of diamonds and $v$-sheaves with internal tangent bundles that are often relative inscribed Banach--Colmez spaces, then apply this theory to the study of $p$-adic periods. In particular, we construct natural inscribed versions of the Hodge and Hodge--Tate period maps and their lattice refinements for de Rham torsors, then compute the derivatives of these period maps in terms of classical structures in $p$-adic Hodge theory. These torsors include infinite level global Shimura varieties and infinite level local Shimura varieties, and for these spaces we also give another moduli-theoretic construction of the inscribed structure; the construction in the local Shimura case applies more generally to the non-minuscule moduli of mixed characterisic local shtukas with one leg. The key new ingredients in our study of inscribed structures on $p$-adic Lie group torsors over smooth rigid varieties over a $p$-adic field are the Liu-Zhu period map, a refinement of the Hodge period map whose derivative is the geometric Sen morphism/canonical Higgs field, and a closely related exact tensor functor from $\mathbb{Q}_p$-local systems to a category of twistor bundles on the relative thickened Fargues--Fontaine curve. These new structures are only visible after passing to the inscribed setting.  We also discuss some possible implications of our computations in the vein of ``differential topology for diamonds." 
\end{abstract}

\maketitle

\let\thefootnote\relax\footnote{MSC 2020 --- 	14F10, 14F30, 14G45}

\tableofcontents

\section{Introduction}

\subsection{Motivation}\label{ss.motivation}
Over the past decade, the fundamental building blocks of $p$-adic geometry have shifted from those of rigid analytic geometry, Tate's Noetherian convergent power series rings, to those of perfectoid geometry, Scholze's typically non-Noetherian perfectoid algebras (characterized by the existence of approximate $p$th roots). This shift has increased the scope and power of the theory, even in the study of rigid analytic varieties and their cohomology. For example, by building from perfectoid algebras to the theory of diamonds and $v$-sheaves, one obtains new period maps in $p$-adic Hodge theory whose domain and/or codomain does not exist in the classical theory. However, this shift comes at a price: by moving away from the familiar convergent power series rings, we lose access to the differentiable toolkit of tangent spaces and derivatives that is so fundamental in classical geometric reasoning. 

\begin{example}\label{example.intro-annihilating-differentials}
Let $C/\mbb{Q}_p$ be an algebraically closed non-archimedean extension. Then $C$ is a perfectoid field, and a basic example of a perfectoid $C$-algebra is the completion $C\langle t^{\pm1/p^\infty}\rangle $ of $\bigcup_{n \geq 1} C[t^{\pm1/p^n}]$ for the supremum norm on the coefficients. Geometrically, this is the ring of functions on the perfectoid annulus, a Galois $\mbb{Z}_p(1)$-cover of the rigid analytic annulus $|t|=1 \subseteq \mathbb{A}^1$ whose ring of functions is $C\langle t^{\pm1}\rangle$. The module of continuous K\"ahler differentials of $C\langle t^{1/p^\infty}\rangle $ over $C$ is trivial: any continuous derivation $d$ from $C\langle t^{1/p^\infty}\rangle$ to a Banach $C\langle t^{1/p^\infty}\rangle$-module is zero because of the formula $p^n d\log t^{1/p^n}=d\log t$.
\end{example}

This phenomenon is general: the existence of approximate $p$-power roots forces any continuous derivation on a perfectoid algebra to be identically zero. Thus, to obtain a broadly applicable differential theory for perfectoid spaces, one cannot proceed directly via the K\"ahler approach as in classical rigid analytic, complex analytic, or algebraic geometry. This issue propagates more broadly to the theory of \emph{diamonds}, which are functors on perfectoid spaces constructed as quotients of representable functors by pro-\'{e}tale equivalence relations, as well as the more general \emph{$v$-sheaves} and \emph{$v$-stacks} that arise naturally in moduli problems related to $p$-adic cohomology and $p$-adic Hodge theory. 

Nevertheless, it is clear that some diamonds and $v$-sheaves do have natural tangent bundles. For example, if $L/\mbb{Q}_p$ is a non-archimedean extension, then smooth rigid analytic varieties over $L$ embed, by their functor of points on perfectoid algebras over $L$, fully faithfully into the category of diamonds over $\Spd L$, and certainly we know how to define the tangent bundle of a smooth rigid analytic variety. More recently, Fargues and Scholze \cite{FarguesScholze.GeometrizationOfTheLocalLanglandsCorrespondence}, in the context of their Jacobian criterion for cohomological smoothness, have defined Banach--Colmez Tangent Bundles for moduli spaces of sections of smooth quasi-projective adic spaces over Fargues--Fontaine curves. Here, the Tangent Bundle is a Vector Bundle (we follow the precepts of Colmez Capitalization in our nomenclature), i.e. a sheaf of modules for the topological constant sheaf $\mathbb{Q}_p^\diamond$ (which is usually denoted $\ul{\mbb{Q}_p}$). 

The same diamond can sometimes be constructed in different ways as a moduli of sections, leading to distinct Tangent Bundles. Thus the Tangent Bundle assigned by such a construction is not intrinsic to the diamond, and can be thought of as a type of additional differentiable structure akin to a differentiable manifold structure on a topological space. 

\begin{example}\label{example.intro-diff-structures}
Let $C$ be an algebraically closed perfectoid field in characteristic $p$. Then, the moduli of sections of the vector bundle $\mc{O}(1/n)$ on the Fargues--Fontaine curve $X_C/\mbb{Q}_p$ is represented by the open perfectoid unit disk over $C$. The Tangent Bundle assigned to it by the Fargues--Scholze construction is the constant Banach--Colmez space $\BC(\mc{O}(1/n)).$ For varying $n$, these are distinct as Vector Bundles on the open perfectoid unit disk.
\end{example}

\subsection{Inscribed $v$-sheaves}\label{ss.intro-inscribed}
The first purpose of the present work is to introduce a robust differentiable structure from which such Tangent Bundles arise naturally. This is the theory of inscribed\footnote{The nomenclature is by analogy with an inscription on a rock-theoretic diamond, which is a piece of extra identifying information laser-etched along the edge at the widest part.} diamonds and $v$-sheaves, where the objects are not functors on perfectoid spaces but instead functors on certain finite locally free thickenings of Fargues--Fontaine curves.

Concretely, for the purposes of the introduction, a test object is an object in the category $X_{\mbb{Q}_p,\Box}^\lfp$ consisting of pairs\footnote{Our notation in the introduction differs slightly from the notation in the body, where we will typically write $P$ for a perfectoid space in characteristic $p$ and $P^\sharp$ for an untilt over $\Spa \mbb{Q}_p$ to allow for more general situations where no untilt has been fixed, i.e. where there is no fixed map $P \rightarrow \Spd\mbb{Q}_p$.} $(P, \mathcal{X}/X_{P^\flat})$, where $P/\Spa \mbb{Q}_p$ is a perfectoid space, $X_{P^\flat}=X_{\mathbb{Q}_p, P^\flat}$ is the associated Fargues--Fontaine curve (an adic space over $\mathbb{Q}_p$), and $\mc{X}/X_{P^\flat}$ is defined by a locally free $\mathcal{O}_{X_{P^\flat}}$-algebra $\mathcal{O}_{\mc{X}}$ such that $\mathcal{O}_{\mc{X}}^\red=\mathcal{O}_{X_{P^\flat}}$. We also require that, as a vector bundle on $X_{P^\flat}$, $\mathcal{O}_{\mc{X}}$ satisfies a slope condition  (see \cref{ss.slope-condition}) that is imposed to ensure that the basic strata in the inscribed moduli of $G$-bundles are deformation-theoretically open (\cref{prop.b-basic-open-stratum}). We will write test objects as $\mc{X}/X_{P^\flat}$ or simply as $\mc{X}$. There is a canonical Cartier divisor $\overline{\infty}: P \hookrightarrow X_{P^\flat}$; we extend this to $\infty: \overline{\infty}^* \mc{X} \hookrightarrow \mc{X}.$ The $v$-topology on the set of all test objects is pulled back from the $v$-topology on perfectoid spaces. 

These thickenings of Fargues--Fontaine curves play a role similar to the artinian rings used in the deformation theory of Galois representations --- in fact, because of the slope condition, they can be thought of as living inside of a category of relative ``artinian effective Banach--Colmez algebras over $\mathbb{Q}_p$ with residue field $\mbb{Q}_p$."

\begin{remark} We use finite locally free thickenings because they are the largest category for which we can easily establish a GAGA equivalence between thickenings of the algebraic Fargues--Fontaine curve and thickenings of the adic Fargues--Fontaine curve. However, we have tried to set up the theory to be compatible with the eventual use of other larger categories of thickenings. In particular, we expect that many of our computations of tangent bundles and derivatives will extend to a category of (derived, analytic) thickenings including all artinian Banach--Colmez algebras over $\mathbb{Q}_p$ with residue field $\mathbb{Q}_p$. Such an extension should allow us to make contact between our computations and the theory of analytic prismatization that has been announced by Ansch\"{u}tz--le Bras--Rodriguez Camargo--Scholze.
\end{remark}

\begin{remark}\label{remark.unique-structure-morphism}
Our thickenings $\mc{X}/X_{P^\flat}$ are, by definition, equipped with a structure morphism back to $X_{P^\flat}$ (the thickening is the map $X_{P^\flat} \hookrightarrow \mc{X}$ given by $\mathcal{O}_{\mc{X}} \twoheadrightarrow \mc{O}_{\mc{X}}^\red=\mc{O}_{X_{P^\flat}}$). This structure morphism can be viewed as a property of a thickening rather than as an additional choice of data: if such a structure morphism exists then it is unique because there are no continuous derivations on a basis of affinoid opens of $X_{P^\flat}$ (which is a preperfectoid space). 
\end{remark}

\subsubsection{} An inscribed $v$-sheaf is a $v$-sheaf on this category of test objects that sends certain coproducts to products (see \cref{def.inscribed-fc}).  For any inscribed $v$-sheaf $\mc{S}$, we define the underlying $v$-sheaf  $\overline{\mc{S}}/\Spd \mbb{Q}_p$ by $\overline{\mc{S}}(P)=\mc{S}(X_{P^\flat}/X_{P^\flat})$. In the other direction, we can and do view $v$-sheaves over $\Spd \mbb{Q}_p$ as the full sub-category of inscribed $v$-sheaves consisting of ``trivially inscribed $v$-sheaves": for a $v$-sheaf $S/\Spd \mbb{Q}_p$, we set $S(\mc{X}/X_{P^\flat})=S(P)$. 

\begin{remark}\label{remark.deRhamStack}In particular, for an inscribed $v$-sheaf $\mc{S}$, if we view $\overline{\mc{S}}$ as a trivially inscribed $v$-sheaf then it is defined by
$\overline{\mc{S}}(\mc{X}/X_{P^\flat})=\mc{S}(X_{P^\flat}).$
Because $X_{P^\flat}=\mc{X}^\red$, $\overline{\mc{S}}$ can be viewed as the de Rham stack of $\mc{S}$. Thus, \cref{remark.unique-structure-morphism} implies that we are in an unusual setting where the usual map $\mc{S} \rightarrow \overline{\mc{S}}$ (induced by the thickening map $X_{P^\flat} \hookrightarrow \mc{X}$) from an object to its de Rham stack admits a canonical section $\overline{\mc{S}} \rightarrow \mc{S}$ (induced by the structure map $\mc{X} \rightarrow X_{P^\flat}$). This reflects that our category is rather flabby, so that, e.g., the functor from smooth rigid analytic varieties to inscribed $v$-sheaves is faithful but not full. This flabbiness is actually very useful at times: for example, it allows us (in \cref{ss.hodge-period-maps}) to define a global Hodge period map for a filtered vector bundles with integrable connection measuring the position of the filtration against bases of flat sections, which does not typically exist rigid analytically.
\end{remark}

\subsubsection{}The inscribed formalism allows us to extract the Tangent Bundles discussed in \cref{ss.motivation} from a  construction of internal tangent bundles by Weil restriction: for $\mc{S}$ an inscribed $v$-sheaf, the tangent bundle $T_\mc{S}$ is obtained by ``adding an $\epsilon$":
\[ T_{\mc{S}}(\mc{X})=\mc{S}(\mc{X}[\epsilon]) \textrm{ where } \mc{O}_{\mc{X}[\epsilon]}:=\mc{O}_{\mc{X}}[\epsilon]/\epsilon^2. \]
The tangent bundle $T_{\mc{S}}$ is itself an inscribed $v$-sheaf with a canonical map to $\mc{S}$, and $T_{\mc{S}}/\mc{S}$ is a module over $\mathbb{Q}_p^\lfid$, where the latter is defined by
 \[ \mathbb{Q}_p^\lfid(\mc{X})=H^0(\mc{X}, \mc{O}_{\mc{X}})=\End(\mc{X}[\epsilon]/\mc{X}).\] 
 The addition law on $T_{\mc{S}}$ comes from the coproduct condition. 

In particular, $\overline{\mathbb{Q}_p^\lfid}=\mathbb{Q}_p^\diamond$, and thus $\overline{T_\mc{S}}$ is a $\mathbb{Q}_p^\diamond$-module over $\overline{\mc{S}}$ so can be thought of as a Tangent Bundle to $\overline{\mc{S}}$ in the sense of \cref{ss.motivation}. This theory encompasses naturally the tangent bundles of rigid analytic varieties: for a smooth rigid analytic variety $Z/L$, we set \[ Z^\lfid(\mc{X}/X_{P^\flat})=\Hom_{\Spa \mbb{Q}_p}(\overline{\infty}^* \mc{X}, Z) \] (which lives over $\Spd L$), and then $T_{Z^\lfid}=(T_Z)^\lfid$ (treating $T_Z$ itself as a smooth rigid analytic space over $L$) so that $\overline{T_{Z^\lfid}} \rightarrow \overline{Z^\lfid}$ is $(T_Z)^\diamond \rightarrow Z^\diamond$.

A similar construction gives the Tangent Bundles of the Fargues--Scholze Jacobian criterion by starting with a smooth morphism $Z \rightarrow X_{P^\flat}$ and considering the sections functor $Z^\lfid/P^\diamond$ which sends $\mc{X}/X_Q, Q \rightarrow P$ to 
$\Hom_{X_{P^\flat}} (\mc{X}, Z)$ (see \cref{example.smooth-rig-and-fs} for details and a refinement). We also treat similar moduli of sections constructions for smooth spaces over other loci in the Fargues--Fontaine curve, and much of the power of the formalism comes from the fact that all of these different constructions now produce differential objects in the same category, so that one can easily study the natural maps between them.  

\subsection{Main results}\label{ss.intro-main-results}
The main goal of this theory to provide a setting for applying differential techniques to the study of the period maps that arise when considering $p$-adic cohomology and its variation in families. We now describe our results in this direction. 

Let $L$ be a $p$-adic field, that is, a complete non-archimedean extension of $\mathbb{Q}_p$ with discrete value group and perfect residue field. Suppose $Z/L$ is a smooth rigid analytic variety, $G/\mathbb{Q}_p$ is a connected linear algebraic group, $K \leq G(\mathbb{Q}_p)$ is an open subgroup, and  $\tilde{Z}/Z$ is a $K$-torsor (on the pro-\'{e}tale site or, equivalently for such a torsor, the $v$-site). We have seen that there is a natural inscribed version of $Z$, $Z^\lfid$, and the $p$-adic manifold $K$ also has an inscribed incarnation, which can be constructed concretely in this case as 
\[ K^\lfid=G(\mathbb{Q}_p^\lfid) \times_{G(\mathbb{Q}_p^\diamond)} K^\diamond \textrm{ (for $K^\diamond:=\underline{K}$ the topological constant sheaf).}\]

Just as $T_{Z^\lfid}=(T_Z)^\lfid$, where $T_Z$ is the rigid analytic tangent bundle, $T_{K^\lfid}= (T_{K})^\lfid$, where $T_K$ is the $p$-adic manifold tangent bundle. Concretely, for $\mf{k}:=\Lie K = \Lie G(\mathbb{Q}_p)=:\mf{g}$, this means that $T_{K^\lfid}$ is identified with the constant bundle $\mathfrak{k} \otimes_{\mathbb{Q}_p} \mathbb{Q}_p^\lfid$ by viewing elements of $\mathfrak{k}$ as left-invariant vector fields. Our main results allow us to construct of a non-trivial inscribed structure $\tilde{Z}^\lfid$ on $\tilde{Z}$ compatible with the inscribed structures on $Z$ and $K$ and an explicit computation of $T_{Z^\lfid}$ as well as the derivatives of associated period maps.

In particular, a key motivation for us to develop this framework was the desire to differentiate Hodge--Tate period maps and their lattice refinements. However, it turns out that the key tool in constructing the inscribed structure on $\tilde{Z}$ is a novel study of the \emph{Hodge} period map, which lies on the other side of $p$-adic comparison theorems, in the inscribed setting. 

\subsubsection{} In fact, the most important period map for us is the Liu-Zhu period map, a generalization of the Hodge period map. We call it the Liu-Zhu period map because it is defined starting from the category of $t$-connections on the ``formal lift of $Z_C$ to $B^+_\dR$" of \cite{LiuZhu.RigidityAndARiemannHilbertCorrespondenceForpAdicLocalSystems} where  $C:=\overline{L}^\wedge$ and $B^+_\dR:=\mathbb{B}^+_\dR(C)$ is the associated Fontaine period ring. Starting from $\tilde{Z}$, the Riemann-Hilbert correspondence $\mathbb{L} \mapsto \RH(\mathbb{L})$ of \cite{LiuZhu.RigidityAndARiemannHilbertCorrespondenceForpAdicLocalSystems} can be used to construct such a $t$-connection with $G$-structure, and from it we construct a new period map associated to $\tilde{Z}$,
\[ \pi_{\LZ}: Z^\lfid \rightarrow \Gr_{G} / G(\overline{\mathbb{B}_\dR}). \]
We explain the notation: for period sheaves in $p$-adic Hodge theory such as $\mathbb{B}^+_\dR$, $\mathbb{B}_\dR$, $\mathcal{O}$, etc., we write the same symbol for the ``natural" extension to inscribed $v$-sheaves, and the symbol with an overline for the ``trivial" extension to inscribed $v$-sheaves. The natural extension is obtained by extending an interpretation of the period sheaf in terms of Fargues--Fontaine curves to thickened Fargues--Fontaine curves: for example, the natural extension $\mathbb{B}^+_\dR$ evaluated on a thickening $\mathcal{X}/X$ of a Fargues--Fontaine curve $X$ is given by the functions on the formal neighborhood in $\mathcal{X}$ of the canonical divisor $\infty$; the trivial extension $\overline{\mathbb{B}^+_\dR}$ evaluated on $\mathcal{X}$ returns instead the functions on the formal neighborhood in $X$ of the canonical divisor $\overline{\infty}$. As noted in \cref{ss.intro-inscribed}, trivial inscriptions exist for all $v$-sheaves but, for period sheaves and moduli spaces in $p$-adic Hodge theory such as the $\mathbb{B}^+_\dR$-affine Grassmannian $\Gr_G$, we show that there are also natural extensions that contain interesting nilpotent information. 

The period map $\pi_{\LZ}$ is only interesting in the inscribed setting: on underlying $v$-sheaves, it factors through the classifying stack $\ast / G(\overline{\mathbb{B}_\dR})$, and is simply the classifying map on $Z^\diamond$ for the push-out $\tilde{Z} \times^{K^\diamond} G(\overline{\mathbb{B}_\dR})$ of $\tilde{Z}$ to a $G(\overline{\mathbb{B}_\dR})$-torsor. On the other hand, the derivative $d\pi_{\LZ}$ is highly non-trivial: its restriction to underlying $v$-sheaves is the geometric Sen morphism / canonical Higgs field of Pan \cite{Pan.OnLocallyAnalyticVectorsOfTheCompletedCohomologyOfModularCurves, Pan.OnLocallyAnalyticVectorsOfTheCompletedCohomologyOfModularCurvesII} and Rodriguez Camargo \cite{RodriguezCamargo.GeometricSenTheoryOverRigidAnalyticSpaces} (see \cref{remark.geometric-sen}). 

\subsubsection{}This interpretation of $d\pi_{\LZ}$ as the canonical Higgs field is largely irrelevant for the present work; what is more important is to understand the relation between $d\pi_{\LZ}$ and the Hodge period map when $\tilde{Z}/Z$ is a de Rham torsor in the sense of \cite{Scholze.pAdicHodgeTheoryForRigidAnalyticVarieties} --- we recall that a torsor is de Rham, for example, if it is the trivializing torsor for the cohomology of a smooth proper family over $Z$;  in fact, by the results of \cite{LiuZhu.RigidityAndARiemannHilbertCorrespondenceForpAdicLocalSystems}, it suffices to check the de Rham condition after push-out by a single faithful representation of $G$ and at a single classical point in each connected component of $Z$, where it is the usual de Rham condition of Fontaine.

In the de Rham case, the comparison of \cite{Scholze.pAdicHodgeTheoryForRigidAnalyticVarieties} produces an associated filtered $G$-bundle with integrable connection satisfying Griffiths transversality $\omega_{\nabla}^\Fil$ on $Z$ (which, in particular, determines the $t$-connection in the Liu-Zhu Riemann-Hilbert correspondence).  From $\omega_{\nabla}^\Fil$ we produce an associated Hodge period map 
\[ \pi_{\Hdg}: Z^\lfid \rightarrow \Fl_{G}^\lfid/G(\overline{\mathcal{O}}), \]
where here $\Fl_G$ is the rigid analytic variety parameterizing filtrations on the trivial $G$-torsor. This period map measures the position of the Hodge filtration against a basis of flat sections; it is a useful feature of the inscribed formalism that this makes sense globally even though a basis of flat sections does not typically exist on opens that cover the rigid analytic variety $Z$ so that there is usually no such global rigid analytic Hodge period map. 

We also produce a lattice refinement $\pi_\Hdg^+: Z^\lfid \rightarrow \Gr_{G}^\lfid/G(\overline{\mathbb{B}^+_\dR}).$ It refines $\pi_\Hdg$ in the following sense: if $\pi_{\Hdg}$ factors through a connected component $\Fl_{[\mu^{-1}]}/G(\overline{\mathcal{O}})$ for $[\mu^{-1}]$ a conjugacy class of cocharacters of $G_{\overline{\mbb{Q}}_p}$ defined over $L$ (this holds, e.g., if $Z$ is geometrically connected), then $\pi_{\Hdg}^+$ factors through an associated Schubert cell $\Gr_{[\mu]} \subseteq \Gr_G$, and $\pi_{\Hdg}^+$ is a lift of $\pi_\Hdg$ along the Bialynicki-Birula map $\BB: \Gr_{[\mu]} \rightarrow \Fl_{[\mu^{-1}]}^\lfid$ which sends a $\mathbb{B}^+_\dR$-lattice on the trivial $G(\mathbb{B}^+_\dR)$-torsor to its trace filtration on the trivial $G(\mathcal{O})$-torsor. In this case, the relation between $\pi_{\Hdg}$, $\pi_{\Hdg}^+$, and $\pi_{\LZ}$ is encapsulated in the commutative diagram
% https://q.uiver.app/#q=WzAsNCxbMCwwLCIgWl57XFxsZmlkfSAiXSxbMSwxLCJcXEdyX3tbXFxtdV19L0coXFxvdmVybGluZXtcXG1hdGhiYntCfV4rX1xcZFJ9KSJdLFsxLDMsIlxcRmxfe1tcXG11XnstMX1dfS9HKFxcb3ZlcmxpbmV7XFxtYXRoY2Fse099fSkiXSxbMiwxLCJcXEdyX0cvRyhcXG92ZXJsaW5le1xcbWF0aGJie0J9X1xcZFJ9KSJdLFswLDEsIlxccGleK19cXEhkZyIsMV0sWzAsMiwiXFxwaV97XFxIZGd9IiwxLHsiY3VydmUiOjJ9XSxbMSwyLCJcXEJCIiwxXSxbMCwzLCJcXHBpX3tcXExafSIsMSx7ImN1cnZlIjotM31dLFsxLDNdXQ==
\[\begin{tikzcd}
	{ Z^{\lfid} } \\
	& {\Gr_{[\mu]}/G(\overline{\mathbb{B}^+_\dR})} & {\Gr_G/G(\overline{\mathbb{B}_\dR})} \\
	\\
	& {\Fl_{[\mu^{-1}]}/G(\overline{\mathcal{O}})}
	\arrow["{\pi^+_\Hdg}"{description}, from=1-1, to=2-2]
	\arrow["{\pi_{\LZ}}"{description}, curve={height=-18pt}, from=1-1, to=2-3]
	\arrow["{\pi_{\Hdg}}"{description}, curve={height=12pt}, from=1-1, to=4-2]
	\arrow[from=2-2, to=2-3]
	\arrow["\BB"{description}, from=2-2, to=4-2]
\end{tikzcd}\]
Moreover, by an essentially classical computation (see \cref{lemma.dpihdg-ks}), 
\[ d\pi_{\Hdg}: T_{Z^\lfid} \rightarrow \pi_{\Hdg}^* T_{\Fl_{[\mu^{-1}]}/G(\overline{\mathcal{O}})}=\omega_{\nabla}^\Fil(\mf{g})/\Fil^0\omega_{\nabla}^\Fil(\mf{g}) \]
is the Kodaira-Spencer map for $\omega_{\nabla}^\Fil$. In particular, $d\pi_{\Hdg}^+=d\pi_{\LZ}$ is a lift of the Kodaira-Spencer map along $d\BB$; on the underlying $v$-sheaves there is a unique such lift and it can also be described using the usual Hodge--Tate comparison theorem by \cite[Theorem A]{Howe.GeometricSenAndKodairaSpencer}. 
\begin{remark}\label{remark.griffiths-transversality}We define a period map $\pi_\Hdg$ for any filtered $G$-bundle with integrable connection on $Z$, and lift it to a canonical map $\pi_\Hdg^+$ whenever the filtration and connection satisfy Griffiths transversality (which is the statement that the Kodaira-Spencer map $d\pi_\Hdg$ factors through $\gr^{-1}\omega_{\nabla}^\Fil(\mf{g})=\Fil^{-1}\omega_{\nabla}^\Fil(\mf{g})/\Fil^0\omega_{\nabla}^\Fil(\mf{g})$). Conversely, because the derivatives $d\BB$, $d\pi_{\Hdg}$, and $d\pi_{\Hdg}^+=d\pi_{\LZ}$ are all $\mathbb{B}^+_\dR$-linear, a computation of $d\BB$ shows that the existence of a factorization of $ \pi_{\Hdg}$ through $\BB$  necessitates Griffiths transversality!
\end{remark}

\subsubsection{}The key point in all of these constructions is an adaptation and extension of Scholze's \cite[\S6]{Scholze.pAdicHodgeTheoryForRigidAnalyticVarieties} functor $\mathbb{M}$ from filtered vector bundles with integrable connection satisfying Griffiths transversality on $Z$ to locally free $\mathbb{B}^+_\dR$-modules on $Z^\diamond$. We produce an enrichment of this functor that is valued in locally free $\mathbb{B}^+_\dR$-modules on $Z^\lfid$. Crucially, when the filtration is not flat, neither is this $\mathbb{B}^+_\dR$-module: our $\mathbb{M}$ is \emph{not} the flat extension of Scholze's $\mathbb{M}$. 

To define our $\mathbb{M}$, we use a geometrization of the $\mathcal{O}\mathbb{B}_\dR$-comparison via the inscribed $\mathbb{B}^+_\dR$-jet sheaf $Z_{(\infty)}^\lfid$. In particular, a filtered vector bundle on $Z$ (or, more generally, a Galois-equivariant integrable $t$-connection on the formal lift of $Z_C$ to $B^+_\dR$) can be pulled back to produce a locally free $\mathbb{B}^+_\dR$-module on $Z_{(\infty)}^\lfid$, and then Griffiths transversality of the connection is precisely what is needed to descend this module to $Z^\lfid$ --- indeed, because the filtrations on $\mathbb{B}_\dR$ and $V$ are convolved, two jets with the same image in $Z^\lfid$ are congruent modulo $\Fil^1 \mathbb{B}^+_\dR$ so that, in the Taylor series for integrating the connection, the decrease in filtration from the partial derivatives is canceled out by the powers of the variables.  

\subsubsection{} The Liu-Zhu period map measures the position of $\mathbb{M}$ against bases of flat sections for $\mathbb{M} \otimes_{\mathbb{B}^+_\dR} \mathbb{B}_\dR$ (which \emph{is} flat), while the lattice Hodge period map measures the position of $\mathbb{M}$ against bases of flat sections for the lattice $\mathbb{M}_0 \subseteq \mathbb{M} \otimes_{\mathbb{B}^+_\dR} \mathbb{B}_\dR$ obtained from the same vector bundle with integrable connection but equipped with the trivial filtration --- unlike $\mathbb{M}$, $\mathbb{M}_0$ is flat since the trivial filtration is flat for any connection. The construction of $\mathbb{M}_0$ has no analog for general $t$-connections, which is why we only see $\pi_{\LZ}$ in the general setting. 

\subsubsection{} The consideration of $\pi_{\LZ}$ and $\mathbb{M}$ leads to a new structure in relative $p$-adic Hodge theory that is the key to our inscribed structure on $\tilde{Z}$: first recall that the category of \'{e}tale $\mathbb{Q}_p$-local systems on $Z$ is equivalent to the category of vector bundles on the relative Fargues--Fontaine curve $Z^\diamond$ that are fiberwise semistable of slope zero; we denote this equivalence by $\mathbb{L} \mapsto \mathbb{L} \otimes_{\mathbb{Q}_p^\diamond} {\mathcal{O}_{X}}$. Given such a vector bundle, we can extend it to a flat vector bundle $\mathbb{L} \otimes_{\mathbb{Q}_p^\diamond} {\mathcal{O}_{\mathcal{\mathcal{X}}}}$ on the relative thickened Fargues--Fontaine curve over $Z^\lfid$, and then modify it at the canonical point $\infty$ using the lattice $\mathbb{M}(\RH(\mbb{L}))$ to obtain a vector bundle on the relative thickened Fargues--Fontaine curve over $Z^\lfid$ with a meromomorphic integrable connection:
\[ \Tw(\mathbb{L}) := (\mathbb{L} \otimes_{\mathbb{Q}_p^\diamond} {\mathcal{O}_{\mathcal{\mathcal{X}}}})_{\mathbb{M}(\RH(\mathbb{L}))}. \]
The meromorphic integrable connection has a simple pole at $\infty$; we call such an object a twistor on the relative thickened Fargues--Fontaine curve over $Z^\lfid$ (see \cref{def.mer-conn-and-twistor}), by analogy with Simpson's variations of twistor structure \cite{Simpson.MixedTwistorStructures}. 

\begin{maintheorem}[see \cref{theorem.twistor-functor}]\label{maintheorem.twistor-correspondence}Let $L$ be a $p$-adic field and let $Z/L$ be a smooth rigid analytic variety. The assignment $\mathbb{L} \rightarrow \Tw(\mathbb{L})$ is a fully faithful exact tensor functor from \'{e}tale $\mathbb{Q}_p$-local systems on $Z$ to twistors on the relative thickened Fargues--Fontaine curve over $Z^\lfid$.
\end{maintheorem}

\begin{remark}The existence of an exact tensor functor from \'{e}tale $\mathbb{Q}_p$-local systems on $Z$ to some type of variation of twistor structures over $Z$ was conjectured by Fargues and Liu-Zhu \cite{LiuZhu.RigidityAndARiemannHilbertCorrespondenceForpAdicLocalSystems}. \cref{maintheorem.twistor-correspondence} is related to this conjecture but, as explained in \cref{Remark.FarguesLiuZhuDiscussion}, is probably not quite what they had in mind!
\end{remark}

\begin{remark}As explained in \cref{remark.extensions-of-twistor}, the de Rham case of the functor $\Tw$ and \cref{maintheorem.twistor-correspondence} can be extended from de Rham local systems to general de Rham bundles on the relative Fargues--Fontaine curve. In the case where $L$ is not a $p$-adic field (e.g. $L=C$), it should be straightforward to obtain a generalization to nilpotent bundles on the relative Fargues--Fontaine curve (here for $\mc{V}$ such a vector bundle, nilpotent refers to nilpotence of the geometric Sen operator on the $\overline{\mc{O}}$-module $\overline{\infty}^* \mc{V}$, which always holds when $L$ is $p$-adic), and in fact we expect a version to hold even without this nilpotence condition.   
\end{remark}

\subsubsection{}Applying the functor $\Tw$ of \cref{maintheorem.twistor-correspondence} allows us to produce from $\tilde{Z}$ a $G$-twistor on the relative thickened Fargues--Fontaine curve over $Z^\lfid$, and then we define $\tilde{Z}^\lfid$ as an open inside of the trivializing $G(\mathbb{Q}_p^\lfid)$-torsor for the underlying $G$-bundle of this $G$-twistor (see \cref{ss.inscribed-torsors}). This construction is functorial in the defining data $(Z, G, K \subseteq G(\mathbb{Q}_p), \tilde{Z})$.

\begin{remark}We expect that $\tilde{Z}^\lfid$ is functorial in $(Z,K, \tilde{Z})$ where $K$ is treated as a $p$-adic Lie group, and even functorial in $\tilde{Z}/Z$ treated as a $p$-adic manifold fibration over $Z$ as defined in \cite{Howe.PAdicManifoldFibrations} --- the latter is a less rigid category where maps do not need to respect the torsor structure but instead only the induced $p$-adic manifold structure on the fibers.\end{remark}

\subsubsection{}The natural map $\rho^\lfid: \tilde{Z}^\lfid \rightarrow Z^\lfid$ is a $K^\lfid$-torsor, extending the given $K$-torsor structure on underlying $v$-sheaves (that it is a torsor rather than a quasi-torsor uses crucially the slope condition on our test objects). In particular, writing $\rho^\lfid: \tilde{Z}^\lfid \rightarrow Z^\lfid$, it follows that we obtain an exact sequence of $\mathbb{Q}_p^\lfid$-modules over $\tilde{Z}^\lfid$,
\begin{equation}\label{eq.intro-torsor-exact-seq} 0 \rightarrow \mf{k} \otimes_{\mathbb{Q}_p} \mathbb{Q}_p^\lfid \xrightarrow{da_e} T_{\tilde{Z}^\lfid} \xrightarrow{d\rho^\lfid} {\pi^{\lfid}}^* T_{Z^\lfid} \rightarrow 0. \end{equation}
where we recall $\mathfrak{k}:=\Lie K$ and the first map $da_e$ is the derivative of the action map $a: \tilde{Z}^\lfid \times K^\lfid \rightarrow \tilde{Z}^\lfid$ at the identity section $e$ of $K^\lfid$, which naturally computes the tangent bundles of the fibers (i.e. the vertical tangent bundle of the fibration $\rho^\lfid$). Note that if we restrict to a geometric point $\tilde{z}:\Spa(C,C^+) \rightarrow \tilde{Z}$ lying above $z: \Spa(C,C^+) \rightarrow Z$, then the resulting Tangent Space $T_{\tilde{z}} \tilde{Z}$ can be thought of as an extension of the finite dimensional $C$-vector space $T_{z} Z_C$ by the finite dimensional $\mathbb{Q}_p$-vector space $\mf{k}$ living in the natural $p$-adic Hodge theoretic home for such a beast: the category of Banach--Colmez spaces.

We note that there is also a trivial such a functorial structure, for which the analogous exact sequence \cref{eq.intro-torsor-exact-seq} is canonically split. We will now describe a computation of $T_{Z^\lfid}$ that shows, in particular, our construction is \emph{not} the trivial one if $d\pi_{\LZ} \neq 0$ (in particular, in the de Rham case, whenever the Kodaira-Spencer map is non-zero). 

\begin{remark}
    Descending \cref{eq.intro-torsor-exact-seq} along the $K^\lfid$-action gives an exact sequence over $Z^\lfid$, the Atiyah sequence for $\tilde{Z}^\lfid/Z^\lfid$. Classically in differential geometry, connections on torsors are equivalent to splittings of the Atiyah sequence, thus the non-splitness wherever $d\pi_{\LZ}$ is non-zero can be thought of as the failure of the existence of a connection even locally.  
\end{remark}

\subsubsection{}To describe the computation, we need to introduce some more notation. We write $\BC(\mathcal{O}_{\mathcal{X}}(\infty))$ for the inscribed $v$-sheaf of functions on $\mc{X}\backslash \infty$ with at most a simple pole at $\infty$ (after a twist by $\mathbb{Q}_p(1)$, this is canonically identified with the natural inscribed version of the Banach--Colmez space $\BC(1)$). There is a natural map $r: \BC(\mathcal{O}_{\mathcal{X}}(\infty)) \rightarrow \Fil^{-1} \mathbb{B}_\dR$ obtained by taking the Laurent expansion at $\infty$; we write $\overline{r}$ for the tail map, i.e. the quotient to a map  $\BC(\mathcal{O}_{\mathcal{X}}(\infty)) \rightarrow \Fil^{-1} \mathbb{B}_\dR/\Fil^0 \mathbb{B}_\dR$, so $\ker \overline{r}=\mathbb{Q}_p^\lfid$ (the global functions on $\mc{X}$).  

\begin{maintheorem}\label{maintheorem.tangent-bundle-computation}(see \cref{theorem.torsor-tangent-bundle-computation})
There is a functorial Cartesian diagram of $\mathbb{Q}_p^\lfid$-modules on $\tilde{Z}^\lfid$ 
% https://q.uiver.app/#q=WzAsNCxbMCwyLCJcXG1me2t9IFxcb3RpbWVzX3tcXG1hdGhiYntRfV9wfSBcXEJDKFxcbWF0aGNhbHtPfV97XFxtY3tYfX0oXFxpbmZ0eSkpIl0sWzEsMCwie1xccmhvXlxcbGZpZH1eKiBUX3taXlxcbGZpZH0iXSxbMSwyLCJcXG1me2t9XFxvdGltZXNfe1xcbWF0aGJie1F9X3B9IFxcZnJhY3tcXEZpbF57LTF9XFxtYXRoYmJ7Qn1fXFxkUn17XFxGaWxeMFxcbWF0aGJie0J9X1xcZFJ9Il0sWzAsMCwiVF97XFx0aWxkZXtafV5cXGxmaWR9Il0sWzAsMiwiXFxtYXRocm17aWR9IFxcb3RpbWVzIFxcb3ZlcmxpbmV7cn0iXSxbMSwyLCJ7XFxyaG9ee1xcbGZpZH19XipkXFxwaV97XFxMWn0iXSxbMywxLCJkXFxyaG9eXFxsZmlkIl0sWzMsMF0sWzMsMiwiIiwxLHsic3R5bGUiOnsibmFtZSI6ImNvcm5lciJ9fV1d
\[\begin{tikzcd}
	{T_{\tilde{Z}^\lfid}} & {{\rho^\lfid}^* T_{Z^\lfid}} \\
	\\
	{\mf{k} \otimes_{\mathbb{Q}_p} \BC(\mathcal{O}_{\mc{X}}(\infty))} & {\mf{k}\otimes_{\mathbb{Q}_p} \frac{\Fil^{-1}\mathbb{B}_\dR}{\Fil^0\mathbb{B}_\dR}}
	\arrow["{d\rho^\lfid}", from=1-1, to=1-2]
	\arrow[from=1-1, to=3-1]
	\arrow["\lrcorner"{anchor=center, pos=0.125}, draw=none, from=1-1, to=3-2]
	\arrow["{{\rho^{\lfid}}^*d\pi_{\LZ}}", from=1-2, to=3-2]
	\arrow["{\mathrm{id} \otimes \overline{r}}", from=3-1, to=3-2]
\end{tikzcd}\]
that identifies the inclusion $da_e$ of \cref{eq.intro-torsor-exact-seq} with the product map 
\[ \left(\mf{k} \otimes_{\mbb{Q}_p} \mathbb{Q}_p^\lfid = \mathrm{Ker}(\mr{id} \otimes \overline{r}) \hookrightarrow \mf{k} \otimes_{\mathbb{Q}_p} \BC(\mathcal{O}_{\mc{X}}(\infty)) \right) \times \left(\mf{k} \otimes_{\mbb{Q}_p} \mathbb{Q}_p^\lfid \xrightarrow{0} {\rho^\lfid}^* T_{Z^\lfid} \right).\]
\end{maintheorem}

In particular, \cref{maintheorem.tangent-bundle-computation} implies the exact sequence \cref{eq.intro-torsor-exact-seq} is non-split already after restriction to any geometric point where $d\pi_{\LZ}$ is non-zero.

\begin{remark}\label{remark.tb-is-bc-of-vb}If $d\pi_{\LZ}$ is injective (in particular, in the de Rham case, if the Kodaira-Spencer map is injective), then $T_{\tilde{Z}^\lfid}$ is the Banach--Colmez space of global sections $\BC(\mathcal{E})$ for $\mc{E}$ the vector bundle on the relative thickened Fargues--Fontaine curve obtained as the effective minuscule modification of $\mf{k} \otimes_{\mbb{Q}_p} \mathcal{O}_{\mathcal{X}}$ by the local direct summand $d\pi_{\LZ}: T_S \hookrightarrow \mf{k} \otimes_{\mbb{Q}_p} \frac{\Fil^{-1}\mathbb{B}_\dR}{\Fil^0 \mathbb{B}_\dR}$. \end{remark}

\subsubsection{}If $\tilde{Z}$ is de Rham, then there are $K^\lfid$-equivariant Hodge--Tate and lattice Hodge--Tate period maps 
\[ \pi_{\HT}: \tilde{Z}^\lfid \rightarrow \Fl_{G}^\lfid \textrm{ and } \pi_{\HT}^+: \tilde{Z}^\lfid \rightarrow \Gr_{G} \]
measuring the position of the (inscribed) Hodge--Tate filtration on $\mathbb{M} \otimes_{\mathbb{B}^+_\dR} \mathcal{O}$ and the lattice $\mathbb{M}_0$ that induces it with respect to a universal trivialization over $\tilde{Z}^\lfid$. The lattice Hodge--Tate period map refines the Hodge--Tate period map: if the Hodge period map $\pi_\Hdg$ factors through $\Fl_{[\mu^{-1}]}/G(\overline{\mc{O}})$ (this holds always for some $[\mu^{-1}]$ if $Z$ is geometrically connected), then $\pi_{\HT}$ factors through $\Fl_{[\mu]}$, $\pi_{\HT}^+$ factors through $\Gr_{[\mu^{-1}]}$, and $\pi_{\HT}=\BB \circ \pi_{\HT}^+$.  

\begin{maintheorem}\label{maintheorem.Hodge--Tate-derivative}(see \cref{theorem.ht-derivative-computation})
    For $\mf{k}^+_\dR=\mathbb{M}_0(\omega_{\nabla}^\Fil(\Lie G))$, 
    \[ d\pi_{\HT}^+: T_{\tilde{Z}^\lfid} \rightarrow \pi_{\HT}^* T_{\Gr_{G}} =  \frac{\mf{k}\otimes_{\mbb{Q}_p}\mathbb{B}_\dR}{\mf{k}^+_\dR}\]
    is the map given by the composition
    \[ T_{\tilde{Z}^\lfid} \rightarrow \mf{k} \otimes_{\mbb{Q}_p} \BC(\mathcal{O}_{\mc{X}}(\infty)) \xrightarrow{\mr{id} \otimes r} \mf{k} \otimes_{\mbb{Q}_p} \mathbb{B}_\dR \rightarrow \frac{\mf{k} \otimes_{\mbb{Q}_p} \mathbb{B}_\dR}{\mf{k}^+_\dR}, \]
    where the first map is the projection coming from the identification of \cref{maintheorem.tangent-bundle-computation}. 
\end{maintheorem}
In the statement we use the canonical identification $\mf{k} \otimes_{\mbb{Q}_p} \mathbb{B}_\dR=\mf{k}^+_\dR \otimes_{\mathbb{B}^+_\dR} \mathbb{B}_\dR$ over  $\tilde{Z}^\lfid$.

\begin{remark}
When $\pi_\HT^+$ factors through $\Gr_{[\mu^{-1}]}$, $d\pi_\HT^+$ factors through 
\[ {\pi_{\HT}^+}^* T_{\Gr_{[\mu^{-1}]}}= \frac{\mf{k} \otimes_{\mbb{Q}_p}\mathbb{B}^+_\dR}{\mf{k}^+_\dR \cap \mf{k} \otimes_{\mbb{Q}_p}\mathbb{B}^+_\dR}. \]
It follows that in this case $d\pi_\HT$ is given by composition with the natural map to
\[ \pi_{\HT}^*T_{\Fl_{[\mu]}^\lfid}=\mf{k} \otimes_{\mbb{Q}_p} \mathcal{O} / \Fil^0_\HT (\mf{k} \otimes_{\mbb{Q}_p} \mathcal{O}) \]
where the filtration used on the right is the (inscribed) Hodge--Tate filtration. 
\end{remark}

\begin{remark} Note that, unlike in the Hodge case, there is no $\Fil^1\mathbb{B}^+_\dR$-torsion constraint on $T_{\tilde{Z}^\lfid}$ imposing Griffiths transversality on the (inscribed) Hodge--Tate filtration. Indeed, since $\mathbb{M}$ is not flat, the best one can do to obtain period maps from $Z^\lfid$ itself is to consider
\[ \pi_\HT^+/K^\lfid: Z^\lfid \rightarrow \Gr_{[\mu^{-1}]}/K^\lfid \textrm{ and }  \pi_{\HT}/K^\lfid: Z^\lfid \rightarrow \Fl_{[\mu]}/K^\lfid,\]
but now the quotient is killing a lot of differential information so that the factorization of $\pi_\HT/K^\lfid$ through $\BB/K^\lfid$ does not have the same strong implications.\end{remark}

\begin{example}\label{example.shimura-varieties}
    Let $\gx$ be a Shimura datum and assume, for simplicity, that $\mathsf{G}$ contains no $\mathbb{Q}$-anisotropic $\mathbb{R}$-split central torus. Let $K^p \leq G(\mathbb{A}_f^{(p)})$ be a compact open subgroup. We write $\Sh^\diamond_{K^p}$ for the diamond infinite level Shimura variety over the completion $L$ of the reflex field at a prime above $p$.  For any compact open subgroup $K_p \leq G(\mathbb{Q}_p)$ such that $K_p K^p$ is neat, $\Sh^\diamond_{K^p}$ is a de Rham $K_p$-torsor over the rigid analytic finite level Shimura variety $\Sh_{K_pK^p}/L$ . We write $\pi_{\Hdg, K_p}^{(+)}$ for the finite level (lattice) Hodge period map. 
    
    From the above construction, we obtain an inscribed version $\Sh_{K^p}^\lfid$ equipped with a (lattice) Hodge--Tate period map $\pi_{\HT}^{(+)}$ and a  (lattice) Hodge period map $\pi_{\Hdg}^{(+)}$ (by composition of $\pi_{\Hdg, K_p}^{(+)}$ with $\Sh_{K^p}^\lfid \rightarrow \Sh_{K_pK^p}^\lfid$) that is functorial in the prime-to-$p$ level structure and independent of the choice of $K_p$. Because the Hodge cocharacter $[\mu]$ is minuscule, the Bialynicki-Birula maps are isomorphisms so that $\pi_\Hdg=\pi_\Hdg^+$ and $\pi_\HT=\pi_{\HT}^+$. These maps fit in the commutative diagram
% https://q.uiver.app/#q=WzAsNCxbMSwwLCJcXFNoX3tLXnB9XlxcbGZpZCJdLFswLDEsIlxcU2hfe0tfcEtecH1eXFxsZmlkIl0sWzIsMiwiXFxGbF97W1xcbXVdfV5cXGxmaWQiXSxbMCwyLCJcXEZsX3tbXFxtdV57LTF9XX1eXFxsZmlkL0coXFxvdmVybGluZXtcXG1hdGhjYWx7T319KSJdLFswLDIsIlxccGlfe1xcSFR9IiwxLHsiY3VydmUiOjF9XSxbMSwzLCJcXHBpX3tcXEhkZywgS19wfSIsMl0sWzAsMSwiXFxyaG9fe0tfcH1eXFxsZmlkIiwxXSxbMCwzLCJcXHBpX3tcXEhkZ30iLDEseyJjdXJ2ZSI6LTF9XV0=
\[\begin{tikzcd}
	& {\Sh_{K^p}^\lfid} \\
	{\Sh_{K_pK^p}^\lfid} \\
	{\Fl_{[\mu^{-1}]}^\lfid/G(\overline{\mathcal{O}})} && {\Fl_{[\mu]}^\lfid}
	\arrow["{\rho_{K_p}^\lfid}"{description}, from=1-2, to=2-1]
	\arrow["{\pi_{\Hdg}}"{description}, curve={height=-6pt}, from=1-2, to=3-1]
	\arrow["{\pi_{\HT}}"{description}, curve={height=6pt}, from=1-2, to=3-3]
	\arrow["{\pi_{\Hdg, K_p}}"', from=2-1, to=3-1]
\end{tikzcd}\]

    By a result of \cite{DiaoLanEtAl.LogarithmicRiemannHilbertCorrespondencesForRigidVarieties}, the associated filtered vector bundle with connection $\omega_{\nabla}^\Fil$ on $\Sh_{K_pK^p}$ is the classical automorphic de Rham bundle; in particular, the Kodaira-Spencer map $\kappa$ is an isomorphism. Thus \cref{remark.tb-is-bc-of-vb} applies and $T_{\Sh_{K^p}^\lfid}=\BC(\mathcal{E})$ for $\mc{E}$ a vector bundle on the relative thickened Fargues--Fontaine curve. Because the Kodaira-Spencer map $\kappa$ for $\omega_{\nabla}^\Fil$ is an isomorphism and the Hodge cocharacter is minuscule, we can describe $\mathcal{E}$ very explicitly: it is the modification of the vector bundle $\mf{k} \otimes_{\mbb{Q}_p} \mathcal{O}_{\mc{X}}$ by the lattice
    \[ \mf{k}^+_\mr{max}=\mf{k}^+_\dR + \mf{k} \otimes_{\mbb{Q}_p} \mathbb{B}^+_\dR \subseteq \mf{k} \otimes_{\mbb{Q}_p} \mathbb{B}_\dR = \mf{k}^+_{\dR} \otimes_{\mbb{B}^+_\dR} \mbb{B}_\dR. \]
    The associated diagram of derivatives can be written as
% https://q.uiver.app/#q=WzAsNyxbMiwwLCJcXEJDKFxcbWN7RX0pIl0sWzAsMiwiKFRfe1xcU2hfe0tfcEtecH19KV5cXGxmaWQiXSxbNCwzLCJcXGdyXnstMX1fe1xcSFR9IChcXG1me2t9XFxvdGltZXMgXFxtYXRoY2Fse099KSJdLFswLDMsIlxcZ3Jeey0xfV97XFxIZGd9XFxvbWVnYV97XFxuYWJsYX1eXFxGaWwoXFxMaWUgRykiXSxbMiwxLCJcXG1me2t9Xitfe1xcbXJ7bWF4fX0iXSxbMywyLCJcXGZyYWN7XFxtZntrfV4rX3tcXG1ye21heH19fXtcXG1me2t9XitfXFxkUn0iXSxbMSwyLCJcXGZyYWN7XFxtZntrfV4rX3tcXG1ye21heH19fXtcXG1me2t9XFxvdGltZXNfe1xcbWJie1F9X3B9XFxtYmJ7Qn1eK19cXGRSfSJdLFswLDIsImRcXHBpX3tcXEhUfSIsMSx7ImN1cnZlIjotMn1dLFsxLDMsIlxca2FwcGEiLDFdLFswLDEsImRcXHJob197S19wfV5cXGxmaWQiLDEseyJjdXJ2ZSI6NH1dLFswLDRdLFs0LDVdLFs1LDIsIj0iLDFdLFswLDMsImRcXHBpX3tcXEhkZ30iLDEseyJjdXJ2ZSI6Mn1dLFs0LDZdLFs2LDMsIj0iLDFdXQ==
\[\begin{tikzcd}
	&& {\BC(\mc{E})} \\
	&& {\mf{k}^+_{\mr{max}}} \\
	{(T_{\Sh_{K_pK^p}})^\lfid} & {\frac{\mf{k}^+_{\mr{max}}}{\mf{k}\otimes_{\mbb{Q}_p}\mbb{B}^+_\dR}} && {\frac{\mf{k}^+_{\mr{max}}}{\mf{k}^+_\dR}} \\
	{\gr^{-1}_{\Hdg}\omega_{\nabla}^\Fil(\Lie G)} &&&& {\gr^{-1}_{\HT} (\mf{k}\otimes \mathcal{O})}
	\arrow[from=1-3, to=2-3]
	\arrow["{d\rho_{K_p}^\lfid}"{description}, curve={height=24pt}, from=1-3, to=3-1]
	\arrow["{d\pi_{\Hdg}}"{description}, curve={height=12pt}, from=1-3, to=4-1]
	\arrow["{d\pi_{\HT}}"{description}, curve={height=-12pt}, from=1-3, to=4-5]
	\arrow[from=2-3, to=3-2]
	\arrow[from=2-3, to=3-4]
	\arrow["\kappa"{description}, from=3-1, to=4-1]
	\arrow["{=}"{description}, from=3-2, to=4-1]
	\arrow["{=}"{description}, from=3-4, to=4-5]
\end{tikzcd}\]
where the top vertical arrow takes a global section to its formal expansion at $\infty$. The left side of the diagram says that, from a differential perspective, $\pi_{\Hdg}$ behaves like a $G(\mathbb{Q}_p^\lfid)$-torsor.
\end{example}

\begin{remark}\label{remark.loc-an-functions}
    In the setting of \cref{example.shimura-varieties}, let $\mc{E} \boxtimes \mathcal{O}$ be the $\mathcal{O}$-module of global sections of $\infty^*\mc{E}$. It can be shown that the restriction of $\mc{E} \boxtimes \mc{O}$ to $|\Sh_{K^p}|$ acts by derivations on the sheaf $\mathcal{O}^\mr{la}$ of locally analytic functions on $|\Sh_{K^p}|$ in the sense of \cite{Pan.OnLocallyAnalyticVectorsOfTheCompletedCohomologyOfModularCurves, Pan.OnLocallyAnalyticVectorsOfTheCompletedCohomologyOfModularCurvesII, RodriguezCamargo.GeometricSenTheoryOverRigidAnalyticSpaces, RodriguezCamargo.LocallyAnalyticCompletedCohomology}; in fact, this is essentially equivalent to the key annihilation property of the canonical Higgs field / geometric Sen morphism. This gives, in particular, another interpretation of the differential operators in \cite{Pan.OnLocallyAnalyticVectorsOfTheCompletedCohomologyOfModularCurvesII}. This construction applies more generally to pro-\'{e}tale $p$-adic manifold fibrations over smooth rigid analytic varieties, and will be treated in this generality in \cite{Howe.PAdicManifoldFibrations}. 
\end{remark}

\subsubsection{}
We obtain a similar computation to \cref{example.shimura-varieties} for local Shimura varieties --- in this case, there is a distinguished global basis of flat sections over the associated rigid analytic finite level space, so that here $\pi_\Hdg$ can be treated as a map to $\Fl_{[\mu^{-1}]}^\lfid$ and it is indeed a $G(\mathbb{Q}_p^\lfid)$-torsor over its image (the open admissible locus). In \cref{ss.main-results-mcls} we give another construction of the  inscribed structure on infinite level local Shimura varieties by making the natural extension of their interpretation as moduli of modifications on the Fargues--Fontaine curve. We obtain the same derivative computations from this moduli perspective.

Similarly, assuming the existence of an Igusa stack in the sense of \cite{Kim.UniquenessAndFunctorialityOfIgusaStacks} (which holds in the Hodge-type case), in \cref{s.Inscribed-global-Shimura} we  give an alternative moduli construction of the same inscribed structure on infinite level global Shimura varieties --- this is an adaptation to our setting of a construction in analytic prismatization that has been announced by Scholze. 

\subsubsection{}One particularly interesting aspect of the moduli construction in the infinite level local Shimura case is that it generalizes to all moduli of mixed characteristic local shtuka with one leg, i.e. to non-minuscule cocharacters (and in fact in \cref{s.moduli-of-mod} we treat even more general moduli of modifications than these). The computation of the derivatives of group actions and period maps in this general context is one of the most satisfying parts of the entire paper, as within it we rediscover all of the fundamental exact sequences of $p$-adic Hodge theory (see, in particular, \cref{cor.unbounded-mod-tangent}, \cref{corollary.bounded-mod-tangent} and \cref{ss.main-results-mcls}). 

These moduli spaces are $G(\mathbb{Q}_p^\lfid)$-covers of the admissible loci $\Gr_{[\mu]}^{b-\adm}$ in Schubert cells of the inscribed $\mathbb{B}^+_\dR$-affine Grassmannian, and this is compatible with the twistor story: the $b$-admissible locus carries a universal meromorphic integrable connection with $G$-structure on the relative thickened Fargues--Fontaine curve, and the associated trivializing torsor on the underlying $G$-bundle is the infinite level space. Its restriction along any map from a rigid analytic variety is a twistor bundle, and this recovers our construction for de Rham torsors over rigid analytic varieties in the (very special) flat crystalline case --- see \cref{ss.mom-revisited}. 

In general, we view the compatibility between our construction of inscribed structures on torsors $\tilde{Z}$ via \cref{maintheorem.twistor-correspondence} with the inscribed structures in these special cases coming from natural extensions of moduli problems in $p$-adic Hodge theory as strong evidence that our inscribed structure, rather than the trivial inscribed structure, is the ``correct" one.

\subsection{Differential topology for diamonds}

Having constructed a differentiable theory and computed the derivatives of period maps, it is natural to ask whether there is anything useful we can do with this information. One promising direction is the connection with differential operators on various types of locally analytic functions noted in \cref{remark.loc-an-functions}, which will be developed more broadly for $p$-adic manifold fibrations over smooth rigid analytic varieties in \cite{Howe.PAdicManifoldFibrations}. Below we discuss other applications in the vein of ``differential topology for diamonds," i.e. using the differentiable structure to deduce properties of the underlying $v$-sheaf.

\subsubsection{A perfectoidness criterion.} For $\tilde{Z}$ as in \cref{ss.intro-main-results} and $C=\overline{L}^\wedge$, it is natural to conjecture that $\tilde{Z}_C$ is perfectoid if, for every geometric point $\tilde{z}$, the Tangent Space $T_{\tilde{z}} \tilde{Z}=\overline{T_{\tilde{z}} \tilde{Z}^\lfid}$ is perfectoid. It is typically simple to check whether these Tangent Spaces are perfectoid because there is a simple classification of perfectoid Banach--Colmez spaces (they are precisely the global section spaces of vector bundles on the Fargues--Fontaine curve with Harder-Narasimhan slopes between 0 and 1). In particular, essentially by \cref{remark.tb-is-bc-of-vb}, this conjecture for $\tilde{Z}$ is equivalent to a conjecture of Rodriguez Camargo in terms of the geometric Sen morphism \cite[Conjecture 3.3.5]{RodriguezCamargo.GeometricSenTheoryOverRigidAnalyticSpaces}, towards which there has been substantial recent progress \cite{He.PerfectoidnessViaSenTheory, BellovinCaiHowe.CharacterizingPerfectoidCoversOfAbelianVarieties}. In \cite{Howe.PAdicManifoldFibrations}, we will extend this Tangent Space perfectoidness conjecture to all pro-\'{e}tale $p$-adic manifold fibrations over smooth rigid analytic varieties, and suggest a differential characterization of the defect to perfectoidness in that setting. 

\subsubsection{Cohomological smoothness}
We say a map of inscribed $v$-sheaves is a \emph{connected submersion} if its derivative is surjective (i.e. it is a differential submersion) and the kernel of the derivative at each geometric point (i.e. the Vertical Tangent Space) is a connected Banach--Colmez space. Morally, we expect any sufficiently natural\footnote{The theory of inscribed $v$-sheaves is somewhat flabby as it includes the trivial inscription on any $v$-sheaf, and any map between two trivially inscribed $v$-sheaves is a connected submersion. Thus some assumption is needed!} connected submersion to be cohomologically smooth in the sense of \cite{FarguesScholze.GeometrizationOfTheLocalLanglandsCorrespondence}.  When applied to morphisms of inscribed $v$-sheaves constructed as moduli of sections, this expectation specializes to the Fargues--Scholze Jacobian criterion for cohomological smoothness, one of the key technical ingredients in \cite{FarguesScholze.GeometrizationOfTheLocalLanglandsCorrespondence}. For the morphisms of inscribed $v$-sheaves coming from morphisms of rigid analytic varieties, it specializes to a classical statement about smooth morphisms of rigid analytic varieties. 

One particularly interesting case is the cohomological smoothness of the structure map for infinite level local Shimura varieties and their non-minuscule analogs, the moduli of mixed characteristic local shtukas with one leg. In the basic Rapoport--Zink EL case, Ivanov--Weinstein \cite{IvanovWeinstein.TheSmoothLocusInInfiniteLevelRapoportZinkSpaces} constructed (components of) these spaces as moduli of sections over the Fargues--Fontaine curve and then used the Fargues--Scholze Jacobian Criterion to show that the open complement of the special locus that parameterizes those $p$-divisible groups with extra endomorphisms is cohomologically smooth; this provided, in particular, a qualitative explanation for some earlier cohomological finiteness results coming from explicit constructions by Weinstein \cite{Weinstein.SemistableModelsForModularCurvesOfArbitraryLevel} of formal models for infinite level Lubin-Tate space. In \cite{Howe.CohomologicalSmoothnessConjecture}, we use the computations of tangent bundles in the present work to compute the locus where the structure map is a connected submersion for all moduli of mixed characteristic local shtukas with one leg (or rather for a natural variant as in \cite{IvanovWeinstein.TheSmoothLocusInInfiniteLevelRapoportZinkSpaces} given by fixing the determinant to kill the central disconnected tangent directions), then prove that this locus is cohomologically smooth in the Rapoport--Zink EL case (removing the basic hypothesis of \cite{IvanovWeinstein.TheSmoothLocusInInfiniteLevelRapoportZinkSpaces}). 

\subsubsection{Perversity}\label{sss.perversity} The fibers of the Hodge--Tate period map have induced inscribed structures from the inscribed Shimura variety structure of \cref{example.shimura-varieties}, and we also construct (in \cref{ss.inscribed-generalized-Newton-strata}) a natural inscribed structure on each Newton stratum in the Hodge--Tate period domain $\Fl_{[\mu]}^\diamond$. All of the Tangent Spaces that arise at geometric points for these objects are Banach--Colmez spaces, thus each admits a dimension that is a pair $(r,d)$ where $r \in \mathbb{Z}$ and $d \in \mathbb{Z}_{\geq 0}$ --- for the Banach--Colmez spaces of global sections associated to a vector bundle with non-negative Harder-Narasimhan slopes on the relative Fargues--Fontaine curve, $r$ is the rank and $d$ is the degree. It is thus natural to speculate on a relation between the dimensions that appear and perversity/smallness phenomena for $\pi_\HT$. In particular, for a geometric point $y$ of $\Sh_{K^p}^\lfid$ such that $\pi_{\HT}(y)$ lies in the Newton stratum $\Fl_{[\mu]}^{[b]}$, our computations show there is an equality (for $\dim_{\BC}$ the Banach--Colmez dimension)
\[ \dim_{\BC} T_y \Sh_{K_p}^\lfid = \dim_{\BC} T_y \pi_{\HT}^{-1}(y) + \dim_{\BC} T_{\pi_{\HT}(y)} \Fl_{[\mu]}^{[b]}.\]
Does this kind of relation between the differential dimension of the fibers and the strata have a connection with the form of $\ell$-adic perversity exploited in \cite{CaraianiScholze.OnTheGenericPartOfTheCohomologyOfCompactUnitaryShimuraVarieties} and subsequent work?

\subsubsection{Transversality}
Our construction of Tangent Bundles also allows us to formulate a natural definition of transverse intersection in some previously inaccessible settings: for example, using period maps, a perfectoid infinite level local Shimura variety can often be embedded inside of a rigid analytic variety. Using the inscribed structure, it then makes sense to consider transverse intersections between this fractal-like perfectoid space and smooth rigid analytic subvarieties of the ambient rigid analytic variety, where the definition of transversality is essentially the usual one from differential topology: the intersection is transverse if, at every geometric point of the intersection, the Tangent Spaces of the subobjects together span (over $\mathbb{Q}_p^\diamond$) the Tangent Space of the ambient variety. In joint work with Klevdal \cite{HoweKlevdal.AdmissiblePairsAndpAdicHodgeStructuresIIIVariationAndUnlikelyIntersection}, we use this notion of transversality to formulate an Ax-Schanuel conjecture for local Shimura varieties (that is closely related to our Ax-Lindemann theorem in \cite{HoweKlevdal.AdmissiblePairsAndpAdicHodgeStructuresIITheBiAnalyticAxLindemannTheorem}). 

\subsection{Outline}
In \cref{s.adic-spaces-and-schemes} we treat some preliminaries on nilpotent thickenings of adic spaces and schemes, and in \cref{s.inscription} we develop a theory of inscription in a general axiomatic setting --- we recommend skipping these sections at first and coming back to them only as necessary. 

In \cref{s.inscribed-contexts} we then specialize to the specific inscribed contexts of interest to us in $p$-adic geometry --- in particular, although in the introduction we have used the category of locally free thickenings of Fargues--Fontaine curves that satisfy a slope condition, it is useful to develop various parts of the theory using thickenings of smaller loci within the Fargues--Fontaine curve (for example, to obtain the $\mathbb{B}^+_\dR$-linearity in \cref{remark.griffiths-transversality}). Moreover, since we do not really believe that the ``right" category of test objects is a settled affair, this generality allows us to avoid imposing the slope condition except when it is actually required. We recommend skimming this to fix notation in your head, then coming back as necessary.

The less formal and more geometric parts of the paper finally kick off in \cref{s.inscribed-vb-g-bun}, where we study the inscribed moduli stacks of vector bundles and $G$-bundles on thickenings of relative Fargues--Fontaine curves (and other loci on relative Fargues--Fontaine curves); the key result is \cref{prop.b-basic-open-stratum} on the deformation-openness of basic strata, which is the ultimate reason for the slope condition on thickenings everywhere that it is imposed.

In \cref{s.affine-grassmannian} we construct the inscribed $\mathbb{B}^+_\dR$-affine Grassmannian, its Schubert cells, and the Bialynicki-Birula maps from these Schubert cells to inscribed flag varieties. 

With both the domains and codomains now constructed, in \cref{s.hodge-etc-period-maps} we are able to construct the Hodge, lattice Hodge, and Liu-Zhu period maps and describe the relation between them. The non-trivial extension of Scholze's functor $\mathbb{M}$ in the inscribed setting carried out in this section is perhaps the key technical point of the paper. 

In \cref{s.modifications} we discuss modifications of bundles on relative thickened Farguess-Fontaine curves and related inscribed objects; we also define generalized inscribed Newton strata. An important technical ingredient in some of our later computations of tangent bundles is the torsor structure on the inscribed Hecke correspondence in \cref{ss.inscribed-hecke}. 

At this point it would be possible to continue immediately to the last section, \cref{s.twistors}, where twistors are defined and we prove \cref{maintheorem.twistor-correspondence}, \cref{maintheorem.tangent-bundle-computation}, and \cref{maintheorem.Hodge--Tate-derivative}. Instead, in \cref{s.moduli-of-mod} we first construct inscribed moduli of modifications (this includes infinite level local Shimura varieties and their non-miniscule generalizations, but we actually work in a more general setting) and compute their tangent bundles, derivatives of their period maps, etc. directly from this definition. We also give a very general formulation of the two towers principle in this inscribed context that may be of independent interest (\cref{ss.two-towers}). Similarly, in \cref{s.Inscribed-global-Shimura} we construct inscribed infinite level Shimura varieties via the natural extension of their Igusa stacks moduli interpretation, and compute their tangent bundles, derivatives of period maps, etc. directly from this definition. All of these moduli-theoretic inscribed computations seem to us clearer and more direct than those of \cref{s.twistors}, so, from an expository perspective, it is useful to treat them first. In the local and global Shimura cases, respectively, they are compared with the constructions from the twistor perspective in \cref{ss.mom-revisited} and \cref{ss.global-sv-revisited}. 

\subsection{Acknowledgements}
We thank Christian Klevdal, Gilbert Moss, and Peter Wear for participating in a reading group on \cite{IvanovWeinstein.TheSmoothLocusInInfiniteLevelRapoportZinkSpaces} that ultimately led to this work. We thank Christian Klevdal for influential conversations related to our joint work \cite{HoweKlevdal.AdmissiblePairsAndpAdicHodgeStructuresITranscendenceOfTheDeRhamLattice, HoweKlevdal.AdmissiblePairsAndpAdicHodgeStructuresIITheBiAnalyticAxLindemannTheorem}. We thank Peter Wear for early collaboration on the closely related \cite{Howe.PAdicManifoldFibrations}. We thank Peter Scholze for the suggestion that, instead of just ``adding an $\epsilon$" to the moduli problems, one might profitably consider larger categories of thickenings, which led to the present formulation of our results. We thank Laurent Fargues for a helpful conversation about the two towers isomorphism and, we thank Laurent Fargues and Jared Weinstein for encouraging our study of moduli of mixed characteristic local shtuka with one leg. We thank Daniel Gulotta for suggesting the consideration of quotients of infinite level Shimura varieties and other helpful discussion. We thank Johannes Ansch\"{u}tz, Arthur-C\'{e}sar le Bras, Juan Esteban Rodr\'{i}guez Camargo, and Peter Scholze for helpful conversations and for sharing ideas and perspectives from their theory of analytic prismatization. 

During the preparation of this work, the author was supported by the National Science Foundation through grants DMS-2201112 and DMS-2501816. The author also received support as a visitor at the 2023 Hausdorff Trimester on The Arithmetic of the Langlands Program at the Hausdorff Research Institute for Mathematics from the Deutsche Forschungsgemeinschaft (DFG, German Research Foundation) under Germany’s Excellence Strategy– EXC-2047/1–390685813, as a member at the Institute for Advanced Study during the academic year 2023-24 from the Friends of the Institute for Advanced Study Membership, and during the academic year 2024-2025 from the University of Utah Faculty Fellow program. We thank these institutes and their staff for the extraordinary working conditions they provided, without which this work would not exist in its present form.

\section{Adic spaces and schemes}\label{s.adic-spaces-and-schemes}

In this section we develop some complements on adic spaces and schemes. In particular, we study properties of finite locally free nilpotent thickenings and smooth morphisms, and the relation between them. Everything we discuss is completely classical in the context of schemes, but a bit more delicate in the category of adic spaces. 

To avoid issues of sheafiness, we will work exclusively with the category of strongly sheafy adic spaces as introduced in \cite{HansenKedlaya.SheafinessCriteriaForHuberRings} (see \cref{def.ss-adic-space}). This category includes the sousperfectoid spaces used in \cite[IV.4]{FarguesScholze.GeometrizationOfTheLocalLanglandsCorrespondence}, and thus Fargues--Fontaine curves, but also includes, e.g., smooth rigid analytic varieties over $\Spa(L)$ for any non-archimedean field $L$ (note that, for $L$ very large, $\Spa L$ may not itself be sousperfectoid; see \cite[Remarks 7.7 and 7.8]{HansenKedlaya.SheafinessCriteriaForHuberRings}). Another option that would accommodate both of these is the category of weakly sousperfectoid spaces of \cite{HansenKedlaya.SheafinessCriteriaForHuberRings}, but we will also want to allow locally free nilpotent thickenings (as defined in \cref{ss.thickenings} below) as well as more exotic non-reduced spaces such as the canonical infinitesimal thickenings of perfectoid spaces over $\mathbb{Q}_p$. All of these non-reduced spaces are strongly sheafy but not weakly sousperfectoid. 

Crucially, there is a reasonable notion of smooth morphisms between strongly sheafy adic spaces such that some of the results on smooth morphisms of sous-perfectoid adic spaces obtained in \cite[IV.4]{FarguesScholze.GeometrizationOfTheLocalLanglandsCorrespondence} still hold with little modification to the proofs. In particular, enough structural properties hold for smooth morphisms in this context for us to describe their local structure around a section, define the relative tangent bundles of a smooth morphism, and relate relative tangent bundles to restrictions of scalars along simple square-zero thickenings.

\subsection{Conventions}
In this paper, all Huber rings are complete Tate rings. We write $\AdicSpaces$ for the category of analytic adic spaces, i.e. the category of adic spaces locally of the form $\Spa(A,A^+)$ for $(A,A^+)$ a sheafy Huber pair such that $A$ is a complete Tate ring.  

\subsection{Strongly sheafy adic spaces}

\begin{definition}[cf. Definition 4.1 of \cite{HansenKedlaya.SheafinessCriteriaForHuberRings}]\label{def.ss-adic-space}
    A Huber pair $(A,A^+)$ is \emph{strongly sheafy} if, for every $n\geq 0$, the Tate algebra
\[ (A\langle t_1, \ldots, t_n\rangle, A^+ \langle t_1, \ldots, t_n \rangle) \]
is sheafy. We say an adic space $X$ is \emph{strongly sheafy} if it can be covered by open affinoids $\Spa(A,A^+)$, for $(A,A^+)$ a strongly sheafy Huber pair. We write $\goodAS \subseteq \AdicSpaces$ for the full subcategory of strongly sheafy adic spaces. 
\end{definition}

\begin{example}\hfill
\begin{enumerate}
\item Any sousperfectoid adic space is strongly sheafy. In particular, this applies to any $P\in \Perf$ (the category of perfectoid spaces in characteristic $p$), the Fargues--Fontaine curve $X_{E,P}$ and its cover $Y_{E,P}$, and any untilt $P^\sharp$ (see \cref{ss.Fargues--Fontaine-curves}).
\item Any strongly Noetherian adic space is strongly sheafy. In particular, this applies to rigid analytic varieties over non-archimedean fields.
\item For $P^\sharp$ a perfectoid space over $\mbb{Q}_p$, its canonical thickenings $P^\sharp_{(i)}$ are strongly sheafy (see \cref{ss.perf-untilt-canonical-thickenings}). 
\end{enumerate}
\item By \cref{cor.lf-nilp-strongly-sheafy} below, any locally free nilpotent thickening of a strongly sheafy adic space is strongly sheafy. 
\end{example}

\subsection{Vector bundles}\label{ss.vector-bundles-adic-spaces-schemes}
We write $\VB$ for the fibered category over ringed spaces whose objects are pairs $(T, \mc{F})$ where $T$ is a ringed space and $\mc{F}$ is a locally free of finite rank sheaf of $\mc{O}_T$-modules. 

Given an affine scheme $\Spec A$, global sections give an equivalence between $\VB(\Spec A)$ and the category of finite projective $A$-modules (see, e.g, \cite[\href{https://stacks.math.columbia.edu/tag/00NX}{Tag 00NX}]{stacks-project}). Similarly, given an affinoid adic space $\Spa(A,A^+)$, global sections gives an equivalence between $\VB(\Spa(A,A^+))$ and the category of finite projective $A$-modules by \cite[Theorem 8.2.22]{KedlayaLiu.RelativepAdicHodgeTheoryFoundations}. 

The restriction of $\VB$ to the category of schemes is an \'{e}tale stack. Similarly, because strongly sheafy adic spaces are stably adic in the sense of \cite[Definition 8.2.19]{KedlayaLiu.RelativepAdicHodgeTheoryFoundations}, \cite[Theorem 8.2.22]{KedlayaLiu.RelativepAdicHodgeTheoryFoundations} implies that the restriction of $\VB$ to $\goodAS$ is an \'{e}tale stack.

\subsection{Thickenings of adic spaces and schemes}\label{ss.thickenings}
\begin{definition}\label{def.thickenings}\hfill
\begin{enumerate}
\item A closed immersion $T \rightarrow T'$ of adic spaces is a \emph{nilpotent thickening} if the ideal sheaf $\mc{I}_T$ of $T$ in $\mc{O}_{T'}$ is locally nilpotent, i.e. after restriction to any quasi-compact open there is an $n$ such that $\mc{I}_T^n=0$. It is \emph{square-zero} if $\mc{I}_T^2=0$. 
\item A \emph{nilpotent thickening} of $T/T$ is a morphism $T/T \rightarrow T'/T$ of adic spaces over $T$ such that $T\rightarrow T'$ is a nilpotent thickening. Given such a nilpotent thickening, we obtain a splitting $\mc{O}_{T'}=\mc{O}_T \oplus \mc{I}_{T}$ in the category of sheaves of complete topological $\mc{O}_T$-modules. We say $T/T \rightarrow T'/T$ is \emph{locally free} if
\begin{enumerate}
\item $\mc{O}_{T'}$, or equivalently $\mc{I}_T$, is locally free of finite rank over $\mc{O}_T$, and
\item For any open affinoid adic $\Spa(A,A^+) \subseteq T$, $\mc{O}_{T'}(\Spa(A,A^+))$, or equivalently $\mc{I}_T(\Spa(A,A^+)$, is equipped with its canonical topology as a finite projective $A$-module. 
\end{enumerate}
\item For $T$ an adic space, we say an $\mc{O}_T$-algebra $\mc{A}$ is \emph{augmented} if it is equipped with a (necessarily surjective) map of $\mc{O}_T$-algebras $\mc{A} \rightarrow \mc{O}_T$, and \emph{nilpotent augmented} if the kernel $\mc{I}$ of the augmentation is locally nilpotent. We say an augmented $\mc{O}_T$-algebra $\mc{A}$ is \emph{locally free} if $\mc{I}$ or equivalently $\mc{A}$ is locally free of finite rank over $\mc{O}_T$.
\end{enumerate}
\end{definition}

\begin{remark}
Any nilpotent thickening $T \rightarrow T'$ induces a homeomorphism $|T|=|T'|$ and for any point $t \in |T|$ with associated valuation $v_t$ on $\mc{O}_{T,t}$, the induced valuation on $\mc{O}_{T', t}$ is given by composition of $v_t$ with $\mc{O}_{T',t} \twoheadrightarrow \mc{O}_{T,t}$. 
\end{remark}

For any locally free nilpotent thickening $T/T\rightarrow T'/T$, it follows from the definitions that $\mc{O}_{T'}$ is a locally free nilpotent augmented $\mc{O}_T$-algebra. Conversely, given a locally free nilpotent augmented $\mc{O}_T$-algebra $\mc{A}$, we may construct a locally free nilpotent thickening $T/T \rightarrow \Spa_T \mc{A}$ as follows: we write $\Spa_T \mc{A}$ for the locally v-ringed space $(T, \mc{A})$, where for any open affinoid $\Spa(A,A^+) \subseteq T$, $\mc{A}(\Spa(A,A^+))$ is equipped with its canonical topology as a finite projective $A$-module and, for each $t \in T$, the valuation on $\mc{A}_t$ is pulled back from the valuation $t$ on $\mc{O}_{T,t}$ along the surjection $\mc{A}_t \rightarrow \mc{O}_{T,t}$. When $T=\Spa(A,A^+)$ is affine, then one easily checks that $\Spa_T \mc{A}=\Spa(\mc{A}(T), A^+ \oplus \mc{I}(T))$, where $\mc{A}(T)$ is equipped with its canonical topology as a finite projective $A$-module. Equipped with the map $T \rightarrow \Spa_T \mc{A}$ coming from the augmentation and the structure map $\Spa_T \mc{A} \rightarrow T$ coming from the $\mc{O}_T$-algebra structure, $\Spa_T \mc{A}/T$ is a thickening of $T/T$. 

Note that we can make analogous definitions with schemes, where everything is simpler as we do not need to keep track of the topology. The following is then immediate from the above discussion. 
\begin{proposition}\label{prop.augmented-equiv}\hfill
\begin{itemize} 
\item For $T$ an adic space (resp. scheme), the functor 
\[ (T/T \hookrightarrow T'/T) \mapsto (\mc{O}_{T'} \rightarrow \mc{O}_T) \]
is an equivalence between locally free nilpotent thickenings of $T/T$ and locally free nilpotent augmented $\mc{O}_T$-algebras, with quasi-inverse
\[ (\mc{A} \rightarrow \mc{O}_T) \mapsto (T/T \hookrightarrow \Spa_T \mc{A}/T) \textrm{ (resp. $(T/T \hookrightarrow \Spec_T \mc{A}/T)$). }\]
\item For $T$ an adic space (resp. scheme), $f: V \rightarrow T$ a map of adic spaces (resp. schemes), and $\mc{A}$ a locally free augmented $\mc{O}_T$-algebra, $f^* \mc{A}$ is a locally free augmented $\mc{O}_V$-algebra, and 
\[ \Spa_V f^* \mc{A}= \Spa_T \mc{A} \times_{T} V \textrm{ (resp. } \Spec_V f^* \mc{A}= \Spec_T \mc{A} \times_{T} V  \textrm{).} \]
In particular, the category of locally free nilpotent thickenings is fibered over adic spaces (resp. schemes). 
\end{itemize}
\end{proposition}

\begin{corollary}
    For any sheafy Huber pair $(A,A^+)$, there is a natural equivalence between the category of locally free thickenings of $\Spa(A,A^+)/\Spa(A,A^+)$ and locally free thickenings of $\Spec A/\Spec A$.
\end{corollary}
\begin{proof}
We apply \cref{prop.augmented-equiv} along with the observation that $\VB(\Spa(A,A^+))$ and $\VB(\Spec A)$ are both equivalent, by global sections, to finite projective $A$-modules (see \cref{ss.vector-bundles-adic-spaces-schemes}). 
\end{proof}

\begin{corollary}\label{cor.lf-nilp-strongly-sheafy} If $T$ is a strongly sheafy adic space, then for any locally free nilpotent thickening $T/T \rightarrow T'/T$, $T'$ is a strongly sheafy adic space. 
\end{corollary}

\begin{remark}\label{remark.constant-thickening-reduced}
If $T$ is reduced, then for any finite locally free thickening $T/T \rightarrow T'/T$, the map $T \rightarrow T'$ is uniquely determined by $T' \rightarrow T$ as the inverse of the induced isomorphism $(T')^\red \rightarrow T$
\end{remark}

\begin{definition}\label{def.constant-thickening} If $T$ is an adic space  (resp. a scheme), we say a locally free nilpotent thickening $T/T\rightarrow T'/T$ is constant if it there a finite free nilpotent augmented $\mbb{Z}$-algebra $A \rightarrow \mbb{Z}$ such that $T'/T$ is isomorphic to $\Spa_T (A \otimes_{\mbb{Z}} \mc{O}_T)$ (resp. $\Spec_T (A \otimes_{\mbb{Z}} \mc{O}_T)$).  
\end{definition}

For $T$ an adic space or scheme, it follows from \cref{prop.augmented-equiv} that the category of square-zero locally free nilpotent thickenings of $T/T$ is equivalent to the category of locally free $\mathcal{O}_T$-modules of finite rank. 

\begin{definition}\label{def.gen-sz-thick}
    Given an adic space or scheme $T$ and a locally free of finite rank $\mathcal{O}_T$-module $\mc{I}$, we write $T[\mc{I}]$ for the associated locally free square-zero thickening of $T$. As a locally ringed space it is $(|T|, \mathcal{O}_T \oplus \mc{I})$ where the multiplication is given by $(a,i)(a',i')=(aa', ai'+a'i)$.
\end{definition}

\begin{definition}\label{def.constant-sz}
    Given an adic space (resp. scheme) $T$ and a finite free $\mbb{Z}$-module $M$, we write 
    \[ T[M]:=T[\mc{O}_T \otimes M]= \Spa_T (\mbb{Z}[M] \otimes \mc{O}_T) \textrm{ (resp. $\Spec_T (\mbb{Z}[M] \otimes \mc{O}_T[M])$)} \]
    where here $\mbb{Z}[M]$ is the nilpotent augmented finite free $\mbb{Z}$-algebra $\mbb{Z}\oplus M$ with  multiplication given by $(a,m)(a',m')=(aa', am'+a'm)$.
\end{definition}

From the discussion above we see a constant square-zero thickening of $T/T$ can be equivalently defined as a thickening isomorphic to $T[M]/T$ for some finite free $\mathbb{Z}$-module $M$. The following is immediate.

\begin{lemma}\label{lemma.module-thickening-functor}
For any adic space $T$ or scheme $T$, $M \mapsto T[M]/T$ defines a contravariant functor from finite free $\mathbb{Z}$-modules to constant square-zero thickenings of $T/T$ such that
\begin{equation}\label{eq.sum-to-coproduct}T[M_1 \times M_2]=T[M_1] \sqcup_{T} T[M_2]\end{equation}
where the inclusion maps $T[M_i] \rightarrow T[M_1 \times M_2]$ are induced by the projections $M_1 \times M_2 \rightarrow M_i$. 
\end{lemma}

\begin{definition}\label{def.thickening-categories} For any subcategory $\mathcal{C}$ of adic spaces or schemes, we write $\mathcal{C}^\lf$ for the category of locally free nilpotent thickenings of objects in $\mathcal{C}$, $\mathcal{C}^\cn$ for the subcategory of $\mathcal{C}^\lf$ consisting of constant nilpotent thickenings of objects in $\mathcal{C}$, and $\mathcal{C}^\epsilon$ for the subcategory of $\mathcal{C}^\cn$ consisting of constant square-zero thickenings of objects in $\mathcal{C}$. In each case $\mathcal{C}^\bullet$ is naturally fibered over $\mathcal{C}$. 
\end{definition}

\subsection{Smooth morphisms}

The results on smooth morphisms of sous-perfectoid adic spaces of \cite[IV.4]{FarguesScholze.GeometrizationOfTheLocalLanglandsCorrespondence} up through \cite[Lemma IV.14]{FarguesScholze.GeometrizationOfTheLocalLanglandsCorrespondence} go through essentially as written in the more general setting of strongly sheafy adic spaces. In this section we give just the specific definitions and statements we will need.

\begin{remark}The remaining statements in \cite[IV.4]{FarguesScholze.GeometrizationOfTheLocalLanglandsCorrespondence} depend on Lemma \cite[Lemma IV.16]{FarguesScholze.GeometrizationOfTheLocalLanglandsCorrespondence}, whose proof in loc. cit. uses the sous-perfectoid condition to reduce to the case of a perfectoid $X$. We do not attempt to extend these results here as, for our purposes, the analog of \cite[Lemma IV.14]{FarguesScholze.GeometrizationOfTheLocalLanglandsCorrespondence} saying that any section factors through a ball is the key structural result we need going forward. The generalization to strongly sheafy adic spaces is \cref{prop.section-smooth-balls} below, and we deduce from it a useful corollary about infinitesimal neighborhoods of sections in \cref{corollary.section-infinitesimal-pullback}.
\end{remark}

\begin{definition}(cf. \cite[Definition IV.4.9]{FarguesScholze.GeometrizationOfTheLocalLanglandsCorrespondence}, \cite[Definition 5.11]{HansenKedlaya.SheafinessCriteriaForHuberRings}).
A morphism $f: Y \rightarrow X$ of strongly sheafy adic spaces is
\begin{enumerate}
    \item \'{e}tale if, locally on $X$ and $Y$, it can be written as an open immersion followed by a finite \'{e}tale map, and
    \item smooth if there is a cover of $Y$ by open sets $V$ such that $f|_V$ can be written as a composition of an \'{e}tale map $V \rightarrow \mathbb{B}^d_X$ followed by projection to $X$.
\end{enumerate}
\end{definition}

\begin{lemma}[cf. Proposition IV.4.10-(iii) of \cite{FarguesScholze.GeometrizationOfTheLocalLanglandsCorrespondence}]\label{lemma.ss-base-change}
Let $Y \rightarrow X$ be a smooth map of strongly sheafy adic spaces, and let $X' \rightarrow X$ be an arbitrary map of strongly sheafy adic spaces. Then, $Y':=Y \times_X X'$ is a strongly sheafy adic space and $Y' \rightarrow X'$ is a smooth map of strongly sheafy adic spaces. Thus, the category of smooth morphisms of strongly sheafy adic spaces is a fibered category over strongly sheafy adic spaces. 
\end{lemma}
\begin{proof}
This follows since balls over, rational localizations of, and finite \'{e}tale covers of a strongly sheafy $\Spa(A,A^+)$ are all strongly sheafy. 
\end{proof}

\begin{definition}[cf. Definition IV.4.11 of \cite{FarguesScholze.GeometrizationOfTheLocalLanglandsCorrespondence}]
For $f:Y \rightarrow X$ a smooth map of strongly sheafy adic spaces, the sheaf of relative differentials $\Omega_{Y/X}$ on $Y$ is $\mc{I}/\mc{I}^2$ for $\mc{I}$ the ideal sheaf of the diagonal $Y \hookrightarrow Y \times_X Y$.
\end{definition}

\begin{lemma} For $f: Y \rightarrow X$ a smooth map of strongly sheafy adic spaces, $\Omega_{Y/X}$ is locally free of finite rank over $\mathcal{O}_Y$. In fact it is free of rank $d$ on any open $V$ as in the definition of a smooth map such that $f|_V$ factors through an \'{e}tale map $V \rightarrow \mathbb{B}^d_X$. For any map of strongly sheafy adic spaces $g: X' \rightarrow X$, $g^* \Omega_{Y/X}=\Omega_{Y \times_{X} X' / X'}.$
\end{lemma}
\begin{proof}
The following is based on the proof of \cite[Proposition IV.4.12]{FarguesScholze.GeometrizationOfTheLocalLanglandsCorrespondence}, but we have modified the structure of the argument to make it more clear (to us). Working locally, we may assume $X=\Spa(A,A^+)$, $Y$ is quasi-compact, and $Y\rightarrow X$ factors through an \'{e}tale map $Y \rightarrow \mbb{B}^d_X=: Y'$. The diagonal $Y \rightarrow Y \times_X Y$ then factors as
\[ Y \rightarrow Y \times_{Y'} Y = (Y \times_X Y) \times_{Y' \times_X Y'} Y' \rightarrow Y \times_X Y. \]
The first map, as the diagonal of an \'{e}tale map, is an open immersion. It follows that $\mc{I}_Y/\mc{I}_Y^2$ is the restriction of $\mc{I}_{Y \times_{Y'} Y}/\mc{I}_{Y \times_{Y'} Y}^2$ to $Y$, where $Y \times_{Y'} Y$ is viewed as a closed in $Y \times_X Y$. 

We thus analyze this ideal sheaf. By a change of coordinates, we may rewrite $Y' \times_X Y'$ as $\mbb{B}^{2d}_X$, such that the diagonal of $Y'$ over $X$ is the inclusion $\mbb{B}^d_X \hookrightarrow \mbb{B}^{2d}_X$ corresponding to setting the first $d$ coordinates to zero. In these coordinates, we write $Z \rightarrow \mbb{B}^{2d}_X$ for the \'{e}tale map $Y \times_X Y \rightarrow Y' \times_X Y'$ and $Z_0 \rightarrow \mbb{B}^{d}_X$ for its restriction to the diagonal of $Y'$ over $X$ obtained by setting the first $d$ coordinates to zero, i.e. for $Y \times_{Y'} Y \rightarrow Y'$. We are thus interested in the ideal sheaf of $Z_0$ in $Z$. By spreading out of the \'{e}tale site \cite[Lemma 15.6, Lemma 12.17]{Scholze.EtaleCohomologyOfDiamonds} we find that, for $n \gg 0$, $Z|_{p^n \mbb{B}^d_X \times_X \mbb{B}^d_X}$ is isomorphic to $p^n \mbb{B}^d_X \times_X Z_0$ over ${p^n \mbb{B}^d_X \times_X \mbb{B}^d_X}$. Covering $Z_0$ by strongly sheafy affinoids $\Spa(R,R^+)$, the corresponding opens ${p^n\mbb{B}^d_X \times_X \Spa(R,R^+)}$ are of the form 
$\Spa (R\langle t_1, \ldots, t_d \rangle, R^+\langle t_1, \ldots, t_d \rangle)$ and the ideal sheaf is $\mc{I}_Z=(t_1, \ldots, t_d)$. It follows that on this affinoid $\mc{I}_Z/\mc{I}_Z^2$ is free of rank $d$ with basis the classes of $t_1, \ldots, t_d$, and that concludes the proof of the local freeness and computation of the rank. Using these local charts, the claim about the base change is also immediate. 

\end{proof}

\begin{proposition}\label{prop.section-smooth-balls}(cf. \cite[Lemma IV.4.14]{FarguesScholze.GeometrizationOfTheLocalLanglandsCorrespondence}). If $f:Y \rightarrow X$ is a smooth map of strongly sheafy adic spaces and $s: X \rightarrow Y$ is a section, then there is a cover of $X$ by open subsets $U$ such that $s|_{U}$ factors through a neighborhood of $Y|_{U}$ isomorphic to $\mbb{B}^d_U$ (as adic spaces over $U$). 
\end{proposition}
\begin{proof}
The proof from \cite[Lemma IV.4.14]{FarguesScholze.GeometrizationOfTheLocalLanglandsCorrespondence} applies with no change.  
\end{proof}

\begin{corollary}\label{corollary.section-infinitesimal-pullback}
If $f:Y \rightarrow X$ is a smooth map of strongly sheafy adic spaces and $s: X \rightarrow Y$ is a section then, for any $n\geq 0$, the $n$th infinitesimal neighborhood $s_{(n)}$ of $s(X)$ in $Y$ is a locally free nilpotent thickening of $X/X$, and there is a natural isomorphism $s_{(1)}=X[s^*\Omega_{Y/X}].$ 
\end{corollary}
\begin{proof}
For the first part, note that over any $U$ as in \cref{prop.section-smooth-balls}, we can translate so that $s$ is the zero section. Then, over $U$, $s_{(n)}$ is the  thickening associated to $\mc{O}_{U}[t_1, \ldots, t_d]/(t_1,\ldots,t_d)^{n+1}$ by \cref{prop.augmented-equiv}. 

For the second part, to see that there is a natural isomorphism $s_{(1)}=X[s^*\Omega_{Y/X}]$, it suffices to construct a natural map $s^* \Omega_{Y/X} \rightarrow \mc{I}_{s(X)}/\mc{I}_{s(X)}^2$ and then verify it is an isomorphism in these local coordinates. Now, for $\Delta: Y \rightarrow Y \times_X Y$ the diagonal, we have 
\[ s^* \Omega_{Y/X}= s^* \Delta^* \mc{I}_\Delta = (\Delta \circ s)^* \mc{I}_\Delta \]
On the other hand, we can also write 
\[ \Delta \circ s = ((s \circ f) \times \Id) \circ s.\]
The restriction of $((s \circ f) \times \Id)^* \mc{O}_{Y \times_X Y} \rightarrow \mc{O}_{Y}$ to $((s \circ f) \times \Id)^* \mc{I}_\Delta$ factors through $\mc{I}_{s(X)}$. Thus we obtain a natural map
\[ s^* \Omega_{Y/X} = (((s \circ f) \times \Id) \circ s)^* \mc{I}_\Delta \rightarrow s^* \mc{I}_{s(X)} = \mc{I}_{s(X)}/\mc{I}_{s(X)}^2. \]
Using the local coordinates as above for the zero section in $\mbb{B}^d_U$, it is elementary to check that this is an isomorphism. 
\end{proof}

\subsection{Tangent bundles and restriction of scalars}\newcommand{\Sym}{\mathrm{Sym}}

We now discuss the relation between tangent bundles and restriction of scalars for smooth morphisms. We first recall the theory for schemes. 

First, we recall that for any scheme $X$ and quasi-coherent sheaf $\mc{F}/X$, we can form $\mbb{V}(\mc{F}):=\Spec_X \Sym^\bullet \mc{F}$, a scheme over $X$ representing
\[ (T/X) \mapsto \Hom_{\mc{O}_T}(\mc{F}_T, \mc{O}_T). \]
This construction is naturally a contravariant functor from quasi-coherent sheafs over $X$ to schemes over $X$.

Now, suppose $Y/X$ is a morphism of schemes, and $\mc{I}$ is a locally free of finite rank sheaf of $\mc{O}_X$-modules on $X$. Then, we may consider the restriction of scalars $R_{X[\mc{I}]/X}(Y[\mc{I}_Y]/X[\mc{I}])$. 
This is the functor on schemes over $X$ sending $T/X$ to 
\begin{align*} \Hom_{X[\mc{I}]}(T \times_{X} X[\mc{I}], Y[\mc{I}_Y])&=\Hom_{X[\mc{I}]}(T \times_{X} X[\mc{I}], Y\times_X 
X[\mc{I}])\\
&=\Hom_{X}(T \times_X X[\mc{I}], Y).
\end{align*}
Pull-back along the closed immersion $X \hookrightarrow X[\mc{I}]$ equips $R_{X[\mc{I}]/X}(Y[\mc{I}_Y]/X[\mc{I}])$ with a structure map to $Y$. The following shows it is represented by a natural scheme over $Y$. 

\begin{proposition}\label{prop.schemes-ros-tangent}
If $f: Y\rightarrow X$ is a morphism of schemes, there is a natural identification of functors on schemes over $Y$
\[ R_{X[\mc{I}]/X}(Y[\mc{I}_Y]/X[\mc{I}])=\mbb{V}(\Omega_{Y/X} \otimes_{\mc{O}_Y} \mc{I}_Y^*).\]
\end{proposition}
\begin{proof}
To give a map from $T\times_X X[I] /X $ to $Y/X$ inducing a fixed map $g: T \rightarrow Y$ is the same as to give a map of $f^{-1}\mc{O}_X$-algebras augmented to $\mc{O}_Y$ 
\[ \mc{O}_{Y} \rightarrow g_*\mc{O}_T \oplus (g_* \mc{I}_T). \]
This is equivalent to an $f^{-1}\mc{O}_X$-linear derivation $\mc{O}_Y \rightarrow g_* \mc{I}_T$, or equivalently an element of 
\begin{align*} 
\Hom_{\mc{O}_Y}(\Omega_{Y/X}, g_* \mc{I}_T)&=\Hom_{\mc{O}_T}(g^*\Omega_{Y/X}, \mc{I}_T)\\
&= \Hom_{\mc{O}_T}( g^* \Omega_Y, g^* \mc{I}_Y) \\
&= \Hom_{\mc{O}_T}( g^* (\Omega_Y \otimes_{\mc{O}_Y} \mc{I}_Y^*), \mc{O}_T)\\
&= \Hom_Y(T/Y, \mbb{V}(\Omega_{Y/X} \otimes_{\mc{O}_Y} \mc{I}_Y^*)).
\end{align*}

\end{proof}

We now want to imitate this with strongly sheafy adic spaces. In this case, the functor $\mbb{V}$ is defined only for $\mc{F}$ locally free (in which case the sheaf of sections of $\mbb{V}(\mc{F})$ is $\mc{F}^*$). This is not an issue, since for strongly sheafy adic spaces we have anyway only defined $\Omega_{Y/X}$ for $Y/X$ smooth. 

Another issue that does arise, however, is that even in the smooth case where we have defined $\Omega_{Y/X}$,  we have not actually established any universality property for the natural derivation $\mc{O}_Y \rightarrow \Omega_{Y/X}$. In fact, the notion of a universal continuous derivation is a bit subtle and may be best understood in a more general context. We avoid this by modifying the structure of the proof to work around the necessity of establishing such a property.  

\begin{proposition}\label{prop.ssadic-ros-tangent}
If $Y/X$ is a smooth morphism of strongly sheafy adic spaces and $\mc{I}$ is a locally free of finite rank $\mc{O}_X$-module, then there is a natural identification of functors on strongly sheafy adic spaces over $Y$
\[ R_{X[\mc{I}]/X}(Y[\mc{I}_Y]/X[\mc{I}])=\mbb{V}(\Omega_{Y/X} \otimes_{\mc{O}_Y} \mc{I}_Y^*).\]
\end{proposition}
\begin{proof}
To give a map from $T\times_X X[\mc{I}] /X $ to $Y/X$ inducing a fixed map $g: T \rightarrow Y$ is the same as to give an extension of the section $s=g \times \Id: T \rightarrow Y \times_X T$ to a section
\[ \tilde{s}: T[\mc{I}_T]/T \rightarrow Y \times_X T / T. \]
Such a map must factor through the first infinitesimal neighborhood $s_{(1)}$ of $s$ in $Y \times_X T$. Thus, by \cref{corollary.section-infinitesimal-pullback}, to give $\tilde{s}$ is the same as to give a map of locally free square-zero thickenings 
\[ T[\mc{I}_T]/T \rightarrow T[s^*\Omega_{Y \times_X T / T}]/T=T[g^*\Omega_{Y/X}]/T. \]
But by \cref{prop.augmented-equiv}, this is the same as to give a map of $\mc{O}_T$-modules $g^*\Omega_{Y/X} \rightarrow \mc{I}_T$. This yields the desired equality
\begin{align*} \Hom(g^* \Omega_{Y/X}, \mc{I}_T)&=\Hom(g^* (\Omega_{Y/X} \otimes_{\mc{O}_Y} \mc{I}_Y^*) , \mc{O}_T) \\
&= \Hom_Y(T/Y, \mbb{V}(\Omega_{Y/X} \otimes_{\mc{O}_Y} \mc{I}_Y^*))\end{align*}
\end{proof}

\begin{example}\label{example.tangent-bundle-via-ros} When $\mc{I}=\mc{O}_X \cdot \epsilon$, we write $X[\epsilon]:=X [\mc{O}_X \cdot \epsilon]$, the thickening obtained by adding a single $\epsilon$ squaring to zero.  In this case, from \cref{prop.schemes-ros-tangent} or \cref{prop.ssadic-ros-tangent} we deduce that, for $T_{Y/X}:=\mbb{V}(\Omega_{Y/X})$ the (geometric) tangent bundle of $Y$ over $X$,
\[ R_{X[\epsilon]/X}(Y[\epsilon]/X[\epsilon])=T_{Y/X}. \]
\end{example}

\section{Inscription}\label{s.inscription}

In this section we introduce the notion of inscription and the construction of tangent bundles for inscribed objects. The general setup starts with a pair $(\mathcal{C}, B)$, where $\mathcal{C}$ is a category and $B$ is a functor from $\mathcal{C}$ to schemes or strongly sheafy adic spaces. In this setup, we define a new category $B^\lf$ whose objects are pairs consisting of an object of $\mathcal{C}$ and a locally free thickening of the associated scheme or strongly sheafy adic space. An inscribed fibered category will then be a fibered category over $B^\lf$ that transforms certain simple coproducts into products. These objects have natural tangent bundles, which are again inscribed fibered categories, and for an inscribed presheaf the tangent bundle has moreover a natural module structure. After developing the basic language, in \cref{ss.inscribed-groups} we explain a simple theory of inscribed groups and their actions, and in \cref{ss.inscribed-mos} we explain a moduli of sections construction of inscribed presheaves that will play a key role in the remainder of the work. In section \cref{ss.restricted-categories-thickenings} we discuss how these results extend to a more general notion where the test objects form only a natural subcategory of $B^\lf$.

The key example for us is when $\mathcal{C}$ is a category of affinoid perfectoid spaces and the functor $B$ is the (adic or scheme-theoretic) Fargues--Fontaine curve. However, we will also find ourselves using variants where $B$ is another natural functor arising in $p$-adic Hodge theory, such as the canonical deformation of a perfectoid space in characteristic zero. Since the basic properties are completely independent of any constructions in $p$-adic Hodge theory, we hope it will be clearer to develop them without mentioning the specific situation. The specific inscribed contexts we are interested in will be described in detail in \cref{s.inscribed-contexts}, along with the results specific to those setups.

\subsection{Inscribed fibered categories}
We adopt the terminology of  \cite[\href{https://stacks.math.columbia.edu/tag/0011}{Tag 0011}]{stacks-project} in our discussion of fibered categories. 

\newcommand{\Spaces}{\mathrm{Spaces}}
Let $\mathcal{C}$ be a category, let $\Spaces$ be either  the category of schemes or strongly sheafy adic spaces and let $B: \mathcal{C} \rightarrow \Spaces$ be a covariant functor. 

\begin{definition}
We write $B^\lf$ for the fibered category over $\mathcal{C}$ whose objects are pairs $(o, \mc{B}/B(o))$ such that $o \in \mathcal{C}$ and $\mc{B}/B(o)$ is a locally free nilpotent thickening of $B(o)/B(o)$ (see \cref{def.thickenings}). The morphisms
\[ \Hom_{B^\lf}\left((o, \mc{B}/B(o)), (o', \mc{B}'/B(o'))\right)\]
are given by the set of pairs consisting of a morphism $o \rightarrow o'$ and a morphism $\mc{B} \rightarrow \mc{B}'$ covering the induced map $B(o) \rightarrow B(o')$. In what follows, we will write an object of $B^\lf$ simply as $\mc{B}$ or $\mc{B}/B(o)$ when it will cause no confusion. 
\end{definition}

\begin{example}
If $\mathcal{C}=\{\ast\}$ and $B(\ast)=\Spec k$ (resp.  $\Spa(k,k^+)$) for $k$ a field (resp. non-archimedean field), then $B^\lf$ is equivalent to the opposite category of nilpotent artininian local $k$-algebras with residue field $k$. 
\end{example}

Note that, given an object $\mc{B}_0/B(o) \in B^\lf$, and locally free nilpotent thickenings $\mc{B}_i/\mc{B}_0$ of $\mc{B}_0/\mc{B}_0$, $i=1,2$, the push-out $\mc{B}_1 \sqcup_{\mc{B}_0} \mc{B}_2$ is naturally a locally free nilpotent thickening of $B(o)/B(o)$ thus gives an object of $B^\lf$. At the level of augmented $\mc{O}_{B(o)}$ algebras as in \cref{prop.augmented-equiv}, this push-out corresponds to the fiber product $\mathcal{O}_{\mc{B}_1} \times_{\mc{O}_{\mc{B}_0}} \mc{O}_{\mc{B}_2}$.  

\begin{definition}[Inscribed fibered categories]\label{def.inscribed-fc}  A fibered category $\mc{S}$ over $B^\lf$ is \emph{inscribed} if, for any $\mc{B}_0 \in B^\lf$ and any pair of locally free nilpotent thickenings $\mc{B}_i/\mc{B}_0$, $i=1,2$, the functor 
\begin{equation}\label{eq.inscribed-def} \mc{S}(\mc{B}_1 \sqcup_{\mc{B}_0} \mc{B}_2) \rightarrow \mc{S}(\mc{B}_1) \times_{\mc{S}(\mc{B}_0)} \mc{S}(\mc{B}_2)\end{equation}
induced by pullback along $\mc{B}_i \hookrightarrow \mc{B}_1 \sqcup_{\mc{B}_0} \mc{B}_2$ is an equivalence (note that this does not depend on the choice of pullbacks for $\mc{S}$). 
\end{definition}

\begin{example} We can view a presheaf $\mc{S}$ on $B^\lf$ as a discrete fibered category, i.e. a category fibered in sets over $B^\lf$. Then $\mc{S}$ is inscribed if and only for any $\mc{B}_i/\mc{B}_0$ as above, \cref{eq.inscribed-def} is a bijection of sets. 
\end{example}

\begin{definition}[Underlying fibered category and trivial inscription]\label{def.underlying-fc-and-trivial-inscription}\hfill
\begin{enumerate}
    \item Pullback along $\mc{C} \rightarrow B^\lf$, $o \mapsto B(o)/B(o)$ is a functor from the (2,1)-category of fibered categories over $B^\lf$ to the (2,1)-category of fibered categories over $\mc{C}$, which we write as $\mc{S} \mapsto \overline{\mc{S}}$. We refer to $\overline{\mc{S}}$ as the underlying fibered category of $\mc{S}$; if $\mc{S}$ is inscribed, we refer to $\mc{S}$ as an inscription on $\overline{\mc{S}}$. 
    \item Pullback along $B^\lf \rightarrow \mc{C}$, $\mc{B}/B(o) \mapsto o$, is a functor from 
   the (2,1)-category of fibered categories over $\mc{C}$ to the (2,1)-category of fibered categories over $B^\lf$, which we write as $S \mapsto S^\triv$. We refer to $S^\triv$ as the trivial inscription on $S$.
\end{enumerate}
\end{definition}

The name trivial inscription is justified by the immediate
\begin{lemma} For $S$ a fibered category over $\mc{C}$,  $S^\triv$ is inscribed. 
\end{lemma}

The composition of $\mc{C} \rightarrow B^\lf \rightarrow \mc{C}$ is the identity functor. We thus obtain a natural isomorphism of functors from fibered categories over $\mc{C}$ to fibered categories over $\mc{C}$, $\overline{(\Box^\triv)} = \Id$. We also have natural transformations $\Id \rightarrow (\overline{\Box})^\triv \rightarrow \Id$ of functors from fibered categories on $B^\lf$ to fibered categories on $B^\lf$ via the natural maps $\mc{S}(\mc{B}/B(o)) \rightarrow \mc{S}(B(o)/B(o)) \rightarrow \mc{S}(\mc{B}/B(o))$ induced by pullbacks along $\mc{B}/B(o) \rightarrow B(o)/B(o) \rightarrow \mc{B}/B(o)$. 

The induced functors
\[ \Hom(S^\triv, \mc{S}) \rightarrow \Hom(S, \overline{\mc{S}})  \textrm{ and } \Hom(\overline{\mc{S}}, S) \rightarrow \Hom(\mc{S}, S^\triv)\]
are equivalences, i.e. trivial inscription and the underlying fibered category are ambidextrously adjoint. In particular, since $\overline{\Box^\triv}=\Id$, we find $\Box^\triv$ is fully faithful. Because of this, we will often drop the superscript $\triv$ and simply treat fibred categories over $\mathcal{C}$ as trivially inscribed fibered categories over $B^\lf$ when it will cause no serious confusion; as in \cref{remark.deRhamStack}, from this perspective we can view $\overline{\mc{S}}$ as a type of de Rham stack for $\mc{S}$. 

We also have the following useful permanence property under limits: 

\begin{lemma}\label{lemma.inscribed-limits}
Inscribed fibered categories are preserved under 2-limits. 
\end{lemma}
\begin{proof}
For $\mc{S}_i$, $j \in J$ a diagram of inscribed fibered categories and $\mc{B}_i/\mc{B}_0$ locally free thickenings of $\mc{B}_0$ as in \cref{def.inscribed-fc} we have 
\begin{align*} ( \lim \mc{S}_j)(\mc{B}_1 \sqcup_{\mc{B}_0} \mc{B}_2) &= \lim ( \mc{S}_j(\mc{B}_1 \sqcup_{\mc{B}_0} \mc{B}_2)) \\
&= \lim (\mc{S}_j(\mc{B}_1) \times_{\mc{S}_j(\mc{B}_0)} \mc{S}_j(\mc{B}_2) )\\
&= (\lim \mc{S}_j)(\mc{B}_1) \times_{ (\lim \mc{S}_j)(\mc{B}_0) } (\lim \mc{S}_j)(\mc{B}_2).
\end{align*}
\end{proof}

\subsection{$\mbb{B}$-modules}
\begin{definition}\label{def.BB} Let $\mbb{B}$ be the presheaf of rings on $B^\lf$ defined by 
\[ \mbb{B}(\mc{B})=H^0(\mc{B}, \mc{O}_{\mc{B}}). \]
\end{definition}

\begin{proposition}\label{prop.BB-inscribed}
$\mbb{B}$ is an inscribed presheaf. 
\end{proposition}
\begin{proof}
To see that $\mbb{B}$ is inscribed, it suffices to note that 
\[ \mbb{B}(\mc{B}/B(o)) = \Hom_{B(o)}(\mc{B}, \mbb{A}^1_{B(o)}) \]
so that it follows from the universal property of the coproduct. 
\end{proof}

\begin{remark}
In \cref{ss.inscribed-mos} we will see that much more general moduli of sections constructions also give rise to inscribed presheaves. 
\end{remark}

A $\mbb{B}$-module over an inscribed presheaf $\mc{S}$ is a morphism of presheaves $\mc{V} \rightarrow \mc{S}$ equipped with a zero section $0_{\mc{V}}:\mc{S} \rightarrow \mc{V}$, a commutative group law $\mc{V} \times_{\mc{S}} \mc{V} \rightarrow \mc{V}$, and an action map $\mbb{B} \times {\mc{V}} \rightarrow \mc{V}$ satisfying the usual properties. A $\mbb{B}$-module $\mc{V}$ over $\mc{S}$ is inscribed if it is inscribed as a presheaf on $B^\lf$.

\subsection{Tangent bundles}
There is a natural functor $B^\lf \rightarrow B^\lf$ of fibered categories over $\mc{C}$ sending $\mc{B}$ to $\mc{B}[\epsilon]$ as in \cref{example.tangent-bundle-via-ros}.  For $\mc{S}$ an inscribed fibered category, we write $T_{\mc{S}}$ for the pullback of $\mc{S}$ along this functor, i.e. $T_{\mc{S}}(\mc{B}/B(o))=\mc{S}(\mc{B}[\epsilon]/B(o))$.  

\begin{lemma}\label{lemma.tangent-bundle-is-inscribed} For $\mc{S}$ an inscribed fibered category, $T_{\mc{S}}$ is an inscribed fibered category. 
\end{lemma}
\begin{proof}
Using the inscribed property of $\mc{S}$, we compute
\begin{align*}(T_\mc{S})(\mc{B}_1 \sqcup_{\mc{B}_0} \mc{B}_2) & = \mc{S}( (\mc{B}_1 \sqcup_{\mc{B}_0} \mc{B}_2)[\epsilon]) \\
&\cong \mc{S}(\mc{B}_1[\epsilon] \sqcup_{\mc{B}_0[\epsilon]} \mc{B}_2[\epsilon]) \\
&\cong \mc{S}(\mc{B}_1[\epsilon]) \times_{\mc{S}(\mc{B}_0[\epsilon])} \mc{S}(\mc{B}_2[\epsilon])\\
&\cong T_\mc{S}(\mc{B}_1) \times_{T_\mc{S}(\mc{B}_0)}  T_\mc{S}(\mc{B}_2).
\end{align*}
\end{proof}

The canonical map $\mc{B} \rightarrow \mc{B}[\epsilon]$ induces a morphism $T_{\mc{S}} \rightarrow \mc{S}$. When $\mc{S}$ is an inscribed presheaf, so is $T_{\mc{S}}$, and we claim this morphism can be upgraded with the natural structure of a $\mbb{B}$-module over $\mc{S}$. Indeed, we will show it obtains a $0$-section $0_{T_{\mc{S}}}: \mc{S} \rightarrow T_{\mc{S}}$ by pullback along the structure maps $\mc{B}[\epsilon] \rightarrow \mc{B}$, an action $a_{T_{\mc{S}}}: \mbb{B}\times\mc{V} \rightarrow \mc{V}$ via the natural identification $\mbb{B}(\mc{B})=\End_{\mc{B}}(\mc{B}[\epsilon])$, and an abelian group structure $+_{\mc{T}_{\mc{S}}}: T_{\mc{S}} \times_{\mc{S}} T_{\mc{S}} \rightarrow T_{\mc{S}}$ by 
\[ \mc{S}(\mc{B}[\epsilon])\times_{\mc{S}(\mc{B})} \mc{S}(\mc{B}[\epsilon]) \rightarrow \mc{S}(\mc{B}[\epsilon]\sqcup_{\mc{B}}\mc{B}[\epsilon]) \rightarrow \mc{S}(\mc{B}[\epsilon]) \]
where the first map is obtained by inverting the bijection coming from the inscribed property and the second map is pullback along the composition of $\mc{B}[\epsilon] \sqcup_{\mc{B}} \mc{B}[\epsilon]=\mc{B}[\epsilon_1, \epsilon_2]$ (where here $\epsilon_1^2=\epsilon_1\epsilon_2=\epsilon_2^2=0$) and the map $\epsilon \mapsto \epsilon_1 + \epsilon_2$. 

\begin{proposition}\label{prop.tangent-bundle-functor}
The assignment 
\[ \mc{S} \mapsto (T_{\mc{S}}/\mc{S}, 0_{T_{\mc{S}}}, +_{T_{\mc{S}}}, a_{T_{\mc{S}}})\] is a functor from the (1-)category of inscribed presheaves to the (1-)category of inscribed presheaves equipped with an inscribed $\mbb{B}$-module. 
\end{proposition}

We will prove the proposition using the following structure: we write $\mc{B}^*\VB$ for the category of pairs $(\mc{B}, \mc{I})$ consisting of a $\mc{B}$ in $B^\lf$ and a locally free of finite rank $\mc{O}_{\mc{B}}$-module $\mc{I}$. There is a functor $F: (\mc{B}^*\VB) \rightarrow B^\lf$ given by 
$(\mc{B},\mc{I}) \mapsto \mc{B}[\mc{I}].$ For $\mc{S}$ an inscribed fibered category, we can thus consider $F^*\mc{S}$. Then, for example, $T_{\mc{S}}$ is the pullback  
\[ (\mc{B} \mapsto (\mc{B},\mc{O}_{\mc{B}} \cdot \epsilon))^* F^* \mc{S},\]
and $0_{T_{\mc{S}}}$ is obtained from $\mc{B} \mapsto (\mc{O}_{\mc{B}}\cdot \epsilon \xrightarrow{0} 0)$. The structures defining the group and $\mbb{B}$-module structure only depend on the restriction to the full fibered subcategory $\mc{O}_{\mc{B}}^\oplus$ whose objects are pairs $(\mc{B}, \mc{V})$ where $\mc{V}$ is a finite free module over $\mc{O}_{\mc{B}}$. We identify this with the category $\mbb{B}^\oplus$ whose objects are pairs $(\mc{B}, M)$, where $M$ is a finite free module over $\mbb{B}(\mc{B})$. 

The result will then follow from the following more general statement: Let $\mc{A}$ be a category, let $\mbb{A}$ be a presheaf of rings on $\mc{A}$, and let $\mbb{A}^\oplus$ be the fibered category of whose objects are pairs $(A, M)$ for $A \in \mc{A}$ and $M$ a finite free $\mbb{A}(A)$-module. Note that the fiber of $\mbb{A}^\oplus$ over $A$ is the opposite category of finite free $\mbb{A}(A)$-modules, so the restriction of any presheaf $\mc{F}$ to this fiber can be viewed as a covariant functor $\mc{F}_A$ from finite free $\mbb{A}(A)$-modules to sets. We say a presheaf $\mc{F}$ on $\mbb{A}^\oplus$ is product-preserving if each of its restrictions $\mc{F}_A$ preserve products over the final object, i.e. if the natural map $\mc{F}_A(M_1 \times M_2) \rightarrow \mc{F}_A(M_1) \times_{\mc{F}_A(0)} \mc{F}_A(M_2)$ is a bijection for any $M_1$, $M_2$. Given a product preserving $\mc{F}$, we write $\mc{F}_i$ for the presheaf on $\mc{A}$ sending $A$ to $\mc{F}_i(\mbb{A}(A)^i)$. Then, there are natural maps 
\[ \mc{F}_1 \rightarrow \mc{F}_0, 0_{\mc{F}}: \mc{F}_0 \rightarrow \mc{F}_1, m_{\mc{F}}: \mc{F}_1 \times_{\mc{F}_0} \mc{F}_1, \textrm{ and } a_{\mc{F}}: \mbb{A} \times \mc{F}_1 \rightarrow \mc{F}_1 \]
defined on the fiber over $A$ as follows, writing $\mbb{A}(A)=:R$,
\begin{enumerate}
    \item We define the structure map $\mc{F}_{1}(A)=\mc{F}_A(R) \rightarrow \mc{F}_A(0)=\mc{F}_0(A)$ by applying $\mc{F}_A$ to the final map $R \rightarrow 0$.
    \item We define $0_{\mc{F},A}: \mc{F}_0(A)=\mc{F}_A(0) \rightarrow \mc{F}_A(R)=\mc{F}_1(A)$ by applying $\mc{F}_A$ to the initial map $0 \rightarrow R$.
    \item We define $m_{\mc{F},A}$ by composing the inverse of the bijection 
    \[ \mc{F}_A(R^2) \xrightarrow{ \mc{F}_A(\begin{bmatrix}1 & 0\end{bmatrix}) \times  \mc{F}_A(\begin{bmatrix}0 & 1\end{bmatrix})} \mc{F}(R) \times_{\mc{F}(0)}\mc{F}(R) = (\mc{F}_1 \times_{\mc{F}_0} \mc{F}_1)(A) \]
    with
    \[ \mc{F}_A(R^2) \xrightarrow{\mc{F}_A(\begin{bmatrix} 1 & 1 \end{bmatrix})} \mc{F}_A(R)=\mc{F}_1(A).\]
    \item We define the action under $a_{\mc{F},A}$ of $r \in \mbb{A}(A)=R$ on $\mc{F}_1(A)=\mc{F}_A(R)$ to be given by $\mc{F}_A([r])$, where $[r]$ is the $1\times 1$ matrix viewed as the homomorphism from $R$ to $R$ given by left multiplication by $r$. 
\end{enumerate}

\begin{lemma}\label{lemma.equiv-module-category}The assignment $\mc{F} \mapsto (\mc{F}_1/\mc{F}_0, 0_{\mc{F}}, m_{\mc{F}}, a_{\mc{F}})$ is an equivalence of categories between product preserving presheaves on $\mbb{A}^\oplus$ and the category of pairs consisting of a presheaf $\mc{F}_0$ on $\mc{A}$ and an $\mbb{A}$-module $\mc{F}_1/\mc{F}_0$. 
\end{lemma}
\begin{proof}
That the data defines an $\mbb{A}$-module follows by arguing, for each $A \in \mc{A}$, fiberwise over $\mc{F}_0(A)$ using the usual result in the deformation theory of functors (see, e.g. \cite[\href{https://stacks.math.columbia.edu/tag/06I6}{Tag 06I6}]{stacks-project}). In fact, one can just rewrite the arguments in this setting: for example, the commutativity of the group law follows from the following commutative diagram
% https://q.uiver.app/#q=WzAsNSxbMCwxLCJcXG1je0Z9X0EoUl4yKSJdLFsyLDEsIlxcbWN7Rn1fQShSXjIpIl0sWzAsMCwiXFxtY3tGfV9BKFIpXFx0aW1lc1xcbWN7Rn1fQShSKSJdLFsyLDAsIlxcbWN7Rn1fQShSKVxcdGltZXNcXG1je0Z9X0EoUikiXSxbMSwyLCJcXG1je0Z9X0EoUikiXSxbMCwyLCIoXFxtY3tGfV9BKFsxXFw7MF0pLCBcXG1je0Z9X0EoWzBcXDsxXSkpIl0sWzEsMywiKFxcbWN7Rn1fQShbMVxcOzBdKSwgXFxtY3tGfV9BKFswXFw7MV0pKSIsMl0sWzIsMywiKHNfMSxzXzIpXFxtYXBzdG8oc18yLHNfMSkiXSxbMCwxXSxbMCw0LCJcXG1je0Z9X0EoWzFcXDsxXSkiLDFdLFsxLDQsIlxcbWN7Rn1fQShbMVxcOzFdKSIsMV1d
\[\begin{tikzcd}
	{\mc{F}_A(R)\times\mc{F}_A(R)} && {\mc{F}_A(R)\times\mc{F}_A(R)} \\
	{\mc{F}_A(R^2)} && {\mc{F}_A(R^2)} \\
	& {\mc{F}_A(R)}
	\arrow["{(s_1,s_2)\mapsto(s_2,s_1)}", from=1-1, to=1-3]
	\arrow["{(\mc{F}_A([1\;0]), \mc{F}_A([0\;1]))}", from=2-1, to=1-1]
	\arrow["{ \mc{F}_A(\left[\substack{0\;1 \\ 1 \; 0}\right])}", from=2-1, to=2-3]
	\arrow["{\mc{F}_A([1\;1])}"{description}, from=2-1, to=3-2]
	\arrow["{(\mc{F}_A([1\;0]), \mc{F}_A([0\;1]))}"', from=2-3, to=1-3]
	\arrow["{\mc{F}_A([1\;1])}"{description}, from=2-3, to=3-2]
\end{tikzcd}\]
and similar diagrams establish the other $R$-module properties. 

An inverse functor can be constructed by sending $\mc{F}_1/\mc{F}_0$ to the presheaf $\mc{F}$ sending $(A, M)$ to the set of pairs $(s,m)$ where $s \in \mc{F}_0(A)$ and $m \in M \otimes_{\mbb{A}(A)} (\mc{F}_1(A) \times_{\mc{F}_0(A)} s)$; we omit the verification that this is an inverse. 
\end{proof}

\begin{proof}[Proof of \cref{prop.tangent-bundle-functor}]
The functor in \cref{prop.tangent-bundle-functor} is given by first pulling back along $\mbb{B}^\oplus \rightarrow \mc{B}^*\VB \rightarrow B^\lf$ and then applying the functor of \cref{lemma.equiv-module-category}. That the pullback to $\mbb{B}^\oplus$ is product preserving follows from the inscribed property.  This shows we obtain a functor from inscribed presheaves to inscribed presheaves equipped with a $\mbb{B}$-module, as claimed. That this $\mbb{B}$-module is also inscribed follows from \cref{lemma.tangent-bundle-is-inscribed}. 
\end{proof}

\begin{remark}
For a general inscribed fibered category $\mc{S}$, it would thus be natural to try to replace the consideration of the tangent bundle with its group structure that we used for inscribed $v$-sheaves with the consideration of the pullback of $\mc{S}$ to $\mbb{B}^\oplus$. We do not consider this further here, but we note it should give some type of Picard stack over $\mc{S}$.  
\end{remark}

\begin{remark}\label{remark.tangent-bundle-equivalent-data}
If $\mc{S}$ is an inscribed presheaf, then $\overline{T_{\mc{S}}}/\overline{\mc{S}}$ is a $\overline{\mbb{B}}$-module. To define it, it suffices to know just the restriction of $\mc{S}$ to the category of finite free square-zero thickenings, and \cref{lemma.equiv-module-category} shows that knowledge of this restriction is in fact equivalent to knowledge of the $\overline{\mathbb{B}}$-module $\overline{T_{\mc{S}}}/\overline{\mc{S}}$. 
\end{remark}

\subsection{Topologies}
Let $\tau$ be a Grothendieck topology on $\mathcal{C}$. Then, since $B^\lf$ is fibered over $\mathcal{C}$, there is a natural $\tau$-topology also on $B^\lf$: a family of morphisms with fixed target is a cover if and only if is the pullback of a cover in $\mathcal{C}$. 

\begin{definition}
An \emph{inscribed} $\tau$-prestack/stack/sheaf\footnote{A fibered category is a prestack if morphisms between objects are sheaves, and a prestack is a stack if objects also satisfy descent. A presheaf, viewed as a fibered category with discrete fibers, is automatically a prestack, and is a stack if and only if it is a sheaf.} is an inscribed fibered category that is also a $\tau$-prestack/stack/sheaf on $B^\lf$. When the topology is implicit, we will often drop $\tau$ from the notation. 
\end{definition}

The following lemmas are immediate from the definitions. We will use them implicitly with no further comment.  

\begin{lemma}\label{lemma.adjunction-prestacks-stacks-sheaves} If $\mc{S}$ is an inscribed $\tau-$prestack/stack/sheaf, then $\overline{\mc{S}}$ is a prestack/stack/sheaf on $\mathcal{C}_\tau$, and if $S$ is a prestack/stack/sheaf on $\mathcal{C}_\tau$, then $S^\triv$ is an inscribed $\tau$-prestack/stack/sheaf. 
\end{lemma}

\begin{lemma}\label{lemma.tangent-prestacks-stacks-sheaves}
    If $\mc{S}$ is an inscribed prestack/stack/sheaf, then $T_{\mc{S}}$ is an inscribed prestack/stack/sheaf. 
\end{lemma}

\subsection{Inscribed abelian sheaves and $\mbb{B}$-modules}
We fix a Grothendieck topology $\tau$ on $\mathcal{C}$. For $\mc{S}$ an inscribed sheaf, we view an abelian sheaf on $\mc{S}$ as a morphism of sheaves on $B^\lf$, $\mc{V}\rightarrow \mc{S}$, with a zero section $0_{\mc{V}}: \mc{S} \rightarrow \mc{V}$ and an addition law $\mc{V} \times_{\mc{S}} \mc{V} \rightarrow \mc{V}$ satisfying the usual compatibilities. An abelian sheaf $\mc{V}$ on $\mc{S}$ is inscribed if $\mc{V}$ is inscribed as a presheaf on $B^\lf$.

\begin{proposition}\label{prop.inscribed-abelian-sheaves-abelian-cat}
For $\mc{S}$ an inscribed sheaf, the category of inscribed abelian sheaves on $\mc{S}$ is a full abelian subcategory of the category of abelian sheaves on $\mc{S}$. 
\end{proposition}
\begin{proof}
The zero object $\mc{S}/\mc{S}$ is inscribed, and \cref{lemma.inscribed-limits} implies that finite products of inscribed abelian sheaves are inscribed and that kernels of maps of inscribed abelian sheaves are inscribed. It remains to see that cokernels are inscribed. Because kernels are inscribed, it suffices to see that quotients are also inscribed. 

We will use the following notation: for $\mc{F}$ a presheaf on $B^\lf$ and $\mc{B}/B(o) \in B^\lf$, we write $\mc{F}_{\mc{B}}$ for the presheaf on $\mathcal{C}/o$ sending $o'\rightarrow o$ to $\mc{F}(\mc{B} \times_{B(o)} B(o'))$.

Let $\mc{V} \subseteq\mc{W}$ be inscribed abelian sheaves over $\mc{S}$. Suppose given locally free nilpotent thickenings $\mc{B}_i / \mc{B}_0$, $i=1,2$ lying over $o \in \mc{C}$, such that $\mc{B}=\mc{B}_1 \sqcup_{\mc{B}_0} \mc{B}_2$. We then have 
\begin{align*}
(\mc{W}/\mc{V})_{\mc{B}}&=\mc{W}_{\mc{B}}/\mc{V}_{\mc{B}} \\
&= \mc{W}_{\mc{B}_1} \times_{\mc{W}_{\mc{B}_0}} \mc{W}_{\mc{B}_2} / \left(V_{\mc{B}_1} \times_{\mc{V}_{\mc{B}_0}} \mc{V}_{\mc{B}_2}\right) \\
&= \mc{W}_{\mc{B}_1}/\mc{V}_{\mc{B}_1} \times_{\mc{W}_{\mc{B}_0}/\mc{V}_{\mc{B}_0}} \mc{W}_{\mc{B}_2}/\mc{V}_{\mc{B}_2}\\
&= (\mc{W}/\mc{V})_{\mc{B}_1} \times_{(\mc{W}/\mc{V})_{\mc{B}_0}} (\mc{W}/\mc{V})_{\mc{B}_2}.\end{align*}
Only the third equality requires some justification: The natural map 
\[ \left( \mc{W}_{\mc{B}_1} \times_{\mc{W}_{\mc{B}_0}} \mc{W}_{\mc{B}_2} \right) \rightarrow \mc{W}_{\mc{B}_1}/\mc{V}_{\mc{B}_1} \times_{\mc{W}_{\mc{B}_0}/\mc{V}_{\mc{B}_0}} \mc{W}_{\mc{B}_2} / \mc{V}_{\mc{B}_2} \]
has kernel $\mc{V}_{\mc{B}_1} \times_{\mc{V}_{\mc{B}_0}} \mc{V}_{\mc{B}_2}$, so it remains to show it is surjective. But, for $(a,b)$ in the image, if we choose on some cover preimages $\tilde{a}$ and $\tilde{b}$, then the images of $\tilde{a}$ and $\tilde{b}$ in $\mc{W}_{\mc{B}_0}$ differ by an element of $\mc{V}_{\mc{B}_0}$. Since $\mc{V}_{\mc{B}_1} \rightarrow \mc{V}_{\mc{B}_0}$ is surjective already as a map of presheaves (it admits a section), we may modify the lift $\tilde{a}$ so that $\tilde{a}$ and $\tilde{b}$ have the same image in $\mc{W}_{\mc{B}_0}$, so that $(\tilde{a}, \tilde{b})$ is a section of $\mc{W}_{\mc{B}_1} \times_{\mc{W}_{\mc{B}_0}} \mc{W}_{\mc{B}_2}$ mapping to $(a,b)$. 
\end{proof}

\begin{corollary}\label{cor.inscribed-B-abelian}
    For $\mc{S}$ an inscribed sheaf, the category of inscribed sheaves of $\mbb{B}$-modules on $\mc{S}$ is a full abelian subcategory of the category of sheaves of $\mbb{B}$-modules on $\mc{S}$.
\end{corollary}

\subsection{Relative tangent bundles and normal bundles}

If $f: \mc{Z} \rightarrow \mc{S}$ is a morphism of inscribed presheaves, we obtain by \cref{prop.tangent-bundle-functor} a morphism of $\mbb{B}$-modules on $\mc{Z}$, 
$df: T_{\mc{Z}} \rightarrow f^* T_{\mc{S}}.$
\begin{definition}\label{def.rel-tan-norm}
If $f:\mc{Z} \rightarrow \mc{S}$ is a morphism of inscribed presheaves, we let $T_{\mc{Z}/\mc{S}}:= \mr{ker}\, df$. If $f:\mc{Z} \rightarrow \mc{S}$ is a morphism of inscribed sheaves, we let $N_{\mc{Z}/\mc{S}}:=\mr{coker}\, df$. 
\end{definition}

In the setting of \cref{def.rel-tan-norm}, it follows from \cref{cor.inscribed-B-abelian} that $T_{\mc{Z}/\mc{S}}$ and $N_{\mc{Z}/\mc{S}}$ are inscribed $\mbb{B}$-modules over $\mc{Z}$. 

\begin{example}\label{cref.example-relative-tangent-bundle-triv-inscribed}
If $S$ is a presheaf on $\mc{C}$ and $\mc{Z}/S^\triv$, then $T_{S^\triv}=0$ so $T_{\mc{Z}/S^\triv}=T_{\mc{Z}}$. In particular, for any inscribed presheaf $\mc{S}$, $T_{\mc{S}}=T_{\mc{S}/B^\lf}=T_{\mc{S}/\overline{\mc{S}}}$, where in the middle $B^\lf$ is treated as the trivial presheaf on $B^\lf$. 
\end{example}

\subsubsection{Variant}\label{sss.variant-def-discrete}
Suppose $\mathcal{Z} \rightarrow \mathcal{S}$ is a morphism of inscribed prestacks. Then we can define $\mathcal T_{\mathcal{Z}/\mathcal{S}}(\mathcal{B}) = \mathcal{S}(\mathcal{B}[\epsilon])\times_{\mathcal{Z}(\mathcal{B}[\epsilon])} \mathcal{Z}$, and when $\mathcal{S}$ and $\mathcal{Z}$ are presheaves, this is equivalent to the previous definition. Moreover, if $\mathcal{Z} \rightarrow \mathcal{S}$ has discrete fibers, then $T_{\mathcal{Z}/\mathcal{S}}$ has a natural $\mathbb{B}$-module structure by the same construction as in \cref{prop.tangent-bundle-functor}. 

\begin{definition}
    An inscribed prestack $\mathcal{S}$ is deformation discrete if $\mathcal{S} \rightarrow \overline{\mc{S}}$ has discrete fibers. In this case we define $\mathcal{T}_{\mathcal{S}}:=\mathcal{T}_{\mathcal{S}/\overline{\mc{S}}}.$ 
\end{definition}

This construction gives a $\mathbb{B}$-module over $\mathcal{S}$ that is functorial in maps of deformation discrete inscribed prestacks. When $\mathcal{S}$ is an inscribed presheaf, it is equivalent to the definition of $T_{\mathcal{S}}$ given above; this more general definition was used implicitly in the discussion of the Hodge, lattice Hodge, and Liu-Zhu period maps in \cref{ss.intro-main-results}.

\subsection{Inscribed groups}\label{ss.inscribed-groups}

For $\mc{S}$ an inscribed presheaf, an inscribed group over $\mc{S}$ is a map of inscribed presheaves $\mc{G}/\mc{S}$ equipped with an identity section $e: \mc{S} \rightarrow \mc{G}$, an inverse map $\mc{G} \rightarrow \mc{G}$, and a multiplication law $\mc{G} \times_S \mc{G} \rightarrow \mc{G}$ satisfying the usual compatibilities.

\begin{example} Inscribed $\mbb{B}$-modules are, in particular, inscribed (abelian) groups. 
\end{example}

\begin{lemma}
    For $\mc{S}$ an inscribed sheaf and $\mc{G}/\mc{S}$ an inscribed group, $T_{\mc{G}/\mc{S}}  / \mc{S}$ admits a canonical inscribed group structure. The structure map $T_{\mc{G}/\mc{S}} \rightarrow \mc{G}$ is a surjective homomorphism. It is canonically split by the zero section $\mc{G} \rightarrow T_{\mc{G}/\mc{S}}$, and on the kernel $\Lie \mc{G}=e^*T_{\mc{G}/\mc{S}}$, the two natural group structures over $\mc{S}$ agree (one as a subgroup of $T_{\mc{G}/\mc{S}}$, and the other by pull-back of the $\mbb{B}$-module structure on $T_{\mc{G}/\mc{S}}$ along $e: \mc{S} \rightarrow \mc{G}$). In particular, $T_{\mc{G}/\mc{S}} =\mc{G} \ltimes \Lie \mc{G}$ for the natural $\mbb{B}$-linear conjugation action on $\Lie \mc{G}$. 
\end{lemma}
\begin{proof}
For any $\mc{B}\in B^\lf$, $\mc{G}(\mc{B}[\epsilon])/\mc{S}(\mc{B}[\epsilon])$ is a group over $\mc{S}(\mc{B}[\epsilon])$, and $T_{\mc{G}/\mc{S}}(\mc{B}[\epsilon])$ is the pullback of this to a group over $\mc{S}(\mc{B})$ along $0_{T_{\mc{S}}}: \mc{S}(\mc{B}) \rightarrow \mc{S}(\mc{B}[\epsilon]).$  This shows $T_{\mc{G}/\mc{S}}$ admits a canonical inscribed group structure, and it is evident that the structure map to $\mc{G}/\mc{S}$ is a surjective group homomorphism split by the zero section. 

To see the two group structures on $\Lie \mc{G}$ agree, we first note that the subgroup structure can be written as $dm_e: \Lie \mc{G} \times_{\mc{S}} \Lie \mc{G} \rightarrow \Lie{\mc{G}}$, where $m: \mc{G} \times_{\mc{S}} \mc{G} \rightarrow \mc{G}$ is the multiplication map. This is a $\mbb{B}$-linear map; in particular we find the two group structures agree because
\[ dm_e( (a,b) )=dm_e( (a,0) + (0,b) )= dm_e((a,0)) + dm_e ((0,b))=a+b.\]
 
\end{proof}

Suppose $f: \mc{Z} \rightarrow \mc{S}$ is a map of inscribed presheaves, and $\mc{G}/\mc{S}$ is an inscribed group over $\mc{S}$. A (right) action of $\mc{G}$ on $\mc{Z}$ is a map $a: \mc{Z} \times_{\mc{S}} \mc{G} \rightarrow \mc{Z}$ satisfying the usual axioms. We note that a (right) action of $\mc{G}$ on $\mc{Z}$ induces a (right) action of $T_{\mc{G}/\mc{S}}$ on $T_{\mc{Z}/\mc{S}}$. We write $da_e$ for the induced map $f^* \Lie \mc{G} \rightarrow T_{\mc{Z}/\mc{S}}$ obtained by pulling back $da$ along $\Id_{\mc{Z}} \times e$. Concretely, given a tangent vector $t: \mc{B}[\epsilon] \rightarrow \mc{G}$ restricting to $e: \mc{B} \rightarrow \mc{G}$ and a $z: \mc{B} \rightarrow \mc{Z}$, $da_e(t)=\tilde{z} \cdot t$, where $\tilde{z}$ is the constant extension of $z$ to a $\mc{B}[\epsilon]$-point of $\mc{Z}$. 

\begin{proposition}\label{prop.inscribed-quotient}
We fix a topology $\tau$, and let $a: \mc{Z} \times_{\mc{S}} \mc{G} \rightarrow \mc{Z}$ be a faithful right action of an inscribed group sheaf $\mc{G}$ over an inscribed sheaf $\mc{S}$ on an inscribed sheaf $\mc{Z}$ over $\mc{S}$. Then the quotient $\mc{Z}/\mc{G}$ is an inscribed sheaf over $\mc{S}$ and $T_{(\mc{Z}/\mc{G})/\mc{S}}=T_{\mc{Z}/\mc{S}}/T_{\mc{G}/\mc{S}}$. In particular, writing $\pi:\mc{Z} \rightarrow \mc{Z}/\mc{G}$ for the quotient map, there is a canonical $\mc{G}$-equivariant identification 
\[ \pi^* T_{(\mc{Z}/\mc{G})/\mc{S}}=\mr{coker}(da_e).\]  
\end{proposition}
\begin{proof} Arguing as in the proof of \cref{prop.inscribed-abelian-sheaves-abelian-cat},
\[ \mc{Z}_{\mc{B}_1 \sqcup_{\mc{B}_0} \mc{B}_2} / \mc{G}_{\mc{B}_1 \sqcup_{\mc{B}_0} \mc{B}_2}=\mc{Z}_{\mc{B}_1} \times_{\mc{Z}_{\mc{B}_0}} \mc{Z}_{\mc{B}_2} / \mc{G}_{\mc{B}_1} \times_{\mc{G}_{\mc{B}_0}} \mc{G}_{\mc{B}_2} = \mc{Z}_{\mc{B}_1}/\mc{G}_{\mc{B}_1} \times_{\mc{Z}_{\mc{B}_0}/\mc{G}_{\mc{B}_0}} \mc{Z}_{\mc{B}_2}/\mc{G}_{\mc{B}_2} \]
we conclude that $\mc{Z}/\mc{G}$ is inscribed. The rest is immediate. 
\end{proof}

\subsection{Moduli of sections}\label{ss.inscribed-mos}

In the following definition we sometimes interpret $\mc{B}$ as the functor from $B^\lf$ to $\Spaces$ sending $\mc{B}/B(o)$ to $\mc{B}$. 

\begin{definition} \hfill
\begin{enumerate}
\item We write $\Sm_\mc{B}$ for the category whose objects are pairs $(\mc{B}, Z/\mc{B})$ where $\mc{B}$ is an object in $B^\lf$ and $Z/\mc{B}$ is a smooth morphism of schemes/strongly sheaf adic spaces 
    \item For $\mc{S}$ an inscribed presheaf, a smooth space over $\mc{B}$ on $\mc{S}$ is a map of fibered categories $Z: \mc{S} \rightarrow \Sm_{\mc{B}}$. 
    \item For $\mc{S}$ an inscribed presheaf and $Z$ a smooth space over $\mc{B}$ on $\mc{S}$, we define $\mc{B}^*h_{Z}$ to be the presheaf sending $\mc{B} \in B^\lf$ to the set of isomorphism classes of pairs $(s,f)$ where $s \in \mc{S}(\mc{B})$ and $f \in \Hom_{\mc{B}}(\mc{B}/\mc{B}, \mc{Z}(s)/\mc{B}).$ 
 \end{enumerate}
\end{definition}

\begin{proposition}\label{prop.mos-inscribed-presheaf}For $\mc{S}$ an inscribed presheaf and $Z$ a smooth space over $\mc{B}$ on $\mc{S}$, $\mc{B}^*h_{Z}$ is an inscribed presheaf over $\mc{S}$, and there is a canonical identification of $\mbb{B}$-modules $T_{\mc{B}^*h_{Z}/\mc{S}}=\mc{B}^*h_{T_{Z/\mc{B}}}$
where $T_{Z/\mc{B}}$ is the smooth space over $\mc{B}$ on $\mc{S}$ sending $s \in S(\mc{B})$ to the geometric tangent bundle $T_{Z(s)/\mc{B}}$ of \cref{example.tangent-bundle-via-ros}. 
\end{proposition}
\begin{proof}
To see that it is inscribed, fix a $\mc{B}_0 \rightarrow \mc{S}$ and locally free nilpotent thickenings $\mc{B}_1, \mc{B}_2$ of $\mc{B}_0$. Then, for $\mc{B}:=\mc{B}_1 \sqcup_{\mc{B}_0} \mc{B}_2$, we want to show 
\[ h_Z(\mc{B}) \rightarrow h_Z(\mc{B}_1) \times_{h_Z(\mc{B}_0)} h_Z(\mc{B}_2) \]
is a bijection. Since $\mc{S}$ is inscribed it suffices to work over an element of $s\in \mc{S}(\mc{B})$ corresponding to $(s_1,s_2) \in \mc{S}(\mc{B}_1) \times \mc{S}(\mc{B}_2)$ lying over a common $s_0$ in ${\mc{S}(\mc{B}_0)}$, and we need to show 
\[ h_Z(\mc{B})_s = h_Z(\mc{B}_1)_{s_1} \times_{h_Z(\mc{B}_0)_{s_0}} h_Z(\mc{B}_2)_{s_2}. \]
The elements of $h_Z(\mc{B})_s$ are given by
\[ \Hom_{\mc{B}}(\mc{B}, Z(s))=\Hom_{\Sch/\mc{B}}(\mc{B}_1 \sqcup_{\mc{B}_0} \mc{B}_2, Z(s)). \]
By the definition of a coproduct, this is equal to
\[ \Hom_{\mc{B}}(\mc{B}_1, Z(s)) \times_{\Hom_{\mc{B}}(\mc{B}_0, Z(s))} \Hom_{\mc{B}}(\mc{B}_2, Z(s)).\]
But, since $Z(s)\times_{\mc{B}} {\mc{B}_i}=Z(s_i)$, this is an element of
\[ \Hom_{\mc{B}_1}(\mc{B}_1, Z(s_1)) \times_{\Hom_{\mc{B}_0}(\mc{B}_0, Z(s_0))} \Hom_{\mc{B}_2}(\mc{B}_2, Z(s_2))\]
which is equal, as desired, to 
\[ h_Z(\mc{B}_1)_{s_1} \times_{h_Z(\mc{B}_0)_{s_0}} h_Z(\mc{B}_2)_{s_2}. \]

Finally, the identity $T_{\mc{B}^*h_Z/\mc{S}} = \mc{B}^*h_{T_{Z/\mc{B}}}$ follows from \cref{example.tangent-bundle-via-ros}.
\end{proof}

\begin{remark}
When $\Spaces$ is the category of schemes, we can define more generally a scheme over $\mc{B}$ on $\mc{S}$ 
and \cref{prop.mos-inscribed-presheaf} holds in this generality; however, we will not need this generality in what follows. When $\Spaces$ is the category of strongly sheafy adic spaces, we must restrict to smooth morphisms for two technical reasons: first, we only know we have fiber products for smooth morphisms (we used this fact implicitly in the claim that $\Sm_{\mc{B}}$ is a fibered category), and more seriously, in this case we have only defined tangent bundles for smooth morphisms. 
\end{remark}

\subsection{Restricted categories of thickenings}\label{ss.restricted-categories-thickenings}
As explained in the introduction, some of our main results will only apply after passing to a restricted category of locally free thickenings of Fargues--Fontaine curves where we impose a natural slope condition. We briefly discuss a general framework for this kind of restriction.

Suppose $B^{\bullet}\subseteq B^\lf$ is a full fibered subcategory such that
\begin{enumerate}
\item For any $o \in \mathcal{C}$, $B(o)/B(o) \in B^\bullet$. 
\item For any $\mc{B}/B(o) \in B^\bullet$, and any finite free $\mbb{B}(\mc{B})$-module $M$, the finite locally free square-zero thickening $\mc{B}[M]/B(o)$ of $\mc{B}/B(o)$ is also an object of $B^\bullet$. 
\end{enumerate}

In this context, we can define a $B^\bullet$-inscribed presheaves / fibered categories / sheaves / prestacks  / stacks by replacing $B^\lf$ everywhere above with $B^\bullet$ and only requiring \cref{eq.inscribed-def} for those push-outs that are contained in $B^\bullet$. The definitions, structures, and results, in the previous sections for $B^\lf$-inscribed objects then make sense also for $B^\bullet$-inscribed objects. 

We will typically apply this only to $B^\bullet$-inscribed objects that are obtained by restricting a $B^\lf$-inscribed object, but for which there is some natural property that only holds only over $\bullet$-thickenings. We note that, in this case, the tangent bundle of the restricted $B^\bullet$ inscribed $v$-sheaf is simply the restriction of the tangent bundle of the $B^\lf$ inscribed $v$-sheaf. 

\begin{example}
Let $B^\epsilon \subseteq B^\lf$ denote the full subcategory whose objects are those finite locally free thickenings isomorphic to $B(o)[M]/B(o)$ for $M$ a finite locally free $\mbb{B}(B(o)/B(o))=\mc{O}(B(o))$-module. Then, as in \cref{remark.tangent-bundle-equivalent-data}, the category of $B^\epsilon$-inscribed presheaves is equivalent, via $\mc{S} \mapsto (\overline{\mc{S}}, \overline{T_{\mc{S}}})$, to the category of presheaves on $\mc{C}$ equipped with a presheaf of $\mbb{B}$-modules. 
\end{example}

\section{Inscribed contexts}\label{s.inscribed-contexts}

In this section we describe the pairs $(\mathcal{C},B)$ consisting of a category $\mathcal{C}$ and a functor $B$ from $\mathcal{C}$ to schemes or strongly sheafy adic spaces to which we will apply the formalism of \cref{s.inscription} in the remainder of this work. After recalling some constructions of adic spaces and schemes attached to perfectoid spaces in \cref{ss.perf-untilt-canonical-thickenings} and \cref{ss.Fargues--Fontaine-curves}, in \cref{ss.the-pairs-we-use} we define these pairs and state their  basic properties (see \cref{prop.inscribed-pairs}).  In \cref{ss.pairs-moduli-of-sections} we revisit the moduli of sections construction of \cref{ss.inscribed-mos} in these contexts. In particular, we verify that it gives rise to inscribed $v$-sheaves in the cases of our main interest. We note that the algebraic moduli of sections construction, which applies only to affine schemes, is relatively straightforward, and suffices for many of our computations. The analytic moduli of sections construction gives a common generalization and inscribed upgrade of the diamonds associated to smooth rigid analytic spaces and Fargues--Scholze moduli of sections in a way that incorporates tangent bundles.

\subsection{Perfectoid spaces, untilts, and canonical thickenings}\label{ss.perf-untilt-canonical-thickenings}
Let $\Perf$ be the category of perfectoid spaces in characteristic $p$ and $\AffPerf \subseteq \Perf$ the subcategory of affinoid perfectoid spaces. We equip the categories $\AffPerf \subseteq \Perf$ with the $v$-topology of \cite[Definition 8.1]{Scholze.EtaleCohomologyOfDiamonds}. We note that, to define a $v$-stack on $\Perf$, it suffices to give its values on $\AffPerf$.

Recall that $\Spd \mbb{Q}_p$ is the $v$-sheaf on $\Perf$ sending $P$ to the set of isomophism classes of untilts $P^\sharp/\Spa \mbb{Q}_p$. Given such an untilt $P^\sharp/\Spa \mbb{Q}_p$, there is a canonical infinitesimal thickening for each $i \geq 0$, $P^\sharp_{(i)}$: When $P^\sharp=\Spa(A,A^+)$ is affinoid perfectoid, 
\[ P^\sharp_{(i)}:=\Spa(A_{(i)}, A_{(i)}^+) \]
where the Huber pair $(A_{(i)}, A_{(i)}^+)$ is defined as follows. First, we write $\mbb{B}^+_\dR$ and $A_\mr{inf}$ for the usual Fontaine functors, $\theta: \mbb{B}^+_\dR(A) \twoheadrightarrow A$ for the usual Fontaine map, whose kernel is a Cartier divisor, and $\Fil^j \mbb{B}^+_\dR(A)=(\mr{Ker} \theta)^j$. Then 
\[ A_{(i)}:=\mbb{B}^+_\dR(A)/\Fil^{i+1}\mbb{B}^+_\dR(A),\; A_{(i)}^+=\theta^{-1}(A^+), \]
and $A_{(i)}$ is equipped with the $f$-adic topology such that a ring of definition is given by the image of $A_\mr{inf}(A^{+,\flat})$. Note that $(A_{(i)},A_{(i)}^+)$, by construction, lies over $(\mbb{Q}_p, \mbb{Z}_p)$. 

\begin{lemma}
The Huber pair $(A_{(i)}, A_{(i)}^+)$ is strongly sheafy. 
\end{lemma}
\begin{proof}
For any $n \geq 0$, we write $P_{(i),n} = \Spa(A_{(i)}\langle t_1, \ldots, t_n\rangle, A^+_{(i)}\langle t_1, \ldots, t_n \rangle).$ Note that because these are nilpotent thickenings, for any $i \geq 0$, $P_{(i),n}$ has the same underlying topological space and valuations as $P_{(0),n}$ and rational opens are naturally identified. 

We must show that the structure presheaf is a sheaf for each $P_{(i),n}$. We argue by induction on $i$. When $i=0$, this holds since perfectoid spaces are strongly sheafy. If we fix a generator $\xi$ for $\ker \theta$, then for any $i \geq 1$ we obtain an exact sequence of presheaves
\[ 0 \rightarrow \mc{O}_{P_{(i-1),n}}\xrightarrow{\cdot \xi} \mc{O}_{P_{(i),n}} \rightarrow \mc{O}_{P_{(0),n}} \rightarrow 0. \]
By the inductive hypothesis, $\mc{O}_{P_{(i-1),n}}$ is a sheaf and $\mc{O}_{P_{(0),n}}$ is a sheaf. It follows that $\mc{O}_{P_{(i),n}}$ is a sheaf.
\end{proof}

Outside of the affinoid case, we obtain $P^\sharp_{(i)}$ by glueing. 

\begin{remark}The canonical thickenings $P^\sharp_{(i)}$, $i \geq 1$, do not fall under the umbrella of \cref{ss.thickenings} because one cannot choose a structure morphism $P^{\sharp}_{(i)} \rightarrow P^{\sharp}$. For example, for $P=\Spa(\mbb{C}_p)$, the associated augmentation is 
\[ \mbb{B}^+_\dR(\mbb{C}_p)/\Fil^2\mbb{B}^+_\dR(\mbb{C}_p) \xrightarrow{\theta} \mathbb{C}_p \]
which does not admit a \emph{continuous} algebra section (e.g., because $\overline{\mathbb{Q}}_p$ is dense in both the target and the source).  
\end{remark}

In the affinoid case $P^\sharp=\Spa(A,A^+)$, for $0 \leq i < \infty$, we write 
\[ P^{\sharp-\alg}_{(i)}:=\Spec \mc{O}(P^\sharp_{(i)}) = \Spec \mbb{B}^+_\dR(A)/\Fil^{i+1} \mbb{B}^+_\dR(A).\]
The category of vector bundles on $P^\sharp_{(i)}$ is equivalent to the category of vector bundles on $P^{\sharp-\alg}_{(i)}$ as both are equivalent to projective modules over $\mc{O}(P^\sharp_{(i)})$ (see \cref{ss.vector-bundles-adic-spaces-schemes}).

We also write 
$P^{\sharp-\alg}_{(\infty)}:=\Spec \lim_i \mc{O}(P^\sharp_{(i)})$, i.e. $P^{\sharp-\alg}_{(\infty)}=\Spec \mbb{B}^+_\dR(A)$ when $P^\sharp=\Spa(A,A^+)$. 

\begin{remark}
    We could can also view the system $(P^{\sharp}_{(i)})_i$ as a formal adic space; this perspective will play a role in the constructions of \cref{s.hodge-etc-period-maps}. 
\end{remark}

\begin{definition}\label{def.infinitesimal-nbhds}\hfill
\begin{enumerate}
\item For $0 \leq i < \infty$, we write $\Box^\sharp_{(i)}$ for the functor from $\AffPerf/\Spd \mbb{Q}_p$ to strongly sheafy adic spaces
\[ P/\Spd \mbb{Q}_p \mapsto P^\sharp_{(i)}. \]
We also write $\Box^\sharp=\Box^\sharp_{(0)}$. 
\item For $0 \leq i \leq \infty$, we write $\Box^{\sharp-\alg}_{(i)}$ for the functor from $\AffPerf/\Spd \mbb{Q}_p$ to schemes
\[ P/\Spd \mbb{Q}_p \mapsto P^{\sharp-\alg}_{(i)}. \]
We also write $\Box^{\sharp-\alg}=\Box^{\sharp-\alg}_{(0)}$. 
\item We write $\Box^{\sharp-\alg}_{(\infty)}\backslash \Box^{\sharp-\alg}$ for the functor from $\AffPerf/\Spd \mbb{Q}_p$ to schemes
\[ P/\Spd \mbb{Q}_p \mapsto P^{\sharp-\alg}_{(\infty)} \backslash P^{\sharp-\alg}. \]
\end{enumerate}
\end{definition}

As noted above, for any $0 \leq i < \infty$, there is a natural equivalence of fibered categories on $\Perf/\Spd \mbb{Q}_p$ 
\begin{equation}\label{eq.alg-an-equivalence-infinitesimal} (\Box_{(i)}^{\sharp-\alg})^*\VB=(\Box^{\sharp}_{(i)})^*\VB.\end{equation}

\begin{lemma}\label{lemma.inf-nbhds-stack}\hfill
\begin{enumerate}
    \item For any any $0 \leq i \leq \infty$, $(\Box_{(i)}^{\sharp-\alg})^*\VB$ is a $v$-stack. 
    \item For $0 \leq i < \infty$, $(\Box_{(i)}^{\sharp})^*\VB$ is a $v$-stack.
    \item $(\Box^{\sharp-\alg}_{(\infty)}\backslash \Box^{\sharp-\alg})^*\VB$ is a $v$-prestack. 
\end{enumerate}
 \end{lemma}
 \begin{proof}
 Part (1) is \cite[Corollary 17.1.9]{ScholzeWeinstein.BerkeleyLecturesOnPAdicGeometryAMS207}, and part (2) then follows from \cref{eq.alg-an-equivalence-infinitesimal}.

 For (3), note that $P^{\sharp-\alg}_{(\infty)}\backslash P^{\sharp-\alg}$ is affine: indeed, for $P^\sharp=\Spa(A,A^+)$, it is $\Spec \mbb{B}_\dR(A)$, where as usual $\mbb{B}_\dR(A)$ is obtained from $\mbb{B}^+_\dR(A)$ by inverting any generator of $\ker \theta$. To that end, we note that $\mbb{B}_\dR$ is a $v$-sheaf: This holds, e.g., since if we restrict this to $\Perf/P$ for any $P/\Spd \mbb{Q}_p$ with $P^\sharp=\Spa(A,A^+)$ and fix a generator $\xi$ for $\ker \theta$ on $\mbb{B}^+_\dR(A)$, then the restriction to $\Perf/P$ is $\bigcup \frac{1}{\xi^i}\mbb{B}^+_\dR$, so it is a $v$-sheaf since $\mbb{B}^+_\dR$ is (that $\mbb{B}^+_\dR$ is a $v$-sheaf is part of the case $i=(\infty)$ of (1)). This implies part (3): note that the presheaf of homomorphisms between any $\mbb{B}_\dR$-modules $M_1$ and $M_2$ is the presheaf of sections of $M_1^* \otimes M_2$. The latter is a projective module so that its sheaf of sections is a summand of $\mbb{B}_\dR^n$ for some $n$.  
 \end{proof}

\subsection{Fargues--Fontaine curves}\label{ss.Fargues--Fontaine-curves}
For a $E/\mbb{Q}_p$ a finite extension with residue field $\mbb{F}_q$ and $P=\Spa(R,R^+) \in \AffPerf/\Spd \mbb{F}_q$, as in \cite[II.1.15]{FarguesScholze.GeometrizationOfTheLocalLanglandsCorrespondence} we write 
\[ Y_{E,P} = \Spa(W_E(R^+), W_E(R^+))\backslash V([\varpi]p)\] where $\varpi$ is any pseudouniformizer in $R^+$. It admits a $q$-power Frobenius $\sigma$, and the Fargues--Fontaine curve is 
\[ X_{E,P} := Y_{E,P}/\sigma^{\mbb{Z}}. \]
As in \cite[\S II.2.3]{FarguesScholze.GeometrizationOfTheLocalLanglandsCorrespondence}, there is an ample line bundle $\mathcal{O}(1)$ on $X_{E,P}$, and defining $X_{E,P}^\alg:=\mathrm{Proj} \bigoplus_{i \geq 0} H^0(X_E, \mathcal{O}(i))$, there is a natural map of ringed spaces $X_{E,P} \rightarrow X_{E,P}^\alg$ such that pullback induces an equivalence 
\begin{equation}\label{eq.GAGA-equiv} \VB(X_{E,P}^\alg)=\VB(X_{E,P}) \end{equation} 
that furthermore identifies cohomology groups on both sides. 

Note that there is a natural map $\Spd E \rightarrow \Spd \mbb{F}_q$. For $P/\Spd E$, Fontaine's map $\theta$ induces functorial closed immersions over $E$ for $0 \leq i < \infty$
\[ P_{(i)}^\sharp \hookrightarrow Y_{E,P}, P_{(i)}^\sharp \hookrightarrow X_{E,P}, \textrm{ and } P_{(i)}^{\sharp-\alg} \hookrightarrow X_{E,P}^\alg. \]
It also induces a functorial map 
\[ P_{(\infty)}^{\sharp-\alg} \rightarrow X_{E,P}^\alg \]
which is the algebraization of the formal neighborhood of $P^{\sharp-\alg}$ in $X_{E,P}^\alg$. 

\begin{definition}\label{def.ff-functors}\hfill
    \begin{enumerate}
    \item We write $X_{E,\Box}$ for the functor from $\AffPerf/\Spd \mbb{F}_q$ to strongly sheafy adic spaces
    \[ P/\Spd \mbb{F}_q \mapsto X_{E,P}. \]
    \item We write $X_{E,\Box}^\alg$ for the functor from $\AffPerf/\Spd \mbb{F}_q$ to schemes
    \[ P/\Spd \mbb{F}_q \mapsto X_{E,P}^\alg. \]
    \item We write $X_{E,\Box}^\alg \backslash \Box^{\sharp-\alg}$ for the functor from $\AffPerf/\Spd E$ to schemes
    \[ P/\Spd E \mapsto X_{E,P}\backslash P^{\sharp-\alg}.\]
    \end{enumerate}
\end{definition}

\begin{lemma}\label{lemma.ff-prestack-vstack}\hfill
\begin{enumerate}
\item $(X_{E,\Box})^* \VB$ and $(X_{E,{\Box}}^\alg)^*\VB$ are both $v$-stacks. They are equivalent by pullback along the natural transformation of functors to ringed spaces $X_{E,\Box} \rightarrow X_{E,\Box}^\alg$.
\item $(X_{E,\Box} \backslash \Box^{\sharp-\alg})^*\VB$ is a $v$-prestack. 
\end{enumerate}
\end{lemma}
\begin{proof}
For (1), the equivalence follows from the GAGA equivalence \cref{eq.GAGA-equiv}, so it suffices to establish the stack property only in the analytic case. To that end, we first note that, for any open $U \subseteq Y_{E,P}$, $U^*\VB$ is a $v$-stack on $\AffPerf/P$ by \cite[Proof of Proposition 19.5.3] {ScholzeWeinstein.BerkeleyLecturesOnPAdicGeometryAMS207}. In particular, since the category of vector bundles on $X_E$ is equivalent to the category of $\varphi$-equivariant bundles on $Y_E$, it follows that $X_{E,\Box}^*\VB$ is a $v$-stack.

For (2), we note that $X_{E,P}^\alg \backslash P^{\sharp-\alg}$ is affine. For $P^\sharp=\Spa(A,A^+)$, its global sections are usually written as $\mbb{B}_e(A)$. As in the proof of \cref{lemma.inf-nbhds-stack}, it suffices to verify these global sections are a $v$-sheaf. This follows, e.g., by writing its restriction to any $P/\Spd E$ as the colimit of the global sections presheaves of $\mathcal{O}(n)$ on $X_{E,P}$, which is a $v$-sheaf by part (1). 
\end{proof}

\subsection{The pairs $(\mathcal{C},B)$ that we will use}\label{ss.the-pairs-we-use}

\begin{proposition}\label{prop.inscribed-pairs}
Consider the pairs $(\AffPerf/S,B)$ for $(S,B)$ as follows:
\begin{enumerate}
\item $S=\Spd \mbb{Q}_p$ and $B=\Box^\sharp_{(i)}$ or $\Box_{(i)}^{\sharp-\alg}$ for any $0 \leq i < \infty$
\item $S=\Spd \mbb{Q}_p$ and $B=\Box^{\sharp-\alg}_{(\infty)}$
\item $S=\Spd \mbb{Q}_p$ and $B=\Box^{\sharp-\alg}\backslash \Box^{\sharp-\alg}_{(\infty)}$
\item For $E/\mathbb{Q}_p$ a finite extension with residue field $\mathbb{F}_q$, $S=\Spd \mbb{F}_q$ and $B=X_{E,\Box}$ or $X_{E,\Box}^\alg$. 
\item For $E/\mathbb{Q}_p$ a finite extension with residue field $\mathbb{F}_q$, $S=\Spd E$ and $B=X_{E,\Box}^\alg\backslash \Box^{\sharp-\alg}$. 
\end{enumerate}
In all cases the presheaf $\mbb{B}$ on $B^\lf$ of \cref{def.BB} is an inscribed $v$-sheaf. In (1) and (4), there is a canonical equivalence between the categories $B^\lf$ for the analytic and algebraic versions, identifying the $v$-sheaf $\mbb{B}$. 
\end{proposition}
\begin{proof}
That $\mbb{B}$ is inscribed is \cref{prop.BB-inscribed} and that it is a $v$-sheaf follows from \cref{lemma.inf-nbhds-stack} in cases (1)-(3) and \cref{lemma.ff-prestack-vstack} in cases (4) and (5).  In light of \cref{prop.augmented-equiv}, the equivalence between the algebraic and analytic categories of thickenings follows in (1) from \cref{eq.alg-an-equivalence-infinitesimal} and in (4) from \cref{eq.GAGA-equiv} (or the corresponding part of \cref{lemma.ff-prestack-vstack}-(1)). 
\end{proof}

\subsection{Moduli of sections}\label{ss.pairs-moduli-of-sections}

We consider now a pair $(\AffPerf/S,B)$ as in \cref{prop.inscribed-pairs}. For $\mc{S}$ an inscribed sheaf, and $Z$ a smooth scheme or strongly sheaf adic space over $\mc{B}$ on $\mc{S}$ as in \cref{ss.inscribed-mos}, we write $Z^\lfid$ for the presheaf $\mc{B}^* h_{Z}$ over $\mc{S}$ of \cref{prop.inscribed-pairs}, i.e. 
\[ Z^{\lfid}(s \in \mc{S}(\mc{B})) = \Hom_{\mc{B}}(\mc{B}, Z(s)). \]

\begin{theorem}\label{theorem.affine-scheme-mos}
Let $(\mc{C}, B)$ be one of the algebraic pairs of \cref{prop.inscribed-pairs}, let $\mc{S}$ be an inscribed $v$-sheaf on $B^\lf$, let $Z$ be a smooth \emph{affine} scheme over $\mc{B}$ on $\mc{S}$. Then $Z^\lfid$ is an inscribed $v$-sheaf, and there is a natural identification $T_{Z^{\lfid}/\mc{S}} = (T_{Z/\mc{B}})^\lfid$ of inscribed sheaves of $\mbb{B}$-modules over $Z^{\lfid}$.
\end{theorem}
\begin{proof}
In \cref{prop.inscribed-pairs} we showed $Z^\lfid$ was an inscribed presheaf, and the made the identification of tangent bundles. Thus it remains only to verify that $Z^{\lfid}$ is a $v$-sheaf. 

It suffices, for each $s:\mc{B}_0/B(P/S) \rightarrow \mc{S}$ and $Z_0:=Z(s)$, to verify that the presheaf on $\AffPerf/P$ 
\[ Q/P \mapsto \Hom_{\mc{B}_{0,B(Q/S)}}(\mc{B}_{0,B(Q/S)}, Z_{0,B(Q/S)})=\Hom_{\mc{B}_{0}}(\mc{B}_{0,B(Q/S)}, Z_0)\]
is a $v$-sheaf. In the cases (1)-(3) and (5) where everything in sight is affine, writing $\mc{B}_0=\Spec C$ and $Z_0=\Spec D$, this is
\[ Q/P \mapsto \Hom_{\mbb{B}(P/S)}(D, \mbb{B}(Q/S) \otimes_{\mbb{B}(P/S)} C). \]
This is a $v$-sheaf since $Q/P \mapsto \mbb{B}(Q/S) \otimes_{\mbb{B}(P/S)} C$ is a $v$-sheaf by \cref{lemma.inf-nbhds-stack} in cases (1)-(3) and \cref{lemma.ff-prestack-vstack} in case (5) (it is the $v$-sheaf of sections of $\mc{O}_{\mc{B}_0}$ viewed as an object of $\VB(B(P/S))$). In case (4), we may fix an untilt $P^\sharp/\Spd \mbb{Q}_p$, then deduce the result from that in cases (2), (3), and (5) by writing 
\[ Z^{\lfid}= (Z|_{X_{E,\Box}^\alg \backslash \Box^{\sharp-\alg}})^\lfid \times_{(Z|_{\Box^{\sharp-\alg}_{(\infty)} \backslash \Box^{\sharp-\alg}})^\lfid}(Z|_{\Box_{(\infty)}^{\sharp-\alg}})^\lfid. \]
\end{proof}

\begin{theorem}\label{theorem.smooth-adic-space-mos}
Let $(\mathcal{C}, B)$ be one of the analytic pairs of \cref{prop.inscribed-pairs}, and let $Z$ be a smooth adic space over $\mc{B}$ on $\mc{S}$. Then $Z^{\lfid}$
is an inscribed $v$-stack, and there is a natural identification $T_{Z^{\lfid}/\mc{S}} = (T_{Z/\mc{B}})^\lfid$ of $\mbb{B}$-modules over $Z^{\lfid}$.
\end{theorem}
\begin{proof}

In \cref{prop.inscribed-pairs} we showed $Z^\lfid$ was an inscribed presheaf and established the identification of tangent bundles. Thus it remains only to verify that $Z^{\lfid}$ is a $v$-sheaf. 

It suffices, for each $s:\mc{B}_0/B(P/S) \rightarrow \mc{S}$ and $Z_0:=Z(s)$, to verify that the presheaf on $\AffPerf/P$ 
\[ Q/P \mapsto \Hom_{\mc{B}_{0,B(Q/S)}}(\mc{B}_{0,B(Q/S)}, Z_{0,B(Q/S)})=\Hom_{\mc{B}_{0}}(\mc{B}_{0,B(Q/S)}, Z_0)\]
is a $v$-sheaf. We claim this formula defines a $v$-sheaf on $\AffPerf/P$ for any analytic adic space $Z_0$ over $E$.

We first treat case (1), so that $B=\Box^\sharp_{(i)}$, $0 \leq i <\infty$.
For $Z_0$ affinoid, it follows as in the proof of \cref{theorem.affine-scheme-mos} or \cite[Lemma 15.1-(ii)]{Scholze.EtaleCohomologyOfDiamonds}, using that $\mbb{B}$ is in fact a sheaf of topological rings. For any rational open $U \subseteq Z_0$, the functor represented by $U$ is an open sub-functor. Thus, as in \cite[\S15]{Scholze.EtaleCohomologyOfDiamonds}, we may glue along these open subfunctors to obtain the result for general $Z_0$. 

We now treat case $(4)$, so that $B=X_{E,\Box}$. We observe that it suffices to prove the analogous statement over $Y_{E,\Box}$, since the property for $X_{E,\Box}$ then follows by viewing morphisms from $X_{E,Q} \times_{X_{E,P}} \mc{B}_0$ as $\varphi$-equivariant morphisms from $Y_{E,Q} \times_{X_{E,P}} \mc{B}_0$. We will deduce this statement from the $i=0$ part of case (1), established above.

To that end, let $E_\infty$ be the completion of the $\mathbb{Z}_p$-subextension of the cyclotomic extension $E(\mu_{p^\infty})$ --- $E_\infty$ is a perfectoid field. We write $\Gamma=\mbb{Z}_p$ for the Galois group, and $\gamma \in \Gamma$ for a topological generator. We first consider the presheaf 
\[ Q/P \mapsto \Hom(\mc{B}_{0, Y_{E,Q}} \times_{\Spa E} \Spa E_\infty, Z_0). \]
It follows from case (1) that this is a $v$-sheaf --- we can apply case (1) here because $Q/P \mapsto Y_{E,Q} \times_{\Spa E} \Spa E_\infty$ is a product preserving functor to perfectoid spaces that sends $v$-covers to $v$-covers. We then conclude by observing that the $v$-sheaf we are interested in is obtained by taking the $\Gamma$-invariant sections in this $v$-sheaf. Indeed, we can check this when $Z_0$ is affinoid, in which case it reduces to the statement that 
\[ \mc{O}(\mc{B}_{0,Y_{E,Q}})=\mc{O}(\mc{B}_{0, Y_{E,Q}} \times_{\Spa E} \Spa E^\cyc)^\Gamma. \]
This follows by reducing to the corresponding statement where $Y$ is replaced by the affinoid $Y_I$ for $I \subseteq (0,\infty)$ a compact interval, which follows because
\[ \mc{O}(\mc{B}_{0, Y_{E,I, Q}} \times_{\Spa E} \Spa E_\infty) = \mc{O}(\mc{B}_{0, Y_{E,I, Q}} \hat{\otimes}_E E_\infty) \]
and, by \cite[Proposition 7]{Tate.pDivisibleGroups}, there is a direct sum decomposition $E_\infty=E \oplus V$ such that $\gamma-1$ acts invertibly on $V$. 
\end{proof}

The following example shows that \cref{theorem.smooth-adic-space-mos} encodes both the tangent bundles of smooth rigid analytic varieties over nonarchimedean fields and the Tangent Bundles arising in the Fargues--Scholze Jacobian criterion. 
\begin{example}\label{example.smooth-rig-and-fs}\hfill
\begin{enumerate}
    \item We work over the pair $(\Spd \mbb{Q}_p, \Box^\sharp)$. Suppose $L/\mathbb{Q}_p$ is a non-archimedean extension and $Y/L$ is a smooth rigid analytic variety. Then, 
\[ Y \times_{\Spa L} \mc{B} \]
is a smooth adic space over $\mc{B}$ on $(\Spd L)^\triv$, and 
\[ \overline{(Y \times_{\Spa L} \mc{B})^\lfid}= Y^\diamond \textrm{ and } \overline{T_{(Y \times_{\Spa L} \mc{B})^\lf}}=(T_{Y/\Spa L})^\diamond\]
\item We work over the pair $(\Spd \mbb{F}_q, X_{E,\Box})$. Suppose $P/\Spd \mbb{F}_q$ and $Z$ is a smooth adic space over $X_{E,P}$. Then 
\[ Z \times_{X_{E,P}} \mc{B} \]
is a smooth adic space over $\mc{B}$ on $P^\triv$, and
\[ \overline{(Z \times_{X_{E,P}} \mc{B})^\lfid} = \mathcal{M}_Z \textrm{ and } \overline{T_{(Z \times_{X_{E,P}} \mc{B})^\lfid}} = T_{\mathcal{M}_Z}\]
where $\mathcal{M}_Z$ is the Fargues--Scholze moduli of sections as in \cite[IV.4]{FarguesScholze.GeometrizationOfTheLocalLanglandsCorrespondence}, $\mathcal{M}_Z(Q/P)=\Hom_{X_{E,P}}(X_{E,Q}, Z)$, and $T_{\mathcal{M}_Z}$ is its Tangent Bundle as implicit in \cite[IV.4]{FarguesScholze.GeometrizationOfTheLocalLanglandsCorrespondence} (cf. \cite{IvanovWeinstein.TheSmoothLocusInInfiniteLevelRapoportZinkSpaces}), sending $f \in \Hom_{X_{E,P}}(X_{E,Q}, Z)$ to 
\[ H^0(X_{E,Q}, f^* T_{Z/X_{E,P}}). \]
\end{enumerate}
\end{example}

\begin{remark} The setups of \cref{theorem.affine-scheme-mos} and \cref{theorem.smooth-adic-space-mos} also allow for more general constructions: for example, the absolute Banach--Colmez spaces of \cite[II.2.2]{FarguesScholze.GeometrizationOfTheLocalLanglandsCorrespondence} are sections of smooth adic spaces over $X_{E,\Box}$ on $\Spd \Fqbar$, and in the rigid analytic case the formalism allows us to consider ineffective descents of rigid analytic varieties, such as the Breuil-Kisin-Fargues twist $\mathbb{A}^1_{\mbb{Q}_p}\{1\}$, as smooth adic spaces over $\Box^\sharp$ on $\Spd \mbb{Q}_p$. 
\end{remark}

\subsection{Change of context}\label{ss.change-of-context}
We now describe how to move between different inscribed contexts, focusing on the inscribed contexts in \cref{prop.inscribed-pairs}. 

We first note that, for a pair $(S/\AffPerf,\mc{B})$, if we have a map of $v$-sheaves on $\AffPerf$, $S' \rightarrow S$, then the categories of inscribed $v$-sheaves over $\mc{B}^\lf$ equipped with a structure morphism to $(S')^\triv$ is naturally equivalent to the category of inscribed $v$-sheaves over $(\mc{B}|_{S'})^\lf$, compatibly with tangent bundles, etc., and \cref{prop.inscribed-pairs} still holds if we replace the $S$ in any pair with such an $S'$. We will use this implicitly below.  

Now, for $E/\mathbb{Q}_p$ a finite extension with residue field $\mbb{F}_q$, we have the natural map $\Spd E \rightarrow \Spd \mbb{F}_q$, and above it, the natural maps
\begin{align*}
\Box^{\sharp}_{(i)} & \rightarrow X_{E,\Box} \textrm{ for } 0\leq i < \infty,\\
\Box^{\sharp-\alg}_{(i)} & \rightarrow X^\alg_{E,\Box} \textrm{ for } 0\leq i \leq \infty\\
X_{E,\Box}^{\alg}\backslash \Box^{\sharp-\alg} &\rightarrow X_{E,\Box}^{\alg}. 
\end{align*}
Over $\Spd E$ we also have the natural map 
\[ \Box^{\sharp-\alg}_{(\infty)}\backslash \Box^{\sharp-\alg} \rightarrow  X_{E,\Box}^{\sharp-\alg}\backslash \Box^{\sharp-\alg}. \]
Thus, for example, we may pullback an inscribed $v$-sheaf $\mc{S}/(\Spd E)^\triv$ for the context $(\Spd E, \Box^{\sharp})$ to an inscribed $v$-sheaf for the context $(\Spd \mathbb{F}_q, X_E)$ lying over $(\Spd E)^\triv$ by 
\[ \mc{S}(\mc{X}/X_{E,P}, P \rightarrow \Spd E):=\mc{S}(\mc{X} \times_{X_{E,P}} P^\sharp). \]
This construction allows us, e.g., to treat both rigid analytic varieties and Fargues--Scholze moduli of sections as in \cref{example.smooth-rig-and-fs} in a common world. The price one pays is the loss of information about the $\mbb{B}$-module structure on the tangent bundle --- for example, if $Z/E$ is a rigid analytic variety and we pullback $(Z/E)^\lfid$ by this construction, then its tangent bundle with respect to the inscribed context $(\Spd \mathbb{F}_q, X_E)$ will only remember the $E^\lfid$-module structure rather than the full $\mathcal{O}$-module structure. 

\subsection{The slope condition $\lfp$}\label{ss.slope-condition}
In the context of \cref{prop.inscribed-pairs}-(4), we will also consider inscribed $v$-sheaves, etc., on a restricted category of thickenings $X_{E,\Box}^{\lfp}$ or equivalently $X_{E,\Box}^{\alg-\lfp}$ (we can work with a restricted category of thickenings via the mechanism explained in  \cref{ss.restricted-categories-thickenings}). This is the category that was used for the statements of our results in the introduction (for $E=\mathbb{Q}_p$ and working over $\Spd \mbb{Q}_p$); we give its definition now and observe some basic properties. We will work just with $X_{E,\Box}^\lfp$ --- the treatment of the algebraic version is identical, and is equivalent via the GAGA equivalence of \cref{prop.inscribed-pairs}.

\begin{definition}
$X_{E,\Box}^\lfp$ is the full subcategory of $X_{E,\Box}^\lf$ whose objects are those thickenings $\mc{X}/X_{E,P}$ of $X_{E,P}/X_{E,P}$ such that, for $\mc{I}:= \ker \left(\mc{O}_{\mc{X}} \rightarrow \mc{O}_{X_P}\right)$, $\mc{I}^{n}/\mc{I}^{n+1}$ is a finite locally free $\mc{O}_{X_{E,P}}$-module that has non-negative Harder-Narasimhan slopes after restriction to $X_{E,\Spa(C,C^+)}$ for any geometric point $\Spa(C,C^+) \rightarrow P$. 
\end{definition}

We now verify that this subcategory satisfies the two conditions enumerated in \cref{ss.restricted-categories-thickenings}. First, evidently $X_{E,P}/X_{E,P} \in X_{E,\Box}^\lfp$ for any $P \in \Perf$. Now, suppose given $\mc{X}/X_{E,P} \in X_{E,\Box}^\lfp$, and a finite free $\mc{O}(\mc{X})$-module $M$ or rank $r$. Then, for $\mc{I}$ the ideal sheaf of $X_{E,P} \hookrightarrow \mc{X}$, the ideal sheaf $\mc{I}_M$ of $X_{E,P} \hookrightarrow \mc{X}[M]$ can be identified with 
\[ \mc{I} \oplus M \cong \mc{I} \oplus \mc{O}_{\mc{X}}^{\oplus r}. \]
In particular, we find
\[ \mc{I}_M^n = \mc{I}^n \oplus (\mc{I}^{n-1})^{\oplus r}. \]
Thus
\[ \mc{I}_{M}^n/\mc{I}_{M}^{n+1}=\mc{I}^n/\mc{I}^{n+1} \oplus (\mc{I}^{n-1}/\mc{I}^{n})^{\oplus r},\]
so that the slope condition on $\mc{I}_{\mc{M}}$ follows from the slope condition on $\mc{I}$.

\section{Inscribed vector bundles and $G$-bundles}\label{s.inscribed-vb-g-bun}

We consider one of the following inscribed contexts $(\AffPerf/S, B)$
\begin{enumerate}
\item For $L$ a $p$-adic field and $0 \leq i \leq \infty$, $(\AffPerf/\Spd L, \Box_{(i)}^{\sharp-\alg})$. 
\item For $L/\mathbb{Q}_p$ a finite extension with residue field $\mbb{F}_q$,  $(\AffPerf/\Spd \mathbb{F}_q, X_{E,\Box}^\alg)$.
\end{enumerate}

In the first part of this section we will show that, for $G/L$ a linear algebraic group, the moduli of $G$-bundles on $\mc{B}$ is an inscribed $v$-stack. The result in the case $i=\infty$ of (1) will be used in \cref{s.affine-grassmannian} to construct and study the inscribed $\mbb{B}^+_\dR$-affine Grassmannian. The result in the case (2) will give us an inscription on the moduli stack $\mathrm{Bun}\,G$ of \cite{Fargues.GeometrizationOfTheLocalLanglandsCorrespondenceAnOverview}; we emphasize that this is \emph{not} the trivial inscription. The precise statements are given in \cref{ss.vb-and-classifying-stack}. The inscribed property is relatively straightforward and the prestack property can be deduced from the non-inscribed version, but in order to obtain descent we have to redo some of the descent arguments of \cite{ScholzeWeinstein.BerkeleyLecturesOnPAdicGeometryAMS207} in our setting --- this is carried out in \cref{ss.descent-lemmas}.

In the remainder of the section we develop some complements that will be used in later sections of the paper: In \cref{ss.BC-spaces}, we briefly discuss the inscribed Banach--Colmez spaces associated to a vector bundle, and in \cref{ss.NewtonStrataBG} we discuss the Newton strata on the inscribed classifying stack. A key result is \cref{prop.b-basic-open-stratum}, which says that the basic strata are open not just topologically, i.e. on the underlying $v$-stack, but also deformation-theoretically when the slope condition $\lfp$ of \cref{ss.slope-condition} is imposed on the test objects. 

\subsection{Vector bundles and the classifying stack}\label{ss.vb-and-classifying-stack}

We write $\mc{B}$ for the functor from $B^\lf$ to schemes over $\Spec L$ sending $\mc{B}/B(P/S)$ to $\mc{B}$, so that $\mc{B}^*\VB$ is the inscribed fibered category over $B^\lf$ whose objects are pairs $(\mc{B}, \mc{V})$ where $\mc{B}\in B^\lf$ and $\mc{V}$ is a locally free of finite rank $\mc{O}_{\mc{B}}$-module. 

\begin{theorem}\label{theorem.vb-inscribed-vstack} $\mc{B}^*\VB$ is an inscribed $v$-stack. 
\end{theorem}
\begin{proof}
That $\mc{B}^*\VB$ is inscribed follows from \cite[Theoreme 2.2-(iv)]{Ferrand.ConducteurDescentEtPincement}. We thus must verify it is a $v$-stack. To see that it is a prestack, observe that if we fix $\mc{E}_1$, $\mc{E}_2$ in the same fiber, then $\Hom(\mc{E}_1, \mc{E}_2)$ is the global sections functor of $\mc{E}_1^* \otimes \mc{E}_2$. By pushing forward to a locally free sheaf of finite rank on $B$, it follow from \cref{lemma.inf-nbhds-stack}/\cref{lemma.ff-prestack-vstack} that this is a $v$-sheaf.  

It remains to establish descent. We do not know how to deduce this directly from the descent for $B^*\VB$, since it is not clear that the descent as a locally free $\mc{O}_B$-module with $\mc{O}_{\mc{B}}$ action is locally free as an $\mc{O}_{\mc{B}}$-module. However, the proofs of descent in the non-inscribed settings can be adapted to the inscribed setting; we carry this out in the next subsection: in case (1), the result it \cref{lemma.thickenings-v-descent}, and in case (2), it follows from \cref{lemma.Y-vb-descent} combined with the trivial analytic descent of vector bundles from $Y_{E}$ to $X_{E}$ and the GAGA equivalence between vector bundles on $X_E$ and $X_E^\alg$.  
\end{proof}

Under our assumptions, the functor $B$ and thus also $\mc{B}$ factors canonically through schemes over $\Spec L$. Thus, for $G/L$ a linear algebraic group, it makes sense to consider also the pull back $\mc{B}^*\BG$ of the classifying stack for $G$. Concretely, $\mc{B}^*\BG$ is the fibered category over $B^\lf$ whose objects are pairs $(\mc{B} \in B^\lf, \mc{G}/B)$ where $\mc{G}$ is a $G$-torsor on $\mc{B}$. It will be convenient at various times to use the Tannakian, \'{e}tale, and geometric perspectives on $G$-torsors, and we move freely between these. 

\begin{theorem}\label{thm.BG-pull-back-inscribed}
Suppose $G/L$ is a linear algebraic group. Then $\mc{B}^*\BG$ is an inscribed $v$-stack. 
\end{theorem}
\begin{proof}
That $\mc{B}^*\BG$ is a $v$-stack is immediate from the Tannakian perspective and \cref{theorem.vb-inscribed-vstack}. To see that it is inscribed, for $\mc{B}=\mc{B}_1 \sqcup_{\mc{B}_0} \mc{B}_2$, we may view each of $G(\mc{O}_{\mc{B}_\bullet})$, $\bullet=0,1,2$, as an \'{e}tale sheaf of sections on $\mc{B}_{0,\et}$, so that $\BG(\mc{B}_\bullet)$ classifies $G(\mc{O}_{\mc{B}_\bullet})$-torsors on $\mc{B}_{0,\et}$. But, as in the proof of \cref{theorem.affine-scheme-mos}, these associated sheaves of sections on $\mc{B}_{0,\et}$ satisfy 
\[ G(\mc{O}_{\mc{B}}) = G(\mc{O}_{\mc{B}_1}) \times_{G(\mc{O}_{\mc{B}_0})} G(\mc{O}_{\mc{B}_2}). \]
To give an \'{e}tale torsor for this group is equivalent to giving \'{e}tale torsors for $G(\mc{O}_{\mc{B}_1})$ and $G(\mc{O}_{\mc{B}_2})$ and an isomorphism of their push-outs to $G(\mc{O}_{\mc{B}_0})$ --- the inverse functor is given by the fiber product of sheaves. 
\end{proof}

\begin{remark}
Except in the $i=\infty$ case of (1), both of the above results hold also for the analytic version $\mc{B}^\an$ of $\mc{B}$ since there is a GAGA equivalence for the stack of vector bundles $\mc{B}^*\VB \cong (\mc{B}^\an)^*\VB$. 
\end{remark}

\subsection{Descent lemmas}\label{ss.descent-lemmas}
We now prove the descent lemmas that were used in the proof of \cref{theorem.vb-inscribed-vstack}. We use the techniques of \cite{ScholzeWeinstein.BerkeleyLecturesOnPAdicGeometryAMS207}. 

\begin{lemma}\label{lemma.perf-vb-descent}
Let $P\in \Perf$ with a map $P\rightarrow \Spd \mbb{Q}_p$ giving an untilt $P^\sharp/\Spa \mathbb{Q}_p$, and let $\mc{B}/P^\sharp$ be a locally free nilpotent thickening of $P^\sharp$. Then, the fibered category on $\Perf/P$ sending $Q/P$ to the category of locally free of finite rank $\mc{O}_{\mc{B}_{Q^\sharp}}$-modules is a $v$-stack. 
\end{lemma}
\begin{proof}
When $\mc{B}=P^\sharp$, this is \cite[Lemma 17.1.8]{ScholzeWeinstein.BerkeleyLecturesOnPAdicGeometryAMS207}. We will bootstrap from this case. 

Suppose $Q/P$ and $\mc{E}$ is a locally free of finite rank $\mc{O}_{\mc{B}_{Q^\sharp}}$. Then, the presheaf on $\AffPerf/Q$, 
\[ Q'/Q \mapsto H^0(\mc{B}_{Q'^\sharp}, \mc{E}|_{Q'^\sharp}) \]
is a $v$-sheaf: indeed, it is the $v$-sheaf of sections of the locally free of finite rank $\mc{O}_{Q^\sharp}$-module $\pi_*\mc{E}|_{Q^\sharp}$ for $\pi:\mc{B}_{Q^\sharp} \rightarrow Q^\sharp$, so this follows from \cite[Lemma 17.1.8]{ScholzeWeinstein.BerkeleyLecturesOnPAdicGeometryAMS207}. In particular, we find that for any two such $\mc{E}_1$ and $\mc{E}_2$, the presheaf 
\[ Q'/Q \mapsto \Hom(\mc{E}_1|_{\mc{B}_{Q'^\sharp}}, \mc{E}_2|_{\mc{B}_{Q'^\sharp}}) \]
is a $v$-sheaf on $\Perf/Q$ since it is the sheaf of sections of  $\mc{E}_1^* \otimes_{\mc{O}_{\mc{B}_{Q'^\sharp}}} \mc{E}_2$. 

It remains to show that all descent data is effective. We may assume $Q=P$. Note that we can use \cite[Lemma 17.1.8]{ScholzeWeinstein.BerkeleyLecturesOnPAdicGeometryAMS207} to carry out the descent as a locally free $\mc{O}_{P^\sharp}$ module with $\mc{O}_{\mc{B}_{Q^\sharp}}$-action, but it is not clear that the result is locally free as an $\mc{O}_{\mc{B}_{Q^\sharp}}$-module. However, for $P$ a geometric point, we can deduce the case of a general $\mc{B}$ from \cite[Lemma 17.1.8]{Scholze.pAdicGeometry} (see below), and in general we will build from the case of a geometric by point adapting the proof of \cite[Lemma 17.1.8]{ScholzeWeinstein.BerkeleyLecturesOnPAdicGeometryAMS207}. 

Let $P^\sharp=\Spa(R,R^+)$ and consider a cover $Q \rightarrow P$ with $Q^\sharp=\Spa(S,S^+)$. By replacing $Q$ with an open cover, it suffices to descend the trivial module $M=S^n \otimes_R \mc{O}(\mc{B})$. If we write $\mc{O}(\mc{B})=A$, a finite projective $R$-module, then our descent data is an element $g\in \GL_n((S \hat{\otimes}_R S) \otimes_R A)$ satisfying a cocycle condition (note the second tensor product does not need to be completed since $A$ is a finite projective $R$-module). We want to show that, locally on $P^\sharp$, we can modify $g$ by a coboundary to obtain the identity matrix. 

We first treat the case that $P^\sharp=\Spa(K,K^+)$ is a perfectoid field. In that case, if we write $I$ for the kernel of $A \rightarrow K$ and
\[ G_i:= \ker \GL_n((S \hat{\otimes}_K S) \otimes_K A) \rightarrow \GL_n((S \hat{\otimes}_K S) \otimes_K A/I^j)\]
then $G_0/G_1 = \GL_n(S \hat{\otimes}_K S)$ and for $i \geq 1$, $G_i/G_{i+1}=M_n( (S \hat{\otimes}_K S) \otimes_K I^{i}/I^{i+1})$
For $G_0/G_1$ descent then applies by \cite[Lemma 17.1.8]{ScholzeWeinstein.BerkeleyLecturesOnPAdicGeometryAMS207} and we can thus modify our descent data $g$ by a cocycle to lie in $G_1$. Then, the Cech cohomology group involving $G_1/G_2$ vanishes by \cite[Theorem 17.1.3]{ScholzeWeinstein.BerkeleyLecturesOnPAdicGeometryAMS207} since it is the Cech cohomology on an affinoid cover of the free $\mc{O}$-module $\mc{O}^{nm}$, $m=\dim_K A$. Thus we may modify the descent data $g$ by a cocycle to lie in $G_2$. We  repeat this argument until we reach an $i$ large enough that $I^i=0$ so that $G_i=\{e\}$.

Now, in the general case, fixing $x \in P^\sharp$, we may replace $P^\sharp$ with a rational neighborhood $U$ of $x$ such that $I$, the kernel of $A \rightarrow R$, is free. We fix a basis $i_1, \ldots, i_m$ for $I$ and set $I_0= R^+ i_1 + \ldots + R^+ i_m \subseteq I$, a free $R^+$ sub-module of $I$. Then, for $a$ sufficiently large, $A_0=R^+ + p^a I_0$ is an open sub-algebra of $R$ (here we use that $(p^a i_1)(p^a i_2)=p^{2a} i_1 i_2$ to see it is a sub-algebra) that is free as an $R^+$-module. Replacing $i_j$ with $p^a i_j$ and $I_0$ with $p^a I_0$, we have $A_0 = R^+ + I_0$. Since the descent data is trivial at $x$, we may choose an element $t(x) \in \GL_n(K(x) \otimes_S A)$ such that $g(x)=(\mr{Pr}_1^* t(x))^{-1}(\mr{Pr}_2^* t(x)).$ Replacing $U$ with a potentially smaller rational neighborhood, we may spread out $t(x)$ to a section $t$ over $U$. Then, replacing $g$ with $(\Pr_1^*t)g(\Pr_2^* t)^{-1}$ so that $g(x)$ is the identity matrix, we may pass to a potentially smaller rational neighborhood $U$ to assume that $g$ lies in $\GL_n( (S^+ \hat{\otimes}_{R^+} S^+) \otimes_{R^+} A_0)$ and even that $g \cong 1 \mod \varpi$ for a pseudo-uniformizer $\varpi$. Then, $g$ mod $\varpi^2$ lies in 
\[ 1 + M_n( (S^+ \hat{\otimes}_{R^+} S^+) \otimes_{R^+} (\varpi A_0 / \varpi^2 A_0)) \cong (R^+/\omega)^{m n}. \]
By the almost vanishing of the cohomology of the plus-structure sheaf on affinoid perfectoids in \cite[Theorem 17.1.3]{ScholzeWeinstein.BerkeleyLecturesOnPAdicGeometryAMS207}, for any small $\epsilon$ we may modify $g$ by a coboundary so that $g \equiv 1 \mod \varpi^{2-\epsilon}$. Iterating, we may modify so that $g \equiv 1 \mod \varpi^{2^k-\epsilon}$ for any $k$ (allowing $\epsilon$ to grow but never past $1$). The product of the elements we are conjugating by converges, so that we find that $g$ is itself a coboundary and the descent data is trivial over $U$.  

By carrying this out in a neighborhood of any point $x$ we obtain an analytic cover where the descent data is effective, and then we conclude by glueing in the analytic topology. 
\end{proof}

As in \cite[Corollary 17.1.9]{ScholzeWeinstein.BerkeleyLecturesOnPAdicGeometryAMS207}, we can also extend over the canonical infinitesimal thickenings. We break this up into two statements; the first giving a geometric statement valid over all finite thickenings $P_{(i)}$, and the second giving a purely algebraic statement valid also over $P_{(\infty)}^\alg$.

\begin{lemma}\label{lemma.thickenings-v-descent}
Let $P/\Spd \mbb{Q}_p$ be a perfectoid space, and for $i \geq 0$ let $\mc{B}/P_{(i)}^\sharp$ be a locally free nilpotent thickening. The fibered category over $\Perf/P$ sending any $Q/P$ to the category $\VB(\mc{B}_{Q^\sharp_{(i)}})$ of locally free of finite rank $\mc{O}_{\mc{B}^\sharp_{Q_{(i)}}}$-modules (for $\mc{B}_{Q^\sharp_{(i)}} = \mc{B} \times_{P^\sharp_{(i)}} Q^\sharp_{(i)}$) is a $v$-stack. 

Moreover, if $Q \in \AffPerf/P$, $\VB(\mc{B}_{Q^\sharp_{(i)}})$ is equivalent to the category of finite projective $\mc{O}(\mc{B}_{Q^\sharp_{(i)}})$-modules and for any $\mc{E} \in \VB(\mc{B}_{Q^\sharp_{(i)}})$, $H^j_v(Q, \mc{E})=0$ for all $j>0$. 
\end{lemma}
\begin{proof}
In general, the equivalence on affinoids with finite projectives is \cite[Theorem 8.2.22]{KedlayaLiu.RelativepAdicHodgeTheoryFoundations}.

For the remainder of the statement, we argue by induction on $i$. For $i=0$, the $v$-stack property is \cref{lemma.perf-vb-descent}, and the vanishing property for higher cohomology on affinoid perfectoids follows from \cite[Theorem 17.1.3]{Scholze.pAdicGeometry} since any finite projective module is a direct summand of a finite rank free module. 

Now let $i>0$. We first show that, for any $Q \in \Perf/P$ and $\mc{E}  \in \VB(\mc{B}_{Q^\sharp_{(i)}})$, $\mc{E}$ defines a $v$-sheaf on $\Perf/Q$ whose higher cohomology vanishes on affinoids. To that end, for any $j \leq i$, we write $\mc{E}_{(j)}= \mc{E} \otimes \mc{O}_{\mc{B}_{Q^\sharp_{(j)}}}$. Then we have an exact sequence 
\[ 0 \rightarrow \mc{E}_{(i-1)} \rightarrow \mc{E}_{(i)} \rightarrow \mc{E}_{(1)} \rightarrow 0\]
This induces maps of presheaves of sections on $\Perf/Q$ forming a short exact sequence after restriction to $\AffPerf/Q$. Since $\mc{E}_{(1)}$ and $\mc{E}_{(i-1)}$ both define sheaves on $\Perf/Q$ and the higher Cech cohomology of the sheaf defined by $\mc{E}_{(i-1)}$ vanishes on any cover in $\AffPerf/Q$, we deduce that $\mc{E}_{(i)}$ defines a sheaf on $\AffPerf/Q$, and then by taking the long exact sequence also that its higher cohomology is zero on $\AffPerf/Q$. Since $\mc{E}_{(n)}$ already defined a sheaf in the analytic topology of any object in $\Perf/Q$, we conclude it defines a sheaf on $\Perf/Q$. 

Now, it follows immediately as in the proof of \cref{lemma.perf-vb-descent} that hom-sets are $v$-sheaves, so it remains only to establish descent. For $0 \leq j \leq i$, writing $G_j$ for the kernel of 
\[ \GL_n(\mc{O}_{\mc{B}_{P^\sharp_{(i)}}}) \rightarrow \GL_n(\mc{O}_{\mc{B}_{P^\sharp_{(j)}}}), \]
we have $G/G_0 = \GL_n(\mc{O}_{\mc{B}_{P^\sharp}})$ and $G_j/G_{j+1} = 1 + M_n (\mc{O}_{\mc{B}_{P^\sharp_{(0)}}}\{j+1\})$ for $j <i$ and $0$ for $j \geq i$ (where the Breuil-Kisin-Fargues twist $\{k\}$ denotes tensor with $\Fil^{k}\mathbb{B}_\dR/\Fil^{k+1}\mathbb{B}_\dR$). Since each of these subquotients has vanishing $H^1$ on affinoid perfectoids (the first by the $i=0$ case of descent), so does $\GL_n(\mc{O}_{\mc{B}_{P^\sharp_{(i)}}})$, and it follows that descent is effective on $\AffPerf/P$. Since analytic descent holds already, we obtain the desired descent result.  
\end{proof}

\begin{lemma} Let $P \in \AffPerf$ and, for $0 \leq i \leq \infty$, let $\mc{B}$ be a finite locally free nilpotent thickening of $\Spec \mc{O}(P_{(i)}).$ Then, the fibered category over $\AffPerf/P$ sending $Q/P$ to the set of projective $\mc{O}(\mc{B}) \otimes_{\mc{O}(P_{(i)})} \mc{O}(Q_{(i)})$-modules is a $v$-stack. Moreover, for any finite projective $\mc{O}(\mc{B})$-module $M$, the associated sheaf of sections 
\[ Q/P \mapsto M \otimes_{\mc{O}(P_{(i)})} \mc{O}_{Q_{(i)}} \]
has vanishing higher cohomology. 
\end{lemma}
\begin{proof} The cases $0 \leq i < \infty$ are rephrasings of \cref{lemma.thickenings-v-descent}. The case $\infty$ follows by passing to the limit using that
\begin{enumerate}
\item The category of finite projective $\mc{O}(\mc{B})$-modules is equivalent to the category of compatible systems of finite projective $\mc{O}(\mc{B}) \otimes_{\mc{O}(P_{(\infty)})} \mc{O}(P_{(i)})$-modules, and
\item The inverse system of sheaves $(Q/P \mapsto M \otimes_{\mc{O}(P_{(\infty)})} \mc{O}_{Q_{(i)}})$ whose limit is $Q/P \mapsto M \otimes_{\mc{O}(P_{(\infty)})} \mc{O}_{Q_{(\infty)}}$ has no higher derived limit. 
\end{enumerate}
\end{proof}

\begin{lemma}\label{lemma.Y-vb-descent}
Let $E/\mbb{Q}_p$ be a finite extension with residue field $\mathbb{F}_q$, let $P \in \Perf/\Spd \mbb{F}_q$ and let $Y=Y_{E,P}$. Then, for any open $U \subseteq Y$ and any locally free nilpotent thickening $\mc{B}/U$, the fibered category over $\Perf/P$ sending any $Q/P$ to the category of locally free of finite rank $\mc{O}_{\mc{B}_{Y_{E,Q}}}$-modules (for $\mc{B}_{Y_{E,Q}}=\mc{B} \times_Y Y_{E,Q}$), is a $v$-stack.
\end{lemma}
\begin{proof}
As in the proof of \cref{theorem.smooth-adic-space-mos} and following the proof of \cite[Proposition 19.5.3]{ScholzeWeinstein.BerkeleyLecturesOnPAdicGeometryAMS207}, we can reduce to descent for locally free thickenings of perfectoid spaces. This holds by the $i=0$ case of \cref{lemma.thickenings-v-descent}.
\end{proof}

\subsection{Inscribed Banach--Colmez spaces}\label{ss.BC-spaces}
\newcommand{\bct}{\boxtimes}

In this subsection we work in the inscribed context $(\Spd \mathbb{F}_q, X^\alg_{E,\Box})$, and write $\mathcal{X}$ for the functor $(X^\alg_{E,\Box})^\lf \rightarrow \Sch/E$, $\mathcal{X}/X_{E,P}^\alg \mapsto \mathcal{X}$. We write $E^\lfid$ for the  $(X^\alg_{E,\Box})^\lf$-inscribed $v$-sheaf $\mathbb{B}$ of \cref{def.BB},
\[ E^\lfid(\mathcal{X})=H^0(\mc{X}, \mc{O}).\] 

Suppose $\mc{S}$ is an inscribed $v$-sheaf and $\mc{E}: \mc{S} \rightarrow \mc{X}^*\VB$. We write $\mc{E}_{\mathcal{X}}$ for $\mc{E}(\mc{X}/X_P \rightarrow \mc{S})$ below when it will cause no confusion. 

We consider the following two functors on $(X^\alg_{E,\Box})^\lf/\mc{S}$:
\begin{align*}
\BC(\mc{E}): &\; \mc{X}/S \mapsto H^0(\mc{X}, \mc{E}_{\mc{X}}) \textrm{, and } \\ 
\BC(\mc{E}[1]): & \textrm{ the sheafification of } \left(\mc{X}/S \mapsto H^1(\mc{X}, \mc{E}_{\mc{X}}) \right)
\end{align*}
which are both naturally $E^{\lfid}$-modules. 

\begin{proposition}\label{prop.BC-inscribed}
Suppose $\mc{S}$ is an inscribed $v$-sheaf, and $\mc{E}: \mc{S} \rightarrow \mc{X}^*\VB$. Then, each of $\BC(\mc{E})$ and $\BC(\mc{E}[1])$ is an inscribed $v$-sheaf.
\end{proposition}
\begin{proof}
For $\BC(\mc{E})$, the property of being an inscribed $v$-sheaf follows from \cref{theorem.affine-scheme-mos} applied to the associated geometric vector bundle. Similarly, we obtain the result for $\BC(\mc{E}[1])$ by applying \cref{lemma.Y-vb-descent} after pulling back to $\mc{X} \times_{X_{E,\Box}} Y_{E,\Box}$ and realizing $\BC(\mc{E}[1])$ as the cokernel of a map of inscribed $v$-sheaves of abelian groups over $Y$ (as in \cite[Proposition II.2.1]{FarguesScholze.GeometrizationOfTheLocalLanglandsCorrespondence}), which remains inscribed by \cref{prop.inscribed-abelian-sheaves-abelian-cat}. 
\end{proof}

\begin{proposition}\label{prop.BC-vanishing} Suppose $\mc{S}$ is an inscribed $v$-sheaf and $\mc{E} :\mc{S} \rightarrow \mc{X}^*\VB$ factors through $\overline{\mc{E}}: \mc{S} \rightarrow \overline{(\mc{X}^*\VB)}=X_{E,\Box}^*\VB$. If the Harder-Narasimhan slopes of $\overline{\mc{E}}$ at each geometric point $\Spd(C,C^+)\rightarrow \overline{\mc{S}}$ are non-negative, then 
\[\BC(\mc{E}[1])|_{(X^\alg_{E,\Box})^\lfp}=0.\]
\end{proposition}
\begin{proof}
Given a map $f:\mc{X}/X_{E,P} \rightarrow \mc{S}$, we consider the pullback of $\BC(\mc{E}[1])$ to $\Perf/P$. Viewing $f^* \mc{E}$ as a vector bundle on $\mc{X}$, it follows from our assumption that it is isomorphic to 
\[ \overline{f}^*\overline{\mc{E}} \otimes_{\mc{O}_{X_{E,P}}} \mc{O}_{\mc{X}}=:\mc{V}. \]
The pullback of $\BC(\mc{E}[1])$ to $\Perf/P$ can thus be viewed as the $v$-sheafification of 
\[ (P'\rightarrow P) \mapsto H^1(X_{E,P'}, \mc{V}|_{X_{E,P'}}).\]
By our condition the slopes of $\overline{f}^*\mc{E}_0$ and $\mc{O}_{\mc{X}}$, the slopes of $\mc{V}$ are all nonnegative, thus this $v$-sheafication is trivial by \cite[Proposition II.3.4-(ii)]{Fargues.GeometrizationOfTheLocalLanglandsCorrespondenceAnOverview}.
\end{proof}

\subsection{Newton strata}\label{ss.NewtonStrataBG}
We maintain the notation of \cref{ss.BC-spaces}, and fix $G/E$ a connected linear algebraic group.  
For any $b \in G(\breve{E})$, we have a canonical map $\mc{E}_b: \Spd \Fqbar \rightarrow \mc{X}^*\BG$. It is induced by pullback along $\mc{X} \rightarrow X_{E,\Box}$ from the usual construction
\[ \mc{E}_b: \Spd \Fqbar \rightarrow \mr{Bun} G=X_{E,\Box}^*\BG \]
which sends $P/\Spd \Fqbar$ to the descent of the trivial $G$-torsor $\mc{E}_\triv$ on $Y_{E,P}$ via the isomorphism $\sigma^* \mc{E}_\triv = \mc{E}_\triv \xrightarrow{b} \mc{E}_\triv$. The isomorphism class of $\mc{E}_b$ depends only on the $\sigma$-conjugacy class $[b]$, and we write $(\mc{X}^*\BG)^{[b]} \subseteq \mc{X}^*\BG \times \Spd \Fqbar$ for the image of the graph of $b$ (i.e. $\mc{E}: \mc{X} \rightarrow \BG\times \Spd \Fqbar$ factors through $(\mc{X}^*\BG)^{[b]}$ if and only if it is $v$-locally isomorphic to $\mc{E}_b$).  

\begin{lemma}$(\mc{X}^*\BG)^{[b]}$ is an inscribed $v$-stack. 
\end{lemma}
\begin{proof}
By \cref{thm.BG-pull-back-inscribed}, it remains only to check the inscribed property, and for this the only question is essential surjectivity in \cref{eq.inscribed-def}. Thus suppose $\mc{E}$ is a $G$-bundle on $\mc{X}_1 \sqcup_{\mc{X}_0} \mc{X}_2$ such that $\mc{E}|_{\mc{X}_i}$ is $v$-locally isomorphic to $b$ for each $i$. Then, passing to a sufficiently large cover, we can assume each is isomorphic to $\mc{E}_b$. Since the automorphisms of $\mc{E}_b|_{\mc{X}_1}$ surject onto those of $\mc{E}_b|_{\mc{X}_0}$, we can then glue to get an isomorphism with $\mc{E}_b$ over $\mc{X}_1 \sqcup_{\mc{X}_0} \mc{X}_2$. 
\end{proof}

Recall that $[b]$ is called basic if, for $\mf{g}:=\Lie G$, equipped with the adjoint action, the associated isocrystal $\mf{g}_{b}$ is trivial (equivalently, the slope morphism for $[b]$ is central in $G$).  

\begin{remark}\label{remark.stratification-bunG}
When $G$ is reductive, we have $\overline{(\mc{X}^*\BG)^{[b]}}=\mr{Bun}_G^{[b]}$, for the right hand side as defined in  \cite[Chapter III]{FarguesScholze.GeometrizationOfTheLocalLanglandsCorrespondence}. This is not ``by definition", as the right-hand side is defined by the condition of being isomorphic to $\mc{E}_b$ at all geometric points, but it follows from the results of \cite[Chapter III]{FarguesScholze.GeometrizationOfTheLocalLanglandsCorrespondence}. For general $G$, the right hand side is defined by the condition at geometric points in \cite[\S4]{HoweKlevdal.AdmissiblePairsAndpAdicHodgeStructuresIITheBiAnalyticAxLindemannTheorem}, and the proof of \cite[Theorem 3]{HoweKlevdal.AdmissiblePairsAndpAdicHodgeStructuresIITheBiAnalyticAxLindemannTheorem} shows that this agrees with $\overline{(\mc{X}^*\BG)^{[b]}}$ when $[b]$ is basic; we expect that they agree in general.
\end{remark}

\begin{proposition}\label{prop.b-basic-open-stratum}
Suppose $[b]$ is basic. Then, the restriction of $(\mc{X}^*\BG)^{[b]}$ to $X_{E,\Box}^\lfp$ is open in the following strong sense:
\begin{enumerate}
\item The underlying $v$-stack $\overline{(\mc{X}^*\BG)^{[b]}}=\Bun_G^{[b]}$ is an open substack of $\Bun_G$, and
\item $(\mc{X}^*\BG)^{[b]}$ is the formal neighborhood of $\Bun_G^{[b]}$, that is, 
\[ (\mc{X}^*\BG)^{[b]} = (\mc{X}^*\BG) \times_{\overline{\mc{X}^*\BG}=\Bun_G} \Bun_G^{[b]}\]
\end{enumerate}
\end{proposition}
\begin{proof}
The first claim follows from \cref{remark.stratification-bunG} and \cite[Theorem 4.3]{HoweKlevdal.AdmissiblePairsAndpAdicHodgeStructuresIITheBiAnalyticAxLindemannTheorem}.

For the second case, we must show that, for $\mc{E}$ a $G$-bundle on $\mc{X}/X_{E,P}$, if $\mc{E}|_{X_{E,P}}$ is $v$-locally isomorphic to $\mc{E}_b$ then so is $\mc{E}$. Writing $\mc{X}_{(i)}$ for the thickening corresponding to $\mc{O}_{\mc{X}}/\mc{I}^{i+1}$, we can assume this holds for $\mc{E}_{\mc{X}_{(i-1)}}$ and then extend to $\mc{E}_{\mc{X}_{(i)}}$. Passing to a cover, we may assume $\mc{E}_{\mc{X}_{(i)}}=\mc{E}_{b}|_{\mc{X}_{(i)}}.$ Then, we claim the isomorphism class of such an extension is classified by an element of $H^1(X_{E,P}, \mf{g} \otimes \mc{I}^{i}/\mc{I}^{i+1})$ --- this follows because the automorphism group over $\mc{X}_{(i+1)}$ of $\mc{E}_b|_{\mc{X}_{(i)}}$ is an extension of the automorphism group of $\mc{E}_b|_{\mc{X}_{(i-1)}}$  
by $\mc{E}(\mf{g}_b) \otimes (\mc{I}^{i}/\mc{I}^{i+1})$ which, because of the basic hypothesis, is just $\mf{g} \otimes (\mc{I}^{i}/\mc{I}^{i+1})$. The $v$-local vanishing of this class then follows from  \cite[Proposition II.3.4-(ii)]{Fargues.GeometrizationOfTheLocalLanglandsCorrespondenceAnOverview} (as in \cref{prop.BC-vanishing}).
\end{proof}

\section{The $\mbb{B}^+_\dR$ affine Grassmannian}\label{s.affine-grassmannian}
In this section we define the inscribed $\mbb{B}^+_\dR$ affine Grassmannian associated to a connected linear algebraic group over a $p$-adic field and study its basic properties. In \cref{ss.bdrag-def-first-prop} we give the definition and establish its first properties, including the computation of its tangent bundle. In \cref{ss.bdrag-schubert-cells} we define its Schubert cells and compute their tangent bundles. In \cref{ss.bdrag-bialynicki-birula} we extend the definition of the Bialynicki-Birula map to the inscribed setting, and compute its derivative. Finally, in \cref{ss.bdrag-big-schubert-cells} we study the Schubert cells in $G(\mbb{B}_\dR)$ and their natural period maps. The main result of that subsection is \cref{theorem.SchubertCellDiagrams}, which can be viewed as a toy version of the more refined computations for moduli of modifications over the Fargues--Fontaine curve of \cref{theorem.bounded-structure} and \cref{corollary.bounded-mod-tangent}. 

\subsubsection{Notation}
We fix a $p$-adic field $L$ and work in the inscribed context $(\AffPerf/\Spd L, \Box_{(\infty)}^{\sharp-\alg})$ (which we can also interpret as working in the inscribed context $(\AffPerf/\Spd \mathbb{Q}_p, \Box_{(\infty)}^{\sharp-\alg})$ but with all objects over $\Spd L$). We write $\mbb{B}^+_\dR$ for the $\Box_{(\infty)}^{\sharp-\alg}$-inscribed $v$-sheaf $\mbb{B}$ of \cref{def.BB}, 
\[ \mbb{B}^+_\dR: \mc{B}/P^{\sharp-\alg}_{(\infty)} \mapsto \mc{O}(\mc{B}).\]
It is a sheaf of $L$-algebras by the natural map $\Box^{\sharp-\alg}_{(\infty)} \rightarrow \Spec L$ on $\Spd L$. We note that the functor $(\mc{B}/P_{(\infty)}^{\sharp-\alg}) \mapsto \mc{B}$ is naturally identified with $\Spec \mbb{B}^+_\dR$. 

 We write $\mbb{B}_\dR$ for the inscribed $v$-sheaf obtained by change of context \cref{ss.change-of-context} from the inscribed $v$-sheaf $\mathbb{B}$ with respect to $\Box_{(\infty)}^{\sharp-\alg}\backslash \Box^{\sharp-\alg}$, 
\[ \mbb{B}_\dR: \mc{B}/P^{\sharp-\alg}_{(\infty)} \mapsto \mc{O}\left(\mc{B} \times_{P^{\sharp-\alg}_{(\infty)}}(P^{\sharp-\alg}_{(\infty)} \backslash P^{\sharp-\alg})\right).\]
Note that $\mathbb{B}_\dR$ is naturally a $\mbb{B}^+_\dR$-algebra.

\subsection{Definition and first properties}\label{ss.bdrag-def-first-prop}

We want to define the $\mathbb{B}^+_\dR$-affine Grassmannian as a fiber product. Before doing so, we note that the fibered category $(\Spec \mathbb{B}_\dR)^* \BG$ is not covered by \cref{thm.BG-pull-back-inscribed} because we do not know if descent holds. However, we still have:

\begin{lemma}\label{lemma:BdR-BG-pullback-inscribed-prestack} Let $G/L$ be a connected linear algebraic group. Then $(\Spec \mbb{B}_\dR)^*\BG$ is an inscribed pre-stack. 
\end{lemma}
\begin{proof}
For any $\mc{B}/P^{\sharp-\alg}_{(\infty)}$ and $G$-torsors $\mathcal{G}_1$, $\mathcal{G}_2$ on $\mc{B}$, the presheaf of homomorphisms is the moduli of sections for the smooth affine scheme $\mathcal{I}som_{\mc{B}}(\mathcal{G}_1, \mathcal{G}_2)$. This is a $v$-sheaf by \cref{theorem.affine-scheme-mos}, so $\Spec(\mbb{B}_\dR)^*\BG$ is a pre-stack. 

It is inscribed by the same argument as in the proof of \cref{thm.BG-pull-back-inscribed}. 
\end{proof}

\begin{definition}For $L$ a $p$-adic field and $G/L$ connected linear algebraic group, $\Gr_G$ is the prestack on $(\Box_{(\infty)}^{\sharp-\alg})^\lf$ making the following diagram Cartesian
% https://q.uiver.app/#q=WzAsNCxbMCwxLCJcXFNwZCBMIl0sWzIsMSwiKFxcU3BlYyBcXG1hdGhiYntCfV9cXGRSKV4qIEJHIl0sWzIsMCwiKFxcU3BlYyBcXG1hdGhiYntCfV4rX1xcZFIpXiogQkciXSxbMCwwLCJcXEdyX0ciXSxbMCwxLCJcXG1je0V9X1xcdHJpdiJdLFsyLDFdLFszLDJdLFszLDBdXQ==
\[\begin{tikzcd}
	{\Gr_G} && {(\Spec \mathbb{B}^+_\dR)^* BG} \\
	{\Spd L} && {(\Spec \mathbb{B}_\dR)^* BG}
	\arrow[from=1-1, to=1-3]
	\arrow[from=1-1, to=2-1]
	\arrow[from=1-3, to=2-3]
	\arrow["{\mc{E}_\triv}", from=2-1, to=2-3]
\end{tikzcd}\]
where the right vertical arrow is restriction of $G$-bundles, $\Spd L$ is equipped with the trivial inscription (i.e. it is the final object), and $\mathcal{E}_\triv$ denotes the trivial $G$-bundle. In other words, $\Gr_G(\mathcal{B})$ classifies $G$-bundles on $\mathcal{B}=\Spec \mathbb{B}^+_\dR(\mc{B})$ equipped with a trivialization after restriction to $\Spec \mbb{B}_\dR(\mc{B})$. 
\end{definition}

Note that the automorphism group of any object in $\Gr_G$ is trivial, so that passing to isomorphism classes we can and do view it as a presheaf instead of as a fibered category. 
\begin{proposition}
$\Gr_G$ is an inscribed $v$-sheaf. 
\end{proposition}
\begin{proof}
It follows from \cref{thm.BG-pull-back-inscribed} and \cref{lemma:BdR-BG-pullback-inscribed-prestack} that it is a $v$-sheaf, and it follows from these results combined with \cref{lemma.inscribed-limits} that it is inscribed. 
\end{proof}

There is a natural action of $G(\mathbb{B}_\dR)=\mathcal{A}ut(\mc{E}_\triv)$ on $\Gr_G$ by changing the trivialization. There is also a canonical point $\ast_1: \Spd L \rightarrow \Gr_G$ given by $\mathcal{E}_\triv \times_{\Id} \mathcal{E}_\triv$, i.e. by the trivial bundle with on $\Spec \mbb{B}^+_\dR$ with its canonical trivialization after restriction to $\Spec \mbb{B}_\dR$. 

\begin{proposition}\label{prop.Bdr-aff-grass-transitive}
The action of $G(\mbb{B}_\dR)$ on $\Gr_G$ is transitive in the \'{e}tale topology. In particular, the orbit map for $\ast_1$ induces an identification 
\[ \Gr_G = G(\mbb{B}_\dR)/G(\mbb{B}^+_\dR), \]
where the quotient can be formed in either the \'{e}tale or $v$-topology. 
\end{proposition}
\begin{proof}
The first claim follows from the second, so it suffices to show that any $G$-torsor on $\mathcal{B}/P^{\sharp-\alg}_{(\infty)}$ is trivial after base change to $P'^{\sharp-\alg}_{(\infty)}$ for an \'{e}tale cover $P' \rightarrow P$. The proof is then exactly as in \cite[Proposition 19.1.2]{ScholzeWeinstein.BerkeleyLecturesOnPAdicGeometryAMS207} (see also \cite[Proposition 2.1]{CesnaviciusYoucis.TheAnalyticTopologySufficesForTheBplusDrGrassmannian} and \cite[Proposition 2.2.2]{Howe.TransitivityOfTheB+dRLoopGroupActionOnSchubertCells}). 
\end{proof}

We will use this transitivity to compute the tangent bundle of $\Gr_G$. Before making this computation, we introduce some notation. 

\begin{definition}\label{def.VplusUnivOnGrG}
    For $G/L$ a connected linear algebraic group and $V \in \Rep_{G}(L)$, 
    \begin{enumerate} 
    \item Let $V^+_\univ$ denote the sheaf of $\mbb{B}^+_\dR$-modules on $\Gr_G$ defined as follows: to give a $\mc{B}$-point of $\Gr_G$ is to give a $G$-torsor $\mc{E}$ on $\mc{B}$ with a trivialization after restriction to $\mc{B} \times_{P_{(\infty)}^{\sharp-\alg}}P_{(\infty)}^{\sharp-\alg}\backslash P^{\sharp-\alg}=\Spec \mbb{B}_\dR(\mc{B})$. To such a point, we associate the projective $\mbb{B}^+_\dR(\mc{B})$-module 
\[ \mc{E}(V)=H^0(\mc{B}, \mc{E} \times^{G} (V \otimes \mc{O}_{\mc{B}})). \]
    \item Let $\varphi_{\univ}: V^+_\univ \otimes_{\mbb{B}^+_\dR} \mbb{B}_\dR \xrightarrow{\sim} V \otimes_{L} \mbb{B}_\dR$ send a point as above to the trivialization $\mc{E}(V) \otimes_{\mbb{B}^+_\dR(\mc{B})} \mbb{B}_\dR(\mc{B}) \rightarrow V \otimes_L \mbb{B}_\dR(\mc{B})$. 
    \end{enumerate}
\end{definition}

Recall that in \cref{ss.inscribed-groups} we have defined, for any group action $a$ of an inscribed group $\mc{G}$ on an inscribed $v$-sheaf $\mc{S}$ its derivative at the identity element $da_e: \Lie \mc{G} \rightarrow T_{\mc{S}}$. In the following, we identify $\Lie (G(\mbb{B}_\dR))=\mf{g}\otimes_L \mbb{B}_\dR$, where $\mf{g}=\Lie G(L)$. 

\begin{corollary}\label{corollary.tangent-bundle-affine-grassmannian}
For $a: G(\mbb{B}_\dR) \times_{\Spd L} \Gr_G \rightarrow \Gr_G$ the action map, the derivative $da_e: \mf{g} \otimes_L \mbb{B}_\dR \rightarrow T_{\Gr_G}$ at the identity section $e$ of $G(\mbb{B}_\dR)$ is a surjection of inscribed $\mbb{B}^+_\dR$-modules over $\Gr_G$ with kernel $\varphi_\univ(\mf{g}^+_{\univ})$. It induces a canonical identification of inscribed $\mbb{B}^+_\dR$-modules over $\Gr_G$
\[ T_{\Gr_G} = \left(\mf{g} \otimes \mbb{B}_\dR\right) / \mf{g}^+_{\univ} = (\mf{g}^+_\univ \otimes_{\mbb{B}^+_\dR} \mbb{B}_\dR)/\mf{g}^+_{\univ}=\mf{g}^+_{\univ} \otimes_{\mbb{B}^+_\dR} (\mbb{B}_\dR/\mbb{B}^+_\dR) \]
where here we use $\varphi_\univ$ to identify the $\mbb{B}_\dR$-modules over $\Gr_G$
\[ \mf{g}\otimes\mbb{B}_\dR=\mf{g}^+_\univ \otimes_{\mbb{B}^+_\dR} \mbb{B}_\dR. \]
\end{corollary}
\begin{proof}
From the transitivity of the action of $G(\mathbb{B}_\dR)$ on $\Gr_G$ established in \cref{prop.Bdr-aff-grass-transitive}, we find $da_e$ is surjective. It remains to compute its kernel.

The stabilizer of a point $(\mathcal{E}, \varphi: \mathcal{E}|_{\Spec \mathbb{B}_\dR} \cong \mathcal{E}_\triv)$ in $G(\mathbb{B}_\dR)$ is 
\begin{equation}\label{eq.stabilizer-conj} \varphi^{-1} \circ \Aut(\mathcal{E}) \circ \varphi \subseteq \Aut(\mathcal{E}_\triv)=G(\mathbb{B}_\dR). \end{equation}

Note that $\ul{\Aut}({\mc{E}_\univ}): (\mathcal{E}, \varphi) \mapsto \Aut(\mathcal{E})$ is the moduli of sections of the smooth affine scheme over $\mc{B}=\Spec \mbb{B}^+_\dR$ on $\Gr_G$, $(\mathcal{E}, \varphi) \mapsto \mc{A}ut_{\mc{B}}(\mathcal{E})$. It is thus an inscribed $v$-sheaf over $\Gr_G$ by \cref{theorem.affine-scheme-mos}. Since $\mc{A}ut_{\mc{B}}(\mathcal{E})$ is naturally identified with $\mathcal{E} \times^G G$ where $G$ on the right is equipped with the adjoint action, it follows from \cref{theorem.affine-scheme-mos} that $\Lie \ul{\Aut}({\mc{E}_\univ}) = \mf{g}^+_\univ$. Using the identification $\Lie (G(\mathbb{B}_\dR))=\mf{g} \otimes \mbb{B}_\dR$, \cref{eq.stabilizer-conj} identifies the tangent space of the stabilizer with $\varphi_{\univ}(\mf{g}^+_\univ)$, giving the result. 
\end{proof}

\subsection{Schubert cells in the $B^+_\dR$-affine Grassmannian}\label{ss.bdrag-schubert-cells}

Let $G/L$ be a connected linear algebraic group, and let $[\mu]$ be a conjugacy class of cocharacters of $G_{\overline{L}}$. For $\mu \in [\mu]$, we write $L(\mu) \subseteq \overline{L}$ for the field of definition of $\mu$. We write $L([\mu]) \subseteq \overline{L}$ for the field of definition of $[\mu]$, i.e. the fixed field of the stabilizer of $[\mu]$ in $\Gal(\overline{L}/L)$. For any $\mu \in [\mu]$, we obtain a point $\ast_{\mu}: \Spd L(\mu) \rightarrow \Gr_G$ whose value on any $\mc{B}$ is $\xi^{\mu} \cdot \ast_1$ where $\xi$ is any generator of $\Fil^1 \mathbb{B}^+_\dR(\mc{B})$. This is well defined because, given another generator $\xi'$, $\xi^{-\mu}  (\xi')^{\mu} =(\xi'/\xi)^\mu \in G(\mathbb{B}^+_\dR)$. 

Because the elements of $[\mu]$ are conjugate over $\overline{L}$, we find
\begin{lemma} 
The $v$-sheaf image of $G(\mbb{B}^+_\dR) \cdot \ast_{\mu} \subseteq \Gr_G \times \Spd L(\mu)$ in $\Gr_{G} \times \Spd L([\mu])$ is independent of the choice of $\mu \in [\mu]$. 
\end{lemma}

\begin{definition}
Let $\Gr_{[\mu]} \subseteq \Gr_G \times \Spd L([\mu])$ be the $v$-sheaf image of $G(\mbb{B}^+_\dR)\cdot \ast_{\mu} \subseteq \Gr_G \times \Spd L(\mu)$ in $\Gr_G \times \Spd L([\mu])$ for any choice of $\mu \in [\mu].$
\end{definition}

\begin{proposition}\label{prop.schubert-cell}
 Let $G/L$ be a connected linear algebraic group. The action of $G(\mbb{B}^+_\dR)$ on $\Gr_{[\mu]}$ is transitive in the $v$-topology and $\Gr_{[\mu]}$ is inscribed. Moreover, the derivative
\[ d\iota_{[\mu]}: T_{\Gr_{[\mu]}} \rightarrow  \iota_{[\mu]}^* T_{\Gr_G} \]
of the inclusion map
\[ \iota_{[\mu]}: \Gr_{[\mu]} \hookrightarrow \Gr_{G} \times_{\Spd L} \Spd L([\mu]) \]
induces, under the identification of \cref{corollary.tangent-bundle-affine-grassmannian}, an isomorphism 
\[ T_{\Gr_{[\mu]}}=\mf{g}\otimes_L \mbb{B}^+_\dR / \left(\mf{g}\otimes_L \mbb{B}^+_\dR \cap \mf{g}^+_\univ\right) = \left(\mf{g}\otimes_L \mbb{B}^+_\dR + \mf{g}^+_\univ\right) / \mf{g}^+_\univ  \]
such that the natural quotient map from $\mf{g} \otimes_L \mbb{B}^+_\dR$ is the derivative $da_e$ at the identity section of the action map $a: G(\mathbb{B}^+_\dR) \times_{\Spd L([\mu])} \Gr_{[\mu]} \rightarrow \Gr_{[\mu]}$. Moreover, both
\[ \mf{g}^+_\mr{min} := \mf{g} \otimes_L \mbb{B}^+_\dR \cap \mf{g}^+_\univ \textrm{ and } \mf{g}^+_\mr{max}:=\mf{g} \otimes_L \mbb{B}^+_\dR + \mf{g}^+_\univ\]
are locally free $\mbb{B}^+_\dR$-modules over $\Gr_{[\mu]}$ whose values on any test object $\mc{B}$ can be formed as a literal intersection/sum of $\mbb{B}^+_\dR(\mc{B})$-submodules in $\mf{g}\otimes_L\mbb{B}_\dR(\mc{B})=\mf{g}^+_\univ(\mc{B}) \otimes_{\mbb{B}^+_\dR(\mc{B})}\mbb{B}_\dR(\mc{B})$. 
\end{proposition}
\begin{proof}
For the first part, we can reduce to the case $L=L(\mu)$ for some $\mu \in [\mu]$, so that 
\[ \Gr_{[\mu]}=G(\mathbb{B}^+_\dR) \cdot \ast_{\mu} = G(\mathbb{B}^+_\dR)/\mathrm{Stab}(\ast_{\mu}). \]
and the transitivity is clear. It follows from \cref{lemma.inscribed-limits} that $\mathrm{Stab}(\ast_{\mu})$ is inscribed, and then from \cref{prop.inscribed-quotient} that $\Gr_{[\mu]}$ is inscribed. The computation of $d\iota_{[\mu]}$ then follows by comparing with \cref{corollary.tangent-bundle-affine-grassmannian}. 

Finally, we verify the claims about $\mf{g}^+_\mr{min}$ and $\mf{g}^+_{\mr{max}}$. For $\mf{g}^+_{\mr{min}}$, its formation on any test object is clearly as a literal intersection, so it suffices to show it is locally free. But this follows from the $v$-stack property of \cref{thm.BG-pull-back-inscribed}, since it is easily seen to be locally free after passage to a $v$-cover where $\mf{g}^+_\univ=g \cdot \ast_{[\mu]}$ for $g \in G(\mbb{B}^+_\dR)$. For $\mf{g}^+_{\mr{max}}$, the same argument shows it is locally free, but we must explain why its formation on any test object is a literal sum (since a priori there is a $v$-sheafification). This follows from the exact sequence 
\[ 0 \rightarrow \mf{g}^+_{\mr{min}} \rightarrow \mf{g}^+_{\univ}(\mc{B}) \oplus \mf{g}\otimes_L \mbb{B}^+_\dR \rightarrow \mf{g}^+_{\mr{max}} \rightarrow 0 \]
and the $v$-acyclicity of locally free $\mathbb{B}^+_\dR$-modules given by the second part of \cref{lemma.thickenings-v-descent} (applied to $\mf{g}^+_{\mr{min}}$ in the long exact sequence of $v$-cohomology).
\end{proof}

\begin{remark}
Note that, in \cite{ScholzeWeinstein.BerkeleyLecturesOnPAdicGeometryAMS207, HoweKlevdal.AdmissiblePairsAndpAdicHodgeStructuresIITheBiAnalyticAxLindemannTheorem}, the Schubert cells in the $\mathbb{B}^+_\dR$-affine Grassmannian on $\AffPerf/\Spd L$ are defined to consist of those sections whose restriction to any geometric point lies in the orbit of $\ast_{\mu}$. With this definition, it is only clear that $\overline{\Gr_{[\mu]}}$ is contained in the Schubert cell of loc. cit., not that they are equal. The equality in general follows from \cite[Theorem 1.0.2]{Howe.TransitivityOfTheB+dRLoopGroupActionOnSchubertCells}, building on the reductive case treated in \cite[Proposition VI.2.4]{FarguesScholze.GeometrizationOfTheLocalLanglandsCorrespondence}.

This is not enough to show that in the inscribed setting we could make the definition of the Schubert cell by only considering thickenings over geometric points, however, by a restriction of scalars argument, it is enough if we restrict from locally free nilpotent thickenings to only considering $L$-constant nilpotent thickenings, i.e. those base changed from $L$ (cf. \cite[Example 1.1.1 and Question 1.1.2]{Howe.TransitivityOfTheB+dRLoopGroupActionOnSchubertCells}). It is evident from the arguments above that the key property for computing the tangent bundles is that the action be transitive, which is why we have adopted the definition above rather than the pointwise definition. 
\end{remark}

\begin{remark} If $G$ is reductive, then as a corollary of \cref{prop.schubert-cell}, we find that $T_{\Gr_{[\mu]}}$ is $v$-locally isomorphic to $\bigoplus_\alpha \mbb{B}^+_\dR/\Fil^{\langle \alpha, \mu\rangle}\mbb{B}^+_\dR$, 
where the sum is over any choice of positive roots $\alpha$ for $G_{\overline{L}}$ and $\mu$ is chosen to be dominant. 
\end{remark}

\subsection{The Bialynicki-Birula map}\label{ss.bdrag-bialynicki-birula}
In the following, we write $\mc{O}$ for the inscribed $v$-sheaf over $\Spd L$, $\mc{B}/P^{\sharp\alg}_{(\infty)}\mapsto \mc{O}(\mc{B} \times_{P^{\sharp-\alg}_{(\infty)}} P^{\sharp-\alg}).$ It is the inscribed $v$-sheaf $\mbb{B}$ associated to the pair $(P^{\sharp-\alg},\Spd L)$ viewed in our setting by change of context as in \cref{ss.change-of-context} along $P^{\sharp-\alg} \rightarrow P^{\sharp-\alg}_{(\infty)}$. 

For $G/L$ a connected linear algebraic group and $V \in \Rep_G(L)$ we have the universal $\mathbb{B}^+_\dR$-lattice $V^+_\univ \subseteq V \otimes_L \mbb{B}_\dR$ over $\Gr_G$ of \cref{def.VplusUnivOnGrG}. This lattice induces a natural trace filtration of $V \otimes_L \mc{O}$ by $\mathcal{O}$-modules 
\begin{align*} \Fil^i_{V^+_\univ} (V \otimes_L \mc{O}) & := \left( \Fil^i \mbb{B}^+_\dR \cdot V^+_\univ \cap V \otimes_L \mbb{B}^+_\dR \right) / \left( \Fil^i \mbb{B}^+_\dR \cdot V^+_\univ \cap V \otimes_L \Fil^1 \mbb{B}^+_\dR \right) \\
& \subseteq \left(V \otimes_L \mbb{B}^+_\dR \right) / \left( V \otimes_L \Fil^1 \mbb{B}^+_\dR \right)=V \otimes_L \mc{O}. \end{align*}

This filtration may not be by locally free modules and even when it is it may not be an exact functor from $\Rep_G(L)$ to filtered $\mathcal{O}$-modules. After restricting to a Schubert cell, however, it is. This can be verified after passing to $v$-cover, which, by definition of the Schubert cell, can be chosen so that that the filtration is isomorphic to that defined by $\ast_{\mu}$ for some $\mu \in [\mu]$. Computing directly one finds that this latter is the filtration $\Fil_{\mu^{-1}}$ over $L(\mu)$, where for any cocharacter $\tau$ we define the associated filtration by
\[ \Fil^i_{\tau}(V) = \bigoplus_{j \geq i} V[j], \textrm{ for $V[j]$ the isotypic subspace where $\tau(z)$ acts as $z^j$.} \]

Recall that for any conjugacy class of cocharacters $[\tau]$ there is a flag variety $\Fl_{[\tau]}/L([\tau])$ parameterizing filtrations on the trivial $G$-torsor that are of type $[\tau]$, i.e. locally isomorphic to $\Fil_{\tau}$ for $\tau \in [\tau]$. 

\begin{theorem}\label{theorem.bb-derivative}
For any conjugacy class of cocharacters $[\mu]$ of $G_{\overline{L}}$, 
the restriction of $V \mapsto \Fil^\bullet_{{V^+_\univ}} (V \otimes \mc{O})$ to $\Gr_{[\mu]}$ is a filtration of the trivial $G$-torsor of type $[\mu^{-1}]$. The resulting map $\BB: \Gr_{[\mu]} \rightarrow \Fl_{[\mu^{-1}]}^\lfid$ is equivariant along the natural map $G(\mbb{B}^+_\dR) \twoheadrightarrow G(\mc{O})$, and its derivative fits into the commuting diagram 
% https://q.uiver.app/#q=WzAsNSxbMCwwLCJUX3tcXEdyX3tbXFxtdV19fSJdLFswLDIsIlxcQkJeKlRfe1xcRmxfe1tcXG11XnstMX1dfV5cXGxmaWR9Il0sWzEsMCwiXFxtZntnfVxcb3RpbWVzX0wgXFxtYmJ7Qn1eK19cXGRSIC9cXGxlZnQoXFxtZntnfVxcb3RpbWVzX0wgXFxtYmJ7Qn1eK19cXGRSIFxcY2FwIFxcbWZ7Z31fe1xcdW5pdn1eK1xccmlnaHQpIl0sWzEsMiwiXFxsZWZ0KFxcbWZ7Z30gXFxvdGltZXNfTCBcXG1je099XFxyaWdodCkvXFxGaWxeMCAoXFxtZntnfSBcXG90aW1lc19MIFxcbWN7T30pIl0sWzEsMSwiXFxtZntnfSBcXG90aW1lc19MIFxcbWJie0J9XitfXFxkUiAvIFxcbGVmdChcXG1me2d9XFxvdGltZXNfTCBcXG1iYntCfV4rX1xcZFIgXFxjYXAgXFxtZntnfV97XFx1bml2fV4rICsgXFxtZntnfSBcXG90aW1lc19MIFxcRmlsXjEgXFxtYmJ7Qn1eK19cXGRSIFxccmlnaHQpIl0sWzAsMSwiZEJCIiwyXSxbMCwyLCI9IiwxLHsic3R5bGUiOnsiaGVhZCI6eyJuYW1lIjoibm9uZSJ9fX1dLFsxLDMsIj0iLDEseyJzdHlsZSI6eyJoZWFkIjp7Im5hbWUiOiJub25lIn19fV0sWzIsNCwiIiwxLHsic3R5bGUiOnsiaGVhZCI6eyJuYW1lIjoiZXBpIn19fV0sWzQsMywiPSIsMSx7InN0eWxlIjp7ImhlYWQiOnsibmFtZSI6Im5vbmUifX19XV0=
\[\begin{tikzcd}
	{T_{\Gr_{[\mu]}}} & {\mf{g}\otimes_L \mbb{B}^+_\dR /\left(\mf{g}\otimes_L \mbb{B}^+_\dR \cap \mf{g}_{\univ}^+\right)} \\
	& {\mf{g} \otimes_L \mbb{B}^+_\dR / \left(\mf{g}\otimes_L \mbb{B}^+_\dR \cap \mf{g}_{\univ}^+ + \mf{g} \otimes_L \Fil^1 \mbb{B}^+_\dR \right)} \\
	{\BB^*T_{\Fl_{[\mu^{-1}]}^\lfid}} & {\left(\mf{g} \otimes_L \mc{O}\right)/\Fil^0 (\mf{g} \otimes_L \mc{O})}
	\arrow["{=}"{description}, no head, from=1-1, to=1-2]
	\arrow["dBB"', from=1-1, to=3-1]
	\arrow[two heads, from=1-2, to=2-2]
	\arrow["{=}"{description}, no head, from=2-2, to=3-2]
	\arrow["{=}"{description}, no head, from=3-1, to=3-2]
\end{tikzcd}\]
where the top horizontal equality is from \cref{prop.schubert-cell} and the bottom horizontal equality follows from \cref{theorem.smooth-adic-space-mos}
and the usual computation of the tangent bundle of the flag variety $\Fl_{[\mu^{-1}]}$. 
\end{theorem}
\begin{proof}
We have explained the existence of $\BB$ above. The equivariance follows from the construction, and then the computation of the derivative is an immediate consequence since the tangent bundles of $\Gr_{[\mu]}$ and $T_{\Fl_{[\mu^{-1}]}^\lfid}$ are both expressed here in the form that is obtained  by differentiating the transitive group actions of $G(\mathbb{B}^+_\dR)$ and $G(\mathcal{O})$. 
\end{proof}

The computation of $d\BB$ is closely related to Griffiths transversality. 

\begin{corollary}\label{cor.griffiths-transversality}
Suppose $G/L$ is a connected linear algebraic group, $M/L$ is a non-archimedean extension, and $S/M$ is a smooth rigid analytic variety. If $f: (S/M)^\lfid \rightarrow \Gr_{[\mu]} \times_{\Spd L} \Spd M$ is a map of inscribed $v$-sheaves over $\Spd M$, then $\BB \circ f$ satisfies Griffiths transversality for the trivial connection on the trivial $G$-torsor, i.e. $d(\BB \circ f )$ factors through
\[ \gr^{-1}\left(\mf{g} \otimes_L \mc{O} \right)=\Fil^{-1} \left(\mf{g} \otimes_L \mc{O} \right) / \Fil^0 (\mf{g} \otimes_L \mc{O}) \subseteq (\BB \circ f)^* T_{\Fl_{[\mu]}^\lfid} \]
\end{corollary}
\begin{proof}
Suppose given $f: (S/M)^\lfid \rightarrow \Gr_{[\mu]} \times_{\Spd L} \Spd M$. Then $df$ is a map of $\mbb{B}^+_\dR$-modules over $(S/M)^\lfid$, and because $T_{(S/M)^\lfid}$ is annihilated by $\Fil^1 \mbb{B}^+_\dR$ we find that $df$ factors through the part of $f^*T_{\Gr_{[\mu]}}$ annihilated by $\Fil^1 \mbb{B}^+_\dR$. Using the description of \cref{prop.schubert-cell}, this is given by
\[ \left( \mf{g}\otimes \mbb{B}^+_\dR \cap \Fil^{-1} \mbb{B}^+_\dR \cdot \mf{g}^+_\univ \right) / \left(\mf{g}\otimes \mbb{B}^+_\dR \cap \mf{g}_{\univ}^+\right). \]
It then follow from \cref{theorem.bb-derivative} that $d (\BB \circ f)$ factors through 
\[ \Fil^{-1} \left(\mf{g} \otimes \mc{O} \right) / \Fil^0 (\mf{g} \otimes \mc{O}), \]
i.e. that $f\circ \BB$ satisfies the Griffiths transversality condition for the trivial connection on the trivial $G$-torsor. 
\end{proof}

We will return to this connection with Griffiths transversality in \cref{ss.hodge-period-maps}. 

\subsection{Schubert cells in $G(\mbb{B}_\dR)$}\label{ss.bdrag-big-schubert-cells}

Note that we have two natural right actions of $G(\mbb{B}_\dR)$ on itself: the action $a_1$ by right multiplication and the action $a_2$ by left multiplication by the inverse. We also have two natural maps 
$\pi_1, \pi_2: G(\mbb{B}_\dR) \rightarrow \Gr_G$ defined 
\[ \pi_1(c)=c \cdot \ast_1 \textrm{ and } \pi_2(c)=c^{-1} \cdot \ast_1 \]

\begin{definition}
Let $G/L$ be a connected linear algebraic group and $[\mu]$ a conjugacy class of cocharacters of $G_{\overline{L}}$. We define the associated Schubert cell in $G(\mathbb{B}_\dR)$ as 
\begin{equation}\label{eq.big-schubert-cell-fiber-prod} C_{[\mu]}:= \pi_1 \times_{\Gr_G} \Gr_{[\mu]}=\pi_2 \times_{\Gr_G} \Gr_{[\mu^{-1}]}.\end{equation} 
\end{definition}
By \cref{lemma.inscribed-limits}, $C_{[\mu]}$ is an inscribed $v$-sheaf.

\begin{remark}Note that, if we fix $\mc{B} / \Spd L$, a lift to $\mc{B} \rightarrow \Spd L(\mu)$, and a generator $\xi$ of $\Fil^1 \mathbb{B}^+_\dR(\mc{B})$, then $C_{[\mu]}|_{\mc{B}}$ is the two-sided orbit $G(\mbb{B}^+_\dR) \cdot \xi^\mu \cdot G(\mbb{B}^+_\dR)$. One could adapt this observation into and equivalent definition of $C_{[\mu]}$  similar to the definition of $\Gr_{[\mu]}$ used above. 
\end{remark}

\begin{theorem}\label{theorem.SchubertCellDiagrams}
The actions $a_1$ and $a_2$ restrict to actions of $G(\mbb{B}^+_\dR)$ on $C_{[\mu]}$ such that:
\begin{enumerate}
    \item the map $\pi_1$ is a $G(\mbb{B}^+_\dR)$-torsor over $\Gr_{[\mu]}$ for the action $a_1$ and is equivariant for the action $a_2$, and 
    \item the map $\pi_2$ is a $G(\mbb{B}^+_\dR)$-torsor over $\Gr_{[\mu^{-1}]}$ for the action $a_2$ and is equivariant for the action $a_1$. 
\end{enumerate}
The $\mbb{B}^+_\dR$-module on $C_{[\mu]}$
\[ \mf{g}^+_{\mr{max}}:=\mf{g}\otimes_L \mbb{B}^+_\dR + \Ad(c^{-1})(\mf{g} \otimes_L \mbb{B}^+_\dR) \cong^{\Ad(c)} \Ad(c)(\mf{g}\otimes_L \mbb{B}^+_\dR) + \mf{g} \otimes_L \mbb{B}^+_\dR. \]
is naturally identified with the preimage under $\pi_1$ or $\pi_2$ of $\mf{g}^+_{\mr{max}}$ of \cref{prop.schubert-cell}. In particular, it is locally free of finite rank and can be formed as literal sum of modules on any test object as in \cref{prop.schubert-cell}. 

The product action $a$ of $G(\mbb{B}^+_\dR) \times G(\mbb{B}^+_\dR)$ is transitive, and $da_e$ induces an isomorphism $T_{C_[\mu]} =\mf{g}^+_{\mr{max}}$ of $\mbb{B}^+_\dR$-modules over $C_{[\mu]}$ fitting into the commutative diagram
% https://q.uiver.app/#q=WzAsOSxbMSwyLCJcXG1me2d9XitfXFxtcnttYXh9Il0sWzQsMiwiVF97Q197W1xcbXVdfX0iXSxbMCwyLCJcXG1me2d9XFxvdGltZXNfTCBcXG1iYntCfV9cXGRSICJdLFsxLDAsIlxcbWZ7Z31cXG90aW1lc19MIFxcbWJie0J9XitfXFxkUiJdLFsxLDQsIlxcbWZ7Z31cXG90aW1lc19MIFxcbWJie0J9XitfXFxkUiJdLFsyLDEsIlxccGlfMV4qVF97XFxHcl97W1xcbXVdfX0iXSxbMiwzLCJcXHBpXzJeKlRfe1xcR3Jfe1tcXG11XnstMX1dfX0iXSxbMSwzLCJcXGZyYWN7XFxtZntnfV4rX3tcXG1ye1xcbWF4fX19e1xcQWQoY157LTF9KShcXG1me2d9XFxvdGltZXNfTCBcXG1iYntCfV4rX1xcZFIpfSJdLFsxLDEsIlxcZnJhY3tcXG1me2d9Xitfe1xcbXJ7XFxtYXh9fX17XFxtZntnfVxcb3RpbWVzX0wgXFxtYmJ7Qn1eK19cXGRSfSAiXSxbMywyLCJ0XFxtYXBzdG8gdCIsMix7ImN1cnZlIjo1fV0sWzQsMiwidFxcbWFwc3RvIC1cXEFkKGNeey0xfSkodCkiLDAseyJjdXJ2ZSI6LTV9XSxbMCwyLCIiLDAseyJzdHlsZSI6eyJ0YWlsIjp7Im5hbWUiOiJob29rIiwic2lkZSI6ImJvdHRvbSJ9fX1dLFszLDEsIihkYV8xKV9lIiwwLHsiY3VydmUiOi01fV0sWzQsMSwiKGRhXzIpX2UiLDIseyJjdXJ2ZSI6NX1dLFsxLDUsImRcXHBpXzEiLDFdLFsxLDYsImRcXHBpXzIiLDFdLFsxLDAsIj0iLDEseyJzdHlsZSI6eyJ0YWlsIjp7Im5hbWUiOiJhcnJvd2hlYWQifX19XSxbNiw3LCI9IiwxLHsic3R5bGUiOnsidGFpbCI6eyJuYW1lIjoiYXJyb3doZWFkIn19fV0sWzUsOCwiPSIsMSx7InN0eWxlIjp7InRhaWwiOnsibmFtZSI6ImFycm93aGVhZCJ9fX1dLFswLDhdLFswLDddLFszLDgsIjAiLDFdLFs0LDcsIjAiLDFdXQ==
\[\begin{tikzcd}
	& {\mf{g}\otimes_L \mbb{B}^+_\dR} \\
	& {\frac{\mf{g}^+_{\mr{\max}}}{\mf{g}\otimes_L \mbb{B}^+_\dR} } & {\pi_1^*T_{\Gr_{[\mu]}}} \\
	{\mf{g}\otimes_L \mbb{B}_\dR } & {\mf{g}^+_\mr{max}} &&& {T_{C_{[\mu]}}} \\
	& {\frac{\mf{g}^+_{\mr{\max}}}{\Ad(c^{-1})(\mf{g}\otimes_L \mbb{B}^+_\dR)}} & {\pi_2^*T_{\Gr_{[\mu^{-1}]}}} \\
	& {\mf{g}\otimes_L \mbb{B}^+_\dR}
	\arrow["0"{description}, from=1-2, to=2-2]
	\arrow["{t\mapsto t}"', curve={height=30pt}, from=1-2, to=3-1]
	\arrow["{(da_1)_e}", curve={height=-30pt}, from=1-2, to=3-5]
	\arrow["{=}"{description}, tail reversed, from=2-3, to=2-2]
	\arrow[from=3-2, to=2-2]
	\arrow[hook', from=3-2, to=3-1]
	\arrow[from=3-2, to=4-2]
	\arrow["{d\pi_1}"{description}, from=3-5, to=2-3]
	\arrow["{=}"{description}, tail reversed, from=3-5, to=3-2]
	\arrow["{d\pi_2}"{description}, from=3-5, to=4-3]
	\arrow["{=}"{description}, tail reversed, from=4-3, to=4-2]
	\arrow["{t\mapsto -\Ad(c^{-1})(t)}", curve={height=-30pt}, from=5-2, to=3-1]
	\arrow["{(da_2)_e}"', curve={height=30pt}, from=5-2, to=3-5]
	\arrow["0"{description}, from=5-2, to=4-2]
\end{tikzcd}\]
\end{theorem}
\begin{proof}
The equivariant torsor structures for $\pi_1$ and $\pi_2$ are immediate from the definitions and \cref{prop.schubert-cell}, and so is the transitivity of the product action $a$. 

The description of $\mf{g}^+_{\mr{max}}$ in terms of $\pi_1$ and $\pi_2$ follows from the definition of $\pi_1$ and $\pi_2$ and \cref{prop.schubert-cell}. Indeed, $\pi_1$ classifies the trivial $G$-torsor on $\Spec \mbb{B}^+_\dR$ equipped with the trivialization on $\Spec \mbb{B}_\dR$ given by left multiplication by $c$, while $\pi_2$ classifies the trivial $G$-torsor on $\Spec \mbb{B}^+_\dR$ equipped with the the trivialization on $\Spec \mbb{B}_\dR$ given by left multiplication by $c^{-1}$.  

The transitivity of the product action $a$ induces a surjection 
\[ da_e: \mf{g} \otimes \mbb{B}^+_\dR \oplus \mf{g} \otimes \mbb{B}^+_\dR \twoheadrightarrow T_{C_{[\mu]}}.  \]
If we compose with $T_{C_{[\mu]}}\hookrightarrow T_{G(\mbb{B}_\dR)}|_{C_{[\mu]}}$ and identify the latter with $\mf{g} \otimes \mbb{B}_\dR$ using right-invariant vector fields, then $da_e$ is given by 
\[ (t_1, t_2) \mapsto t_1 - \Ad(c^{-1})(t_2). \]
Indeed, we have $t_2^{-1} c t_1= c (c^{-1} t_2^{-1} c) t_1$. The image of $da_e$ is thus $\mf{g}^+_{\mr{max}}$, giving the claimed isomorphism, and the commutativity of the diagram follows also from this computation and comparison with \cref{prop.schubert-cell}. 
\end{proof}

\section{Inscribed Hodge, lattice Hodge,  and Liu-Zhu period maps}\label{s.hodge-etc-period-maps}

In this section, we define the inscribed Hodge, lattice Hodge, and Liu-Zhu period maps that were described in \cref{ss.intro-main-results}. The lattice Hodge and Liu-Zhu period maps depend on a non-trivial inscribed extension of Scholze's functor $\mathbb{M}$ of \cite{Scholze.pAdicHodgeTheoryForRigidAnalyticVarieties} that will also play a key role also in the construction of twistor bundles in \cref{s.twistors}. 

\subsection{Hodge period maps}\label{ss.hodge-period-maps}
Let $L/\mathbb{Q}_p$ be an arbitrary non-archimedean extension. For this subsection we work in the inscribed context $(\Spd L, \Box^\sharp)$. We write $\mathcal{O}$ the sheaf $\mathbb{B}$ of \cref{def.BB}, i.e. $\mathcal{O}(\mc{P}/P^\sharp)=\mathcal{O}_{\mc{P}}(\mc{P})$. In particular, we note $\overline{\mc{O}}(\mc{P}/P^\sharp)=\mathcal{O}_{P^\sharp}(P^\sharp)$. 

\subsubsection{}
Let $Z/L$ be smooth rigid analytic variety. Let $G/L$ be a connected linear algebraic group, and let $\omega_\nabla$ be a $G$-bundle with integrable connection, i.e. an exact tensor functor from $\Rep G$ to vector bundles with integrable connection on $Z$. We write $\omega$ for the composition of $\omega_\nabla$ with the forgetful functor to vector bundles, and $\mc{G}:=\ul{\Isom}(\omega_\std, \omega)$ for the associated \'{e}tale $G$-torsor. This torsor is represented by a smooth rigid analytic space over $Z$ which we also write as $\mathcal{G}$. The associated inscribed $v$-sheaf $\mathcal{G}^\lfid$ is a $G(\mc{O})$-torsor over $Z^\lfid$. 

We are going to construct a natural reduction of structure group of $\mc{G}^\lfid$ to a $G(\overline{\mc{O}})$-torsor $\mc{G}^{\lfid}_{\nabla} \subseteq \mc{G}^\lfid$ of flat sections over $Z^\lfid$.  To that end, we write $r:Z^\lfid \rightarrow \overline{Z}$ for the natural map as in \cref{def.underlying-fc-and-trivial-inscription}, which maps $f:\mc{P} \rightarrow Z \in Z^\lfid(\mc{P}/P^\sharp)$ to $f|_{P^\sharp}$. Noting $\mc{G}^\diamond=\overline{\mc{G}^\lfid}$, we claim that the integrable connection induces a $G(\overline{\mc{O}})$-equivariant map $\exp_{\nabla}: r^{-1}\mc{G}^\diamond \rightarrow \mc{G}^\lfid$. Indeed, to give a section of $r^{-1}\mc{G}^\diamond$ on $\mc{P}/P^\sharp$ is to give a map $f: \mc{P} \rightarrow Z$ over $\Spa L$ and a section $\overline{s}$ of $\mc{G}$ over $f|_{P^\sharp}$, and the integrable connection promotes this uniquely (and $G(\overline{\mc{O}})$-equivariantly) to a flat section of $s$ of $\mc{G}$ over $f$. 

\begin{remark}
In other words, we have shown $\exp_\nabla$ induces 
\[ r^{-1}\mc{G}^\diamond \times^{G(\overline{\mc{O}})} G(\mc{O}) = \mc{G}^\lfid. \]
\end{remark}

We define $\mc{G}_{\nabla}^\lfid$ to be the image of $r^{-1} \mc{G}^\diamond$ under $\exp_{\nabla}$. 

\subsubsection{}
Suppose now $\omega$ is furthermore equipped with a filtration, i.e. with a factorization $\omega^\Fil$ through the exact category of filtered vector bundles, so that we obtain a rigid analytic period map $\pi_\Fil: \mc{G} \rightarrow \Fl_G$.
This map induces $\pi_\Fil^{\lfid}: \mc{G}^\lfid \rightarrow \Fl_G^\lfid$ and then, by restriction $\tilde{\pi}_\Hdg: \mc{G}_{\nabla}^\lfid \rightarrow \Fl_G^{\lfid}$. Since $\pi_{\Fil}$ is $G$-equivariant, $\tilde{\pi}_{\Hdg}$ is $G(\overline{\mc{O}})$-equivariant, and passing to the quotient we obtain
\[ \pi_\Hdg: Z^\lfid \rightarrow \Fl_{G}^\lfid /G(\overline{\mc{O}}). \]

\subsubsection{} 
The derivative of $\pi_\Hdg$ is a map 
\[ d\pi_\Hdg: (T_Z)^\lfid \rightarrow \pi_{\Hdg}^* T_{\Fl_{G}^\lfid/G(\overline{\mc{O}})}. \]
The pullback appearing as the codomain can be computed by first pulling back to $\Fl_{G}^\lfid$, where it is $T_{\Fl_{G}^\lfid}=(T_{\Fl_G})^\lfid$ with its natural $G(\overline{\mc{O}})$-equivariant structure, then pulling back along $\tilde{\pi}_\Hdg$ and descending along the $G(\overline{\mc{O}})$-action. We thus find $\pi_{\Hdg}^* T_{\Fl_{G}^\lfid/G(\overline{\mc{O}})}$ is naturally identified with $(\omega(\mathfrak{g})/\Fil^0 \omega(\mathfrak{g}))^\lfid$ where $\mathfrak{g}=\Lie G$ is equipped with the adjoint action. We claim that, under this identification, $d\pi_\Hdg$ is simply the map induced by the Kodaira-Spencer map. This is close to a tautology, after we recall the definition of the latter.

\subsubsection{} The Kodaira-Spencer map $\kappa_{\omega_\nabla^\Fil}$ is an $\mc{O}_Z$-linear homomorphism
\[ T_Z \xrightarrow{\kappa_{\omega_\nabla^\Fil}} \omega(\mf{g})/\Fil^0(\omega(\mf{g})).\]
It can be defined as follows: \'{e}tale locally, we may choose another connection $\nabla'$ on $\omega$ that preserves the filtration (e.g. by choosing a trivialization where the filtration is constant and taking the trivial connnection). The difference $\nabla-\nabla'$ assigns to any tangent vector $t$ and representation $V \in \Rep G$ an endomorphism $f_{t,V}=(\nabla_{t,V}-\nabla'_{t,V})$ of $\omega(V)$, functorially in $V$ and compatibly with the tensor product in that $f_{t,V_1 \otimes V_2}=f_{t,V_1} \otimes 1 + 1 \otimes f_{t,V_2}$. By the Tannakian formalism, it is given by an element $f_t \in \omega(\mf{g})$ (which maps to $\omega(\End(V))=\mc{E}nd(\omega(V))$ for any $V \in \Rep G$), and the map $t \mapsto f_t$ is a homomorphism $T_Z \rightarrow \omega(\mf{g})$ on the \'{e}tale cover where $\nabla'$ was chosen. The composition of $t \mapsto f_t$ with projection to $\omega(\mf{g})/\Fil^0(\omega(\mf{g}))$ does not depend on the choice of $\nabla'$, thus descends to give $\kappa_{\omega_{\nabla}^\Fil}$. 

\begin{lemma}\label{lemma.dpihdg-ks}
With the identifications above, $d\pi_\Hdg=\left(\kappa_{\omega_{\nabla}^\Fil} \right)^\lfid$. 
\end{lemma}
\begin{proof}
Working \'{e}tale locally we may assume our $\omega$ is trivialized such that the filtration is constant. Then writing a trivialization of the filtration as $s_0 g$ where $s_0$ is a flat section and $g \in G(\mathcal{P}[\epsilon])$, the derivative of the period map lifts as $\frac{1}{\epsilon}(\tilde{g}^{-1}g - 1) \in \Lie G \otimes \mathcal{O}(\mathcal{P})$, where $\tilde{g}$ is the constant extension to $G(\mathcal{P}[\epsilon])$ of the restriction of $g$ to $G(\mathcal{P})$. For $V \in \Rep G$ and $v \in V$, $\epsilon\nabla (s_0(gv))= s_0(gv) - s_0(\tilde{g}v)=s_0((1-\tilde{g}g^{-1})gv)$. Since $\epsilon \nabla'( s_0(gv))=s_0((1-p)gv)$ where $p$ is $1$ mod $\epsilon$ in the stabilizer $P$ of the filtration, we find $\epsilon(\nabla-\nabla')(s_0(gv))=(s_0((-\overline{g}g^{-1}+p)gv)$. It follows that, in the trivialization given by $s_0$, $\nabla-\nabla'$ is the endomorphism $\frac{1}{\epsilon}(-\tilde{g}g^{-1}+p)$ which, modulo $\Lie P$, agrees with $\frac{1}{\epsilon}(\tilde{g}^{-1}g-1)$, as desired. 
\end{proof}

\subsubsection{} Recall our convention that any $v$-stack can be viewed as an inscribed $v$-stack via the trivial inscription; in particular, under this convention we have $Z^\diamond=\overline{(Z^{\lfid})}$ and $\Fl_G^\diamond/G(\overline{\mc{O}})= \overline{{\Fl_G^\lfid}/G(\overline{\mc{O}})}$. We write $\overline{\pi}_\Hdg$ for the induced map $Z^\diamond \rightarrow \Fl_G^\diamond/G(\overline{\mc{O}})$.

\begin{remark}$\overline{\pi}_\Hdg$ could also be interpreted as the diamondification of the period map of rigid analytic stacks $Z \rightarrow \Fl_G/G$ obtained by quotienting $\pi_\Fil$ by $G$; note, however, that $\pi_\Hdg$ is \emph{not} the $\diamond^\lf$ of this map, since the target of $\pi_\Hdg$ is $\Fl_G^\lfid/G(\overline{\mc{O}})$ rather than $\Fl_G^\lfid/G(\mc{O})$. This is crucial as only the former contains interesting differential information.  
\end{remark}

When the Kodaira-Spencer map is an isomorphism, $\pi_\Hdg$ is an isomorphism of the formal neighborhood in $Z^\lfid$ of any point in $Z^\diamond$ into the formal neighborhood in $\Fl_G^\lfid/G^\diamond$ of its image in $\Fl_G^\diamond/G^\diamond$. In particular, we obtain the following description of $Z^{\diamond_\lf}$ in terms of $Z^\diamond$ and the flag variety.

\begin{proposition}\label{prop.ks-iso-formal-nbhd}
If $\kappa_{\omega_{\nabla}^\Fil}$ is an isomorphism, then the following square is Cartesian
\[\begin{tikzcd}
	{Z^\lfid} & {\Fl_G^\lfid}/G(\overline{\mc{O}}) \\
	{Z^\diamond} & {\Fl_G^\diamond}/G(\overline{\mc{O}})
	\arrow["{\pi_{\Hdg}}", from=1-1, to=1-2]
	\arrow[from=1-1, to=2-1]
	\arrow[from=1-2, to=2-2]
	\arrow["{\overline{\pi}_{\Hdg}}"', from=2-1, to=2-2]
\end{tikzcd}\]
\end{proposition}
\begin{proof}
Since we can quotient by $G(\overline{\mc{O}})$, it suffices to show the following diagram is Cartesian
\[\begin{tikzcd}
	{\mc{G}^\lfid_\nabla} & {\Fl_G^\lfid} \\
	{\mc{G}^\diamond} & {\Fl_G^\diamond}
	\arrow["{\tilde{\pi}_{\Hdg}}", from=1-1, to=1-2]
	\arrow[from=1-1, to=2-1]
	\arrow[from=1-2, to=2-2]
	\arrow["\overline{{\tilde{\pi}}}_{\Hdg}"', from=2-1, to=2-2]
\end{tikzcd}\]

We fix a map $\tilde{f}_0: P^\sharp \rightarrow \mc{G}$ lying above $f_0: P^\sharp \rightarrow Z$. Since $\mc{G}_{\nabla}^\lfid$ is a $G(\overline{\mc{O}})$-torsor over $Z^\lfid$, for any $\mc{P}/P^\sharp$ and $f:\mc{P} \rightarrow Z$ with $f|_P^\sharp=f_0$, there is a unique $\tilde{f}: \mc{P} \rightarrow \mc{G}$ lying above $f$ with $\tilde{f}|_{P^\sharp}=\tilde{f}_0$. The assignment $f \mapsto \pi_{\Fil}(\tilde{f})$ is, by construction, induced by an isomorphism from the formal neighborhood of the graph of $f_0: P^\sharp \rightarrow Z$ to the formal neighborhood of the graph of $\pi_{\Hdg}\circ f_0: P^\sharp \rightarrow \Fl_G$. By our computation of $d\pi_\Hdg$, the derivative of this map of formal neighborhoods is the Kodaira-Spencer map, which is an isomorphism. Since these formal neighborhoods are smooth, this map is an isomorphism, so the assignment 
$\tilde{f} \mapsto \pi_{\Hdg}(\tilde{f})$ is a bijection between the set of $\tilde{f}$ deforming $f_0$ and the set of deformations of $\pi_{\Hdg} \circ f_0$. This is exactly the statement that the diagram is Cartesian.  

\end{proof}

\subsection{The lattice Hodge period map}\label{ss.latice-Hodge-period-map}

\subsubsection{} Suppose now that $L$ is a $p$-adic field and $Z/L$ is a smooth rigid analytic variety. In this subsection, we work in the inscribed context $(\Spd L, \square^{\sharp-\alg}_{(\infty)})$. We write $\mathbb{B}^+_\dR$ for the sheaf $\mathbb{B}$ of \cref{def.BB}, sending $\mathcal{B}/P^{\sharp-\alg}_{(\infty)}$ to $\mathcal{O}(\mathcal{B})$. We write $\mathcal{O}$ for the sheaf 
\[ \mathcal{O}(\mc{B})=\mc{O}_{P^{\sharp-\alg} \times_{P^{\sharp-\alg}_{(\infty)}} \mc{B}}(P^{\sharp-\alg} \times_{P^{\sharp-\alg}_{(\infty)}} \mc{B})\]
and 
$\mathbb{B}_\dR$ for the sheaf 
\[ \mathbb{B}_\dR(\mc{B})=\mathcal{O}_{\mc{B}}(P^{\sharp-\alg}_{(\infty)}\backslash P^{\sharp-\alg} \times_{P^{\sharp-\alg}_{(\infty)}} \mathcal{B}),\]
both of which are obtained by change of context as in \cref{ss.change-of-context}. It will be convenient below to write $t$ to mean any choice of a generator of $\Fil^1 \mathbb{B}^+_\dR$ (we cannot make a global choice but over any test object generators exist and nothing will be depend on the choice). 

\subsubsection{}
Suppose $G/L$ is a connected linear algebraic group, and $\omega$ is a $G$-bundle on $Z$ with filtration and integrable connection satisfying Griffiths transversality, i.e. an exact tensor functor from $\Rep G$ to the category of filtered vector bundles on $Z$ with integrable connection satisfying Griffiths transversality. In this subsection we will construct a lattice Hodge period map 
\[ \pi^+_\Hdg: Z^{\lfid} \rightarrow \Gr_{G}/G(\overline{\mathbb{B}^+_\dR}).\]

If we suppose furthermore that the type $[\mu^{-1}]$ of the filtration is constant (this holds, e.g., if $Z$ is geometrically connected), then $\pi_\Hdg^+$ will factor through $\Gr_{[\mu]}/G(\overline{\mathbb{B}^+_\dR})$, and, noting that we can view the map $\pi_\Hdg$ of the previous subsection within this inscribed context by change of context as a map $Z^\lfid \rightarrow \Fl_{[\mu^{-1}]}/G(\overline{\mathcal{O}})$, it will fit into a commutative diagram 
% https://q.uiver.app/#q=WzAsMyxbMCwwLCIgWl57XFxsZmlkfSAiXSxbMiwwLCJcXEdyX3tbXFxtdV19L0coXFxvdmVybGluZXtcXG1hdGhiYntCfV4rX1xcZFJ9KSJdLFsyLDIsIlxcRmxfe1tcXG11XnstMX1dfS9HKFxcb3ZlcmxpbmV7XFxtYXRoY2Fse099fSkiXSxbMCwxLCJcXHBpXitfXFxIZGciLDFdLFswLDIsIlxccGlfe1xcSGRnfSIsMV0sWzEsMiwiXFxCQiIsMV1d
\[\begin{tikzcd}
	{ Z^{\lfid} } && {\Gr_{[\mu]}/G(\overline{\mathbb{B}^+_\dR})} \\
	\\
	&& {\Fl_{[\mu^{-1}]}/G(\overline{\mathcal{O}})}
	\arrow["{\pi^+_\Hdg}"{description}, from=1-1, to=1-3]
	\arrow["{\pi_{\Hdg}}"{description}, from=1-1, to=3-3]
	\arrow["\BB"{description}, from=1-3, to=3-3]
\end{tikzcd}\]

\subsubsection{}The key point is the construction of an exact tensor functor $\mathbb{M}$ from from filtered vector bundles with connection to locally free $\mathbb{B}^+_\dR$-modules on $S^\lfid$. On restriction to $Z^\diamond=\overline{Z^\lfid}$, this will recover the functor $\mathbb{M}$ of \cite{Scholze.pAdicHodgeTheoryForRigidAnalyticVarieties}; it is a crucial point, however, that when the filtration is not flat, then $\mathbb{M}$ \emph{not} simply the constant extension of $\mathbb{M}|_{Z^\diamond}$. 

\subsubsection{} Rather than imitating the $\mathcal{O}\mathbb{B}_\dR$-formalism of \cite{Scholze.pAdicGeometry}, we find it clearer in this case (and for the generalization in the next section) to utilize a geometrization of the construction via the $\mathbb{B}^+_\dR$-jet sheaf: for each $i \geq 0$, we define the $\mathbb{B}^+_\dR/\Fil^{i+1}\mathbb{B}^+_\dR$-jet sheaf $Z_{(i)}^\lfid$ to be the moduli of sections for $Z \times_{\Spd L} \square_{(i)}^\sharp / \square_{(i)}^\sharp$ (viewed in our inscribed setting by change of context), and we define the $\mathbb{B}^+_\dR$-jet sheaf as $Z_{(\infty)}^\lfid = \varprojlim Z_{(i)}$ --- in particular, these are inscribed $v$-sheaves by \cref{theorem.smooth-adic-space-mos}. 

Explicitly, for a test object $(P/\Spd L, \mc{B})$ with $P^\sharp=\Spa(R,R^+)$, if we write 
\[ \mc{B}_{(i)}^\an=\Spa( \mathcal{O}(\mc{B}) \otimes_{\mathbb{B}^+_\dR(R)} \mathbb{B}^+_\dR(R)/\Fil^{i+1} \mathbb{B}^+_\dR(R)) \]
(where the plus ring is the preimage of $R^+$), then we have
\[ Z_{(i)}^\lfid(\mc{B}) = \Hom_{L}(\mc{B}_{(i)}^\an, Z) \textrm{ and } Z_{(\infty)}(\mc{B})=\varprojlim_i \Hom_{L}(\mc{B}_{(i)^\an}, Z). \]

\newcommand{\fan}{\mathrm{f-an}}

In particular, in this setting, a point of $Z_{(\infty)}^\lfid$ gives a map of ringed spaces $\mc{B}^\fan \rightarrow Z$, where $\mc{B}^\fan$ is the ringed space with underlying topological space 
\[ |\mc{B}^\fan|:=|P|=|P_{(i)}^\sharp|=|\mc{B}_{(i)}^\an| \textrm{ for any $0 \leq i < \infty$} \]
and is equipped with the structure sheaf 
\[ \mathcal{O}_{\mc{B}^\fan}=\varprojlim \mathcal{O}_{\mc{B}_{(i)}} \]
Here the superscript $\fan$ stands for formal analytic. We also write $\mc{B}^\fan[\frac{1}{t}]$ for the ringed space obtained by base-change of the structure sheaf to $\mathbb{B}_\dR$.

\subsubsection{} Suppose given a filtered vector bundle $V$ on $Z$ (note that, by convention, our filtrations are always by local direct summands). Then, for $j$ a $\mc{B}$-point of $Z_{(\infty)}^\lfid$, writing $j_{\fan}: \mc{B}^\fan\rightarrow Z$ for the associated map, we obtain a vector bundle $j_{\fan}^*V[\frac{1}{t}]$ on $\mc{B}^\fan[\frac{1}{t}]$. It is equipped with a filtration by locally free $\mathcal{O}_{\mc{B}^\fan}$-modules by convolving the filtration on $V$ with that on $\mathcal{O}_{\mc{B}^\fan[\frac{1}{t}]}$ (since the filtration on $V$ is by local direct summands, we can choose a splitting locally to see that each $\Fil^i \left(j_{\fan}^*V[\frac{1}{t}]\right)$ is locally free). From this construction, we obtain a locally free $\mathbb{B}^+_\dR$-module $\mathcal{L}_{V,\Fil}$ on $Z_{(\infty)}^\lfid$ whose fiber over $j$ is  $H^0(\mc{B}^\fan, \Fil^0 \left(j_{\fan}^*V[\frac{1}{t}]\right))$. 

\subsubsection{} Suppose furthermore that the filtered vector bundle $V$ is equipped with an integrable connection $\nabla$ satisfying Griffiths transversality, i.e. $\nabla|_{\Fil^i V}$ factors through $\Fil^{i+1} V \otimes_{\mathcal{O}_L} \Omega_{Z/L}$.

\begin{lemma}The connection $\nabla$ induces a canonical descent data on $\mathcal{L}_{V,\Fil}$ from $Z_{(\infty)}^\lfid$ to $Z^\lfid$. 
\end{lemma}
\begin{proof}
We first consider the case when the filtration is trivial (in which case Griffiths transversality holds for any connection), so that $\Fil^0 \left(j_{\fan}^*V[\frac{1}{t}]\right)=j_{\fan}^*V$; we write $\mathcal{L}_V$ in place of $\mathcal{L}_{V, \Fil}$ when the filtration is trivial. In this case, the descent data is obtained via the usual construction viewing the integrable connection $\nabla$ as an isomorphism $\epsilon:p_1^* V \rightarrow p_2^* V$ on the formal neighborhood of the diagonal in $Z \times_{\Spa L} Z$ and then pulling back: for each $i$, given two $\mathcal{B}$-points $j$, $j'$ in $Z_{(\infty)}^\lfid(\mc{B})$ that restrict to the same $j_{(0)}=j'_{(0)}$ in $Z^\lfid(\mc{B})$, the associated map $j_{(i)}\times j'_{(i)}: \mc{B}_{(i)} \rightarrow Z \times_{\Spa L} Z$ factors through this formal neighborhood, and so we can pull back the isomorphism $\epsilon$ and take global sections to obtain the desired descent data on $\mathcal{L}_{V}$. In particular, we also obtain descent data on $\mathcal{L}_{V}[\frac{1}{t}]$. 

Returning to the case of a general filtration, since $\mathcal{L}_{V,\Fil} \subseteq \mathcal{L}_{V}[\frac{1}{t}]$, it suffices to show the descent data on the latter preserves the former. This can be seen by expressing $\epsilon$ via a Taylor expansion in local coordinates: suppose given two $\mathcal{B}$-points $j$, $j'$ that restrict to the same morphism $j_{(0)}$ from $\mc{B}_{(0)}$. We replace $|P|$ with a rational open that factors through an open in $Z$ with local torus coordinates $x_1, \ldots, x_n$ and write $\tau_i= x_1 \otimes 1 - 1 \otimes x_n$. Then, by the usual computation, the isomorphism $\epsilon$ can be written as 
\[ \epsilon(1 \otimes v)=\sum_{\ul{k}=(k_1, \ldots, k_n) \in \mathbb{Z}_{\geq 0}^k} \partial^{\ul{k}}v \otimes \frac{\tau^{\ul{k}}}{\ul{k}!} \]
where $\partial^{\ul{k}}=\partial_{x_1}^{k_1}\ldots\partial_{x_n}^{k_n}$, $\tau^{\ul{k}}=\tau_1^{k_1}\ldots\tau_{n}^{k_n}$, and $\ul{k}!=k_1!\ldots k_n!$. Now, since $(j^\fan \times' j'^\fan)^*  \tau_i$ lies in $t\mathcal{O}_{\mathcal{B}}^{f-\an}$, we find $(j^\fan \times' j'^\fan)^*  \tau^{\ul{k}}$ lies in $t^{|k|}\mathcal{O}_{\mathcal{B}}^{f-\an}$, where $|k|:= k_1 + \ldots k_n$. On the other hand, by Griffiths transversality for $v \in \Fil^{i} V$, $\partial^{\ul{k}} v \in \Fil^{i-|k|}v$, and thus, by the definition of the convolution filtration, the pullback of $\epsilon$ preserves $\Fil^0 (j^*_{\fan}V[\frac{1}{t}]) \subseteq j^*_{\fan}V[\frac{1}{t}]$. Thus the descent data preserves $\mathcal{L}_{V,\Fil}$, so that $\mathcal{L}_{V,\Fil}$ descends to $Z^\lfid$. 
\end{proof}

\subsubsection{}
We write $\mathbb{M}$ for the resulting functor from filtered vector bundles with integrable connection satisfying Griffiths transversality on $Z$ to locally free $\mathbb{B}^+_\dR$-modules on $Z^\lfid$. Since it is constructed using pullback and descent, it is an exact tensor functor. 

We also write $\mathbb{M}_0$ for the functor from vector bundles with integrable connection on $Z$ to locally free $\mathbb{B}^+_\dR$-modules on $Z^\lfid$ obtained by applying $\mathbb{M}$ to the trivial filtration; by abuse of notation, we also write $\mathbb{M}_0$ for the functor on filtered vector bundles with integrable connection satisfying Griffiths transversality obtained by composing with the forgetful functor (i.e. by applying $\mathbb{M}$ after replacing the given filtration with the trivial filtration). In this setting, by construction, we have a canonical isomorphism 
\begin{equation}\label{eq.m-and-m0}\mathbb{M} \otimes_{\mathbb{B}^+_\dR} \mathbb{B}_\dR = \mathbb{M}_0 \otimes_{\mathbb{B}^+_\dR} \mathbb{B}_\dR\end{equation}

If $V$ is a vector bundle with integrable connection on $Z$, then the connection on $V$ gives an isomorphism $j_{\fan}^*V=\overline{j}_{\fan}^* V$, where $\overline{j}$ is the image of $j$ under the composition of the natural maps \[ Z_{(\infty)}^\lfid \rightarrow Z_{(\infty)}^\diamond \rightarrow Z_{(\infty)}^\lfid. \] 
Taking global sections, we obtain

\begin{lemma}\label{lem.m0-connection}
    For $r: Z^\lfid \rightarrow Z^\diamond$ the natural map, there is a canonical natural isomorphism
\[ r^*\overline{\mathbb{M}_0} \otimes_{\overline{\mathbb{B}^+_\dR}}  \mathbb{B}^+_\dR \rightarrow \mathbb{M}_0.\]
\end{lemma}

\begin{remark}
    Such an isomorphism should be thought of as an integrable connection in the inscribed setting. 
\end{remark}

\begin{remark}
On underlying $v$-sheaves, the functors $\mathbb{M}$ and $\mathbb{M}_0$ recover the constructions of \cite[\S6]{Scholze.pAdicHodgeTheoryForRigidAnalyticVarieties} (cf. also \cite[\S3.2]{HoweKlevdal.AdmissiblePairsAndpAdicHodgeStructuresIITheBiAnalyticAxLindemannTheorem}). \cref{lem.m0-connection} says that $\mathbb{M}_0$ contains no additional information in the inscribed setting --- it is isomorphic to the trivial extension of the construction in the non-inscribed setting. On the other hand, if the filtration on $V$ is not flat, then the induced integrable connection on $\mathbb{M}_0 \otimes_{\mathbb{B}^+_\dR} \mathbb{B}_\dR$ does \emph{not} preserve $\mathbb{M}$, i.e. does not identify $r^*\overline{\mathbb{M}} \otimes_{\overline{\mbb{B}^+_\dR}} \mathbb{B}^+_\dR$ with $\mathbb{M}$. Thus $\mathbb{M}$ carries genuinely interesting inscribed information; this will manifest in the derivatives of the period maps constructed below. 
\end{remark}

\subsubsection{}
Finally, we can construct our lattice Hodge period map. Consider the composition $\mathbb{M} \circ \omega$. We have an isomorphism 
\[ (\mathbb{M} \circ \omega) \otimes_{\mathbb{B}^+_\dR} \mathbb{B}_\dR = (\mathbb{M}_0 \circ \omega) \otimes_{\mathbb{B}^+_\dR} \mathbb{B}_\dR = (r^*\overline{\mathbb{M}_0} \circ \omega) \otimes_{\overline{\mathbb{B}^+_\dR}} \mathbb{B}_\dR\]
coming from \cref{eq.m-and-m0} and \cref{lem.m0-connection}. Thus, we obtain a period map to $\Gr_G$ on the trivializing cover for $r^*\overline{\mathbb{M}_0} \circ \omega$. Since this trivializing cover is a $G(\overline{\mathbb{B}^+_\dR})$-torsor and the period map is equivariant, we can view this instead as a period map 
\[ \pi_{\Hdg}^+: Z^\lfid \rightarrow \Gr_{G}/G(\overline{\mathbb{B}^+_\dR}).\]

\subsubsection{}
The map $\pi_\Hdg^+$ factors through $\Gr_{[\mu]}$ by construction when the filtration on $V$ is of type $[\mu^{-1}]$, and in this case the derivative of $\pi_{\Hdg}^+$ is a map of $\mathbb{B}^+_\dR$-modules 
\[ d\pi_{\Hdg}^+ : T_{Z^\lfid} \rightarrow \mathbb{M}_0(\omega(\mf{g}))/\mathbb{M}(\omega(\mf{g}))\cap\mathbb{M}_0(\omega(\mf{g})).  \]
The composition $\BB \circ \pi_{\Hdg}^+$ is $\pi_\Hdg$, since, by construction, $\Fil^\bullet$ is the trace filtration of $\mathbb{M}$ on $\mathbb{M}_0 \otimes_{ \mathbb{B}^+_\dR} \mathcal{O}$.

From \cref{lemma.dpihdg-ks} it follows, in particular, that $d\pi_{\Hdg}^+$ is a lift of the Kodaira spencer map $\kappa_{\omega^{\nabla}_\Fil}$ along the canonical surjection. 
\[ \mathbb{M}_0(\omega(\mf{g}))/\left(\mathbb{M}(\omega(\mf{g}))\cap\mathbb{M}_0(\omega(\mf{g}))\right) \rightarrow \omega(\mf{g})/\Fil^0 \omega(\mf{g}). \]

\begin{remark}
    Note that Griffiths transversality on $\omega^\Fil_{\nabla}$ is also a \emph{necessary} condition for such a lift to exist. Indeed, as in \cref{cor.griffiths-transversality}, since $d\pi_\Hdg$ is a map of $\mathbb{B}^+_\dR$-modules, it factors through the $t$-torsion in the codomain
    \[ \left(\mathbb{M}_0(\omega(\mf{g}))\cap \frac{1}{t}\mathbb{M}(\omega(\mf{g}))\right)/\left(\mathbb{M}(\omega(\mf{g}))\cap\mathbb{M}_0(\omega(\mf{g}))\right)\]
    In particular, the image of this submodule under $d\BB$ lies in 
    \[ \gr^{-1} \omega(\mf{g})=\Fil^{-1}\omega(\mf{g})/\Fil^0\omega{(\mf{g})}\]
    but the condition that $(\nabla, \Fil)$ satisfies Griffiths transversality is precisely the condition that $d\pi_{\Hdg}=\kappa_{\omega_{\nabla}^\Fil}$ factor through $\gr^{-1} \omega(\mf{g})$.
\end{remark}

\subsection{The Liu-Zhu period map}\label{ss.liu-zhu-period-map}
Note that, in the setting of \cref{ss.latice-Hodge-period-map}, we could also push-out to $G(\overline{\mathbb{B}_\dR})$ to obtain from $\pi_\Hdg^+$ a map 
\[ \pi_{\mathrm{LZ}}: Z^\lfid \rightarrow \Gr_G / G(\overline{\mathbb{B}_\dR}). \]
This map can be described without reference to $\mathbb{M}_0$: indeed, it is simply the period map measuring the position of $\mathbb{M} \circ \omega$ against a flat trivialization of $(\mathbb{M} \circ \omega) \otimes_{\mathbb{B}^+_\dR} \mathbb{B}_\dR$. 

\newcommand{\Spaf}{\mathrm{Spaf}}
This interpretation depends on less data than a filtered vector bundle with connection on $Z$: let $C=\overline{L}^\wedge$, let $B^+_\dR=\mathbb{B}^+_\dR(C)$, and let $Z_{\Spaf B^+_\dR} = \varprojlim Z_{B^+_\dR/\Fil^i B^+_\dR}$, where the limit is taken in ringed spaces: in particular the underlying topological space is $|Z_C|$ and the structure sheaf agrees with that constructed in \cite[\S3.1]{LiuZhu.RigidityAndARiemannHilbertCorrespondenceForpAdicLocalSystems}. We write $Z_{\Spaf B^+_\dR}[\frac{1}{t}]$ for the same space with $t$ inverted in the structure sheaf. Given a vector bundle $W$ on $Z_{\Spaf B^+_\dR}[\frac{1}{t}]$, a connection on $W$ is a map $\nabla: W \rightarrow W \otimes_{\pi^{-1}\mathcal{O}_Z} \pi^{-1}\Omega_Z$ satisfying the Leibniz rule, where $\pi$ is the natural map of ringed spaces $Z_{\Spaf B^+_\dR}[\frac{1}{t}] \rightarrow Z$; it is integrable 
if $\nabla^2=0$. 

In particular, suppose given a vector bundle $V$ on $Z_{\Spaf B^+_\dR}$ and an integrable connection $\nabla$ on $V[\frac{1}{t}]$ such that $\nabla|_V$ factors through $\frac{1}{t} V \subseteq V[\frac{1}{t}]$ --- this condition is Griffiths transversality for the filtrations by powers of $t$, and we call such a pair $(V,\nabla)$ an integrable $t$-connection\footnote{In fact $t\nabla: V \rightarrow V$ is what one would normally call $t$-connection, but this is equivalent data.}. Then, arguing almost exactly as in the previous subsection, we can pull back $V$ to a locally free $\mathbb{B}^+_\dR$-module on the jet sheaf $Z_{(\infty)}^\lfid \times_{\Spd L} \Spd C$, and then use $\nabla$ to produce descent data down to $Z_{C}^\lfid$, defining $\mathbb{M}_C(V)$. The $\mathbb{B}_\dR$-module $\mathbb{M}_C(V)[\frac{1}{t}]$ is equipped with an integrable connection
\[ r^* \overline{\mathbb{M}_C(V)[\frac{1}{t}]} \otimes_{\overline{\mathbb{B}^+_\dR}} \mathbb{B}^+_\dR \]
just as in \cref{lem.m0-connection}. 

In particular, if $\omega$ is an exact functor from $\Rep G$ to integrable $t$-connections on $Z_{\Spaf B^+_\dR}$, then we obtain a period map 
\[ \pi_{\mathrm{LZ}}: Z_C^\lfid \rightarrow \Gr_G/G(\overline{\mathbb{B}_\dR}) \]
measuring the position of $\mathbb{M}_C \circ \omega$ against a flat basis of $\mathbb{M}_C[\frac{1}{t}] \circ \omega$. 

Note that the derivative can be written as a map
\[ d\pi_{\LZ}: T_{Z_C^\lfid} \rightarrow \frac{\mathbb{M}_{C}(\omega(\mf{g}))[\frac{1}{t}]}{\mathbb{M}_{C}(\omega(\mf{g}))}. \]
Since it is a map of $\mathbb{B}^+_\dR$-modules, it factors through 
\[ \frac{\frac{1}{t}\mathbb{M}_{C}(\omega(\mf{g}))}{\mathbb{M}_{C}(\omega(\mf{g}))} = \mathbb{M}_C(\omega(\mf{g})) \otimes_{\mathbb{B}^+_\dR} \gr^{-1}\mathbb{B}^+_\dR. \]

\begin{remark}\label{remark.geometric-sen}
The restriction of $d\pi_{\LZ}$ to the underlying $v$-sheaf $Z^\diamond_C$ is the geometric Sen morphism / canonical Higgs field for the $G(\overline{\mc{O}})$-bundle $\overline{(\mathbb{M}_C \circ \omega) \otimes_{\mathbb{B}^+_\dR} \mc{O}}$ on $Z^\diamond_C$. This can be shown by using an $\mathcal{O}\mathbb{C}$-comparison to construct this $v$-bundle from $\gr^0 \omega$ with the twisted Higgs field induced by $\nabla$ and then applying \cite[Proposition 3.5.2]{RodriguezCamargo.GeometricSenTheoryOverRigidAnalyticSpaces} to compute the geometric Sen morphism. In particular, if $\omega= \RH \circ \omega_\et$ for $\RH$ the Riemann-Hilbert correspondence of \cite{LiuZhu.RigidityAndARiemannHilbertCorrespondenceForpAdicLocalSystems} and  $\omega_\et$ an exact tensor functor from $\Rep G$ to $\mathbb{Q}_p$-local systems on $Z$, then the restriction of $d\pi_\LZ$ is the geometric Sen morphism / canonical Higgs field associated to the $G(\mathbb{Q}_p)$-torsor of trivializations of $\omega_\et$ in \cite{RodriguezCamargo.GeometricSenTheoryOverRigidAnalyticSpaces}. We do not need this comparison for the present work, so we do not give a more detailed argument, but see \cite{Howe.GeometricSenAndKodairaSpencer} for the de Rham case. 
\end{remark}

\subsubsection{}If $V$ is moreover equipped with a continuous action of $\Gal(\overline{L}/L)$ and $\nabla$ is equivariant for this action, then this data all descends to $Z^\lfid$ and defines a tensor functor $\mathbb{M}$ from continuous $\Gal(\overline{L}/L)$-equivariant vector bundles with integrable $t$-connections on $Z_{\Spaf B^+_\dR}$ to $\mathbb{B}^+_\dR$-modules on $Z^\lfid$ and a connection 
\[ r^*\overline{\mathbb{M}[\frac{1}{t}]} \otimes_{\overline{\mathbb{B}_\dR}} \mathbb{B}_\dR \xrightarrow{\sim} \mathbb{M}[\frac{1}{t}]. \]
In particular, if $\omega$ is an exact tensor functor from $\Rep G$ to continuous $\Gal(\overline{L}/L)$-equivariant integrable $t$-connections on $Z_{\Spaf B^+_\dR}$, then we obtain a period map 
\[ \pi_{\mathrm{LZ}}: Z^\lfid \rightarrow \Gr_G/G(\overline{\mathbb{B}_\dR}). \]

\subsubsection{}There is a natural exact tensor functor from filtered vector bundles with integrable connection satisfying Griffiths transversality on $Z$ to vector bundles with integrable $t$-connection on $Z_{\Spaf B^+_\dR}$ sending $(V, \nabla, \Fil)$ to $(\Fil^0(\pi^* V), \pi^*\nabla [1/t])$, where $\pi: Z_{\Spaf B^+_\dR}[\frac{1}{t}] \rightarrow Z$ is the natural map of ringed spaces and $\pi^*V$ is equipped with the convolution of the filtrations on $V$ and $\mathcal{O}_{Z_{\Spaf B^+_\dR}[\frac{1}{t}]}$. Using this functor, $\pi_\LZ$ as defined above recovers the period map described at the start of this subsection obtained from $\pi_\Hdg^+$ by push-out to $G(\overline{\mathbb{B}_\dR})$.

\section{Modifications}\label{s.modifications}
In this section we discuss modifications of $G$-torsors in the inscribed setting. In particular, in \cref{ss.mod-fes} we explain the fundamental exact sequences of $p$-adic Hodge theory in the inscribed setting, which play a key role in the computation of tangent bundles for the moduli of modifications in the sections to come. After some preliminary discussion of automorphism groups of $G$-bundles in \cref{ss.aut-groups}, we then construct the inscribed Hecke correpsondence in \cref{ss.inscribed-hecke}. Using the inscribed Hecke correspondence, we construct the inscribed generalized Newton strata of the $B^+_\dR$-affine Grassmannian and its Schubert cells in \cref{ss.inscribed-generalized-Newton-strata}.

\subsubsection{Notation}
Let $E/\mbb{Q}_p$ be a finite extension. In this section, we work in the inscribed context $(\Spd \mbb{F}_q, X_{E,\Box})$. We will need to consider the sheaf $\mathbb{B}$ not just for $(X^\alg_{E,\Box})^\lf$, but also for related inscribed contexts by change of base. To disambiguate, we begin by fixing some names for these sheaves. 

\begin{itemize}
\item 
We write $E^\lfid$ for the  $(X^\alg_{E,\Box})^\lf$-inscribed $v$-sheaf $\mathbb{B}$ of \cref{def.BB},
\[ E^\lfid(\mathcal{X})=H^0(\mc{X}, \mc{O}).\] 

\item We write $\mbb{B}_e$ for the $(X^\alg_{E,\Box})^\lf$-inscribed $v$-sheaf $\mathbb{B}$ over $\Spd E$ associated, by change of base as in \cref{ss.change-of-context}, to $\mbb{B}$ on $(X_{E,\Box}^\alg\backslash \Box^\sharp)^\lf$: 
\[ \mathbb{B}_e(\mc{X}/X_{E,P}^\alg, P/\Spd E)=H^0(\mc{X}_{X_{E,P}^\alg\backslash P^{\sharp-\alg}}, \mathcal{O}).\]

\item We write $\mbb{B}^+_\dR$ for the $(X^\alg_{E,\Box})^\lf$-inscribed $v$-sheaf over $\Spd E$ associated, by change of base as in \cref{ss.change-of-context}, to $\mbb{B}$ on $(\Box_{(\infty)}^{\sharp-\alg})^\lf$:
\[  \mbb{B}^+_\dR(\mc{X}/X_{E,P}^\alg, P/\Spd E)=H^0(\mathcal{X}_{P_{(\infty)}^{\sharp-\alg}}, \mc{O}) \]

\item We write $\mbb{B}_\dR$ for the $(X^\alg_{E,\Box})^\lf$-inscribed $v$-sheaf over $\Spd E$ associated, by change of base as in \cref{ss.change-of-context}, to $\mbb{B}$ on
$(\Box_{(\infty)}^{\sharp-\alg}\backslash \Box^{\sharp-\alg})^\lf$,
\[  \mbb{B}_\dR(\mc{X}/X_{E,P}^\alg, P/\Spd E)=H^0(\mathcal{X}_{P_{(\infty)}^{\sharp-\alg} \backslash P_{(\infty)}^{\sharp-\alg}}, \mc{O}). \]
\end{itemize}

Note that $\Spec \mbb{B}_e$ (resp. $\Spec \mbb{B}^+_\dR$, resp. $\Spec \mbb{B}_\dR$) is canonically identified with the functor  sending $(\mc{X}/X_{E,P}, P/\Spd E)$ to $\mc{X}_{X_{E,P}}\backslash \mc{X}_{{P}^{\sharp-\alg}}$ (resp. $\mc{X}_{P_{(\infty)}^{\sharp-\alg}}$, resp. $\mc{X}_{P_{(\infty)}^{\sharp-\alg}} \backslash \mc{X}_{{P}^{\sharp-\alg}}$). We will sometimes write $\mc{X}\backslash \infty$ for $\Spec \mbb{B}_e$. 

\subsection{Fundamental exact sequences}\label{ss.mod-fes}
If $\mc{S}$ is an inscribed $v$-sheaf and $\mc{E}: \mc{S} \rightarrow \mc{X}^*\VB$, recall that in \cref{ss.BC-spaces} we have defined $\BC(\mc{E})$ and $\BC(\mc{E}[1])$. If $\mc{S}/\Spd E$, then we also consider, for $\mbb{B}=\mbb{B}_e, \mbb{B}^+_\dR, \textrm{ or } \mbb{B}_\dR$, 
\[ \mc{E} \bct \mbb{B} : \mc{X}/\mc{S} \mapsto H^0(\Spec \mbb{B}(\mc{X}), \mc{E}_{\mc{X}}|_{\Spec \mbb{B}(\mc{X})}). \]
This is naturally a $\mbb{B}$-module, thus, in particular an $E^{\lfid}$-module.  If $W$ is an isocrystal, we write $\BC(W):=\BC(\mc{E}(W))$, $\BC(W[1]):=\BC(\mc{E}(W)[1])$, which lies over $\Spd \Fqbar$, and $W\boxtimes \mbb{B}:= \mc{E}(W) \boxtimes \mbb{B}$, which lies over $\Spd \breve{E}$.

Working over $\Spd E$, if we consider the open immersion $j: \Spec \mbb{B}_e \rightarrow \mc{X}$ and closed immersion $i: \mc{X}_{\Box^{\sharp-\alg}} \hookrightarrow \mc{X}$, then for any inscribed $v$-sheaf $\mc{S}/\Spd E$ and vector bundle $\mc{E}: \mc{S} \rightarrow \mc{X}^*\VB$, there is a natural associated exact sequence of sheaves on $\mc{X}$ over $\mc{S}$ 
\begin{equation}\label{eq.ses-sheaves-on-curlyX} 0 \rightarrow \mc{E}_{\mc{X}} \rightarrow j_*j^*\mc{E} \rightarrow i_* \left((\mc{E}\boxtimes \mbb{B}_\dR)/(\mc{E}\boxtimes \mbb{B}^+_\dR) \right) \rightarrow 0 \end{equation}

\begin{lemma}\label{lemma.fes}
Suppose $\mc{S}$ is an inscribed $v$-sheaf, and $\mc{E}: \mc{S} \rightarrow \mc{X}^*\VB$. Then, each of $\BC(\mc{E})$ and $\BC(\mc{E}[1])$ is an inscribed $v$-sheaf. Moreover, if $\mc{S}/\Spd E$, then, $\mc{E} \boxtimes \mbb{B}$ is an inscribed $v$-sheaf for $\mbb{B}=\mbb{B}_e, \mbb{B}^+_\dR,$ \textrm{ or } $\mbb{B}_\dR$, and the cohomology long exact sequences for 
\cref{eq.ses-sheaves-on-curlyX} induce, by $v$-sheafification, an exact sequence of  $E^{\lfid}$-modules over $\mc{S}$
\begin{equation}\label{eq.fes-inscribed-vb} 0 \rightarrow \BC(\mc{E}) \rightarrow \mc{E} \bct \mbb{B}_{e} \rightarrow   \mc{E} \bct \mbb{B}_\dR / \mc{E} \bct \mbb{B}^+_\dR \rightarrow \BC(\mc{E}[1]) \rightarrow 0.  \end{equation}
\end{lemma}
\begin{proof}
We have already seen in \cref{ss.BC-spaces} that $\BC(\mc{E})$ and $\BC(\mc{E}[1])$ are inscribed $v$-sheaves, and that the $\mc{E} \boxtimes \mbb{B}$ are inscribed $v$-sheaves follows, e.g., from \cref{theorem.affine-scheme-mos} applied to the associated geometric vector bundles (or just from the property for $\mbb{B}$ itself by writing $\mc{E}\boxtimes \mbb{B}$ locally as a direct summand of $\mbb{B}^n$). The exact sequence is then immediate from the definitions and the vanishing of quasi-coherent cohomology on affine schemes. 
\end{proof}

\begin{definition} In the setting of \cref{lemma.fes}, we refer to \cref{eq.fes-inscribed-vb} as the \emph{fundamental exact sequence} for $\mc{E}$. 
\end{definition}

\subsection{Automorphism groups}\label{ss.aut-groups}
Given an inscribed $v$-sheaf $\mc{S}$ and $\mc{E}: \mc{S} \rightarrow \mc{X}^*\BG$, we write $\mc{G}_{\mc{E}}$ for the smooth affine scheme over $\mc{X}$ on $\mc{S}$ of automorphisms of $\mc{E}$. We write $\tilde{G}_{\mc{E}}$ for its moduli of sections, which is an inscribed $v$-sheaf by \cref{theorem.affine-scheme-mos}. If $\mc{S}/\Spd E$, then we can also form its moduli of sections $\mc{G}_{\mc{E}}(\mbb{B}_e)$, $\mc{G}_{\mc{E}}(\mbb{B}^+_\dR)$, and $\mc{G}_{\mc{E}}(\mbb{B}_\dR)$, which are again inscribed $v$-stacks by \cref{theorem.affine-scheme-mos}.
We note also that 
\begin{equation}\label{eq.global-aut-intersection} \tilde{G}_{\mc{E}}=\mc{G}_{\mc{E}}(\mbb{B}_e) \times_{\mc{G}_\mc{E}(\mbb{B}_\dR)} \mc{G}_{\mc{E}}(\mbb{B}^+_\dR)=: \mc{G}_{\mc{E}}(\mbb{B}_e) \cap \mc{G}_{\mc{E}}(\mbb{B}^+_\dR).\end{equation}

Using the computation of the tangent bundle in \cref{theorem.affine-scheme-mos}, we find canonical identifications, compatible with restriction,
\begin{equation}\label{eq.LieAlgebras-of-aut-groups} \Lie {\tilde{G}_{\mc{E}}}=\BC(\mc{E}(\mf{g})) \textrm{ and } \Lie \mc{G}_{\mc{E}}(\mbb{B})=\mc{E}(\mf{g})\boxtimes \mbb{B} \textrm { for } \mbb{B}=\mbb{B}_e, \mbb{B}^+_\dR, \mbb{B}_\dR. \end{equation}

\subsection{An inscribed Hecke correspondence}\label{ss.inscribed-hecke}

\begin{lemma}\label{lemma.vb-glueing}The natural map given by restriction of vector bundles
\[ \mc{X}^* \VB \rightarrow (\Spec \mbb{B}_e)^* \VB \times_{(\Spec \mbb{B}_\dR)^*\VB} (\Spec \mbb{B}^+_\dR)^* \VB \]
is an equivalence of (inscribed) fibered categories.
\end{lemma}
\begin{proof}
Immediate from Beauville-Laszlo glueing. 	
\end{proof}

\begin{proposition}\label{prop.torsor-glueing}
    The natural map given by restriction of torsors 
\[ \mc{X}^* \BG \rightarrow (\Spec \mbb{B}_e)^* \BG \times_{(\Spec \mbb{B}_\dR)^*\BG} (\Spec \mbb{B}^+_\dR)^* \BG \]
is an equivalence of (inscribed) fibered categories.
\end{proposition}
\begin{proof}Follows from \cref{lemma.vb-glueing} by the Tannakian formalism. 
\end{proof}

Now, suppose given 
\[ \mc{E}_0: \mc{S} \rightarrow \mc{X}^*\BG \textrm{ and a trivialization } \varphi_0: \mc{E}_0|_{\Spec \mbb{B}_\dR} \xrightarrow{\sim} \mc{E}_{\triv}.\]
Then, using \cref{prop.torsor-glueing}, we obtain an induced map 
\begin{equation}\label{eq.modification-map} m_{(\mc{E}_0, \varphi_0)}: \mc{S} \times_{\Spd E} \Gr_{G} \rightarrow \mc{X}^*\BG=
 (\Spec \mbb{B}_e)^* \BG \times_{(\Spec \mbb{B}_\dR)^*\BG} (\Spec \mbb{B}^+_\dR)^* \BG \end{equation}
defined by
\[ (s, (\mc{E}, \varphi: \mc{E}|_{\Spec \mbb{B}_\dR} \xrightarrow{\sim} {\mc{E}_\triv})) \mapsto  (\mc{E}_0|_{\Spec \mbb{B}_e}, \mc{E}, \mc{E}_0|_{\Spec{\mbb{B}_\dR}} \xrightarrow{\varphi^{-1} \circ \varphi_0} \mc{E}|_{\Spec{\mbb{B}_\dR}}).  \]

We write $\mc{G}_0=\mc{G}_{\mc{E}_0}$ for the automorphism scheme of $\mc{E}_0$ as in \cref{ss.aut-groups}. Note that restriction $\mc{G}_0(\mbb{B}_e) \hookrightarrow \mc{G}_0(\mbb{B}_\dR)$ followed by conjugation by  $\varphi_0$ (an isomorphism $\mc{G}_0(\mbb{B}_\dR) \xrightarrow{\sim} G(\mbb{B}_\dR)$) induces a map $\mc{G}_0(\mbb{B}_e) \hookrightarrow G(\mbb{B}_\dR)$. 

\begin{lemma}\label{lemma.quasi-torsor-mod}
The map $p_1 \times m_{\mc{E}_0, \varphi_0}: \mc{S} \times \Gr_G \rightarrow \mc{S} \times \mc{X}^*\BG$ of inscribed $v$-stacks over $\mc{S}$ is a quasi-torsor for the action of $\mc{G}_0(\mbb{B}_e) \leq G(\mbb{B}_\dR)$.
\end{lemma}
\begin{proof}
Suppose $(s, (\mc{E}, \varphi))$ and $(s, (\mc{E}', \varphi'))$ map to the same object. Then $\mc{E}$ is isomorphic to $\mc{E'}$, so we may assume $\mc{E}=\mc{E}'$. It then follows that there is an automorphism $\psi$ of $(\mc{E}_0|_{\Spec \mbb{B}_e})$ such that 
\[ \varphi^{-1} \circ \varphi_{0} \circ \psi = (\varphi')^{-1} \circ \varphi_0.\]
and thus 
\[ \varphi^{-1} \circ (\varphi_{0} \circ \psi \circ \varphi_0^{-1}) = (\varphi')^{-1}. \]
Thus $\varphi_0 \circ \psi \circ \varphi_0^{-1}$ gives an element of $\mc{G}_0(\mbb{B}_e) \subseteq G(\mathbb{B}_\dR)$ mapping the one pre-image to the other. Reversing the argument we find that if there is $g \in G(\mbb{B}_\dR)$ such that $(s, (\mc{E}, g\varphi))$ and $(S, (\mc{E}, \varphi))$ map to the same object, then $g \in \mc{G}_0(\mbb{B}_e)$. 
\end{proof}

\subsection{Generalized Newton strata on the $\mbb{B}^+_\dR$-affine Grassmannian}\label{ss.inscribed-generalized-Newton-strata}
Let $G/E$ be a connected linear algebraic group and let $b_1 \in G(\breve{E})$. We pullback the bundle $\mc{E}_{b_1}$ of \cref{ss.NewtonStrataBG} from $\Spd \Fqbar$ to $\Spd \breve{E}$.  

By construction, there is a canonical trivialization $\varphi_{b_1}: \mc{E}_{b_1}|_{\Spec \mbb{B}_\dR} \xrightarrow{\sim} \mc{E}_{\triv}.$ By the construction of \cref{ss.inscribed-hecke}, we obtain an induced map 
\[ m=m_{b_1,\varphi_{b_1}}: \Spd \breve{E} \times_{\Spd E} \Gr_{G} \rightarrow \mc{X}^*\BG. \]
For $[b_2] \in B(G)$, the set of $\sigma$-conjugacy classes in $G(\breve{E})$, we write 
\[ \Gr_G^{b_1\rightarrow[b_2]} := m \times_{\mc{X}^*\BG} ((\mc{X}^*\BG)^{[b_2]} \rightarrow \mc{X}^*\BG). \]
If we fix also $[\mu]$ a conjugacy class of cocharacters of $G_{\overline{L}}$, we write 
\[ \Gr_{[\mu]}^{b_1 \rightarrow [b_2]} = (\Gr_{[\mu]} \times_{\Spd E([\mu])} \Spd \breve{E}([\mu]))\times_{\Gr_G \times_{\Spd E} \Spd \breve{E}} \Gr_G^{b_1\rightarrow[b_2]}. \]

\begin{example}
By \cref{remark.stratification-bunG} and the results of \cite{FarguesScholze.GeometrizationOfTheLocalLanglandsCorrespondence}, when $G$ is reductive, the $v$-sheaves $\overline{\Gr_G^{b_1 \rightarrow [b_2]}}$ as $[b]$ varies give a stratification of $\overline{\Gr_G \times_{\Spd E} \Spd \breve{E}}$ by locally closed subsheaves. For $b_1=1$, this is called the Newton stratification. 

These stratifications are usually studied after restricting to the Schubert cells $\Gr_{[\mu]}$. In particular, for a conjugacy class $[\mu]$, le $b_1$ lie in the Kottwitz set $B(G,[\mu^{-1}])$. Then $\overline{\Gr_{[\mu]}^{b_1 \rightarrow [1]}}$ is the open non-empty admissible locus for $b_1$ in $\overline{\Gr_{[\mu]}}$, and the other non-empty terms $\overline{\Gr_{[\mu]}^{b_1 \rightarrow [b_2]}}$ stratify the boundary. 
\end{example}

\begin{remark}\label{remark.canonical-torsor-over-newton-stratum}
If we fix $b \in [b]$, then we obtain a canonical $\tilde{G}_b:=\tilde{G}_{\mc{E}_b}$-torsor 
\[ (m_{b_1,\varphi_{b_1}} \times \Id_{\Spd \breve{E}}) \times_{\mc{X}^*\BG \times \Spd \breve{E}} (\mc{E}_b \times \Id_{\Spd \breve{E}}) \rightarrow \Gr_{G}^{b_0 \rightarrow [b]}. \]
parameterizing trivializations of the modified bundle to $\mc{E}_b$. This space also admits a natural action of $\mc{G}_{b_0}(\mbb{B}_e):=\mc{G}_{\mc{E}_{b_0}}(\mbb{B}_e)$ and  it follows from \cref{lemma.quasi-torsor-mod} that this action realizes the structure map to $\Spd \breve{E}$ as a quasi-torsor. We will see in \cref{theorem.unbounded-structure} that this is in fact a torsor (i.e., it is surjective or equivalently in this case non-empty) and admits a simpler description that highlights the symmetry between $b_1$ and $b_2$. Using this description, we will then deduce an explicit computation of the tangent and normal bundles of $\Gr^{b_1 \rightarrow [b_2]}_G$ (\cref{cor.newton-strata-tangent-normal}) and $\Gr^{b_1 \rightarrow [b_2]}_{[\mu]}$  (\cref{cor.bdd-newton-strata-tangent-normal}).  
\end{remark}

\section{Moduli of modifications}\label{s.moduli-of-mod}

In this section we define moduli of modifications between vector bundles on the relative thickened Fargues--Fontaine curve, and establish their main properties. We first treat the unbounded moduli space in \cref{ss.unbounded-moduli}. Its main structures are described in \cref{theorem.unbounded-structure}, including a key description as a bitorsor over $\Spd \breve{E}$ (cf. \cref{remark.canonical-torsor-over-newton-stratum}). In \cref{cor.unbounded-mod-tangent} we then deduce a computation of its tangent bundle, the derivatives of its natural period maps, and the tangent and normal bundles of the associated generalized Newton strata (\cref{cor.newton-strata-tangent-normal}). In \cref{ss.bounded-moduli} we then cut out the bounded moduli space inside by taking the preimage of a Schubert cell under a period map. Combining the results on the unbounded moduli space of \cref{ss.unbounded-moduli} and the results on the $\mathbb{B}^+_\dR$-affine Grassmannian and its Schubert cells of \cref{s.affine-grassmannian}, we obtain in \cref{theorem.bounded-structure} a description of the main structures of the bounded moduli of modifications, in \cref{corollary.bounded-mod-tangent} a description of its tangent bundle and the derivatives of its period maps, and in \cref{cor.bdd-newton-strata-tangent-normal} a description of the tangent and normal bundles of the associated generalized Newton strata. In \cref{ss.two-towers}, we describe a very general two towers isomorphism for inscribed moduli of modifications and explain how it interacts with our computations of tangent bundles and derivatives.

When one of the bundles in the modification is the trivial bundle $\mc{E}_1$, the moduli of modifications is an inscription on the moduli of mixed characteristic local shtuka with one leg as in \cite{ScholzeWeinstein.BerkeleyLecturesOnPAdicGeometryAMS207, HoweKlevdal.AdmissiblePairsAndpAdicHodgeStructuresIITheBiAnalyticAxLindemannTheorem}; in particular, in the minuscule bounded case, these are inscribed infinite level local Shimura varieties. For applications, this is by far the most important case, so we conclude in \cref{ss.main-results-mcls} by specializing our computations to this setting and discussing some complements --- in particular, in the local Shimura case, we discuss the relation between the infinite level moduli inscription and the alternative inscription on the finite level spaces obtained by viewing them as rigid analytic varieties. 

\subsubsection{Notation}
We fix a finite extension $E/\mbb{Q}_p$ with residue field $\mathbb{F}_q$. In this section we work in the inscribed context $(\Spd \mbb{F}_q, X_{E,{\Box}}^\alg)$, and use freely the notation of \cref{s.modifications}. We also use the notation of \cref{s.affine-grassmannian}, transported into this inscribed setting by change of context as in \cref{ss.change-of-context}. In particular, we view $\Gr_G$ as an inscribed $v$-sheaf on $X_{E,\Box}^\lf$ over $\Spd E$, i.e. by 
\[ \Gr_{G}(\mc{X}/X_{E,P}^\alg) = \{ (P/\Spd E, s: \mc{X}|_{P^{\sharp-\alg}_{(\infty)}} \rightarrow \Gr_G) \}, \]
and similarly for the Schubert cells $\Gr_{[\mu]}$, etc.

\subsection{The unbounded moduli space}\label{ss.unbounded-moduli}
Recall from \cref{ss.NewtonStrataBG} that, for $G/E$ a connected linear algebraic group and any $b \in G(\breve{E})$, we have defined a $G$-bundle $\mc{E}_b: \Spd \Fqbar \rightarrow \mc{X}^*\BG$. We write $\mc{G}_b=\mc{G}_{\mc{E}_b}$ for the automorphism scheme of $\mc{E}_b$ as in \cref{ss.aut-groups} and $\tilde{G}_b=\tilde{G}_{\mc{E}_b}$ for the moduli of  sections of $\mc{G}_b$.

After restriction to $\Spd \breve{E}=\Spd E \times_{\Spd \mbb{F}_q} \Spd \Fqbar$, there is a canonical trivialization 
\[ \triv_b: \mc{E}_{\triv} \rightarrow  \mc{E}_b|_{\Spec \mbb{B}_\dR}. \]
In the remainder of this subsection, all inscribed $v$-sheaves have been base changed to lie over $\Spd \breve{E}$ (with its trivial inscription). Equivalently, as described in \cref{ss.change-of-context}, we could work in the inscribed context $(\Spd \breve{E}, X_{E, \Box})$.

\begin{definition}
Let $G/E$ be a connected linear algebraic group, and let $b_1, b_2 \in G(\breve{E})$. We define $\M_{b_1 \rightarrow b_2}$ to be the following presheaf on $X_{E,\Box}^\lf$ over $\Spd \breve{E}$: 
\[ \M_{b_1 \rightarrow b_2} := (\Spd \breve{E} \xrightarrow{\mc{E}_{b_1}|_{\mc{X}\bs\infty}} (\mc{X}\bs\infty)^* \BG) \times_{(\mc{X}\bs\infty)^* \BG} (\Spd \breve{E} \xrightarrow{\mc{E}_{b_2}|_{\mc{X}\bs\infty}} (\mc{X}\bs\infty)^* \BG). \]
Equivalently, $\M_{b_1 \rightarrow b_2}$ is the functor
\[ (\mathcal{X}/X_{E,P}^\alg, P/\Spd \breve{E}) \mapsto \{ \varphi: \mc{E}_{b_1}|_{\mathcal{X}\bs\infty} \xrightarrow{\sim} \mc{E}_{b_2}|_{\mathcal{X}\bs\infty} \}. \]
\end{definition}

It follows from \cref{thm.BG-pull-back-inscribed} that $\M_{b_1 \rightarrow b_2}$ is an inscribed $v$-sheaf over $\Spd \breve{E}$, and it admits obvious actions of $\mc{G}_{b_i}(\mbb{B}_e)$, $i=1,2$, by precomposition and postcomposition with $\varphi$. Explicitly,  we define the right action maps
\begin{align*} a_i: \mc{M}_{b_1\rightarrow b_2} \times_{\Spd \breve{E}}  \mc{G}_{b_i}(\mbb{B}_e) \rightarrow \mc{M}_{b_1 \rightarrow b_2}, i=1,2, \textrm{ by } \\
 a_1(\varphi, g)= \varphi \circ g\textrm{ and } a_2(\varphi, g)=g^{-1} \circ \varphi.
 \end{align*}

\begin{remark}
Unwinding the definitions and proof of \cref{thm.BG-pull-back-inscribed}, we see that $\mc{M}_{b_1 \rightarrow b_2}$ is the moduli of sections over $\mc{X}\backslash \infty$ as in \cref{theorem.affine-scheme-mos} for the affine scheme $\mc{G}_{b_1 \rightarrow b_2}$ over $\mc{X}$ on $\Spd \breve{E}$ of isomorphisms of $G$-torsors, $\mathcal{I}som(\mc{E}_{b_1}, \mc{E}_{b_2})$, and the actions of $\mc{G}_{b_i}(\mbb{B}_e)$  are induced by the actions of $\mc{G}_{b_i}$ on this scheme. 
\end{remark}

We also define $c_\dR: \mc{M}_{b_1 \rightarrow b_2} \rightarrow G(\mbb{B}_\dR)$ to be the map
\[ \varphi \mapsto \triv_{b_2}^{-1} \circ \varphi|_{\Spec \mbb{B}_\dR} \circ \triv_{b_1}.\]
We define period maps $\pi_i: \mc{M}_{b_1 \rightarrow b_2} \rightarrow \Gr_G$ by
\begin{align*} \pi_1(\varphi)&= (\mc{E}_{b_1}|_{\Spec \mbb{B}^+_\dR},  \triv_{b_2}^{-1} \circ \varphi|_{\Spec \mbb{B}_\dR}) \textrm{ and }\\
\pi_2(\varphi)&= (\mc{E}_{b_2}|_{\Spec \mbb{B}^+_\dR}, \triv_{b_1}^{-1}\circ \varphi^{-1}|_{\Spec \mbb{B}_\dR}). \end{align*}

\begin{lemma}\label{lemma.period-maps-cdr}
The maps $\pi_i: \mc{M}_{b_1 \rightarrow b_2} \rightarrow \Gr_G$ are computed via $c_\dR$ as
\[ \pi_1(\varphi) = c_\dR(\varphi) \cdot \ast_1 \textrm{ and } \pi_2(\varphi) =(c_\dR(\varphi))^{-1} \cdot \ast_1. \]
\end{lemma}
\begin{proof}
For $\pi_1$, the map $\triv_{b_2}$ gives an isomorphism
\[ c_\dR(\varphi)\cdot \ast_1=(\mc{E}_{1}|_{\Spec \mbb{B}^+_\dR}, c_\dR(\varphi))   \xrightarrow{\sim} (\mc{E}_{b_2}|_{\Spec \mbb{B}^+_\dR}, \varphi|_{\Spec \mbb{B}_\dR} \circ \triv_{b_1}) \]
The argument is similar for $\pi_2$.
\end{proof}

\begin{theorem}\label{theorem.unbounded-structure}
Let $G/E$ be a connected linear algebraic group and let $b_1, b_2 \in G(\breve{E})$. Then
$\M_{b_1\rightarrow b_2}$ is an inscribed $v$-sheaf over $\Spd \breve{E}$. Moreover, 
\begin{enumerate} 
\item The action maps $a_1$ and $a_2$ realize $\M_{b_1 \rightarrow b_2}$ as a bitorsor over $\Spd \breve{E}$, trivializeable over $\Spd \mbb{C}_p$.
\item The map $\pi_1$ factors through $\Gr_G^{b_2 \rightarrow [b_1]}$ and the map $\pi_2$ factors through $\Gr_{G}^{b_1 \rightarrow [b_2]}$. The restriction of the action map $a_1$ to $\tilde{G}_{b_1}$ realizes $\pi_1$ as the canonical $\mc{G}_{b_2}(\mbb{B}_e)$-equivariant $\tilde{G}_{b_1}$-torsor over $\Gr_G^{b_2 \rightarrow [b_1]}$ of \cref{remark.canonical-torsor-over-newton-stratum} and the restriction of the action map $a_2$ to $\tilde{G}_{b_2}$ realizes $\pi_2$ as the canonical $\mc{G}_{b_1}(\mbb{B}_e)$-equivariant $\tilde{G}_{b_2}$-torsor  over $\Gr_G^{b_1 \rightarrow [b_2]}$ of \cref{remark.canonical-torsor-over-newton-stratum}. 
\end{enumerate}
In addition, the following diagram commutes:
% https://q.uiver.app/#q=WzAsMTAsWzAsMywiXFxNX3tiXzFcXHJpZ2h0YXJyb3cgYl8yfSJdLFszLDMsIkcoXFxtYmJ7Qn1fXFxkUikiXSxbMiwxLCJcXEdyX0dee2JfMlxccmlnaHRhcnJvd1tiXzFdfSJdLFsyLDUsIlxcR3JfR157Yl8xXFxyaWdodGFycm93W2JfMl19ICJdLFswLDIsIlxcTV97Yl8xXFxyaWdodGFycm93IGJfMn0gIFxcdGltZXMgXFxtY3tHfV97Yl8xfShcXG1iYntCfV9lKSAiXSxbMCw0LCJcXE1fe2JfMVxccmlnaHRhcnJvdyBiXzJ9IFxcdGltZXMgXFxtY3tHfV97Yl8yfShcXG1iYntCfV9lKSJdLFszLDAsIiBHKFxcbWJie0J9X1xcZFIpIFxcdGltZXMgRyhcXG1iYntCfV9cXGRSKSJdLFszLDYsIkcoXFxtYmJ7Qn1fXFxkUilcXHRpbWVzIEcoXFxtYmJ7Qn1fXFxkUikiXSxbMiwyLCJcXEdyX0ciXSxbMiw0LCJcXEdyX0ciXSxbMCwxLCJjX1xcZFIiLDIseyJzdHlsZSI6eyJ0YWlsIjp7Im5hbWUiOiJob29rIiwic2lkZSI6InRvcCJ9fX1dLFswLDMsIlxccGlfezJ9IiwxXSxbNCwwLCJhXzEiLDJdLFs1LDAsImFfMiJdLFs0LDYsImNfXFxkUiBcXHRpbWVzIChnXFxtYXBzdG9cXHRyaXZfe2JfMX1eey0xfVxcY2lyYyBnIFxcY2lyY1xcdHJpdl97Yl8xfSkiLDAseyJjdXJ2ZSI6LTV9XSxbNiwxLCIoYyxnKVxcbWFwc3RvIGNnIiwxXSxbNywxLCIgKGMsZylcXG1hcHN0byBnXnstMX1jIiwxXSxbNSw3LCJjX1xcZFIgXFx0aW1lcyAoZ1xcbWFwc3RvXFx0cml2X3tiXzJ9XnstMX1cXGNpcmMgZyBcXGNpcmNcXHRyaXZfe2JfMn0pIiwyLHsiY3VydmUiOjV9XSxbMCwyLCJcXHBpX3sxfSIsMV0sWzQsNSwiXFxzdWJzdGFja3soXFx2YXJwaGksZylcXG1hcHN0byBcXFxcKFxcdmFycGhpLCBcXHZhcnBoaSBnXnstMX0gXFx2YXJwaGleey0xfSl9IiwyLHsib2Zmc2V0Ijo1LCJjdXJ2ZSI6M31dLFsxLDgsImMgXFxtYXBzdG8gYyBcXGNkb3QgXFxhc3RfMSIsMV0sWzIsOCwiIiwxLHsic3R5bGUiOnsidGFpbCI6eyJuYW1lIjoiaG9vayIsInNpZGUiOiJib3R0b20ifX19XSxbMSw5LCIgY1xcbWFwc3RvICBjXnstMX1cXGNkb3QgXFxhc3RfMSIsMV0sWzMsOSwiIiwxLHsic3R5bGUiOnsidGFpbCI6eyJuYW1lIjoiaG9vayIsInNpZGUiOiJ0b3AifX19XV0=
\[\begin{tikzcd}
	&&& { G(\mbb{B}_\dR) \times G(\mbb{B}_\dR)} \\
	&& {\Gr_G^{b_2\rightarrow[b_1]}} \\
	{\M_{b_1\rightarrow b_2}  \times \mc{G}_{b_1}(\mbb{B}_e) } && {\Gr_G} \\
	{\M_{b_1\rightarrow b_2}} &&& {G(\mbb{B}_\dR)} \\
	{\M_{b_1\rightarrow b_2} \times \mc{G}_{b_2}(\mbb{B}_e)} && {\Gr_G} \\
	&& {\Gr_G^{b_1\rightarrow[b_2]} } \\
	&&& {G(\mbb{B}_\dR)\times G(\mbb{B}_\dR)}
	\arrow["{(c,g)\mapsto cg}"{description}, from=1-4, to=4-4]
	\arrow[hook', from=2-3, to=3-3]
	\arrow["{c_\dR \times (g\mapsto\triv_{b_1}^{-1}\circ g \circ\triv_{b_1})}", curve={height=-30pt}, from=3-1, to=1-4]
	\arrow["{a_1}"', from=3-1, to=4-1]
	\arrow["\begin{array}{c} \substack{(\varphi,g)\mapsto \\(\varphi, \varphi g^{-1} \varphi^{-1})} \end{array}"', shift right=5, curve={height=18pt}, from=3-1, to=5-1]
	\arrow["{\pi_{1}}"{description}, from=4-1, to=2-3]
	\arrow["{c_\dR}"', hook, from=4-1, to=4-4]
	\arrow["{\pi_{2}}"{description}, from=4-1, to=6-3]
	\arrow["{c \mapsto c \cdot \ast_1}"{description}, from=4-4, to=3-3]
	\arrow["{ c\mapsto  c^{-1}\cdot \ast_1}"{description}, from=4-4, to=5-3]
	\arrow["{a_2}", from=5-1, to=4-1]
	\arrow["{c_\dR \times (g\mapsto\triv_{b_2}^{-1}\circ g \circ\triv_{b_2})}"', curve={height=30pt}, from=5-1, to=7-4]
	\arrow[hook, from=6-3, to=5-3]
	\arrow["{ (c,g)\mapsto g^{-1}c}"{description}, from=7-4, to=4-4]
\end{tikzcd}\]

\end{theorem}
\begin{proof}
It is an inscribed $v$-sheaf by \cref{thm.BG-pull-back-inscribed}. It is evidently a quasi-bitorsor for $\mc{G}_{b_1}(\mbb{B}_e)$ and $\mc{G}_{b_2}(\mbb{B}_e)$. It is trivialized over $\Spd \mbb{C}_p$ (and thus, in particular, a bitorsor) by \cite[Theorem 6.5]{Anschuetz.ReductiveGroupSchemesOverTheFarguesFontaineCurve} (to apply this result in the case of $G$ non-reductive, we note that any $b$ may be $\sigma$-conjugated into a Levi subgroup of $G$; cf. \cite[\S2.2.3]{HoweKlevdal.AdmissiblePairsAndpAdicHodgeStructuresITranscendenceOfTheDeRhamLattice}). The commutativity of the diagram is a chase through the definitions after applying \cref{lemma.period-maps-cdr}. 
\end{proof}

We now describe the differentials of the maps in the commutative diagram of \cref{theorem.unbounded-structure}. To that end, note that, writing $\mf{g}_{b_i}$ for the isocrystal associated to $b_i$ by the adjoint representation on $\mf{g}$,  we have $\pi_i^* \mf{g}^+_\univ=\mc{E}(\mf{g}_{b_i}) \boxtimes \mbb{B}^+_\dR=\mf{g}_{b_i} \otimes_{\qpbreve} \mbb{B}^+_\dR$. Below we will simplify notation by writing 
\[ \mf{g}_{b_i} \boxtimes \mbb{B}^+_\dR := \mc{E}(\mf{g}_{b_i}) \boxtimes \mbb{B}^+_\dR \textrm{ and } \mf{g}_{b_i} \boxtimes \mbb{B}_e :=  \mc{E}(\mf{g}_{b_i}) \boxtimes \mbb{B}_e. \]

Then, from \cref{corollary.tangent-bundle-affine-grassmannian}, we obtain canonical isomorphisms 
\begin{equation}\label{eq.pullback-isos} c_i: (\mf{g}_{b_i} \boxtimes \mbb{B}_\dR) / (\mf{g}_{b_i} \boxtimes \mbb{B}^+_\dR) \xrightarrow{\sim} \pi_i^* T_{\Gr_G}.\end{equation}

\begin{corollary}\label{cor.unbounded-mod-tangent}
The following diagram of $E^\lfid$-modules on $\M_{b_1 \rightarrow b_2}$ commutes, where the left and right columns are the fundamental exact sequences of \cref{lemma.fes} for $\mf{g}_{b_1}$ and $\mf{g}_{b_2}$ and the morphisms $c_i$ are as in \cref{eq.pullback-isos}. Moreover, the horizontal arrows are all isomorphisms.
% https://q.uiver.app/#q=WzAsMTUsWzIsMiwiVF97XFxNX3tiXzFcXHJpZ2h0YXJyb3cgYl8yfX0iXSxbMCwyLCJcXG1me2d9X3tiXzF9IFxcYm94dGltZXMgXFxtYmJ7Qn1fZSJdLFs0LDIsIlxcbWZ7Z31fe2JfMn0gXFxib3h0aW1lcyBcXG1iYntCfV9lIl0sWzEsNCwiXFxwaV97MX1eKlRfe1xcR3JfR30iXSxbNCw0LCJcXGZyYWN7XFxtZntnfV97Yl8yfVxcYm94dGltZXNcXG1iYntCfV9cXGRSfXtcXG1me2d9X3tiXzJ9IFxcYm94dGltZXMgXFxtYmJ7Qn1eK19cXGRSfSJdLFswLDQsIlxcZnJhY3tcXG1me2d9X3tiXzF9XFxib3h0aW1lc1xcbWJie0J9X1xcZFJ9e1xcbWZ7Z31fe2JfMX0gXFxib3h0aW1lcyBcXG1iYntCfV4rX1xcZFJ9Il0sWzAsMSwiXFxCQyhcXG1me2d9X3tiXzF9KSJdLFswLDAsIjAiXSxbMCw1LCJcXEJDKFxcbWZ7Z31fe2JfMX1bMV0pIl0sWzAsNiwiMCJdLFs0LDYsIjAiXSxbNCwxLCJcXEJDKFxcbWZ7Z31fe2JfMn0pIl0sWzQsMCwiMCJdLFszLDQsIlxccGlfezJ9XipUX3tcXEdyX0d9Il0sWzQsNSwiXFxCQyhcXG1me2d9X3tiXzJ9WzFdKSJdLFswLDMsImRcXHBpX3sxfSIsMl0sWzIsNF0sWzUsMywiY18xIiwxXSxbMSwwLCIoZGFfMSlfZSIsMV0sWzYsMV0sWzEsNV0sWzcsNl0sWzUsOF0sWzgsOV0sWzEyLDExXSxbMiwwLCIoZGFfMilfZSIsMV0sWzQsMTMsImNfMiIsMV0sWzAsMTMsImRcXHBpX3syfSJdLFsxMSwyXSxbNCwxNF0sWzE0LDEwXSxbMSwyLCItXFxBZF8qXFx2YXJwaGkiLDEseyJjdXJ2ZSI6LTR9XV0=
\[\begin{tikzcd}
	0 &&&& 0 \\
	{\BC(\mf{g}_{b_1})} &&&& {\BC(\mf{g}_{b_2})} \\
	{\mf{g}_{b_1} \boxtimes \mbb{B}_e} && {T_{\M_{b_1\rightarrow b_2}}} && {\mf{g}_{b_2} \boxtimes \mbb{B}_e} \\
	\\
	{\frac{\mf{g}_{b_1}\boxtimes\mbb{B}_\dR}{\mf{g}_{b_1} \boxtimes \mbb{B}^+_\dR}} & {\pi_{1}^*T_{\Gr_G}} && {\pi_{2}^*T_{\Gr_G}} & {\frac{\mf{g}_{b_2}\boxtimes\mbb{B}_\dR}{\mf{g}_{b_2} \boxtimes \mbb{B}^+_\dR}} \\
	{\BC(\mf{g}_{b_1}[1])} &&&& {\BC(\mf{g}_{b_2}[1])} \\
	0 &&&& 0
	\arrow[from=1-1, to=2-1]
	\arrow[from=1-5, to=2-5]
	\arrow[from=2-1, to=3-1]
	\arrow[from=2-5, to=3-5]
	\arrow["{(da_1)_e}"{description}, from=3-1, to=3-3]
	\arrow["{-\Ad_*\varphi}"{description}, curve={height=-24pt}, from=3-1, to=3-5]
	\arrow[from=3-1, to=5-1]
	\arrow["{d\pi_{1}}"', from=3-3, to=5-2]
	\arrow["{d\pi_{2}}", from=3-3, to=5-4]
	\arrow["{(da_2)_e}"{description}, from=3-5, to=3-3]
	\arrow[from=3-5, to=5-5]
	\arrow["{c_1}"{description}, from=5-1, to=5-2]
	\arrow[from=5-1, to=6-1]
	\arrow["{c_2}"{description}, from=5-5, to=5-4]
	\arrow[from=5-5, to=6-5]
	\arrow[from=6-1, to=7-1]
	\arrow[from=6-5, to=7-5]
\end{tikzcd}\]
\end{corollary}
\begin{proof}
It follows from the torsor property that each of $(da_i)_e$ is an isomorphism, and the identity $(d a_1)_e =  - (da_2)_e \circ \Ad_* \varphi$ is immediate from the commutative circle at the top of the diagram in \cref{theorem.unbounded-structure}.

The commutativity of the left middle quadrilateral comes from using \cref{theorem.unbounded-structure} to compute $d\pi_1 \circ da_1$ as the derivative of 
\begin{align*} (\varphi, g) & \mapsto (c_\dR(\varphi) \triv_{b_1}^{-1}g\triv_{b_1} c_\dR(\varphi)^{-1}) c_\dR(\varphi) \cdot \ast_1 \\
&= \left((\triv_{b_2}^{-1}\circ \varphi) g (\triv_{b_2}^{-1} \circ \varphi)^{-1} \right) \cdot (c_\dR(\varphi) \cdot \ast_1)  
\end{align*}
Indeed, the identification $c_1$ is given by composing the derivative of the action map $G(\mbb{B}_\dR) \times_{\Spd \breve{E}} \Gr_G \rightarrow \Gr_G$ at the identity in $G(\mbb{B}_\dR)$ with the isomorphism $\mf{g}_{b_1} \boxtimes \mbb{B}_\dR \rightarrow \mf{g} \boxtimes \mbb{B}_\dR$ induced by $\triv_{b_2}^{-1} \circ \varphi|_{\Spec \mbb{B}_\dR}$. 

Similarly, for commutativity of the right middle quadrilateral we compute $d\pi_2 \circ da_2$ as the derivative of 
\begin{align*} (\varphi, g) & \mapsto (c_\dR(\varphi)^{-1} \triv_{b_2}^{-1}g \triv_{b_2} c_\dR(\varphi)) c_\dR(\varphi)^{-1} \cdot \ast_1 \\
&= \left((\triv_{b_1}^{-1}\circ \varphi^{-1}) g (\triv_{b_1}^{-1} \circ \varphi^{-1})^{-1} \right) \cdot (c_\dR(\varphi)^{-1} \cdot \ast_1)  
\end{align*}
where we note the inverses in the definitions of $a_2$ and $\pi_2$ are cancelling to give the term $g$. Indeed, the identification $c_2$ is given by composing the derivative of the action map $G(\mbb{B}_\dR) \times_{\Spd \breve{E}} \Gr_G \rightarrow \Gr_G$ at the identity in $G(\mbb{B}_\dR)$ with the isomorphism $\mf{g}_{b_2} \boxtimes \mbb{B}_\dR \rightarrow \mf{g} \boxtimes \mbb{B}_\dR$ induced by $\triv_{b_1}^{-1} \circ \varphi^{-1}|_{\Spec \mbb{B}_\dR}$. 
We obtain the negative sign in the commutative diagram because of the $g^{-1}$ that appears instead of a $g$ within the conjugation.  
\end{proof}

\begin{corollary}\label{cor.newton-strata-tangent-normal}
    Let $\mc{E}_{[b_2]}$ be the restriction of the universal $G$-bundle on $\mc{X}^*\BG$ to $(\mc{X}^*\BG)^{[b_2]}$ and let $\mf{g}_{[b_2]}=\mc{E}_{[b_2]}(\mf{g})$ be its push-out by the adjoint representation. The short exact sequence over $\Gr_{G}^{b_1 \rightarrow [b_2]}$ induced by the fundamental exact sequence of \cref{lemma.fes} for $\mf{g}_{[b_2]}$,  
\[ 0 \rightarrow \mf{g}_{[b_2]} \boxtimes \mbb{B}_e / \BC(\mf{g}_{[b_2]})  \rightarrow \mf{g}_{[b_2]}  \boxtimes \mbb{B}_\dR /\mf{g}_{[b_2]}  \boxtimes \mbb{B}^+_\dR \rightarrow \BC(\mf{g}_{[b_2]} [1]) \rightarrow 0, \]
is canonically identified with the short exact sequence
\[ 0 \rightarrow T_{\Gr_G^{b_1 \rightarrow[b_2]}} \xrightarrow{d\iota} \iota^* T_{\Gr_G} \rightarrow N_{\Gr_G^{b_1 \rightarrow [b_2]}} \rightarrow 0 \]
where $\iota: \Gr_G^{b_1 \rightarrow [b_2]} \hookrightarrow \Gr_G$ is the inclusion. 
\end{corollary}
\begin{remark}
The situation is symmetric in $b_1$ and $b_2$, so we only need to state one version in \cref{cor.newton-strata-tangent-normal}. 
\end{remark}

\subsection{The bounded moduli space}\label{ss.bounded-moduli}

We fix $G/E$ a connected linear algebraic group and $[\mu]$ a conjugacy class of cocharacters of $G_{\overline{E}}$. In this subsection all objects are base changed to $\Spd(\breve{E}([\mu]))$. For $b_1, b_2 \in G(\breve{E})$, 
\[ \M_{b_1 \rightarrow b_2, [\mu]} := c_\dR \times_{G(\mbb{B}_\dR)} C_{[\mu]} = \pi_1 \times_{\Gr_G} \Gr_{[\mu]} = \pi_2 \times_{\Gr_G} \Gr_{[\mu^{-1}]} \]
where the two equalities are immediate from the definitions. The left and right multiplication actions of $G(\mathbb{B}_\dR)$ on itself restrict to left and right multiplication actions of $G(\mathbb{B}^+_\dR)$ on $C_{[\mu]}$, thus by \cref{eq.global-aut-intersection}, the actions of $\mc{G}_{b_i}(\mbb{B}_e)$ on $\M_{b_1 \rightarrow b_2}$ restrict to actions of $\tilde{G}_{b_i}$ on $\M_{b_1 \rightarrow b_2, [\mu]}$. Below we write $a_i$ for the restrictions of the right action maps used in \cref{ss.unbounded-moduli}. 

Below we also write $\pi_i$ for the restriction of the period maps to $\M_{b_1 \rightarrow b_2, [\mu]}$, so that $\pi_1$ factors through $\Gr_{[\mu]}$ and $\pi_2$ factors through $\Gr_{[\mu^{-1}]}$. We then obtain two filtration period maps, $\pi_1^f=\BB \circ \pi_1: \M_{b_1 \rightarrow b_2, [\mu]} \rightarrow \Fl_{[\mu^{-1}]}^\lfid$ and $\pi_2^f=\BB \circ \pi_2: \M_{b_1 \rightarrow b_2, [\mu]} \rightarrow \Fl_{[\mu]}^\lfid.$ 

\begin{theorem}\label{theorem.bounded-structure}
Let $b_1, b_2 \in G(\breve{E})$. Then
$\M_{b_1 \rightarrow b_2,[\mu]}$ is a inscribed $v$-sheaf over $\Spd \breve{E}([\mu])$.  The map $\pi_1$ factors through $\Gr_{[\mu]}^{b_2 \rightarrow [b_1]}$ and the map $\pi_2$ factors through $\Gr_{[\mu]}^{b_1 \rightarrow [b_2]}$, and the actions realize $\pi_1$ as the canonical $\tilde{G}_{b_2}$-equivariant $\tilde{G}_{b_1}$-torsor over $\Gr_{[\mu]}^{b_2 \rightarrow [b_1]}$ of \cref{remark.canonical-torsor-over-newton-stratum} and $\pi_2$ as the canonical $\tilde{G}_{b_1}$-equivariant $\tilde{G}_{b_2}$-torsor over $\Gr_{[\mu^{-1}]}^{b_1 \rightarrow [b_2]}$ of \cref{remark.canonical-torsor-over-newton-stratum}. 
The following extended subdiagram of the diagram in \cref{theorem.unbounded-structure} commutes: 
% https://q.uiver.app/#q=WzAsMTMsWzAsMywiXFxNX3tiXzFcXHJpZ2h0YXJyb3cgYl8yLFtcXG11XX0iXSxbMiwyLCJcXEZsX3tbXFxtdV57LTF9XX1eXFxsZmlkIl0sWzMsMiwiXFxHcl97W1xcbXVdfSJdLFsyLDQsIlxcRmxfe1tcXG11XX1eXFxsZmlkIl0sWzMsNCwiXFxHcl97W1xcbXVeey0xfV19Il0sWzQsMywiQ197W1xcbXVdfSAiXSxbMCw0LCJcXE1fe2JfMSBcXHJpZ2h0YXJyb3cgYl8yLFtcXG11XX1cXHRpbWVzXFx0aWxkZXtHfV97Yl8yfSJdLFswLDIsIlxcTV97Yl8xXFxyaWdodGFycm93IGJfMixbXFxtdV19IFxcdGltZXMgXFx0aWxkZXtHfV97Yl8xfSJdLFs0LDAsIkNfe1tcXG11XX0gXFx0aW1lcyBHKFxcbWJie0J9XitfXFxkUikiXSxbNCw2LCJDX3tbXFxtdV19XFx0aW1lcyBHKFxcbWJie0J9XitfXFxkUikiXSxbNSwzXSxbMywxLCJcXEdyX3tbXFxtdV19XntiXzIgXFxyaWdodGFycm93IFtiXzFdfSJdLFszLDUsIlxcR3Jfe1tcXG11XnstMX1dfV57Yl8xIFxccmlnaHRhcnJvdyBbYl8yXX0iXSxbNCwzLCJcXEJCIl0sWzAsMywiXFxwaV8yXmYiLDFdLFswLDUsImNfXFxkUiIsMSx7InN0eWxlIjp7InRhaWwiOnsibmFtZSI6Imhvb2siLCJzaWRlIjoidG9wIn19fV0sWzYsMCwiYV8yIl0sWzcsMCwiYV8xIiwyXSxbNyw4LCJjX1xcZFIgXFx0aW1lcyAoZyBcXG1hcHN0byBcXHRyaXZfe2JfMX1eey0xfVxcY2lyYyBnIFxcY2lyYyBcXHRyaXZfe2JfMX0pIiwwLHsiY3VydmUiOi00fV0sWzYsOSwiY19cXGRSIFxcdGltZXMgKGcgXFxtYXBzdG8gXFx0cml2X3tiXzJ9XnstMX1cXGNpcmMgZyBcXGNpcmMgXFx0cml2X3tiXzJ9KSIsMix7ImN1cnZlIjozfV0sWzgsNSwiKGMsZylcXG1hcHN0byBjZyIsMV0sWzksNSwiKGMsIGcpIFxcbWFwc3RvIGdeey0xfWMiLDFdLFs1LDIsImMgXFxtYXBzdG8gYyBcXGNkb3QgXFxhc3RfMSIsMV0sWzUsNCwiY1xcbWFwc3RvIGNeey0xfVxcY2RvdFxcYXN0XzEiLDFdLFsyLDEsIlxcQkIiLDJdLFsxMSwyLCIiLDAseyJzdHlsZSI6eyJ0YWlsIjp7Im5hbWUiOiJob29rIiwic2lkZSI6InRvcCJ9fX1dLFswLDEsIlxccGlfMV5mIiwxXSxbMCwxMSwiXFxwaV8xIiwxLHsiY3VydmUiOi0yfV0sWzEyLDQsIiIsMSx7InN0eWxlIjp7InRhaWwiOnsibmFtZSI6Imhvb2siLCJzaWRlIjoidG9wIn19fV0sWzAsMTIsIlxccGlfMiIsMSx7ImN1cnZlIjoyfV1d
\[\begin{tikzcd}
	&&&& {C_{[\mu]} \times G(\mbb{B}^+_\dR)} \\
	&&& {\Gr_{[\mu]}^{b_2 \rightarrow [b_1]}} \\
	{\M_{b_1\rightarrow b_2,[\mu]} \times \tilde{G}_{b_1}} && {\Fl_{[\mu^{-1}]}^\lfid} & {\Gr_{[\mu]}} \\
	{\M_{b_1\rightarrow b_2,[\mu]}} &&&& {C_{[\mu]} } & {} \\
	{\M_{b_1 \rightarrow b_2,[\mu]}\times\tilde{G}_{b_2}} && {\Fl_{[\mu]}^\lfid} & {\Gr_{[\mu^{-1}]}} \\
	&&& {\Gr_{[\mu^{-1}]}^{b_1 \rightarrow [b_2]}} \\
	&&&& {C_{[\mu]}\times G(\mbb{B}^+_\dR)}
	\arrow["{(c,g)\mapsto cg}"{description}, from=1-5, to=4-5]
	\arrow[hook, from=2-4, to=3-4]
	\arrow["{c_\dR \times (g \mapsto \triv_{b_1}^{-1}\circ g \circ \triv_{b_1})}", curve={height=-24pt}, from=3-1, to=1-5]
	\arrow["{a_1}"', from=3-1, to=4-1]
	\arrow["\BB"', from=3-4, to=3-3]
	\arrow["{\pi_1}"{description}, curve={height=-12pt}, from=4-1, to=2-4]
	\arrow["{\pi_1^f}"{description}, from=4-1, to=3-3]
	\arrow["{c_\dR}"{description}, hook, from=4-1, to=4-5]
	\arrow["{\pi_2^f}"{description}, from=4-1, to=5-3]
	\arrow["{\pi_2}"{description}, curve={height=12pt}, from=4-1, to=6-4]
	\arrow["{c \mapsto c \cdot \ast_1}"{description}, from=4-5, to=3-4]
	\arrow["{c\mapsto c^{-1}\cdot\ast_1}"{description}, from=4-5, to=5-4]
	\arrow["{a_2}", from=5-1, to=4-1]
	\arrow["{c_\dR \times (g \mapsto \triv_{b_2}^{-1}\circ g \circ \triv_{b_2})}"', curve={height=18pt}, from=5-1, to=7-5]
	\arrow["\BB", from=5-4, to=5-3]
	\arrow[hook, from=6-4, to=5-4]
	\arrow["{(c, g) \mapsto g^{-1}c}"{description}, from=7-5, to=4-5]
\end{tikzcd}\]
\end{theorem}
\begin{proof}
That $\M_{b_1 \rightarrow b_2,[\mu]}$ is an inscribed $v$-sheaf follows from the corresponding property of the constituents of the fiber product and \cref{lemma.inscribed-limits}. The rest of the theorem follows by restriction from \cref{theorem.bounded-structure}.
\end{proof}

\begin{remark}
Note that $\mc{M}_{b_1 \rightarrow b_2,[\mu]}$ may be empty. When $b_1=1$, it is non-empty exactly when $[b_2]$ lies in the Kottwitz set $B(G,[\mu^{-1}])$.
\end{remark}

As in the unbounded case, we can now describe the derivatives. To that end, we need to consider some bounded analogs of the fundamental exact sequence. We write $\mc{E}_{\mr{max}}$ for the minimal common modification of $\mc{E}(\mf{g}_{b_1})$ and $\mc{E}(\mf{g}_{b_2})$ on $\mc{M}_{b_1\rightarrow b_2, [\mu]}$, i.e. for the modification associated by \cref{lemma.vb-glueing} to
\[ \mf{g}_{b_1} \boxtimes \mbb{B}^+_\dR + \mf{g}_{b_2} \boxtimes \mbb{B}^+_\dR = \pi_1^*\mf{g}^+_{\mr{max}} = \pi_2^*\mf{g}^+_{\mr{max}}, \]
which is a lattice in $\mf{g}_{b_1} \boxtimes \mbb{B}_\dR=\mf{g}_{b_2} \boxtimes \mbb{B}_\dR$ by \cref{prop.schubert-cell}.

We write this lattice as $\mf{g}^+_{\mr{max}}$. Then, for each $i$, we have an exact sequence of sheaves on $\mc{X}$ over $\mc{M}_{b_1\rightarrow b_2, [\mu]}$: 
\begin{equation}\label{eq.ses-for-bdd-fes} 0 \rightarrow \mc{E}(\mf{g}_{b_i}) \rightarrow \mc{E}_{\mr{max}} \rightarrow \infty_* (\mf{g}^+_{\mr{max}} / \mf{g}_{b_i} \boxtimes \mbb{B}^+_\dR) \rightarrow 0. \end{equation}
The $v$-sheafification of the associated long exact sequence of cohomology gives rise to an exact sequence of $E^\lfid$-modules on $\M_{b_1\rightarrow b_2, [\mu]}$, 
\begin{equation}\label{eq.bdd-fes} 0 \rightarrow \BC(\mf{g}_{b_i}) \rightarrow \BC(\mc{E}_{\mr{\max}}) \rightarrow \frac{\mf{g}^+_{\mr{max}}}{\mf{g}_{b_i} \boxtimes \mbb{B}^+_\dR} \rightarrow \BC(\mf{g}_{b_i}[1]) \rightarrow \BC(\mc{E}_{\mr{max}}[1]) \rightarrow 0\end{equation}
where the last zero is simply because the third term in \cref{eq.ses-for-bdd-fes} is a quasi-coherent sheaf supported on a closed affine subscheme of $\mathcal{X}$ so has vanishing cohomology. 

\begin{corollary}\label{corollary.bounded-mod-tangent}
The following diagram of inscribed $E^\lfid$-modules on $\M_{b_1 \rightarrow b_2, [\mu]}$, which is an extended subdiagram of (the pullback to $\M_{b_1\rightarrow b_2,[\mu]}$ of) the diagram of \cref{cor.unbounded-mod-tangent}, commutes. The left and right columns are the bounded fundamental exact sequences of \cref{eq.bdd-fes} for $\mf{g}_{b_1}$ and $\mf{g}_{b_2}$, and the morphisms $c_i$ are obtained by restricting \cref{eq.pullback-isos}. Moreover, the horizontal arrows are all isomorphisms.
% % https://q.uiver.app/#q=WzAsMTcsWzIsMiwiVF97XFxNX3tiXzEgXFxyaWdodGFycm93IGJfMiwgW1xcbXVdfX0iXSxbMCwyLCJcXEJDKFxcbWN7RX1fe1xcbXJ7bWF4fX0pIl0sWzQsMiwiXFxCQyhcXG1je0V9X3tcXG1ye21heH19KSJdLFsxLDQsIlxccGlfezF9XipUX3tcXEdyX3tbXFxtdV19fSJdLFs0LDQsIlxcZnJhY3tcXG1me2d9Xitfe1xcbXJ7bWF4fX19e1xcbWZ7Z31fe2JfMn0gXFxib3h0aW1lcyBcXG1iYntCfV4rX1xcZFJ9Il0sWzAsNCwiXFxmcmFje1xcbWZ7Z31eK197XFxtcnttYXh9fX17XFxtZntnfV97Yl8xfSBcXGJveHRpbWVzIFxcbWJie0J9XitfXFxkUn0iXSxbMCwxLCJcXEJDKFxcbWZ7Z31fe2JfMX0pIl0sWzAsMCwiMCJdLFswLDUsIlxcQkMoXFxtZntnfV97Yl8xfVsxXSkiXSxbMCw3LCIwIl0sWzQsNywiMCJdLFs0LDEsIlxcQkMoXFxtZntnfV97Yl8yfSkiXSxbNCwwLCIwIl0sWzMsNCwiXFxwaV97Mn1eKlRfe1xcR3Jfe1tcXG11XnstMX1dfX0iXSxbNCw1LCJcXEJDKFxcbWZ7Z31fe2JfMn1bMV0pIl0sWzAsNiwiXFxCQyhcXG1je0V9X3tcXG1ye21heH19WzFdKSJdLFs0LDYsIlxcQkMoXFxtY3tFfV97XFxtcnttYXh9fVsxXSkiXSxbMCwzLCJkXFxwaV97MX0iLDJdLFsyLDRdLFs1LDMsImNfMSIsMV0sWzEsMCwiKGRhXzEpX2UiLDFdLFs2LDFdLFsxLDVdLFs3LDZdLFs1LDhdLFsxMiwxMV0sWzIsMCwiKGRhXzIpX2UiLDFdLFs0LDEzLCJjXzIiLDFdLFswLDEzLCJkXFxwaV97Mn0iXSxbMTEsMl0sWzQsMTRdLFs4LDE1XSxbMTUsOV0sWzE0LDE2XSxbMTYsMTBdLFsxLDIsIi1cXEFkXypcXHZhcnBoaSIsMSx7ImN1cnZlIjotNH1dXQ==
\[\begin{tikzcd}
	0 &&&& 0 \\
	{\BC(\mf{g}_{b_1})} &&&& {\BC(\mf{g}_{b_2})} \\
	{\BC(\mc{E}_{\mr{max}})} && {T_{\M_{b_1 \rightarrow b_2, [\mu]}}} && {\BC(\mc{E}_{\mr{max}})} \\
	\\
	{\frac{\mf{g}^+_{\mr{max}}}{\mf{g}_{b_1} \boxtimes \mbb{B}^+_\dR}} & {\pi_{1}^*T_{\Gr_{[\mu]}}} && {\pi_{2}^*T_{\Gr_{[\mu^{-1}]}}} & {\frac{\mf{g}^+_{\mr{max}}}{\mf{g}_{b_2} \boxtimes \mbb{B}^+_\dR}} \\
	{\BC(\mf{g}_{b_1}[1])} &&&& {\BC(\mf{g}_{b_2}[1])} \\
	{\BC(\mc{E}_{\mr{max}}[1])} &&&& {\BC(\mc{E}_{\mr{max}}[1])} \\
	0 &&&& 0
	\arrow[from=1-1, to=2-1]
	\arrow[from=1-5, to=2-5]
	\arrow[from=2-1, to=3-1]
	\arrow[from=2-5, to=3-5]
	\arrow["{(da_1)_e}"{description}, from=3-1, to=3-3]
	\arrow["{-\Ad_*\varphi}"{description}, curve={height=-24pt}, from=3-1, to=3-5]
	\arrow[from=3-1, to=5-1]
	\arrow["{d\pi_{1}}"', from=3-3, to=5-2]
	\arrow["{d\pi_{2}}", from=3-3, to=5-4]
	\arrow["{(da_2)_e}"{description}, from=3-5, to=3-3]
	\arrow[from=3-5, to=5-5]
	\arrow["{c_1}"{description}, from=5-1, to=5-2]
	\arrow[from=5-1, to=6-1]
	\arrow["{c_2}"{description}, from=5-5, to=5-4]
	\arrow[from=5-5, to=6-5]
	\arrow[from=6-1, to=7-1]
	\arrow[from=6-5, to=7-5]
	\arrow[from=7-1, to=8-1]
	\arrow[from=7-5, to=8-5]
\end{tikzcd}\]
\end{corollary}

Finally, we obtain also a description of the tangent bundle and normal bundle for the bounded generalized Newton strata.

\begin{corollary}\label{cor.bdd-newton-strata-tangent-normal}
    Let $\mc{E}_{[b_2]}$ be the restriction of the universal $G$-bundle on $\mc{X}^*\BG$ to $(\mc{X}^*\BG)^{[b_2]}$ and let $\mf{g}_{[b_2]}=\mc{E}_{[b_2]}(\mf{g})$ be its push-out by the adjoint representation. The short exact sequence over $\Gr_{[\mu]}^{b_1 \rightarrow [b_2]}$ induced by the bounded fundamental exact sequence analogous to  \cref{eq.bdd-fes} for $\mf{g}_{[b_2]}$,  
\[ 0 \rightarrow \frac{\BC(\mc{E}_{\mr{max}})}{ \BC(\mf{g}_{[b_2]})}  \rightarrow \frac{\mf{g}^+_{\mr{max}}}{\mf{g}_{[b_2]}  \boxtimes \mbb{B}^+_\dR} \rightarrow \mr{Ker}\left(\BC(\mf{g}_{[b_2]} [1]) \rightarrow \BC(\mc{E}_{\mr{max}}[1]) \right) \rightarrow 0, \]
is canonically identified with the short exact sequence
\[ 0 \rightarrow T_{\Gr_{[\mu]}^{b_1 \rightarrow[b_2]}} \xrightarrow{d\iota} \iota^* T_{\Gr_G} \rightarrow N_{\Gr_{[\mu]}^{b_1 \rightarrow [b_2]}} \rightarrow 0 \]
where $\iota: \Gr_{[\mu]}^{b_1 \rightarrow [b_2]} \hookrightarrow \Gr_{[\mu]}$ is the inclusion. This isomorphism is moreover compatible with that of \cref{cor.newton-strata-tangent-normal} by the natural inclusion maps. 
\end{corollary}

\subsection{The two towers}\label{ss.two-towers}
There is an evident isomorphism 
\[ \M_{b_1 \rightarrow b_2} \xrightarrow{\sim} \M_{b_2 \rightarrow b_1}, \varphi \mapsto \varphi^{-1}. \]
This isomorphism reflects the diagrams of \cref{theorem.unbounded-structure} and \cref{theorem.bounded-structure} along the central horizontal axis, acting as $c \mapsto c^{-1}$ on $G(\mbb{B}_\dR)$ (and thus it reflects the diagrams of \cref{cor.unbounded-mod-tangent} and \cref{corollary.bounded-mod-tangent} along the central vertical axis). 

One also obtains interesting isomorphisms by changing the group; combining these observations will give the traditional two towers isomorphism  (in the form discussed, e.g., in \cite[\S 8.5]{HoweKlevdal.AdmissiblePairsAndpAdicHodgeStructuresIITheBiAnalyticAxLindemannTheorem}) in a very general setting. To state this cleanly, it is useful to consider the following generalization of our constructions:

\begin{definition}\label{def.X-pure-inner-form}
For $G/E$ a connected linear algebraic group, an affine group scheme $\mc{G}$ on $\mc{X}$ over $\Spd \breve{E}$ is an $\mc{X}$-pure inner form of $G$ if it is isomorphic to the automorphism group scheme $\mc{A}ut_{\mc{X}}(\mc{E})$ of a $G$-torsor $\mc{E}: \Spd \breve{E} \rightarrow \mc{X}^*\BG$. 
\end{definition}

\begin{example} For $b \in G(\breve{E})$, $\mc{A}ut(\mc{E})$ is the affine group scheme we have denoted by $\mc{G}_b$ above. If $b$ is basic, then $\mc{G}_b=G_b \times_E \mc{X}$ where $G_b$ is the automorphism group of the $G$-isocrystal $b$. 
\end{example}

Given $\mc{G}$ that is an $\mc{X}$-pure inner form of a connected linear algebraic group $G/E$, if we fix a $G$-bundle $\mc{E}$ and an isomorphism $\mc{G}=\mc{A}ut(\mc{E})$, we obtain a twisting isomorphism
\[ \mc{X}^*\BG \xrightarrow{\sim} \mc{X}^* \mr{B}\mc{G}, \mc{E}' \mapsto \mc{I}som_G(\mc{E}, \mc{E}').\]
In particular, it follows that $\mc{X}^* \mr{B}\mc{G}$ is an inscribed $v$-stack. 

Suppose now given such a $\mc{G}$, and $\mc{G}$-torsors $\mc{E}_1, \mc{E}_2: \Spd \breve{E} \rightarrow \mc{X}^*\mr{B}\mc{G}$ equipped with trivializations $\triv_i:   \mc{E}_\triv \xrightarrow{\sim} \mc{E}_i|_{\Spec \mbb{B}_\dR}$. Then, we can define the moduli of modifications $\mc{M}_{\mc{E}_1 \rightarrow \mc{E}_2}$, the maps $c_\dR$ and $\pi_i$, the actions of automorphism groups, etc., by imitating the discussion given above. In particular, one again obtains an isomorphism $\M_{\mc{E}_1 \rightarrow \mc{E}_2}$ that acts as $c\mapsto c^{-1}$ on $G(\mbb{B}_\dR)$ and reflects the same diagrams. 

Combining these two constructions, one obtains (an extension of) the classical two towers duality for the infinite level moduli space $\M_{1 \rightarrow b}$: first, we apply $\mc{I}som_G(\mc{E}_b, \bullet)$ to obtain an isomorphism with 
$\M_{\mc{E}' \rightarrow \mc{E}'_\triv}$ where in the subscript we have the $\mc{G}_b$-torsors $\mc{E}'=\mc{I}som_G(\mc{E}_b, \mc{E}_1)$ and $\mc{E}'_\triv=\mc{I}som_G(\mc{E}_b, \mc{E}_b)$ (the latter is the trivial $\mc{G}_b=\mc{I}som_G(\mc{E}_b, \mc{E}_b)$-torsor). Then, we take an inverse as above to reverse the arrows to get $\M_{\mc{E}' \rightarrow \mc{E}'_\triv} \xrightarrow{\sim} \M_{\mc{E}_\triv' \rightarrow \mc{E}'}$. Composing the two isomorphisms yields
\[ \M_{1 \rightarrow b} \xrightarrow{\sim} \M_{\mc{E}'_\triv \rightarrow \mc{E}'}. \]
When $b$ is basic so that $\mc{G}_b=G_b \times_E \mc{X}$, the right-hand side is canonically identified with 
$\M_{1 \rightarrow b^{-1}}$ where here $b^{-1}$ is viewed as an element of ${G_b(\breve{E})=G(\breve{E})}$ (see, e.g., \cite[\S 8.5]{HoweKlevdal.AdmissiblePairsAndpAdicHodgeStructuresIITheBiAnalyticAxLindemannTheorem}). This construction thus extends the classical two towers duality to the non-basic case (as well as the inscribed setting). Outside of the basic case, however, $\mc{G}_b$ is not the base change of a group over $\Spec E$, so one needs to allow the more general moduli spaces for $\mc{X}$-pure inner forms considered in this subsection in order to state it.  

\subsection{Moduli of mixed characteristic local shtukas with one leg}\label{ss.main-results-mcls}

For $[\mu]$ a conjugacy class of cocharacters and $b \in G(\breve{E})$ such that $b \in B(G,[\mu^{-1}])$, we write $\M_{b,[\mu]}:=\M_{1 \rightarrow b}$. In other words, 
\[ \M_{b,[\mu]}(\mc{X}/X_{E,P}^\alg, P/\Spd \breve{E}([\mu])) \]
is the set of meromorphic isomorphisms $\varphi: \mc{E}_1|_{\mc{X}\backslash \infty} \rightarrow \mc{E}_b|_{\mc{X}\backslash \infty}$ whose period matrix  
\[ c_\dR(\varphi):=\triv_b^{-1} \circ \varphi \in G(\mbb{B}_\dR) \]
lies in the Schubert cell $C_{[\mu]} \subseteq G(\mbb{B}_\dR)$.  The underlying $v$-sheaf is the infinite level moduli of mixed characteristic shtuka with one leg studied in \cite{ScholzeWeinstein.BerkeleyLecturesOnPAdicGeometryAMS207, HoweKlevdal.AdmissiblePairsAndpAdicHodgeStructuresIIIVariationAndUnlikelyIntersection}. In particular, when $[\mu]$ is minuscule, the underlying $v$-sheaf is an infinite level local Shimura variety. 

We rename the period maps in this setting: we write
$\pi_{\Hdg}^+:=\pi_1$, $\pi_{\Hdg}:=\pi_1^f$, $\pi_{\HT}^+:=\pi_2$, and $\pi_{\HT}:=\pi_2^f$ (note that, in \cite[\S8]{HoweKlevdal.AdmissiblePairsAndpAdicHodgeStructuresIITheBiAnalyticAxLindemannTheorem}, $\pi_{\Hdg}^+$ is written as $\pi_{\mc{L}_\et}$ and $\pi_{\HT}^+$ is written as $\pi_{\mc{L}_\dR}$ --- this notation comes from the fact that, for a point coming from $p$-adic cohomology, the lattice $\mc{L}_\et$ deforms \'{e}tale cohomology while $\mc{L}_\dR$ deforms de Rham cohomology). 

The groups that act are $\tilde{G}_b$, the automorphism group of $\mc{E}_b$, and $G(E^\lfid)$, the automorphism group of $\mc{E}_1$. Concretely, the action maps are
\begin{align*} a_1: \M_{b,[\mu]} \times G(E^\lfid) & \rightarrow \M_{b,[\mu]}, (\varphi, g) \mapsto (\varphi \circ g) \\
\textrm{ and } a_2: \M_{b,[\mu]} \times \tilde{G}_b & \rightarrow \M_{b,[\mu]}, (\varphi, g) \mapsto (g^{-1} \circ \varphi).
\end{align*}

In this setting, the commutative diagram of \cref{theorem.bounded-structure}
becomes 
% https://q.uiver.app/#q=WzAsMTAsWzAsMiwiXFxNX3tiLFtcXG11XX0iXSxbMiwxLCJcXEZsX3tbXFxtdV57LTF9XX1eXFxsZmlkIl0sWzMsMSwiXFxHcl97W1xcbXVdfSJdLFsyLDMsIlxcRmxfe1tcXG11XX1eXFxsZmlkIl0sWzMsMywiXFxHcl97W1xcbXVeey0xfV19Il0sWzQsMiwiQ197W1xcbXVdfSAiXSxbMCwzLCJcXE1fe2IsW1xcbXVdfVxcdGltZXNcXHRpbGRle0d9X2IiXSxbMCwxLCJcXE1fe2IsW1xcbXVdfSBcXHRpbWVzIEcoRV5cXGxmaWQpIl0sWzQsMCwiQ197W1xcbXVdfSBcXHRpbWVzIEcoXFxtYmJ7Qn1eK19cXGRSKSJdLFs0LDQsIkNfe1tcXG11XX1cXHRpbWVzIEcoXFxtYmJ7Qn1eK19cXGRSKSJdLFswLDIsIlxccGlfe1xcSGRnfV4rIiwxLHsiY3VydmUiOjF9XSxbMCwxLCJcXHBpX1xcSGRnIl0sWzIsMSwiXFxCQiIsMl0sWzQsMywiXFxCQiJdLFswLDQsIlxccGlfe1xcSFR9XisiLDEseyJjdXJ2ZSI6LTF9XSxbMCwzLCJcXHBpX1xcSFQiLDJdLFswLDUsImNfXFxkUiIsMSx7InN0eWxlIjp7InRhaWwiOnsibmFtZSI6Imhvb2siLCJzaWRlIjoidG9wIn19fV0sWzYsMCwiYV8yIl0sWzcsMCwiYV8xIiwyXSxbNyw4LCJjX1xcZFIgXFx0aW1lcyAoZ1xcbWFwc3RvIGcpIiwwLHsiY3VydmUiOi0xfV0sWzYsOSwiY19cXGRSIFxcdGltZXMgKGcgXFxtYXBzdG8gXFx0cml2X2Jeey0xfVxcY2lyYyBnIFxcY2lyYyBcXHRyaXZfYikiLDIseyJjdXJ2ZSI6MX1dLFs4LDUsIihjLGcpXFxtYXBzdG8gY2ciLDFdLFs5LDUsIihjLCBnKSBcXG1hcHN0byBnXnstMX1jIiwxXSxbNSwyLCJjIFxcbWFwc3RvIGMgXFxjZG90IFxcYXN0XzEiLDFdLFs1LDQsImNcXG1hcHN0byBjXnstMX1cXGNkb3RcXGFzdF8xIiwxXV0=
\begin{equation}\label{eq.mmcls-diagram}\begin{tikzcd}
	&&&& {C_{[\mu]} \times G(\mbb{B}^+_\dR)} \\
	{\M_{b,[\mu]} \times G(E^\lfid)} && {\Fl_{[\mu^{-1}]}^\lfid} & {\Gr_{[\mu]}} \\
	{\M_{b,[\mu]}} &&&& {C_{[\mu]} } \\
	{\M_{b,[\mu]}\times\tilde{G}_b} && {\Fl_{[\mu]}^\lfid} & {\Gr_{[\mu^{-1}]}} \\
	&&&& {C_{[\mu]}\times G(\mbb{B}^+_\dR)}
	\arrow["{(c,g)\mapsto cg}"{description}, from=1-5, to=3-5]
	\arrow["{c_\dR \times (g\mapsto g)}", curve={height=-6pt}, from=2-1, to=1-5]
	\arrow["{a_1}"', from=2-1, to=3-1]
	\arrow["\BB"', from=2-4, to=2-3]
	\arrow["{\pi_\Hdg}", from=3-1, to=2-3]
	\arrow["{\pi_{\Hdg}^+}"{description}, curve={height=6pt}, from=3-1, to=2-4]
	\arrow["{c_\dR}"{description}, hook, from=3-1, to=3-5]
	\arrow["{\pi_\HT}"', from=3-1, to=4-3]
	\arrow["{\pi_{\HT}^+}"{description}, curve={height=-6pt}, from=3-1, to=4-4]
	\arrow["{c \mapsto c \cdot \ast_1}"{description}, from=3-5, to=2-4]
	\arrow["{c\mapsto c^{-1}\cdot\ast_1}"{description}, from=3-5, to=4-4]
	\arrow["{a_2}", from=4-1, to=3-1]
	\arrow["{c_\dR \times (g \mapsto \triv_b^{-1}\circ g \circ \triv_b)}"', curve={height=6pt}, from=4-1, to=5-5]
	\arrow["\BB", from=4-4, to=4-3]
	\arrow["{(c, g) \mapsto g^{-1}c}"{description}, from=5-5, to=3-5]
\end{tikzcd}\end{equation}

In particular the group actions realize $\pi_{\Hdg}^+$ as a $\tilde{G}_b$-equivariant $G(E^\lfid)$-torsor over its image $\Gr_{[\mu]}^{b-\adm}:=\Gr_{[\mu]}^{b\rightarrow[1]}$, the inscribed $b$-admissible locus, and $\pi_2$ as a $G(E^\lfid)$-equivariant $\tilde{G}_b$-torsor over its image $\Gr_{[\mu^{-1}]}^{[b]}:=\Gr_{[\mu^{-1}]}^{1\rightarrow[b]}$, the inscribed Newton stratum.

After restricting test objects to $(X^\alg_{E,\Box})^\lfp$ to ensure that $\BC(\mf{g}[1])$ vanishes (i.e. restricting to the case where the $b$-admissible locus is deformation-theoretically open by \cref{prop.b-basic-open-stratum} and thus has trivial normal bundle), \cref{corollary.bounded-mod-tangent} and \cref{theorem.SchubertCellDiagrams} combine into the following commutative diagram of derivatives for \cref{eq.mmcls-diagram}: 

\[\begin{tikzcd}
	{0\rightarrow \mf{g}\otimes_E E^\lfid} & {\BC(\mc{E}_{\mr{max}})} & {\frac{\mf{g}^+_{\mr{max}}}{\mf{g} \otimes_E \mbb{B}^+_\dR}} & 0 \\
	& {\frac{\mf{g}_b\otimes_{\breve{E}} \mc{O}}{\Fil^0_\Hdg (\mf{g}_b \otimes_{\breve{E}}\mc{O})}} & {} \\
	& {\pi_\Hdg^*T_{\Fl_{[\mu^{-1}]}^\lfid}  } & {{\pi_{\Hdg}^+}^*T_{\Gr_{[\mu]}}} \\
	{T\M_{b,[\mu]}} &&& {c_\dR^*TC_{[\mu]}=\mf{g}^+_\mr{\max}} \\
	& {\pi_\HT^*T_{\Fl_{[\mu]}^\lfid}  } & {{\pi_{\HT}^+}^*T_{\Gr_{[\mu^{-1}]}}} \\
	& {\frac{\mf{g}\otimes_E \mc{O}}{\Fil^0_\HT (\mf{g}\otimes_E \mc{O})}} && {{\pi_{\HT}^+}^*N\Gr_{[\mu^{-1}]}^{[b]}} \\
	{0\rightarrow\BC(\mf{g}_{b})} & {\BC(\mc{E}_{\mr{max}})} & {\frac{\mf{g}^+_{\mr{max}}}{\mf{g}_{b} \otimes_{\breve{E}} \mbb{B}^+_\dR}} & {\BC(\mf{g}_{b}[1])\rightarrow0}
	\arrow[from=1-1, to=1-2]
	\arrow["{(da_1)_e}"{description}, from=1-1, to=4-1]
	\arrow[from=1-2, to=1-3]
	\arrow["\sim"{description}, curve={height=12pt}, from=1-2, to=4-1]
	\arrow[from=1-3, to=1-4]
	\arrow[two heads, from=1-3, to=2-2]
	\arrow["\sim"{description}, from=1-3, to=3-3]
	\arrow["\sim"{description}, from=2-2, to=3-2]
	\arrow["{{\pi_{\Hdg}^+}^*d\BB}"', from=3-3, to=3-2]
	\arrow["{d\pi_\Hdg}"{description}, from=4-1, to=3-2]
	\arrow["{d\pi_{\Hdg}^+}"{description}, curve={height=6pt}, from=4-1, to=3-3]
	\arrow["{dc_\dR}"{description}, from=4-1, to=4-4]
	\arrow["{d\pi_\HT}"{description}, from=4-1, to=5-2]
	\arrow["{d\pi_{\HT}^+}"{description}, curve={height=-6pt}, from=4-1, to=5-3]
	\arrow[curve={height=-12pt}, two heads, from=4-4, to=1-3]
	\arrow[curve={height=12pt}, two heads, from=4-4, to=7-3]
	\arrow["{{\pi_{\HT}^+}^*d\BB}", from=5-3, to=5-2]
	\arrow[dashed, two heads, from=5-3, to=6-4]
	\arrow["\sim"{description}, from=6-2, to=5-2]
	\arrow["{(da_2)_e}"{description}, from=7-1, to=4-1]
	\arrow[from=7-1, to=7-2]
	\arrow["\sim"{description}, curve={height=-12pt}, from=7-2, to=4-1]
	\arrow[from=7-2, to=7-3]
	\arrow["\sim"{description}, from=7-3, to=5-3]
	\arrow[two heads, from=7-3, to=6-2]
	\arrow[from=7-3, to=7-4]
	\arrow["\sim"{description}, from=7-4, to=6-4]
\end{tikzcd}\]

Note that, when $b$ is basic, the normal bundle in of the Newton stratum also vanishes in this diagram, since then the Newton stratum is also deformation-theoretically open (still under the assumption that we have restricted to $(X^\alg_{E,\Box})^\lfp$!).

\begin{remark} One can hope that these descriptions of the tangent and normal bundles of the generalized Newton stratifications will have applications, e.g., to a Banach--Colmez theory of holonomic $D$-modules. 
\end{remark}

\begin{remark}[Finite level local Shimura varieties]\label{remark.local-shimura-finite-level}
For $[\mu]$ minuscule, i.e. the local Shimura case, and $K_p \leq G(\mathbb{Q}_p)$ compact open, it is shown in \cite{ScholzeWeinstein.BerkeleyLecturesOnPAdicGeometryAMS207} that there exist rigid analytic local Shimura varieties $\M^{\mr{rig}}_{b,[\mu], K_p}$ representing the diamonds $\overline{\M_{b,[\mu]}}/K_p$. Indeed, in the minuscule case the Bialynicki-Birula map is an isomorphism $\overline{\Gr_{[\mu]}}=\Fl_{[\mu^{-1}]}^\diamond$, so that one can construct $\M^{\mr{rig}}_{b,[\mu], K_p}$ by invoking the equivalence of \'{e}tale sites $(\Fl_{[\mu^{-1}]})_\et= (\Fl_{[\mu^{-1}]}^\diamond)_\et$. 

We claim that, in this minuscule case, there is a canonical identification of $\M_{b,[\mu], K_p}:=\M_{b,[\mu]}/K_p^\lfid$ with $(\M^{\mr{rig}}_{b,[\mu], K_p})^\lfid$ after restriction to $(X^\alg_{E,\Box})^\lfp$. To see this, the first point is to construct a map from $\M_{b,[\mu]}$ to $(\M^{\mr{rig}}_{b,[\mu], K_p})^\lfid$. To that end, we note that, as in \cref{prop.ks-iso-formal-nbhd}, since the map $\mc{M}_{b,[\mu],K_p}^{\mr{rig}} \rightarrow \Fl_{[\mu^{-1}]}$ is \'{e}tale we have 
\[ (\M^{\mr{rig}}_{b,[\mu], K_p})^\lfid = \M_{b,[\mu], K_p}^\diamond \times_{\Fl_{[\mu^{-1}]}^\diamond} \Fl_{[\mu^{-1}]}^\lfid. \]
Thus we obtain a map from $\mc{M}_{b,[\mu]}$ as the product of the obvious map
\[ \overline{\M_{b,[\mu]}}  \rightarrow \overline{\M_{b,[\mu]}}/K_p= \M_{b,[\mu],K_p}^{\mr{rig},\diamond} \]
and $\pi_{\Hdg}$. The map is evidently $K_p^\lfid$-equivariant for the trivial action on the target, so factors through a map from $\M_{b,[\mu], K_p}$. This map is an isomorphism on underlying $v$-sheaves. We claim that, because we have imposed the slope condition, it is also an isomorphism on the formal neighborhood of each point of the underlying $v$-sheaf. 

It suffices to see that the map $\pi_{\Hdg}/K_p^\lfid: \M_{b,[\mu],K_p} \rightarrow \Fl_{[\mu^{-1}]}^\lfid$ is an isomorphism on the formal neighborhood of each point of the underlying $v$-sheaf. To obtain this, observe the map from $\M_{b,[\mu]}$ to $\Fl_{[\mu]^{-1}}^\lfid$ is a $G(E^\lfid)$-torsor over a deformation-theoretic open by \cref{prop.b-basic-open-stratum} (here is where we use the $\lfp$ slope condition), thus surjective on the formal neighborhood of any point of the underlying $v$-sheaf. Because $G(E^\lfid)/K_p^\lfid$ is trivially inscribed, the induced map from $\M_{b,[\mu],K_p}$ to  $\Fl_{[\mu]^{-1}}^\lfid$ becomes also injective on these formal neighborhoods, and thus an isomorphism on these formal neighborhoods. 
\end{remark}

\section{Inscribed global Shimura varieties via Igusa stacks}\label{s.Inscribed-global-Shimura}

In this section we construct a moduli-theoretic inscription on any infinite level Shimura varieties for which there exists an Igusa stack in the sense of \cite{Kim.UniquenessAndFunctorialityOfIgusaStacks}, and show that it is a torsor over the natural inscriptions on finite level rigid analytic Shimura varieties (mirroring \cref{remark.local-shimura-finite-level} for local Shimura varieties). We also recover the description of the tangent bundles in \cref{example.shimura-varieties} from this perspective --- see \cref{theorem.inscribed-global-igusa-theorem} for a full statement of the main results of this section.  

\subsubsection{}In this section we work in the inscribed context $(\AffPerf/\Spd \mbb{F}_p, X_{\mathbb{Q}_p,\Box})$ of \cref{s.inscribed-contexts}. 

\subsection{Setup and the Igusa stack}\label{ss.setup-and-igusa-stacks}

\subsubsection{}Let $\gx$ be a Shimura datum, and let $G=\mathsf{G}_{\mathbb{Q}_p}$. We fix a $p$-adic field $L$ and an embedding $\mathbb{Q}([\mu]) \rightarrow L$, where $\mathbb{Q}([\mu])$ is the reflex field for $\gx$, i.e. the field of definition of the conjugacy class of Hodge cocharacters $[\mu]$ (a subfield of $\mathbb{C}$ that is finite over $\mathbb{Q}$). We assume:
\begin{enumerate}
    \item There exists a  Igusa stack $\Igs\gx$ for $\gx$ in the sense of \cite[Definition 1.1]{Kim.UniquenessAndFunctorialityOfIgusaStacks} (which is necessarily unique by \cite[Theorem A]{Kim.UniquenessAndFunctorialityOfIgusaStacks}).
    \item The maximal $\mathbb{R}$-split $\mathbb{Q}$-anisotropic central torus in $G$ is trivial.  
\end{enumerate}
The second assumption is just to simplify statements (it ensures our infinite level Shimura variety is a torsor over the finite level Shimura variety for a compact open subgroup of $G(\mathbb{Q}_p)$), but it is also fairly innocent given the first assumption: the most general current existence result, \cite[Theorem D]{Kim.UniquenessAndFunctorialityOfIgusaStacks} (building on \cite{Zhang.APelTypeIgusaStackAndThePAdicGeometryOfShimuraVarieties} and \cite{DanielsVanHoftenKimZhang.IgusaStacksAndTheCohomologyOfShimuraVarieties}), shows that Igusa stacks exist when $\gx$ is of Hodge type, but the assumption (2) is automatic in the Hodge type case. 

\subsubsection{}
For $K \leq G(\mbb{A}^{\infty})$ a neat compact open, we write $\Sh_K$ for the associated Shimura variety of level $K$ as a (smooth) rigid analytic variety over $L$. For $K^p \leq G(\mathbb{A}^{\infty p})$ compact open, let 
\[ \Sh_{K^p}^\diamond := \varprojlim_{K_p} \Sh_{K_pK^p}^\diamond\]
where the limit is over compact open subgroups $K_p \leq G(\mathbb{Q}_p)$ such that $K_p K^p$ is neat. There is an action of $G(\mathbb{Q}_p^\diamond)$ on $\Sh_{K^p}^\diamond$ and a $G(\mathbb{Q}_p^\diamond)$-equivariant Hodge--Tate period map $\pi_{\HT}^\diamond: \Sh_{K^p}^\diamond \rightarrow \Fl_{[\mu]}^\diamond$. We write $\pi_{K_p}^\diamond: \Sh_{K^p}^\diamond \rightarrow \Sh_{K_pK^p}^\diamond$ for the natural projection map --- it is a torsor for $K_p^\diamond$. 

We write $\Igs_{K^p}=\Igs\gx/K^p$ for the finite level Igusa stack.   We write $\Bun_G$ for the $v$-stack of $G$-bundles on the Fargues--Fontaine curve and $\BL:\Fl_{[\mu]}^\diamond \rightarrow \Bun_G$ for the map obtained by making the Beauville-Laszlo modification of the trivial bundle along the $\mbb{B}^+_\dR$-lattice associated to the minuscule filtration parameterized by $\Fl_{[\mu]}^\diamond$. By the definition of an Igusa stack \cite[Definition 1.1]{Kim.UniquenessAndFunctorialityOfIgusaStacks}, there is a map 
\[ \overline{\pi}_\HT: \Igs_{K^p}\rightarrow \mr{Bun}_G \]
such that the following diagram is Cartesian:

\begin{equation}\label{eq.igusa-stacks-conj-diag}\begin{tikzcd}
	{\Sh_{K^p}^\diamond } & {\Fl_{[\mu]}^\diamond} \\
	{\Igs_{K^p}} & {\Bun_G}
	\arrow["{\pi_\HT^\diamond}", from=1-1, to=1-2]
	\arrow["{\pi_{K^p}^\diamond}"', from=1-1, to=2-1]
	\arrow["\lrcorner"{anchor=center, pos=0.125}, draw=none, from=1-1, to=2-2]
	\arrow["\BL", from=1-2, to=2-2]
	\arrow["{\overline{\pi}_\HT}", from=2-1, to=2-2]
\end{tikzcd}\end{equation}

\subsection{Construction}

We will now define $\Sh_{K^p}^\lfid$ by producing inscribed versions of the two maps $\overline{\pi}_\HT$ and $\BL$ and then taking their fiber product as in \cref{eq.igusa-stacks-conj-diag}.  

\subsubsection{}We write $\BL^\lfid: \Gr_{G} \rightarrow \mc{X}^*\BG$ for the modification map $m_{\mc{E}_\triv, \Id}$ in the notation of \cref{eq.modification-map}. Its restriction to $\Gr_{[\mu^{-1}]}$ can be viewed as a map $\Fl_{[\mu]}=\Gr_{[\mu^{-1}]} \rightarrow \mc{X}^*\BG$, since the Bialynicki-Birula map $\BB: \Gr_{[\mu^{-1}]} \rightarrow \Fl_{[\mu]}^\lfid$ is an isomorphism for $[\mu]$ minuscule. The induced map on underlying $v$-sheaves is the map $\BL$ appearing in \cref{eq.igusa-stacks-conj-diag}. 

\subsubsection{}We write $\overline{\pi}_\HT^\lfid$ for the composition of $\overline{\pi}_\HT$ (viewed as a map of trivially inscribed $v$-stacks) with the natural map $\Bun_G \rightarrow \mc{X}^*\BG$ (by pullback along $\mc{X}\rightarrow X_{\mbb{Q}_p,P}$). 

\subsubsection{}We define 
\begin{equation}\label{eq.min-comp-fiber-product-mod} {\Sh_{K^p}}^\lfid := (\Igs_{K^p} \xrightarrow{\overline{\pi}_\HT^\lfid} \mc{X}^*\BG) \times_{\mc{X}^*\BG} (\Fl_{[\mu]} \xrightarrow{\BL^\lfid} \mc{X}^*\BG).\end{equation}
By \cref{eq.igusa-stacks-conj-diag}, we have an identification 
\[ \overline{\Sh_{K^p}^\lfid} = (\Igs_{K^p}^\ast \xrightarrow{\overline{\pi}_\HT^\diamond} \Bun_G) \times_{\Bun_G} (\Fl_{[\mu]}^\diamond \xrightarrow{\BL} \Bun_G) = \Sh_{K^p}^\diamond. \]

\subsubsection{}By construction, there is a map $\pi_\HT^\lfid: (\Sh_{K_p})^\lfid \rightarrow \Fl_{[\mu]}^\lfid$ given by projection to the second factor in the product whose map on underlying $v$-sheaves is $\pi_\HT^\diamond$.

\begin{lemma}
The action of $G(\mathbb{Q}_p^\lfid)$ on $\Fl_{[\mu]}^\lfid$ preserves $\BL^\lfid$, thus induces
\[ a: \Sh_{K^p}^\lfid \times G(\mathbb{Q}_p^\lfid) \rightarrow \Sh_{K^p}^\lfid \]
such that $\pi_\HT$ is $G(\mathbb{Q}_p^\lfid)$-equivariant. \end{lemma}
\begin{proof}The existence of the action on $\Sh_{K^p}^\lfid$ follows from \cref{lemma.quasi-torsor-mod}, since $G(\mathbb{Q}_p^\lfid)\subseteq G(\mbb{B}_e)$ preserves $\Fl_{[\mu]}^\lfid = \Gr_{[\mu^{-1}]} \subseteq \Gr_G$.
\end{proof}

\subsection{Statement of result}
Let $\mf{g}:=\Lie G(\mathbb{Q}_p)$. Note that on $\Sh_{K_p}^\lfid$ we have the trivial bundle $\mf{g} \otimes_{\mbb{Q}_p} \mathcal{O}_{\mc{X}}$ over $\mc{X}$ and the pullback of the universal bundle associated to the adjoint representation over $\mc{X}^*\BG$, $\mf{g}_\univ$. By the construction of the map $\BL^\lfid$, we have fixed an isomorphism between these two bundles after restriction to $\mc{X}\bs\infty$.

In the following, we consider the inscribed finite level Shimura varieties $\Sh_{K_pK^p}^\lfid$ by change of context as in \cref{ss.change-of-context}, i.e. 
\[ \Sh_{K_pK^p}^\lfid(\mc{X}/X_{\mbb{Q}_p,P}, P/\Spd L) = \Hom_L( P^\sharp \times_{X_{\mbb{Q}_p,P}} \mc{X}, \Sh_{K_pK^p}). \]

\begin{theorem}\label{theorem.inscribed-global-igusa-theorem} With notation and assumptions as in \cref{ss.setup-and-igusa-stacks},
\begin{enumerate}
\item For compact opens $K_p \leq G(\mathbb{Q}_p)$ such that $K_p K^p$ is neat, there are canonical maps $\pi_{K_p}^\lfid: \Sh_{K^p}^\lfid \rightarrow \Sh_{K_pK^p}^\lfid$ extending the projection maps on the underlying $v$-sheaves. These maps are compatible as $K_p$ and $K^p$ vary and, after restriction of the test objects to $X_{\mbb{Q}_p,\Box}^\lfp$, realize $\Sh_{K^p}^\lfid$ as a $K_p^\lfid$-torsor over $\Sh_{K_pK^p}^\lfid$.  
\item 
The $\mbb{B}^+_\dR$-module on $\Sh_{K_p}^\lfid$
\[ \mf{g}^+_{\mr{max}}:=\mf{g}_\univ \boxtimes \mbb{B}^+_\dR + \mf{g} \otimes_E \mbb{B}^+_\dR \subseteq \mf{g}_\univ \boxtimes \mbb{B}_\dR = \mf{g} \otimes_E \mbb{B}_\dR\]
is locally free. In particular, there is an associated vector bundle $\mc{E}_{\mr{max}}$ on $\mc{X}$ over $\Sh_{K^p}^\lfid$ fitting into two canonical modification exact sequences of sheaves on $\mc{X}$
\begin{align}
\label{eq.intro-shimura-mod-1}0 \rightarrow \mf{g} \otimes_E \mc{O}_{\mc{X}} \rightarrow \mc{E}_{\mr{max}} \rightarrow \infty_* \left(\mf{g}^+_{\mr{max}}/\mf{g}\otimes_E \mbb{B}^+_\dR\right) \rightarrow 0\\
\label{eq.intro-shimura-mod-2} \textrm{ and } 0 \rightarrow \mf{g}_\univ \rightarrow \mc{E}_{\mr{max}} \rightarrow \infty_* \left(\mf{g}^+_{\mr{max}}/\mf{g}_\univ \boxtimes \mbb{B}^+_\dR \right) \rightarrow 0. 
\end{align}
There is a canonical identification $\BC(\mc{E}_{\mr{max}})=T_{\Sh_{K^p}^\lfid}$ such that the $v$-sheafification of the associated long exact sequences of cohomology for \cref{eq.intro-shimura-mod-1} and \cref{eq.intro-shimura-mod-2}  are identified, after restriction of the test objects to $X_{\mbb{Q}_p,\Box}^\lfp$, with
\begin{equation}\label{eq.global-shim-first-mod-seq} 0 \rightarrow \mathfrak{g}\otimes \mbb{Q}_p^\lfid \xrightarrow{da_e} \BC(\mc{E}_{\max})=T_{\Sh_{K^p}^\lfid}\xrightarrow{d\pi_{K_p}^\lfid} (\pi_{K_p}^\lfid)^* T_{\Sh_{K_pK^p}^\lfid} \rightarrow 0 \end{equation}
and
\begin{equation}\label{eq.global-shim-second-mod-seq} 
 0 \rightarrow \BC(\mathfrak{g}_\univ)\rightarrow \BC(\mathfrak{g}_{\max})=T_{(\Sh_{K^p}^\circ)^\lfid} \xrightarrow{d\pi_{\HT}^\lfid} (\pi_{\HT}^{\lfid})^* T_{\Fl_{[\mu]}^\lfid} \rightarrow \BC(\mathfrak{g}_\univ[1]) \rightarrow 0. \end{equation}
\end{enumerate}
\end{theorem}
\begin{proof}This combines \cref{lemma.shimura-tangent-computation} and \cref{lemma.torsor-global-shimura} to be proved below. 
\end{proof}

\begin{remark}
The fibers of $\pi_\HT^\lfid$ are identified, by the definition using a product, with inscribed variants of the Caraiani-Scholze type Igusa varieties $\Igs_{K^p} \times_{\Bun_G} ({\Bun_G} \xleftarrow{b}  \Spd \breve{\mbb{Q}}_p)$. In particular, if we look above a point of the flag variety lying in the Newton stratum for $b \in G(\breve{\mbb{Q}}_p)$, then the restriction of $\mf{g}_\univ$ is the bundle associated to the isocrystal $\mf{g}_b$. The kernel $\BC(\mf{g}_b)$ of $d \pi_\HT^\lfid$ at such a point, which is the tangent bundle of the fiber, can also be viewed as coming from differentiating the action on this Igusa variety of the group $\tilde{G}_b$ of automorphisms of the associated $G$-bundle on the relative thickened Fargues--Fontaine curve. Note that the term $\BC(\mf{g}_b[1])$ appearing in the restriction of \cref{eq.global-shim-second-mod-seq} is zero if and only if $b$ is basic. Its role here is as the normal bundle of the associated Newton stratum on the flag variety (cf. \cref{cor.bdd-newton-strata-tangent-normal} and the surrounding discussion). In particular, $\pi_\HT^{\lfid}$, after imposing the slope condition, is a submersion over the open basic locus but nowhere else; in general, it is only a submersion after pullback to a Newton stratum (and this gives the dimension computation of \cref{sss.perversity}).  
\end{remark}

\subsection{Computation of the tangent bundle}

As in the case of local Shimura varieties and more general moduli of modifications treated in \cref{s.moduli-of-mod}, the computation of the tangent bundle in \cref{theorem.inscribed-global-igusa-theorem} is accomplished by first treating an unbounded analog, and then cutting out the result we are interested in within. 

To that end, let $\mc{S} := (\Igs_{K^p} \xrightarrow{\overline{\pi}_\HT^\lfid} \mc{X}^*\BG) \times_{\mc{X}^*\BG} (\Gr_G \xrightarrow{\BL^\lfid} \mc{X}^*\BG)$. 

\begin{lemma}$\mc{S}$ is an inscribed $v$-sheaf. 
\end{lemma}
\begin{proof}
    Since each of the terms in the fiber product is an inscribed $v$-stack, it is an inscribed $v$-stack by \cref{lemma.inscribed-limits}. It is a sheaf rather than a more general stack because $\Gr_G$ is a sheaf and $\overline{\pi}_{\HT}^\lfid$ is 0-truncated. 
\end{proof}

We note that on $\mc{S}$ we have the trivial bundle $\mf{g} \otimes \mathcal{O}_{\mc{X}}$ over $\mc{X}$ and the pullback of the universal bundle associated to the adjoint representation over $\mc{X}^*\BG$, $\mf{g}_\univ$. By the construction of the map $\BL^\lfid$, we have fixed an isomorphism between these two bundles after restriction to $\mc{X}\bs\infty$. 

By \cref{lemma.quasi-torsor-mod}, $\mc{S}$ is a $G(\mbb{B}_e)$ quasi-torsor over $\Igs_{K^p}$ (in fact a torsor since the map to $\Igs_{K^p}$ is surjective, but we will not actually need this fact). Since $\Igs_{K^p}$ is trivially inscribed, we deduce that, writing $a: \mc{S} \times G(\mbb{B}_e) \rightarrow \mc{S}$ for the action map, $da_e$ induces an isomorphism $\mf{g} \otimes \mbb{B}_e = \Lie G(\mbb{B}_e) \xrightarrow{\sim} T_{\mc{S}}$.

The map $\pi_{2}: \mc{S} \rightarrow \Gr_G$ is equivariant along the inclusion $G(\mbb{B}_e) \subseteq G(\mbb{B}_\dR)$. We note also that, by construction, $\pi_2^* \mf{g}^+_{\univ}=\mf{g}_\univ \boxtimes \mbb{B}^+_\dR$. Comparing with \cref{corollary.tangent-bundle-affine-grassmannian}, we find $d\pi_2$ can be identified with the natural map $\mf{g} \otimes \mbb{B}_e \rightarrow \mf{g} \otimes \mbb{B}_\dR/\mf{g}_\univ \boxtimes \mbb{B}^+_\dR$. Since 
\[ \Sh_{K^p}^\lfid = (\mc{S} \xrightarrow{\pi_2} \Gr_G) \times_{\Gr_G} \Gr_{[\mu^{-1}]}, \]
comparing with \cref{prop.schubert-cell} we obtain 
\[ T_{\Sh_{K^p}^\lfid} = \mf{g} \otimes \mbb{B}_e \times_{\mf{g} \otimes \mbb{B}_\dR/\mf{g}^+_\univ} \mf{g}^+_\mr{max}/\mf{g}^+_\univ= \BC(\mc{E}_{\mr{max}}), \]
where $\mf{g}^+_\mr{max} = \mf{g} \otimes \mbb{B}^+_\dR + \mf{g}_\univ \boxtimes \mbb{B}^+_\dR$ and $\mc{E}_{\mr{max}}$ is the associated modification of $\mf{g} \otimes \mathcal{O}_{\mc{X}}$. We have thus established:

\begin{lemma}\label{lemma.shimura-tangent-computation} With notation above, the $\mbb{B}^+_\dR$-module on $\Sh_{K_p}^\lfid$
\[ \mf{g}^+_{\mr{max}}:=\mf{g}_\univ \boxtimes \mbb{B}^+_\dR + \mf{g} \otimes_E \mbb{B}^+_\dR \subseteq \mf{g}_\univ \boxtimes \mbb{B}_\dR = \mf{g} \otimes_E \mbb{B}_\dR\]
is locally free. In particular, there is a vector bundle $\mc{E}_{\mr{max}}$ on $\mc{X}$ over $\Sh_{K_p}^\lfid$ fitting into two canonical modification exact sequences of sheaves on $\mc{X}$
\begin{align}
\label{eq.body-shimura-mod-1}0 \rightarrow \mf{g} \otimes_E \mc{O}_{\mc{X}} \rightarrow \mc{E}_{\mr{max}} \rightarrow \infty_* \left(\mf{g}^+_{\mr{max}}/\mf{g}\otimes_E \mbb{B}^+_\dR\right) \rightarrow 0\\
\label{eq.body-shimura-mod-2} \textrm{ and } 0 \rightarrow \mf{g}_\univ \rightarrow \mc{E}_{\mr{max}} \rightarrow \infty_* \left(\mf{g}^+_{\mr{max}}/\mf{g}_\univ \boxtimes \mbb{B}^+_\dR \right) \rightarrow 0. 
\end{align}
There is a canonical identification $\BC(\mc{E}_{\mr{max}})=T_{\Sh_{K^p}^\lfid}$ such that the $v$-sheafification of the associated long exact sequences of cohomology for \cref{eq.body-shimura-mod-2} is identified with 
\begin{equation}\label{eq.second-mod-seq} 
 0 \rightarrow \BC(\mathfrak{g}_\univ)\rightarrow \BC(\mathfrak{g}_{\max})=T_{(\Sh_{K^p}^\circ)^\lfid} \xrightarrow{d\pi_{\HT}^\lfid} (\pi_{\HT}^{\lfid})^* T_{\Fl_{[\mu]}^\lfid} \rightarrow \BC(\mathfrak{g}_\univ[1]) \rightarrow 0. \end{equation}
\end{lemma}

\subsection{The Hodge period map at infinite level.}
We note that we can define a natural \emph{Hodge} period map on $\Sh_{K^p}^\lfid$. Indeed, already on $\Fl^\lfid_{[\mu]}$, we can define $\mc{E}_\dR:= \infty^*(\BL^*\mc{E}_{\univ})$, a $G(\mc{O})$-torsor equipped with a Hodge filtration $\Fil^\bullet_{\Hdg} \mc{E}_\dR$ of type $\mu^{-1}$ by the modification $\mc{E}_{\triv} \dashrightarrow \mc{E}_\dR$. Then, because $\pi_{\overline{\HT}}$ factors through $\Bun_G$, the pullback of $\mc{E}_\dR$ is equipped with a canonical integrable connection (cf. \cref{s.hodge-etc-period-maps}) 
\[ r^*\overline{\mc{E}_{\dR}} \times ^{G(\overline{\mathcal{O}})} G(\mathcal{O}) \xrightarrow{\sim} \mc{E}_\dR \]
where here $r: \Sh_{K_p}^\lfid \rightarrow \overline{\Sh_{K_p}^\lfid}=\Sh_{K_p}^\diamond$. In particular, we obtain a flat sections reduction of structure group on $\mc{E}_\dR$ to a $G(\overline{\mc{O}})$-torsor $\mc{E}_{\dR,\nabla}$ over $\Sh_{K^p}^\lfid$ and then, quotienting by $G(\overline{\mathcal{O}})$, a period map classifying the Hodge filtration 
\[ \pi_{\Hdg}: \Sh_{K^p}^\lfid \rightarrow \Fl_{[\mu^{-1}]}/G(\overline{\mathcal{O}}). \]

We fix a $K_p$ such that $K_pK^p$ is neat. It follows from the comparison between the automorphic de Rham local systems and those constructed via $p$-adic Hodge theory as established in \cite{LiuZhu.RigidityAndARiemannHilbertCorrespondenceForpAdicLocalSystems} that, on underlying $v$-sheaves, $\pi_{\Hdg}$ restricts to the composition of $\pi_{K_p}^\diamond$ with the usual Hodge period map as constructed in \cref{ss.hodge-period-maps} for the automorphic de Rham $G$-torsor $\mc{G}_\dR$-over $\Sh_{K_pK^p}$ (equipped with its integrable connection and Hodge filtration), 
\[ \pi_{\Hdg,K_p}^\diamond: \Sh_{K_pK^p}^\diamond \rightarrow \Fl_{[\mu^{-1}]}^\diamond/G(\overline{\mc{O}}). \]

The associated Kodaira-Spencer map at finite level is, by construction, an isomorphism, thus \cref{prop.ks-iso-formal-nbhd} applies so that
\[ \Sh_{K_pK^p}^\lfid = \Sh_{K_pK^p}^\diamond \times_{\Fl_{[\mu^{-1}]}^\diamond/G(\overline{\mc{O}})} \Fl_{[\mu^{-1}]}^\lfid/G(\overline{\mc{O}}). \]

Thus, the infinite level inscribed Hodge period map $\pi_{\Hdg}$ on $\Sh_{K^p}^\lfid$ combined with $\pi_{K_p}^\diamond$ induces a map 
\[ \pi_{K_p}^\lfid: \Sh_{K^p}^\lfid \rightarrow \Sh_{K_pK^p}^\lfid. \]

Arguing as in \cref{remark.local-shimura-finite-level}, $\pi_{K_p}^\lfid$ it is a quasi-torsor for $K_p^\lfid$ and, by applying \cref{prop.b-basic-open-stratum}, in fact a torsor (i.e. surjective) after restriction to $X_{\mbb{Q}_p,\Box}^\lfp$. Comparing to the computations of \cref{lemma.shimura-tangent-computation} and the usual $p$-adic comparisons, we obtain:

\begin{lemma}\label{lemma.torsor-global-shimura}
After restricting to $X_{\mbb{Q}_p,\Box}^\lfp$, the map $\pi_{K_p}^\lfid: \Sh_{K^p}^\lfid \rightarrow \Sh_{K_pK^p}^\lfid$ is a $K_p^\lfid$-torsor.  Moreover, after the same restriction, the $v$-sheafification of the long exact sequence associated to \cref{eq.body-shimura-mod-1} is canonically identified with
\[ 0 \rightarrow  \mf{g}^{\lfid} \xrightarrow{da_e} T_{\Sh_{K^p}^\lfid} \xrightarrow{d\pi_{K_p}^\lfid} T_{\Sh_{K_pK^p}^\lfid} \rightarrow 0. \]
\end{lemma}

\section{Twistors on the relative thickened Fargues--Fontaine curve}\label{s.twistors}

In this section we define our category of twistors and prove the main theorems of the paper (\cref{maintheorem.twistor-correspondence} as \cref{theorem.twistor-functor}, \cref{maintheorem.tangent-bundle-computation} as \cref{theorem.torsor-tangent-bundle-computation}, and \cref{maintheorem.Hodge--Tate-derivative} as \cref{theorem.ht-derivative-computation}). In \cref{ss.mom-revisited} we compare with the construction of \cref{s.moduli-of-mod} in the case of local Shimura varieties and more general flat crystalline local systems. In \cref{ss.global-sv-revisited} we compare with the construction of \cref{s.Inscribed-global-Shimura} in the case of global Shimura varieties (note that the construction of this section does not use Igusa stacks so is unconditional!).  

\subsubsection{Setup}
We work in the inscribed context $(\Spd \mbb{Q}_p, X^\alg_{\mathbb{Q}_p,{\Box}})$. We write $\pi: \mc{X} \rightarrow X^\alg_{\mathbb{Q}_p, \Box}$ and $\iota: X^\alg_{\mathbb{Q}_p, \Box} \rightarrow \mc{X}$ for the natural morphisms. We use the notation $\mathbb{Q}_p^\lfid$, $\mathbb{B}_e, \mathbb{B}^+_\dR, \mathbb{B}_\dR,$ and $\mc{O}$ for period sheaves as in \cref{s.modifications}, and for any of these $\mathbb{B}$ and $\mc{E}$ a vector bundle on $\mc{X}$ over an inscribed $v$-sheaf $\mc{S}$, we write $\mc{E} \boxtimes \mathbb{B}$ for the corresponding $\mathbb{B}$-module of sections over $\Spec \mathbb{B}$.

\subsection{Twistors}
\begin{definition}\label{def.mer-conn-and-twistor} Let $\mc{S}/\Spd \mbb{Q}_p$ be an inscribed $v$-sheaf. A \emph{vector bundle with meromorphic integrable connection on the relative thickened Fargues--Fontaine curve over $\mc{S}$} is a pair $(\mc{E}, e)$ where
\begin{enumerate}
    \item $\mc{E}$ is a vector bundle on $\mc{X}$ over $\mc{S}$, and
    \item $e$ is an integrable connection on $\mc{E}|_{\mc{X}\backslash \infty}$ that is, an isomorphism $e: \mc{E}|_{\mathcal{X}\backslash \infty} \xrightarrow{\sim} \pi^* \iota^* \mc{E}|_{\mathcal{X}\backslash \infty}$.
\end{enumerate}
Given a vector bundle with meromorphic integrable connection on the relative thickened Fargues--Fontaine curve over $\mc{S}$, we write $\nabla: T_{\mc{S}} \rightarrow \mc{E}nd (\mc{E} \boxtimes \mathbb{B}_e)$ for the map such that, for $\gamma \in T_{\mc{S}}(\mc{X})=\mc{S}(\mc{X}[\epsilon]),$
\[ e(\gamma)e(0)^{-1}=\Id + \nabla(t)\epsilon. \]
We say that $\mc{E}$ is a \emph{twistor bundle} if $\nabla$ has a simple pole at $\infty$, i.e., for any $\gamma$ as above, $\nabla(\gamma) \otimes_{\mbb{B}_e} \mbb{B}_\dR$ restricts to a map
\[ \Fil^1\mathbb{B}^+_\dR(\mc{X}[\epsilon]) \cdot \left(\mc{E} \boxtimes \mathbb{B}^+_\dR(\mc{X}[\epsilon])\right) \rightarrow  \mc{E} \boxtimes \mathbb{B}^+_\dR(\mc{X}[\epsilon]). \]
\end{definition}

The category of vector bundles with meromorphic integrable connection on the relative thickened Fargues--Fontaine curve over $\mc{S}$ is naturally a rigid tensor category, and the subcategory of twistor bundles is also a rigid tensor category --- to see that the twistor condition is preserved under tensor products and duals, it suffices to observe 
\[ \nabla_{\mc{E}^*}(t)=\nabla_{\mc{E}}(t)^* \textrm{ and } \nabla_{\mc{E}_1 \otimes \mc{E}_2}=\nabla_{\mc{E}_1}\otimes 1 + 1 \otimes \nabla_{\mc{E}_2}. \]

\subsection{From $\mathbb{Q}_p$-local systems to twistors}
Let $L$ be a $p$-adic field, let $C=\overline{L}^\wedge$, let $B^+_\dR=\mathbb{B}^+_\dR(C)$, and let $Z/L$ be a smooth rigid analytic variety. For $\mathbb{L}$ a $\mathbb{Q}_p$-local system on $Z$, we consider the vector bundle $\mathbb{L} \otimes_{\mathbb{Q}_p^\diamond} \mathcal{O}_{\mc{X}}$ on the relative thickened Fargues--Fontaine curve over $Z^\lfid$. It is equipped with a canonical integrable connection by writing
\[ V \otimes_{\mbb{Q}_p^\diamond} \mathcal{O}_{\mc{X}}= (V \otimes_{\mbb{Q}_p^\diamond} \mathcal{O}_{X}) \otimes_{\mathcal{O}_X} \mathcal{O}_{\mathcal{X}}. \]
In particular, any modification at $\infty$ of $V \otimes_{\mbb{Q}_p^\diamond} \mathcal{O}_{\mc{X}}$ is equipped with a canonical meromorphic integrable connection; we will construct an associated twistor bundle $\Tw(V)$ in this way.  

To obtain the desired modification, we recall that Liu and Zhu \cite{LiuZhu.RigidityAndARiemannHilbertCorrespondenceForpAdicLocalSystems} have constructed an exact tensor functor $\RH$ from $\mathbb{Q}_p$-local systems on $Z$ to $\Gal(\overline{L}/L)$-equivariant $t$-connections on $Z_{\Spaf(\mathbb{B}^+_\dR)}$. We may compose this with the functor $\mathbb{M}$ of \cref{ss.liu-zhu-period-map}, and it is immediate from the constructions that 
\[ (\mathbb{M} \circ \RH) \otimes_{\mathbb{B}^+_\dR} \mathbb{B}_\dR = V \otimes_{\mbb{Q}_p^\diamond} \mathbb{B}_\dR =  (V \otimes_{\mbb{Q}_p^\diamond} \mathcal{O}_{\mc{X}}) \boxtimes \mathbb{B}_\dR. \]
Thus we may define $\Tw(V)$ by taking the modification of $V \otimes_{\mbb{Q}_p^\diamond} \mathcal{O}_{\mc{X}}$ by $\mathbb{M}(\RH(V))$.

\begin{theorem}\label{theorem.twistor-functor}
    The assignment $V \mapsto \Tw(V)$ is a fully faithful exact tensor functor from $\mathbb{Q}_p$-local systems on $Z$ to twistors on the relative thickened Fargues--Fontaine curve over $Z^\lfid$. 
\end{theorem}
\begin{proof}
For $V$ a $\mathbb{Q}_p$-local system, $\Tw(V)$ is a twistor because the associated $\nabla \otimes_{\mbb{B}_e}\mbb{B}_\dR$ as in \cref{def.mer-conn-and-twistor} comes from the $t$-connection on $\RH(V)$. That $V \mapsto \Tw(V)$ is an exact tensor functor follows from the same property for $V \mapsto V \otimes_{\mbb{Q}_p^\diamond} \mathcal{O}_{\mc{X}}$ and for $V \mapsto \RH(V)$. Finally, the full-faithfulness follows from the full-faithfulness of $\mathbb{L} \mapsto \mathbb{L} \otimes_{\mathbb{Q}_p^\diamond} \mathcal{O}_X$, since this functor is recovered by restricting to $Z^\diamond$ (faithfulness) and all morphisms are uniquely determined by their action here using the integrable connection (fullness). 
\end{proof}

\begin{remark}\label{Remark.FarguesLiuZhuDiscussion}
    Fargues and Liu-Zhu \cite[p.327]{LiuZhu.RigidityAndARiemannHilbertCorrespondenceForpAdicLocalSystems} have conjectured the existence of a functor from $\mathbb{Q}_p$-local systems to $p$-adic twistors. Since they did not give a precise definition of a $p$-adic twistor (note, in particular, that there is no base in the fiber product $X \times \FF$ in the discussion on \cite[p.327]{LiuZhu.RigidityAndARiemannHilbertCorrespondenceForpAdicLocalSystems}), we are free to claim that \cref{theorem.twistor-functor} proves this conjecture! This is not really an honest claim, however: by analogy with the $p$-adic Simpson correspondence, what we have constructed here is more like the canonical twisted Higgs field on a $v$-vector bundle over $Z^\diamond$ than it is like its $p$-adic Simpson correspondent, which is an \'{e}tale vector bundle on $Z$ equipped with a Higgs field. However, unlike in the $p$-adic Simpson correspondence, it is not clear to us what category should appear on the other side. 
\end{remark}

\begin{remark}[Extensions of \cref{theorem.twistor-functor}]\label{remark.extensions-of-twistor}
We recall that \'{e}tale $\mathbb{Q}_p$-local systems on $Z$ are equivalent, via the map $\mbb{L} \mapsto \mbb{L}  \otimes_{\mbb{Q}_p^\diamond} \mathcal{O}_X$, to vector bundles on the relative Fargues--Fontaine curve over $Z^\diamond$ that are pointwise semistable of slope zero. More generally, given a vector bundle $\mc{V}$ on the relative Fargues--Fontaine curve over $Z^\diamond$, we say $\mc{V}$ is de Rham if $\mc{V} \boxtimes \mathbb{B}^+_\dR$ is associated to a filtered vector bundle with integrable connection satisfying Griffiths transversality $(V,\nabla, \Fil)$ on $Z$ as in \cite[\S6]{Scholze.pAdicHodgeTheoryForRigidAnalyticVarieties} (i.e. if $\mc{V} \boxtimes \mathbb{B}^+_\dR$ is equal to the restriction to $Z^\diamond$ of the bundle $\mathbb{M}( (V, \nabla, \Fil))$ of \cref{ss.liu-zhu-period-map}). Noting that in this case the desired bundle $\mathbb{M}$ over $Z^\lfid$ can be constructed from the filtered vector bundle with integrable connection as in \cref{ss.latice-Hodge-period-map}, one obtains an exact tensor functor $\mc{V} \mapsto \Tw(\mc{V})$ from de Rham vector bundles with integrable connection to twistors on the relative thickened Fargues--Fontaine curve over $S^\lfid$ by modifying $\mc{V} \otimes_{\mathcal{O}_X} \mathcal{O}_{\mathcal{X}}$ by $\mathbb{M}$. 

Note that this extends the restriction of the functor of \cref{theorem.twistor-functor} to de Rham $\mathbb{Q}_p$-local systems, but it is not strictly more general than \cref{theorem.twistor-functor}, since not all $\mathbb{Q}_p$-local systems are de Rham. In general, for $L$ any non-archimedean extension over $\mathbb{Q}_p$ and $Z/L$ a rigid analytic variety equipped with an equivariant lift to $\Spaf \mathbb{B}^+_\dR(\overline{L}^\wedge)$, we expect to be able to refine the construction of $\RH$ in \cite{LiuZhu.RigidityAndARiemannHilbertCorrespondenceForpAdicLocalSystems} in order to obtain an exact tensor functor 
\[ \Tw: \parbox{16em}{Nilpotent vector bundles on the relative Fargues--Fontaine curve over $Z^\diamond$} \longrightarrow \parbox{16em}{Twistors on the relative thickened Fargues--Fontaine curve over $Z^\lfid$.}\]
Here the nilpotence condition is that the canonical Higgs field / geometric Sen morphism on $\mc{V} \boxtimes \mathcal{O}$ is nilpotent. In particular, this nilpotence is automatic when $L$ is a $p$-adic field (which is why no analog appears as a condition in \cref{theorem.twistor-functor}), and holds more generally for any $\mathbb{Q}_p$-local systems coming from the cohomology of a smooth proper family of rigid analytic varieties over $Z$ (by spreading out). We note that even this nilpotence condition is likely not necessary (but with the nilpotence condition it should be a relatively straightforward generalization of our construction here and \cite{LiuZhu.RigidityAndARiemannHilbertCorrespondenceForpAdicLocalSystems}).  
\end{remark}

\subsection{Inscribed $p$-adic Lie group torsors}\label{ss.inscribed-torsors}
Suppose $G/\mathbb{Q}_p$ is a connected linear algebraic group and $K \leq G(\mathbb{Q}_p)$ is an open subgroup. Let $L$ be a $p$-adic field, and let $Z/L$ be a smooth rigid analytic variety. Let $\tilde{Z}/Z^\diamond$ be a $K$-torsor. In this subsection we use the functor $\Tw$ in order to define a natural non-trivial inscription $\tilde{Z}^\lfid$ on $\tilde{Z}$.

To obtain this inscription, we first note that $\tilde{Z}$ gives rise to an exact tensor functor from $\Rep\, G$ to $\mathbb{Q}_p$-local systems on $Z^\diamond$ by push-out:
\[ \omega_{\et}: V \mapsto \tilde{Z} \times^{K^\diamond} V^\diamond.\]
We write $\omega_\Tw$ for the composition of $\Tw \circ \omega_\et$ with the forgetful functor to vector bundles on the relative twisted Fargues--Fontaine curve over $Z^\lfid$ (i.e. forgetting the integrable meromorphic connection). We then define 
\[ \tilde{Z}^\lfid := \tilde{Z} \times_{\ul{\Isom}(\omega_\triv, \omega_\et)} \ul{\Isom}(\omega_\triv \otimes \mathcal{O}_{\mc{X}}, \omega_\Tw) \]
where here we have used the identity
\[ \overline{\ul{\Isom}(\omega_\triv \otimes \mathcal{O}_{\mc{X}}, \omega_\Tw)}=\ul{\Isom}(\omega_{\triv} \otimes \mathcal{O}_{X}, \omega_\et \otimes \mc{O}_X)=\ul{\Isom}(\omega_{\triv}, \omega_{\et}). \]

Note that there is a natural action of $K^\lfid = K^\diamond \times_{G(\mathbb{Q}_p^\diamond)} G(\mathbb{Q}_p^\lfid)$ induced by the action of $K^\diamond $ on $\tilde{Z}$ and the action of $G(\mathbb{Q}_p^\lfid)$ as the automorphism group of $\omega_\triv \otimes \mathcal{O}_{\mc{X}}$. 

\begin{theorem}\label{theorem.ztilde-torsor}
$\tilde{Z}^\lfid$ is a $K^\lfid$ quasi-torsor over $Z^\lfid$, and it is a torsor after restriction of the test objects to $X^\lfp$.     
\end{theorem}
\begin{proof}
    We can immediately reduce to the case $K=G(\mathbb{Q}_p)$, in which case we are simply studying 
    \[ \ul{\Isom}(\omega_\triv \otimes \mathcal{O}_{\mc{X}}, \omega_\Tw).\]
    By construction it is a quasi-torsor; it remains to show that it is actually a torsor after restriction to $X^\lfp$. But this is immediate from \cref{prop.b-basic-open-stratum} (with $b=1$) and the torsor property for the underlying map of $v$-sheaves $\tilde{Z} \rightarrow Z^\diamond$.
\end{proof}

\subsubsection{} Let $\mathfrak{k}=\Lie K=\Lie G(\mathbb{Q}_p)$, a $\mathbb{Q}_p$-vector space. Writing $\rho^\lfid: \tilde{Z}^\lfid \rightarrow Z^\lfid$, it follows from \cref{theorem.ztilde-torsor} that we obtain an exact sequence of $\mathbb{Q}_p^\lfid$-modules over $\tilde{Z}^\lfid$,
\begin{equation}\label{eq.torsor-exact-seq} 0 \rightarrow \mf{k} \otimes_{\mathbb{Q}_p} \mathbb{Q}_p^\lfid \xrightarrow{da_e} T_{\tilde{Z}^\lfid} \xrightarrow{d\pi^\lfid} {\rho^{\lfid}}^* T_{Z^\lfid} \rightarrow 0. \end{equation}

The following result identifies this extension. Recalling that $\mc{O}_{\mc{X}}(\infty)$ denotes the sheaf on $\mc{X}$ of functions that are holomorphic on $\mc{X}\backslash \infty$ and have at most a simple pole at $\infty$, let 
\[ r: \mf{k} \otimes_{\mathbb{Q}_p} \BC(\mc{O}_{\mc{X}}(\infty)) \rightarrow \mf{k} \otimes_{\mathbb{Q}_p} \mbb{B}_\dR\]
denote the  map by restriction of global sections to $\infty$, pulled back to a map of $\mathbb{Q}_p^\lfid$-modules over $\tilde{Z}^\lfid$, and let $\overline{r}$ denote the induced map to $\mf{k} \otimes_{\mathbb{Q}_p} \frac{\mbb{B}_\dR}{\mbb{B}^+_\dR}$. 
Note that we also have a canonical map of $\mathbb{Q}_p^\lfid$-modules over $\tilde{Z}^\lfid$ 
\[ \tilde{\kappa}:={\rho^\lfid}^* d\pi_{\LZ}: {\rho^\lfid}^* T_{Z^\lfid} \rightarrow \mf{k}\otimes_{\mathbb{Q}_p} \mathbb{B}_\dR/\mathbb{B}^+_\dR. \]   
Both maps $\tilde{\kappa}$ and $\overline{r}$ factor through $\mf{k} \otimes \mathcal{O}\{-1\}$, where 
\[ \mathcal{O}\{-1\}:\gr^{-1}\mathbb{B}_\dR=\Fil^{-1}\mathbb{B}_\dR/\Fil^0\mathbb{B}_\dR =\Fil^{-1}\mathbb{B}_\dR/\mathbb{B}^+_\dR. \]
We define $V_{\overline{r},\tilde{\kappa}}$ as the fiber product

\[\begin{tikzcd}
	{V_{\overline{r}, \tilde{\kappa}}} & {{\rho^\lfid}^* T_{Z^\lfid}} \\
	\\
	{\mf{k} \otimes_{\mathbb{Q}_p} \BC(\mc{O}_{\mc{X}}(\infty))} & {\mf{k}\otimes_{\mathbb{Q}_p} \mathcal{O}\{1\}}
	\arrow[from=1-1, to=1-2]
	\arrow[from=1-1, to=3-1]
	\arrow["\lrcorner"{anchor=center, pos=0.125}, draw=none, from=1-1, to=3-2]
	\arrow["{\tilde{\kappa}}", from=1-2, to=3-2]
	\arrow["\overline{r}", from=3-1, to=3-2]
\end{tikzcd}\]
In particular, $V_{\overline{r}, \tilde{\kappa}}$ sits in the  natural exact sequence
\begin{equation}\label{eq.putative-tangent-bundle} 0 \rightarrow \mf{k} \otimes_{\mathbb{Q}_p} \mathbb{Q}_p^\lfid \rightarrow V_{\overline{r},\tilde{\kappa}} \rightarrow {\rho^{\lfid}}^* T_{Z^\lfid} \rightarrow 0. \end{equation}

\begin{theorem}\label{theorem.torsor-tangent-bundle-computation}
The two exact sequences \cref{eq.torsor-exact-seq} and \cref{eq.putative-tangent-bundle} are canonically identified.
\end{theorem}
\begin{proof}
    It suffices to produce a map from \cref{eq.putative-tangent-bundle} to \cref{eq.torsor-exact-seq} that is the identity on $\mf{k} \otimes_{\mathbb{Q}_p} \mathbb{Q}_p^\lfid$ and on ${\rho_{\lfid}}^* T_{Z^\lfid}$. To that end, suppose given an $\mathcal{X}$-point 
    \[ (a,\gamma) \in V_{\overline{r},\tilde{\kappa}}(\mc{X}) \subseteq \mf{k}\otimes_{\mbb{Q}_p} \BC(\mc{O}_{\mc{X}}(\infty))(\mc{X}) \times {\rho^\lfid}^*T_{Z^\lfid}(\mc{X})\]
    lying above $\tilde{z}$ in $\tilde{Z}^\lfid(\mc{X})$. Write $z$ for the image of $\tilde{z}$ in $Z^\lfid(\mc{X})$ so that $\gamma$ can be viewed as a tangent vector to $z$, i.e. as a point above $z$ in $Z^\lfid(\mc{X}[\epsilon])$.  
    
    We write $\gamma_0$ for the zero tangent vector at the point $\tilde{z}$. Then, we construct a point of $\tilde{Z}^\lfid(\mathcal{X}[\epsilon])$ above $\gamma \in Z^\lfid(\mathcal{X}[\epsilon])$ by precomposing the given trivialization of $\gamma_0^* \omega_{\Tw}$ (corresponding to the constant extension of the trivialization parameterized by $\tilde{z}$), after restriction to $\mathcal{X}\backslash\infty[\epsilon]$, with multiplication by $1+a\epsilon \in G(\mathbb{B}_e[\epsilon])$. The compatibility between $a$ and $\gamma$ in the fiber product defining $V_{\overline{r}, \tilde{\kappa}}$ is precisely equivalent to the condition that this extends to a trivialization of $\gamma^*\omega_\Tw$ on $\mathcal{X}[\epsilon]$. 
\end{proof}

\begin{remark}
    Note that if $d\pi_{\LZ}$ is injective, then the fiber product that describes $T_{\tilde{Z}^\lfid}$ in \cref{theorem.torsor-tangent-bundle-computation} is $\BC(\mathcal{E})$, where $\mathcal{E}$ is the vector bundle on the relative thickened Fargues--Fontaine curve over $\tilde{Z}^\lfid$ obtain as the minuscule effective modification of $\mf{k} \otimes_{\mathbb{Q}_p} \mathcal{O}_{\mathcal{X}}$ by the local direct summand $\tilde{\kappa}$, i.e. by changing the holomorphic functions at $\infty$ to be those sections of $\mf{k} \otimes_{\mathbb{Q}_p} \mathcal{O}_{\mathcal{X}}(\infty)$ that evaluate at $\infty$ into $T_{\tilde{Z}^\lfid} \subseteq \mf{k} \otimes_{\mbb{Q}_p} \mathcal{O}\{-1\}$ (containment via $\tilde{\kappa}$). 
\end{remark}

\subsection{Hodge--Tate and lattice Hodge--Tate period maps}
Continuing with the notation of \cref{ss.inscribed-torsors}, let us assume $\tilde{Z}$ is furthermore a de Rham torsor, i.e. that $\omega_{\et}$ factors through the full subcategory of de Rham local systems (it suffices to check this condition on one faithful representation and, by \cite[Theorem 1.3]{LiuZhu.RigidityAndARiemannHilbertCorrespondenceForpAdicLocalSystems}, even at a single point in each connected component). Let $\omega_{\nabla}^{\Fil}$ denote the associated filtered $G$-bundle with integrable connection satisfying Griffiths transversality.  

In this setting, we write $\omega^+_\dR$ for $\mathbb{M}_0\circ \omega_{\nabla}^\Fil$. By construction, we have an isomorphism 
\[ \omega_\Tw \boxtimes \mathbb{B}_\dR = (\mathbb{M} \circ \omega_{\nabla}^\Fil) \otimes_{\mathbb{B}^+_\dR} \mathbb{B}_\dR = (\mathbb{M}_0 \circ \omega_{\nabla}^\Fil) \otimes_{\mathbb{B^+_\dR}} \mathbb{B}_\dR.\]
Thus on $\tilde{Z}^\lfid$ we obtain a lattice Hodge--Tate period map measuring the lattice  $(\mathbb{M}_0 \circ \omega_{\nabla}^\Fil)$ against the trivialization $\omega_\Tw \boxtimes \mathbb{B}_\dR$ induced by the trivialization of $\omega_\Tw$ over $\tilde{Z}^\lfid$,  
\[ \pi_{\HT}^+ : \tilde{Z}^\lfid \rightarrow \Gr_{G}. \]

Let $\mf{k}^+_\dR:=\mathbb{M}_0(\omega_{\nabla}^\Fil(\mf{g}))$. The derivative of $\pi_\HT^+$ is a map
\[ d\pi_{\HT}^+: T_{\tilde{Z}^\lfid} \rightarrow {\pi_{\HT}^+}^* T_{\Gr_G} =\frac{\mf{k} \otimes_{\mbb{Q}_p}\mathbb{B}_\dR}{\mf{k}^+_\dR} = \mf{k}^+_\dR \otimes_{\mathbb{B}^+_\dR} \frac{\mathbb{B}_\dR}{\mathbb{B}^+_\dR}. \]

Unwinding the definitions, we find:
\begin{theorem}\label{theorem.ht-derivative-computation}
Under the identification of $T_{\tilde{Z}^\lfid}=V_{\overline{r},\tilde{\kappa}}$ of \cref{theorem.torsor-tangent-bundle-computation}, $d\pi_\HT^+$ is the composition 
\[ V_{\overline{r},\tilde{\kappa}} \rightarrow \mf{k} \otimes_{\mbb{Q}_p} \BC(\mathcal{O}_{\mc{X}}(\infty)) \xrightarrow{\mr{id} \otimes r} \mf{k} \otimes_{\mbb{Q}_p} \mathbb{B}_\dR \twoheadrightarrow \frac{\mf{k} \otimes_{\mbb{Q}_p} \mathbb{B}_\dR}{\mf{k}^+_\dR}. \]
\end{theorem}

\subsection{Moduli of modifications revisited}\label{ss.mom-revisited}

 Consider a space $\Gr_{G}^{b_1 \rightarrow [b_2]}$ as in \cref{ss.inscribed-generalized-Newton-strata} (for $E=\mathbb{Q}_p$). By construction, the universal modification $\mc{E}$ of $\mc{E}_{b_1}$ carries a meromorphic integrable connection induced by 
\[ \mc{E}|_{\mathcal{X}\backslash\infty}=\mc{E}_{b_1}|_{\mathcal{X}\backslash\infty}, \]
and, as in \cref{s.moduli-of-mod}, the infinite level space $\M_{b_1 \rightarrow b_2}$ is the torsor of trivializations to $\mathcal{E}_{b_2}$ for the underlying $G$-bundle $\mc{E}$. 

In particular, suppose $[\mu]$ is a conjugacy class of characters of $G_{\qpbar}$, $L/\mathbb{Q}_p([\mu])$ is a $p$-adic field, and $Z^\diamond \rightarrow \Gr_{[\mu]}^{b \rightarrow [1]}$ for $Z/L$ a smooth rigid analytic variety --- note that, by \cite[Theorem 5.0.4-(4)]{HoweKlevdal.AdmissiblePairsAndpAdicHodgeStructuresIITheBiAnalyticAxLindemannTheorem}, such maps are determined by the induced map to $\Fl_{[\mu^{-1}]}^\diamond$, which corresponds to a map of rigid analytic varieties $Z \rightarrow \Fl_{[\mu^{-1}]}$ satisfying Griffiths transversality for the trivial connection on the trivial $G$-torsor; thus we get an induced map as in \cref{ss.latice-Hodge-period-map} 
\[ f: Z^\lfid \rightarrow \Gr_{[\mu]} \]
that, if we restrict our test objects to $X^\lfp$, factors through $\Gr_{[\mu]}^{b \rightarrow [1]}$ (by \cref{prop.b-basic-open-stratum} since it factors through this locus on the underlying $v$-sheaves). 

In particular, the meromorphic integrable connection on $f^*\mc{E}$ is a $G$-twistor, and it is the $G$-twistor associated to the flat crystalline $G(\mathbb{Q}_p)$-torsor $f^*\overline{\M_{b,[\mu]}} / Z^\diamond$. In this case, the computation of \cref{theorem.torsor-tangent-bundle-computation} can be identified with the natural computation of the tangent bundle of the fiber product $\M_{b \rightarrow [1]} \times_{\Gr_{[\mu]}} Z^\lfid$, noting that in this case $df=d\pi_{\LZ}=d\pi^+_\Hdg$.  

When the cocharacter $[\mu]$ is minuscule (the local Shimura case), this discussion applies to the whole space $\overline{\Gr_{[\mu]}^{b \rightarrow [1]}}$, whose associated rigid analytic variety is the $b$-admissible locus $\Fl_{[\mu^{-1}]}^{b-\adm}$, giving an alternate perspective on the inscribed structure on infinite level local Shimura varieties in \cref{ss.main-results-mcls}. 

\subsection{Global Shimura varieties revisited}\label{ss.global-sv-revisited}
Let $\gx$ be a Shimura datum and fix a completion $L$ of the reflex field at a place above $p$. Let $K^p\leq G(\mathbb{A}_f^{(p)})$ be a compact open subgroup. For $K_p \leq G(\mathbb{Q}_p)$ compact open such that $K_pK^p$ is neat, we write $\Sh_{K_pK^p}$ for the associated rigid analytic Shimura variety over $L$. Then, for $G^c$ the quotient of $G$ by its maximal $\mathbb{Q}$-anisotropic $\mathbb{R}$-split central torus and $\overline{K}_p$ the image of $K_p$ in $G^c(\mathbb{Q}_p)$, by the usual constructions we have a de Rham $\overline{K_p}$-torsor $\tilde{\Sh}_{K^p} / \Sh_{K_pK^p}^\diamond$. We thus obtain, by the construction of \cref{ss.inscribed-torsors}, an inscribed infinite level Shimura variety $\tilde{\Sh}_{K^p}^\lfid$. 

In particular, in the setting of \cref{s.Inscribed-global-Shimura} (where the conditions imply $G^c=G$), we have two a priori distint constructions of an inscription on the infinite level Shimura variety; we write $(\Sh_{K_pK^p}^\lfid)'$ for the inscription constructed in \cref{s.Inscribed-global-Shimura}. Since both are $K_p^\lfid$-torsors over $\Sh_{K_pK^p}^\lfid$, to show they agree it suffices to write down a $K_p^\lfid$-equivariant map $\Sh_{K_p}^\lfid \rightarrow (\Sh_{K_pK^p}^\lfid)'$ over $\Sh_{K_pK^p}^\lfid$. This is straightforward using the definition of $(\Sh_{K_pK^p}^\lfid)'$ as a fiber product: since the Igusa stack is trivially inscribed, we obtain a map to the Igusa stack from the underlying $v$-sheaf, and then we have already constructed our map $\pi_\HT$ to $\Fl_{[\mu]}$; the identification of the modified bundle with the one pulled back from the Igusa stack is obtained from the integrable connection on $\mathbb{M}_0 \circ \omega_{\nabla}^\Fil$.

\bibliographystyle{plain}
\bibliography{references, preprints}

\end{document}